\documentclass[a4paper,11pt]{book}
\usepackage[latin1]{inputenc}
\usepackage{color}

\usepackage{anysize}

\usepackage{makeidx}
\makeindex

\marginsize{3.5cm}{1.7cm}{2.5cm}{2cm}

\usepackage{fancyhdr}
 \newcommand{\helv}{
   \fontsize{11}{11}\selectfont}
\usepackage[Sonny]{fncychap}

\ChNameVar{\Large\sf}

\ChNumVar{\Huge}

\ChTitleVar{\Large\sf}

\ChRuleWidth{0.5pt}

\ChNameUpperCase

\ChNameVar{\huge\rm}

\usepackage{amsxtra}
\usepackage{amssymb, amscd}
\usepackage{graphicx}
\usepackage{amsmath}
\usepackage{amsthm}

\newcommand\so{\mathfrak{{so}}}

\newcommand\lieg{\mathfrak{ g}}
\newcommand\liek{\mathfrak{ k}}
\newcommand\lieh{\mathfrak{ h}}

\newcommand\lieso{\mathfrak{ so}}

\newcommand\CC{\mathbb C}

\newcommand\RR{\mathbb R}
\newcommand\ZZ{\mathbb Z}
\newcommand\NN{\mathbb N}
\newcommand\HH{\mathbb H}

\newcommand\mm{\mathbf m}
\newcommand\rr{\mathbf r}
\newcommand\ee{{\mathbf e}}
\newcommand\ttt{{\mathbf t}}
\newcommand\sss{{\mathbf s}}

\newcommand\SL{{\mathrm{SL}}}

\newcommand\GL{{\mathrm{GL}}}
\newcommand\SO{{\mathrm{SO}}}
\newcommand\OO{{\mathrm{O}}}
\newcommand\SU{{\mathrm{SU}}}
\newcommand\U{{\mathrm{U}}}
\newcommand\diag{{\mathrm{diag}}}
\newcommand\B{{\sf B}}

\newcommand\End{\operatorname{End}}

\newcommand\tr{\operatorname{tr}}

\newcommand\ad{\operatorname{ad}}

\theoremstyle{remark}

\newcommand\matr[4]{\left( {\hfill #1\@@atop\hfill #3}{\hfill
#2\@@atop\hfill #4}\right)}
\newcommand\matl[4]{\left( { #1\@@atop #3}{ #2\@@atop\hfill #4}\right)}
\makeatother

\newcommand\z{{\mathbb Z}}

\newcommand\vz[1]{\mathchoice{\left\{ #1 \right\}}{\left\{ #1
\right\}}{\{ #1 \}}{\{ #1 \}}}

\newcommand\vzm[2]{\mathchoice{\left\{\, #1 : #2 \,\right\}}{\{\, #1
:\allowbreak #2 \,\}}{\{ #1 :\allowbreak #2 \}}{\{ #1 :\allowbreak #2
\}}}

\newcommand\lw[1]{\hbox{}_{#1}\!}

\theoremstyle{plain}

\newcommand\g{\gamma}
\newcommand\Dt{\Delta}

\newcommand\x{\xi}
\newcommand\p{\pi}

\newcommand\Ph{\Phi}
\newcommand\ph{\varphi}

\def\gr{\mathop{\mbox{\rm gr}}}

\def\lim{\mathop{\mbox{\rm l\'{\i}m}}}

\def\sen{\mathop{\operator@font sen}\nolimits}
\def\arcsen{\mathop{\operator@font arcsen}\nolimits}
\def\senh{\mathop{\operator@font senh}\nolimits}
\def\max{\mathop{\mbox{\rm m\'ax}}}
\def\min{\mathop{\mbox{\rm m\'{\i}n}}}

\makeatletter 

\newcommand\widearray[1]{\renewcommand\arraystretch{1.4}
\begin{array}{#1}}

\numberwithin{equation}{chapter}

\theoremstyle{plain}
  \newtheorem{thm}{Teorema}[chapter]
\newtheorem{lem}[thm]{Lema}
\newtheorem{prop}[thm]{Proposici\'on}
\newtheorem{cor}[thm]{Corolario}

\newtheorem{defn}[thm]{Definici\'on}

\newtheorem{remark}[thm]{Observaci\'on}

\begin{document}




\begin{center}
\thispagestyle{empty}

\

{\Huge {\sc Funciones Esf\'ericas Matriciales }
} 
\vspace{.75cm}

{\Huge {\sc Asociadas a las Esferas y a}
}
\vspace{.75cm}
 
{\Huge {\sc los Espacios Proyectivos Reales  }
} 

\vspace{3cm}

{\Large\sc Ignacio Nahuel Zurrián}\\
\vspace{6cm}



{\Large { Presentado ante la Facultad de Matemática, Astronomía y Física como parte de los requerimientos para la obtención del grado de Doctor en Matemática de la
}}\\

\vspace{1cm}

{\Large {\sc Universidad Nacional de Córdoba }}\\

\vspace{1.5cm}

{\Large {\sc Córdoba}\hspace{6cm}{\sc Argentina}}\\
\ \\
{\Large \sc Marzo de 2013}\ 

\vspace{1cm}

{\Large\sc \copyright Fa.M.A.F. - U.N.C.}

\vspace{1.5cm}

{\Large\sc Director: Dr. Juan Alfredo Tirao}

\end{center}

\newpage
\mbox{}
\thispagestyle{empty} 

\thispagestyle{empty}

\begin{center}
\bf Resumen
\end{center}

En este trabajo, primero, se determinan todas las funciones esféricas irreducibles
$\Phi$ de cualquier $K $-tipo asociadas al par
$(G,K)=(\SO(4),\SO(3))$. Para esto asociamos a
$\Phi$ una función vectorial $H=H(u)$ de una variable real $u$,
la cual es analítica en $u=0$ y cuyas componentes son soluciones de dos sistemas acoplados de ecuaciones diferenciales ordinarias. A través de una apropiada conjugación que involucra a los polinomios de Hahn conseguimos desacoplar uno de estos sistemas, que luego llevamos a un sistema de ecuaciones diferenciales hipergeométricas.
Encontrando entonces como solución la función vectorial $P=P(u)$, cuyas entradas son múltiplos de polinomios de Gegenbauer.
Posteriormente, identificamos aquellas soluciones simultáneas y usamos la teoría de representaciones de $\SO(4)$ para caracterizar todas las funciones esféricas irreducibles.
Las funciones $P=P(u)$ correspondientes a las funciones esféricas irreducibles de un $K$-tipo fijo $\pi_\ell$ son cuidadosamente empaquetadas  en una sucesión de polinomios matriciales $\{P_w\}_{w\ge0}$ de tamaño $(\ell+1)\times(\ell+1)$. Finalmente probamos que  $\widetilde P_w={P_0}^{-1}P_w$ es una sucesión de polinomios ortogonales con respecto a un peso matricial $W$. Más aún, probamos que $W$ admite un operador diferencial simétrico de segundo orden $\widetilde D$ y un operador diferencial simétrico de primer orden $\widetilde E$.

Luego se establece una directa relación entre las funciones esféricas de la esfera  $n$-di\-men\-sional $S^n\simeq\SO(n+1)/\SO(n)$ y las  funciones esféricas del  espacio proyectivo real $n$-dimensional $P^n(\mathbb{R})\simeq\SO(n+1)/\mathrm{O}(n)$. Precisamente, para $n$ impar una función en $\SO(n+1)$
es una función esférica irreducible de algún tipo $\pi\in\hat\SO(n)$ si y solo si  es una función esférica irreducible de algún tipo $\gamma\in\hat {\mathrm{O}}(n)$. Cuando $n$ es par esto también es cierto para ciertos tipos, y en los otros casos exhibimos una clara correspondencia entre las  funciones esféricas irreducibles de ambos pares  $(\SO(n+1),\SO(n))$ y $(\SO(n+1),\mathrm{O}(n))$. Entonces, encontrar todas las funciones esféricas de un par es equivalente a hacer lo mismo con  el otro.

Finalmente, estudiamos las funciones esféricas de ciertos tipos de la esfera $n$-dimensional $S^{n}\simeq \SO(n+1)/\SO(n)$, para cualquier $n$.
Más precisamente, explicitamos todas las funciones esféricas cuyas función asociada $H$ es escalar, esto incluye a las de tipo  trivial, y luego estudiamos todas las de tipo fundamental, describiéndolas  en  términos de funciones hipergeométricas matriciales $_2\!F_1$. Para esto trabajamos con las realizaciones explícitas de las representaciones fundamentales del grupo especial ortogonal real.
 Posteriormente construimos para cada tipo fundamental una sucesión de polinomios ortogonales con respecto a un peso $W$, las cuales están asociadas a las funciones esféricas. Y probamos que, para cualquier $n$, $W$ admite un operador diferencial simétrico de segundo orden.

\

\



\newpage

\thispagestyle{empty}

\mbox{}
\newpage

\thispagestyle{empty}

\begin{center}
\bf Abstract
\end{center}

In this work we start by determining all irreducible spherical functions
$\Phi$ of any $K $-type associated to the pair
$(G,K)=(\SO(4),\SO(3))$. This is accomplished by associating to
$\Phi$ a vector valued function $H=H(u)$ of a real variable $u$,
which is analytic at $u=0$ and whose components are solutions of two
coupled systems of ordinary differential equations. By an
appropriate conjugation involving Hahn polynomials we uncouple one
of the systems. Then this is taken to an uncoupled system of hypergeometric equations,
leading to a vector valued solution $P=P(u)$, whose entries are Gegenbauer's polynomials.
Afterward, we identify those simultaneous solutions and use the representation theory of $\SO(4)$ to characterize all irreducible spherical functions.
The functions $P=P(u)$ corresponding to the irreducible spherical functions of a fixed $K$-type $\pi_\ell$ are appropriately packaged into a sequence of matrix valued polynomials $(P_w)_{w\ge0}$ of size $(\ell+1)\times(\ell+1)$. Finally we prove that $\widetilde P_w={P_0}^{-1}P_w$ is a sequence of matrix orthogonal polynomials with respect to a weight matrix $W$. Moreover, we show  that $W$ admits a second order symmetric hypergeometric operator $\widetilde D$ and a first order symmetric differential operator $\widetilde E$.

Later, we establish a direct relationship between the spherical functions of the  $n$-di\-men\-sional sphere $S^n\simeq\SO(n+1)/\SO(n)$ and the spherical functions of the $n$-dimensional  real  projective space $P^n(\mathbb{R})\simeq\SO(n+1)/\mathrm{O}(n)$. Precisely, for $n$ odd a function on $\SO(n+1)$
is an irreducible spherical function of some type $\pi\in\hat\SO(n)$ if and only if  it is an irreducible spherical function of some type $\gamma\in\hat {\mathrm{O}}(n)$. When $n$ is even this is also true for certain types, and in the rest of the cases we exhibit a clear correspondence between the irreducible spherical function
of both pairs  $(\SO(n+1),\SO(n))$ and $(\SO(n+1),\mathrm{O}(n))$. Concluding that to find all the spherical functions of one of these pairs is equivalent to do the same it with the other.

Finally, we study the spherical functions of certain types of the $n$-dimensional  sphere $S^{n}\simeq \SO(n+1)/\SO(n)$, for any  $n$.
More precisely, we give explicitly all the spherical functions whose associated functions $H$ are scalar valued, including those of trivial type, and then we study the irreducible spherical functions of fundamental type, describing them in terms of matrix valued hypergeometric functions $_2\!F_1$. To do this we worked with the explicit realizations of the fundamental representations  of the real special ortogonal group. Thereafter, for every fundamental type we build a sequence of ortogonal matrix valued polynomials with respect to a weight $W$, which are associated to the spherical functions. We also prove that, for any $n$, $W$ admits a second order symmetric differential operator.

\

\


\newpage
\mbox{}
\thispagestyle{empty}

%

\thispagestyle{empty}

\

{\it
Inclemente fortuna la que te trae a mis hojas, que soy árbol sin flores, de cuatro inviernos mis años.
Y en este  mar no alcanzarás naufragio, pues ni cobija todos los tiempos y sus trazos,  ni sus tintas tan profundas son. Sin embargo evidencia con negros y blancos una parte de esta quieta carrera de colores tantos.
}

\

\

A Juan Alfredo Tirao, ser su discípulo es una gran honra y una enorme responsabilidad.

A Inés Pacharoni, sin su dirección y apoyo las cosas no serían lo que son, sin su fuerza tampoco.

A Alberto Grünbaum, que cruzó los paralelos para regalarme su lección.

A Roberto Miatello, mi profesor primero.

A Pablo Román, camarada del más chico.

A José Vargas Soria, ejemplo digno a seguir y difícil de alcanzar.

A Elva Josefina Deabate, madre de todas mis cosas.

A Ricardo Alberto Zurrián,  mi admiración nunca será suficientemente justa, mi amor sí.

A Ricardo Alberto Zurrián, el invencible, que hizo su sendero y enseña cómo es andar la vida sin rendirse jamás.

A mis amigos, maestros de mi vida que alientan mis horas. Al que me compartió la matemática, al que me mostró el romanticismo, al que me enseñó la música, al que me enseñó a andar en la calle, al que me enseñó a pelear, al que me enseñó a no pelear, al que me dio sus ratos en los minutos lerdos, al que me enseñó el trabajo, al que me habló del alma, al que me hizo mirar más adentro, al que me leyó, al que me escribió, al que se fue, al que siempre está, a los que me mostraron a mí mismo.

A todos quienes han enriquecido mi camino con su esencia. Profesores con vocación, compañeros solidarios, y la familia.

A los que ya no están.

\

\vfill
\begin{flushright}{\it
El aprendizaje ha sido vasto. \\El autor sigue siendo el mismo de ayer.}
\end{flushright}

\newpage
\thispagestyle{empty}
\mbox{}
\thispagestyle{empty}

\tableofcontents
\pagestyle{fancy}

  \chapter{Introducción}\label{capintro} 
\textit{Desocupado lector, sin juramento me podrás creer que quisiera que este libro, como hijo del entendimiento, fuera el más hermoso, el más gallardo y más discreto que pudiera imaginarse. Pero no he podido yo contravenir al orden de naturaleza; que en ella cada cosa engendra su semejante.\footnote{Cita textual del comienzo del prólogo de {\it Don Quijote de la Mancha} por Miguel de Cervantes Saavedra.}}

En los últimos dos siglos la teoría de funciones especiales ha
dado he\-rra\-mientas que hicieron posible la aplicación de la
matemática a las ciencias físicas: desde la teoría de la
conducción del calor, pasando por el electromagnetismo y la
mecánica cuántica hasta la reconstrucción de imágenes en la
física médica, se han beneficiado con su desarrollo. Es
conocido que casi todas las llamadas funciones especiales, que
juegan un rol importante en las soluciones explícitas de
problemas en física matemática, son o bien casos particulares
de funciones hipergeométricas de Euler y Gauss o bien funciones de
Bessel.

El estudio de las funciones especiales, comenzó durante el siglo
XIX y de un modo más bien caótico. Durante el último siglo, este
conocimiento comenzó a ser unificado y organizado cuando se
conectaron estas funciones  con la teoría de  representaciones
de los grupos clásicos.

Fueron  E. Cartan (1929) y H. Weyl (1934)  quienes desarrollaron la
teoría de funciones esféricas para espacios simétricos
compactos y variedades riemannianas compactas, en par\-ti\-cu\-lar
probaron que los armónicos esféricos surgen de manera natural a
partir del estudio de funciones en $G/K$, donde $G=\SO(n)$ y
$K=\SO(n-1)$.

Las funciones esféricas asociadas con la representación trivial
de $K$, dan origen a numerosos ejemplos, que incluyen funciones
especiales tales como los polinomios de Jacobi, Hermite, Laguerre y
las funciones de Bessel, Legendre, Jacobi,  etc. Por ejemplo, los
polinomios de Legendre son funciones esféricas asociadas al par
$(G,K)$ con $G$ el grupo de rotaciones de $\RR^3$ y $K$ el grupo de
rotaciones de $\RR^2$. Las funciones de Bessel, surgen al considerar
el plano de dimensión dos como el cociente del grupo de  todos los
movimientos rígidos del plano por el grupo de rotaciones.

I. Gelfand, R. Godement  y  Harish-Chandra, desarrollaron la teoría de funciones esféricas
 para espacios simétricos no compactos. La propiedad crucial que
 permitió unificar los diversos ejemplos de funciones especiales
 es el hecho que todas estas funciones satisfacen una ecuación
 integral. En general, dado un grupo \(G\) y \(K\) un subgrupo compacto de \(G\),
una función \(\ph\) en \(G\) se dice esférica zonal si satisface
\[\ph(x)\ph(y)=\int_K \ph(xky)\, dk,\qquad \quad \text{ para todo }x,y\in G.\]
Partiendo de esta ecuación integral,  las funciones esféricas se
conectan con  la teoría general de desarrollos de funciones en autofunciones
asociadas con operadores diferenciales de segundo orden, la teoría general de representaciones de grupos y la teoría clásica
 de funciones especiales.
Todos estas conexiones son una parte importante de la teoría de
grupos de Lie, área que se expandió y ahora alimenta otras
áreas importantes de la matemática, como las ecuaciones
diferenciales ordinarias y parciales, el análisis armónico, la
física matemática, la teoría de invariantes, la teoría
de números y la geometría algebraica.

Tomando como punto de partida la ecuación integral de las
funciones esféricas zonales uno puede obtener fácilmente, por
ejemplo, las fórmulas integrales para las funciones de Bessel y
para los polinomios de Legendre y Gegenbauer.

Esta interpretación de las funciones especiales permitió
modernizar un gran número de hechos conocidos  hasta ese momento.
También mostró una manera de buscar otros ejemplos de estas
situaciones, incluyendo grupos discretos, con interesantes
aplicaciones en combinatoria. Estos desarrollos han tenido un
importante impacto en la teoría de códigos y otras áreas no
tradicionales donde los matemáticos pueden aún jugar un rol
importante.

De manera independiente  Tirao (\cite{T77}) y Gangoli-Varadarajan
(\cite{GV88})  generalizaron la teoría clásica de funciones
esféricas, poniendo en relieve la función matricial subyacente
en el concepto de función esférica (escalar) desarrollada por
Godement y Harish-Chandra. Este punto de vista matricial ha cobrado
mucha importancia en los últimos a\~nos, por ejemplo en el
desarrollo de la teoría de polinomios ortogonales matriciales.

Sea $\hat K$ el conjunto de todas las clases de equivalencia de
representaciones irreducibles de dimensión finita de $K$; para
cada  $\delta\in \hat K$ sea $\xi_{\delta}$ el carácter de
$\delta$, $d(\delta)$ la dimensión de $\delta$ y
$\chi_\delta=d(\delta)\xi_\delta$. Una función esférica de tipo
$\delta$ es una función continua $\Phi:G\longrightarrow \End(V)$
tal que   $\Ph(e)=I$ y que satisface la siguiente ecuación
integral.
$$\Ph(x)\Ph(y)=\int_K \chi_{\delta}(k^{-1})\Ph(xky)\, dk,\;\text{ para todo } x,y\in G.$$

 Siendo $D(G)^K$ el álgebra de los operadores  diferenciales en   $G$ que son
invariantes por multiplicación a izquierda por  $G$ y a derecha por  $K$, tenemos que una  función
esférica de tipo $\delta$ está caracterizada por las siguientes
propiedades:
\begin{enumerate}
\item[i)] $\Ph:G\longrightarrow \End(V)$ es una función analítica.
\item[ii)] $\Ph(k_1gk_2)=\pi(k_1)\Ph(g)\pi(k_2)$, para todo $k_1,k_2\in K$, $g\in G$, y $\Phi(e)=I$.
\item[iii)] $[\Dt\Ph ](g)=\Ph(g)[\Dt\Ph ](e)$, para todo $g\in G$ y  $\Delta\in D(G)^K$.
\end{enumerate}
Donde  $(\pi, V)$ es una representación finita de $K$ tal que
$\pi=m\delta$ con $\delta\in \hat K$.

Por otra parte las funciones esféricas de tipo $\delta$ aparecen de manera natural al considerar representaciones de $G$.
Si $g\mapsto U(g)$ es una representación continua
de $G$ en un espacio vectorial de dimensión finita
$E$, entonces $$P_\delta=\int_K \chi_\delta(k^{-1})U(k)\, dk$$ es
una proyección de $E$ sobre $P_\delta E=E(\delta)$. La función
$\Phi:G\longrightarrow \End(E(\delta))$ definida por
\begin{equation}\label{mirada}
 \Phi(g)a=P_\delta U(g)a,\quad g\in G,\; a\in E(\delta),
\end{equation}
es una función esférica de tipo $\delta$. En efecto, si $a\in E(\delta)$ tenemos que
\begin{align*}
\Phi(x)\Phi(y)a&= P_\delta U(x)P_\delta U(y)a=\int_K \chi_\delta(k^{-1})
P_\delta U(x)U(k)U(y)a\, dk\\
&=\left(\int_K\chi_\delta(k^{-1})\Phi(xky)\, dk\right) a.
\end{align*}

Si la representación $g\mapsto U(g)$ es irreducible entonces la función esférica asociada $\Phi$ es irreducible también. Recíprocamente, cualquier función esférica irreducible en un grupo compacto $G$ se obtiene de esta forma a partir de una representación irreducible de dimensión finita de $G$.

Cuando $V$ es de dimensión $1$ las funciones esféricas son a valores escalares. En el caso
general las  funciones que se obtienen son a valores matriciales.

Si bien  la teoría fue desarrollada hace tiempo, los ejemplos
concretos de estas funciones no se conocían hasta hace pocos
a\~nos, aún en los casos de dimensión uno (dados por Heckman y
Opdam, en 1990). Gr\"unbaum, Pacharoni   y Tirao, en
\cite{GPT02a}, describen las funciones esféricas de cualquier $K$-tipo asociadas al plano
proyectivo complejo, que se identifica con $\SU(3)/\U(2)$. Este trabajo fue un motivador crucial para un posterior desarrollo de investigaciones en el área,  incluyendo \cite{GPT02b,GPT03,PT04, GPT05,PT07a, PR08, PT12b, PT13}, donde uno considera funciones esféricas a valores matriciales asociadas a un par simétrico compacto $(G,K)$, generando sucesiones de polinomios matriciales ortogonales de una variable real que satisfacen una relación recursiva de tres términos y siendo autofunciones de un operador diferencial matricial de segundo orden (propiedad biespectral).

En 1929,  Bochner planteó y resolvió el problema de determinar
todas las familias de polinomios ortogonales a valores escalares que
fueran autofunciones de un operador diferencial de segundo orden,
arbitrario pero fijo.
Las  funciones esféricas zonales dan los ejemplos más simples de
tales funciones. Mas explícitamente los así llamados
polinomios ortogonales de Jacobi, Hermite o Laguerre son los
primeros ejemplos que surgen. La teoría de polinomios ortogonales
matriciales, sin ninguna consideración de ecuaciones diferenciales
se remonta a \cite{K49} y \cite{K71}. Después la teoría fue revivida por A. Durán en \cite{D97}, quién planteó el problema de buscar pesos matriciales $W$ con operadores diferenciales matriciales simétricos $D$ de segundo orden. Pero la existencia de tales ``pares clásicos'' $(W,D)$ fue demostrada por primera vez en \cite{G03} y \cite{GPT03} como resultado de lo obtenido en \cite{GPT02a}. De hecho, en \cite{GPT02a} para cualquier $K$-tipo $\pi$ a partir de las funciones esféricas del par $(\SU(3),\U(2))$ se construyeron un peso matricial $W$ de tamaño $m=\dim \pi$, un operador diferencial simétrico $D$ de segundo orden y una sucesión de polinomios matriciales $\{P_w\}_{\ge0}$.
Tal sucesión tiene las siguientes propiedades:
$\gr P_w=w+m$, $\det P_0\not \equiv 0$,
$P_w$ y $P_{w'}$ son ortogonales con respecto a $W$ para todo $n\neq n'$,  $DP_w=P_w\Lambda_n$ y la sucesión $\{P_w\}_{w\ge0}$ satisface un relación de recurrencia de tres términos. De todos modos, la sucesión $\{P_w\}_{w\geq 0}$ no encaja directamente en la teoría existente de polinomios ortogonales matriciales tal cual se la presenta en  \cite{D97}. En \cite{G03} y \cite{GPT03} se establece tal conexión definiendo la función matricial $Q_w$ a través de $Q_w=P_0^{-1}P_w$.
Vale la pena notar que siempre que se tengan estas hipótesis se puede probar que $\{Q_w\}_{w\geq 0}$ es una sucesión de polinomios ortogonales matriciales con  respecto a  $\widetilde W= P_{0}^*WP_{0}$, y que
$\widetilde D= P_{0}^{-1}DP_{0}$ es simétrico.
Resultados de naturaleza comparable pueden ser hallados en \cite{GPT03,PT04, RT06, GPT07, RT12, PT12b}.

Un camino diferente para encontrar polinomios ortogonales matriciales clásicos se puede encontrar por ejemplo en \cite{DG04}.
\medskip

A continuación se describe sucintamente el contenido de la tesis.

En el Capítulo \ref{S3} describimos todas las funciones esféricas de cualquier $K$-tipo
asociadas a la esfera tridimensional $S^3$, que podemos identificar con
$G/K=\SO(4)/\SO(3)$,  como autofunciones simultáneas de dos operadores diferenciales que conmutan, uno de orden dos y el otro de orden uno; luego las funciones esféricas  del mismo $K$-tipo son dispuestas en una sucesión de polinomios matriciales ortogonales.

En la Sección \ref{DE}, luego de un desarrollo preliminar, describimos explícitamente todas las funciones esféricas irreducibles del par simétrico $(\SO(4),\SO(3))$ para cualquier $K$-tipo predeterminado en términos de una función a valores vectoriales $P=P(u)$, cuyas entradas son múltiplos de polinomios de Gegenbauer en una variable $u$. Esto se alcanza al desacoplar un sistema de ecuaciones diferenciales lineales de segundo orden usando una matriz constante $U$ de polinomios de Hahn evaluados en $1$, ver Corolario \ref{DEhyp}.

Siguiendo las técnicas y estrategias
desarrolladas en \cite{GPT02a} asociamos a cada función esférica irreducible
$\Phi:G\longrightarrow \End(V_\pi)$ de tipo $\pi=\pi_{\ell}$, una
función $H$ definida por
$$H(g)=\Phi(g)\, \Phi_\pi(g)^{-1},$$
donde $\Phi_\pi$ es una función esférica particular de tipo $\pi$.
Entonces $H$ tiene las siguientes propiedades
\begin{enumerate}
\item [i)] $H(e)=I$.
\item [ii)]$ H(gk)=H(g)$, para todo $g\in G, k\in K$. \item [iii)]
$H(kg)=\pi(k)H(g)\pi(k^{-1})$, para todo $g\in G, k\in K$.
\end{enumerate}

\noindent La propiedad ii) indica que $H$ se puede considerar como
una función en $S^3$.

El grupo $G=\SO(4)$ actúa de manera natural en la esfera tridimensional.
Esta acción es transitiva siendo $K=\SO(3)$ el subgrupo de isotropía del polo norte $e_4=(0,0,0,1)\in
S^3$. Por lo tanto $S^3=G/K$.
Más aún, la acción de $G$ corresponde a la acción inducida por la multiplicación a izquierda en $G/K$.

En el hemisferio norte de $S^3$
 $$(S^3)^{+}=\vzm{x=(x_1,x_2,x_3,x_4)\in S^3}{x_4>0},$$ consideraremos el sistema de coordenadas
$p:(S^3)^{+}\longrightarrow \RR^3$ dado por
\begin{equation*}
p(x)=\left(\frac{x_1}{x_4},\frac{x_2}{x_4},\frac{x_3}{x_4}\right)=(y_1,y_2,y_3).
 \end{equation*}

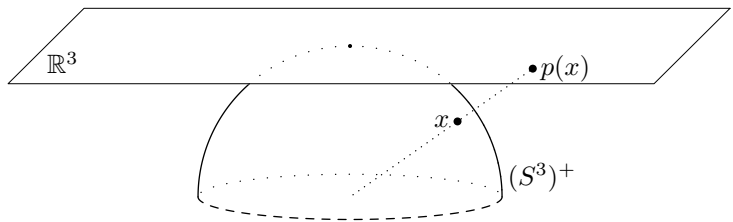
\begin{figure}[h]
  \centering
      \includegraphics{fig1}
\end{figure}

El mapa coordenado $p$ lleva las $K$-órbitas en $(S^3)^+$ a las $K$-órbitas en $\RR^3$, las cuales son las esferas
 $$S_r=\vzm{(y_1,y_2,y_3)\in \RR^3}{|| y||^2=|y_1|^2+|y_2|^2+|y_3|^2=r^2}, \qquad 0\leq r <\infty.$$
En cada órbita $S_r$ elegimos como representante el punto $(r,0,0) \in \RR^3$ con $0\leq r<\infty$.
Por lo que el intervalo $[0,\infty)$ parametriza el conjunto de $K$-órbitas de $\RR^3$.

En este caso el álgebra $D(G)^K$ de operadores diferenciales en $G$,
invariantes a izquierda por $G$ y a derecha por $K$ es un álgebra de
polinomios en dos generadores algebraicamente independientes
$\Delta_1$ y $\Delta_2$. El hecho que $\Phi$ sea una autofunción de $\Delta_1$ y
$\Delta_2$, se traduce en el hecho que $H$ sea autofunción de dos
operadores diferenciales de segundo orden $D$ y $E$ en $\RR^3$.

Por lo tanto existen operadores diferenciales ordinarios $\tilde D$
y $\tilde E$ en el intervalo abierto $(0,\infty)$ tales que
$$(D\,H)(r,0,0)=(\tilde D\tilde H)(r)\, , \quad (E H)(r,0,0)=(\tilde E\tilde
H)(r),$$ donde $\tilde H(r)=H(r,0,0)$, $r\in(0,\infty)$. Los operadores
$\tilde D$ y $\tilde E$ están dados explícitamente en los Teoremas
\ref{D} y \ref{E}. Estos teoremas están dados en
términos de transformaciones lineales. Las funciones $\tilde H$
son diagonalizables (Proposición \ref{Hdiagonal}). Por lo tanto, en
una base apropiada de $V_\pi$ podemos escribir $\tilde
H(r)=(h_0(r),\cdots, h_\ell(r))$. Luego, en los Corolarios
\ref{sistema1} y \ref{sistema2} damos los enunciados
correspondientes de los Teoremas \ref{D} y \ref{E} en
términos de las funciones escalares $h_i$.

Después del
cambio de variables $u=\tfrac {1} {\sqrt{1+r^2}}$ e interpretando ambos sistemas de ecuaciones diferenciales como un par de ecuaciones diferenciales matriciales, los operadores diferenciales $\widetilde D$ y $\widetilde E$,
mencionados en los Teoremas \ref{E} y \ref{D}, se transforman en
$ D$ y $E$, ver \eqref{opDH} y \eqref{opEH}.

Luego introducimos los operadores
$$\bar D =\left(UT(u)\right)^{-1} D \left(UT(u)\right) \quad \text{ y } \quad \bar E =\left(UT(u)\right)^{-1} E \left(UT(u)\right),$$
con $ T(u)=\sum_{j=0}^\ell (1-u^2)^{j/2}E_{jj} $.

\medskip

La mirada algebraica dada en \eqref{mirada} nos permite obtener nuevos resultados y completar otros ya ob\-te\-ni\-dos, en particular determinamos 	explícitamente todos los autovalores de los operadores $\Delta_1$ y $\Delta_2$ asociados a cada función esférica irreducible, ver Teorema \ref{param2}.
En la Sección \ref{depafi}, establecemos cuáles de aquellos polinomios a valores vectoriales $P=P(u)$ se corresponden con funciones esféricas irreducibles, y se muestra cómo construir las funciones esféricas a partir de éstos.

Resumiendo, en la Sección \ref{depafi} somos capaces finalmente de describir
todas las funciones esféricas irreducibles asociadas a la esfera tridimensional en términos
de polinomios de Gegenbauer. Cada una de las entradas de $P$ es un múltiplo de un polinomio de Jacobi,
solución de la ecuación diferencial
\begin{equation*}
(1-u^2) p_j''(u)-(2j+3)\,u p_j'(u)+ (n-j)(n+j+2) p_j(u) =0.
\end{equation*}

\medskip

%
%

En la Sección \ref{hyper}, para un $\ell\in2\ZZ_{\ge0}$ fijo, construimos una sucesión  de polinomios matriciales $\{P_w\}_{w\ge0}$, de modo que la $k$-ésima columna está dada por la función vectorial asociada a la función esférica $\Phi^{({w}+\ell/2,-{k}+\ell/2)}_{\ell}$, para $k =0,1,2,\dots ,\ell$ (ver Teorema \ref{param2}).
Luego usamos el primer polinomio matricial $P_0$ de la sucesión para considerar una nueva sucesión $\{\widetilde P_w\}$, donde $\widetilde P_w=P_0^{-1}P_w$, para $w\ge0$.

Demostramos que si conjugamos a $\bar D$ con esta función polinomial $P_0$ obtenemos que el operador $\widetilde D={P_0}^{-1}\bar D{P_0}$ es un operador hipergeométrico
matricial, obteniendo también un resultado análogo para el operador de primer orden $\widetilde E={P_0}^{-1}\bar E{P_0}$, ver Teorema \ref{hyp}. Tal demostración entrelaza varias propiedades satisfechas por los polinomios de Gegenbauer.

En la Sección \ref{Matrix Orthogonal Polynomials} exhibimos explícitamente la sucesión de polinomios ortogonales matriciales $\widetilde P_w$, $w\ge0$, con respecto al peso $W$ dado en \eqref{pes}.
En el caso clásico las funciones esféricas de tipo trivial
asociadas a las esferas pueden
ser identificadas con polinomios de Gegenbauer, con
una parametrización conveniente. Aquí se prueban
resultados similares pero para cualquier $K$-tipo. Las funciones $P(u)$ 
asociadas a las funciones esféricas, además de poseer en cada una de sus entradas un múltiplo de un polinomio de Gegenbauer como mencionamos antes, dan un ejemplo
de polinomios de Jacobi matriciales de grado $w$: si definimos $F(s)=P(u)$ con $s=(1-u)/2$, tenemos que $F$ es solución de

\begin{equation*}
s(1-s) F'' + \left(B-sC\right)  F'+  \left(\Lambda_0 -\lambda\right)F=0,
\end{equation*}
con \begin{align*}
B=  \sum_{j=0}^\ell (j+\tfrac 32 )E_{jj} - \sum_{j=0}^{\ell-1} (j+1)E_{j,j+1},\quad C= \sum_{j=0}^\ell (2j+ 3 )E_{jj},\quad   \Lambda_0  =- \sum_{j=0}^\ell j(j+2 )E_{jj},
\end{align*}
que es un ejemplo de la ecuación hipergeométrica matricial introducida por Tirao, en \cite{T03}.
Y ya que los autovalores de $C$ no están en $-\NN$ la función $F$ está caracterizada por su valor en $0$, más aún para $|s|<1$  ella está dada por
$$F(s)={}_2\!H_1\left(\begin{smallmatrix}C,-\Lambda_0+\lambda\\  B \end{smallmatrix};s\right)F_0=\sum_{j=0}^{w}\frac{s^j}{j!} [B;C;-\Lambda_0+\lambda]_j F_0, \qquad F_0\in \CC^{\ell+1},
$$
donde el símbolo $[B;C;-\Lambda_0+\lambda]_j$ está definido inductivamente por
\begin{align*}
[B;C;-\Lambda_0+\lambda]_0   &=1,\\
[B;C;-\Lambda_0+\lambda]_{j+1} &=
\left(B+j\right)^{-1}(j(C+j-1)-\Lambda_0+\lambda)[B;C;-\Lambda_0+\lambda]_j ,
\end{align*}
para todo $j\geq 0$.

Entonces, el espacio vectorial $V(\lambda)$ de todas las soluciones polinomiales con valores vectoriales de la ecuación hipergeométrica es no trivial si y
sólo si
$$\lambda=\lambda_w(k)=-(k+w)(k+w+2),$$
con $w$ y $k$ enteros no negativos tales que $k\le \ell$. La solución polinomial es única (salvo escalares) para cada $\lambda_w(k)$, y es de grado $w$.

Con este conocimiento en mano, en el Teorema \ref{sucesion} probamos que $\{\widetilde P_w\}_{w\ge0}$ es una
sucesión ortogonal de polinomios matriciales tal que
$$\widetilde D\widetilde P_w=\widetilde P_w\Lambda_w,\qquad \widetilde E\widetilde P_w=\widetilde P_wM_w, $$  donde $\Lambda_n$ y $M_w$ son las matrices diagonales reales dadas por
$\Lambda_w= \sum_{k=0}^\ell\lambda_w(k)E_{kk}$, y  $M_w= \sum_{k=0}^\ell \mu_w(k)E_{kk}$, con
\begin{align*}
\lambda_w(k)=-(w+k)(w+k+2) \qquad \text{y}\qquad \mu_w(k)=w(\tfrac\ell2-k)-k(\tfrac\ell2+1).
\end{align*}

Este capítulo está contenido en \cite{PTZ12} y es la finalización de un trabajo iniciado en \cite{Z08}.
Cabe mencionar que más recientemente en \cite{KPR11}  los autores estudiaron las funciones esféricas irreducibles del par $(G,K)=(\SU(2)\times\SU(2), \SU(2))$ ($\SU(2)$ embebido diagonalmente) como proyecciones en componentes $K$-isotípicas de representaciones   irreducibles de $G$. Este desarrollo es comparable con la construcción de polinomios vectoriales dada en  \cite{K85}. También, en \cite{KPR12} los autores vuelven al tema pero empezando con la cons\-truc\-ción de los polinomios ortogonales matriciales usando una relación de recurrencia y las relaciones de ortogonalidad, y terminando con los operadores diferenciales.

El grupo $\SU(2)\times\SU(2)$ es el cubrimiento universal de $\SO(4)$ y la imagen de $\SU(2)$ por este cubrimiento es $\SO(3)$. Por lo tanto, los pares $(\SU(2)\times\SU(2), \SU(2))$ y
$(\SO(4),\SO(3))$ están muy relacionados. De todos modos las investigaciones se realizaron de manera independiente y los tratamientos son muy diferentes.

\medskip

En el Capítulo \ref{pn} establecemos una enfática y directa relación entre las funciones esféricas de la esfera $n$-dimensional $S^n\simeq\SO(n+1)/\SO(n)$ y las funciones esféricas del espacio proyectivo real $n$-dimensional $P^n(\mathbb{R})\simeq\SO(n+1)/\mathrm{O}(n)$. Precisamente, para $n$ impar una función en $\SO(n+1)$ es una función esférica irreducible de algún tipo  $\pi\in\hat\SO(n)$ si y sólo si ella misma es una función esférica irreducible de algún tipo $\gamma\in\hat {\mathrm{O}}(n)$.

Cuando $n$ es par esto también es cierto para ciertos tipos, y en los otros casos damos una clara y explícita correspondencia entre las funciones esféricas irreducibles de ambos pares $(\SO(n+1),\SO(n))$ y $(\SO(n+1),\mathrm{O}(n))$. En el Teorema \ref{Matrix} demostramos finalmente que encontrar todas las funciones esféricas de un par es equivalente a hacer lo mismo con el otro.

Es sabido que las funciones esféricas zonales son polinomios de Jacobi de la forma
$$
\varphi_j^*(\theta) =c_j\, P_j^{(\alpha,\beta)}(\cos \theta), \qquad \theta\in[0,\pi],
$$
donde $c_j$ está definido por al condición por $\varphi_j^*(0)=1$, con  $\alpha$ y $ \beta $  dependiendo del par $(G,K)$:

\begin{enumerate}
\item[i)]  {\ \hbox to 6cm{$ G/K\simeq  S^n:$\hfill}                 \ \hbox to 5cm{$\alpha=(n-2)/2$,\hfill}$\beta=(n-2)/2$.}
\item[ii)] {\ \hbox to 6cm{$ G/K\simeq  P^n(\RR):$\hfill}            \ \hbox to 5cm{$\alpha=(n-2)/2$,\hfill}$\beta=-1/2$.}
\item[iii)]{\ \hbox to 6cm{$ G/K\simeq  P^n(\CC):$\hfill}            \ \hbox to 5cm{$\alpha=n-1$,\hfill}$\beta=0$.}
\item[iv)] {\ \hbox to 6cm{$ G/K\simeq  P^n(\mathbb H):$\hfill}      \ \hbox to 5cm{$\alpha=2n-1$,\hfill}$\beta=1$.}
\item[v)]  {\ \hbox to 6cm{$ G/K\simeq  P^2(Cay):$\hfill}            \ \hbox to 5cm{$\alpha=7$,\hfill}$\beta=3$.}
\end{enumerate}
Lo que a simple vista podría sugerir que las funciones esféricas de $S^n$ y $P^n(\RR)$ no guardan relación directa, lo cual no se condice con nuestros resultados, por lo que  en la Sección \ref{appendix} miramos el caso particular de las esféricas zonales para aclarar esta aparente inconsistencia.

Como consecuencia inmediata de este capítulo conocemos todas las funciones esféricas del espacio proyectivo real
$\SO(4)/\mathrm{O}(3)$. Puesto que en el Capítulo \ref{S3} estudiamos todas las del par $(\SO(4),\SO(3))$, aplicando el Teorema \ref{par} también tenemos todas las funciones esféricas del par $(\SO(4),\OO(3))$ de cualquier tipo, las cuales, como funciones en $\SO(4)$ son las mismas que las funciones esféricas del par $(\SO(4),\mathrm{SO}(3))$. El contenido de este capítulo forma parte de \cite{TZ12}.

\medskip

Finalmente en el Capítulo \ref{sn} desarrollamos el estudio de las funciones esféricas en el caso general de la esfera $n$-dimensional $\SO(n+1)/\SO(n)$, es decir pensamos en las funciones esféricas $\Phi$ del par $(G,K)=(\SO(n+1),\SO(n))$.
Luego de un cierto trabajo preliminar, en la Sección \ref{eloperadordelta} nos concentramos en estudiar un operador $\Delta\in D^K(G)$, particularmente llevamos la identidad $\Delta\Phi=\lambda \Phi$ con $\lambda\in\CC$ a una ecuación diferencial de orden dos en una variable real $\widetilde D H=\lambda H$, donde las funciones $H$ se corresponden con las funciones esféricas irreducibles $\Phi$ mediante la restricción de ésta última a un subgrupo monoparamétrico $A$, de modo que resolver cualquiera de estas ecuaciones implica resolver la otra. A continuación analizamos cuándo los $\SO(n)$-tipos son $\SO(n-1)$-irreducibles, es decir, cuándo $\Phi$ es a valores escalares. Entonces, a partir de lo obtenido en la Sección \ref{eloperadordelta} describimos explícitamente todas las funciones esféricas escalares de la esfera $n$-dimensional, entre las cuales están las funciones zonales o de tipo trivial; vale observar que combinando esto con el resultado del Capítulo \ref{pn} también se conocen las todas
funciones escalares del espacio proyectivo $n$-dimensional.

Posteriormente reescribimos más explícitamente la ecuación diferencial $\widetilde D H=\lambda H$ para el caso en que la función esférica $\Phi$ sea de tipo fundamental, es decir cuando  el peso máximo del tipo $\pi$ sea de la forma $(1,\dots,1,0,\dots,0)$ si $n$ es impar o  $(1,\dots,1,0,\dots,0)$ con al menos un cero si $n$ es par. Debido a la naturaleza del grupo ortogonal especial real trabajamos por separadas las situaciones $n$ par y $n$ impar; más aún, cuando $n$ es impar  el caso $(1,\dots,1)$ es  estudiado aparte en la Subsección \ref{111}.
En cada una de estas situaciones consideramos la realización de la representación $\pi$ para poder así obtener los coeficientes matriciales de la ecuación diferencial, puesto que, entre otras dificultades, la base de Gelfand-Tsetlin para las representaciones del álgebra de Lie no es una base de vectores pesos.

En la Sección \ref{hiper}, para cada caso, conjugamos el operador $\widetilde D$ por una función matricial apropiada $\Psi$, hipergeometrizando la ecuación $\widetilde D H=\lambda H$, lo que nos lleva a  una ecuación hipergeométrica matricial $D P=\lambda P$. Estudiamos los posibles valores de $\lambda$ y en cada una de las situaciones escribimos la ecuación diferencial $DP=\lambda P$ de la forma
\begin{equation*}
y(1-y) P''(y)+(C-y(A+B+1)) P'(y)-AB\, P(y)=0,
\end{equation*}
con $A$, $B$ y $C$ matrices cuadradas.

A partir de este punto dejamos de lado el caso en que $\pi$ tenga peso máximo $(1,\dots,1)$ y nos concentramos en todos los restantes. Demostramos que todas las funciones $P$ que surgen a partir de funciones esféricas irreducibles son polinomios, y entonces en el Teorema \ref{columnsSn}, a  través de la función hipergeométrica matricial,  describimos explícitamente cada función $P$ con una expresión del tipo
\begin{equation*}
P(y)={}_2\!F_1\left(\begin{smallmatrix}
                   A,B\\C
                  \end{smallmatrix}
\right)
=\sum_{j=0}^{w}\frac{y^j}{j!} (C;A;B)_j  P(0),
\end{equation*}
donde $P(0)$ es un vector que sabemos calcular y el símbolo $(C;A;B)_j$ se define inductivamente por
\begin{align*}
(C;A;B)_0   =1,\qquad
(C;A;B)_{j+1} =
\left(C+j\right)^{-1}\left(A+j\right)\left(B+j\right) (C;A;B)_j,
\end{align*}
para todo $j\geq 0$.
Finalmente en la Sección \ref{POM} construimos una sucesión $\{P_w\}_{w\ge0}$ que es una sucesión de polinomios ortogonales matriciales con respecto a un peso $W$ construido explícitamente. Además el operador $D$ es de segundo orden y simétrico con respecto a $W$; más aún  tenemos que $D P_w=P_w\Lambda_w $, con $\Lambda_w$ una matriz diagonal real.

  \chapter{Funciones Esféricas}\label{funciones_esfericas}
\begin{flushright}{\it 
``No hay árbol bueno que pueda dar fruto malo, ni árbol malo que pueda dar fruto bueno.\\ Porque cada árbol se conoce por su fruto; pues no se cosechan higos de los espinos,\\ ni de las zarzas se vendimian uvas."}\\
Lucas 6:43-44.
\end{flushright}

\

Sea $G$ un grupo localmente compacto unimodular y sea $K$ un
subgrupo compacto de $G$. Sea $\hat K$ el conjunto de todas las
clases de equivalencia de representaciones complejas irreducibles de
dimensión finita de $K$; para cada $\delta\in \hat K$, sea
$\x_\delta$ el carácter de $\delta$, $d(\delta)$ el grado de
$\delta$, i.e. la dimensión de cualquier representación en la
clase $\delta$ y $\chi_\delta=d(\delta)\x_\delta$. Elegimos de ahora
en adelante la medida de Haar $dk$ en $K$ normalizada por $\int_K
dk=1$.

Denotaremos por $V$ un espacio vectorial de dimensión finita sobre
el cuerpo $\CC$ de números complejos y por $\End(V)$ el espacio de
todas las transformaciones lineales de $V$ en $V$. Siempre que
hagamos referencia a la topología de dicho espacio, estaremos
hablando de la única topología Hausdorff lineal en él.

Por definición una función esférica zonal \index{Función esférica !zonal}(\cite{He00}) $\ph$ sobre $G$ es una función continua a valores complejos que
satisface $\ph(e)=1$ y
\begin{equation}\label{defclasica}
\ph(x)\ph(y)=\int_K \ph(xky)\, dk, \qquad \qquad x,y\in G.
\end{equation}

Una fructífera generalización del concepto anterior está
dada en la si\-guien\-te definición
\begin{defn}[\cite{T77},\cite{GV88}]\label{defesf}
Una función esférica \index{Función esférica !de tipo $\delta$} $\Ph$ sobre $G$ de
tipo $\delta\in \hat
K$ es una función continua sobre $G$ con valores en $\End(V)$ tal que
\begin{enumerate} \item[i)] $\Ph(e)=I$. ($I$= transformación identidad).

\item[ii)] $\Ph(x)\Ph(y)=\int_K \chi_{\delta}(k^{-1})\Ph(xky)\, dk$, para todo $x,y\in G$.
\end{enumerate}
\end{defn}

\begin{prop}[\cite{T77},\cite{GV88}]\label{propesf} Si $\Ph:G\longrightarrow
\End(V)$ es una función esférica de tipo $\delta$ entonces
\begin{enumerate}
\item[i)] $\Ph(kgk')=\Ph(k)\Ph(g)\Ph(k')$, para todo $k,k'\in K$, $g\in G$.
\item[ii)] $k\mapsto \Ph(k)$ es una representación de $K$ tal que cualquier subrepresentación pertenece a $\delta$.
\end{enumerate}
\end{prop}

Con respecto a la definición notemos que la función esférica
$\Ph$ determina un\'\i\-vocamente su tipo (Proposición
\ref{propesf}). El número de veces que $\delta$ ocurre en la
representación $k\mapsto \Ph(k)$ se llama la {\em altura} de
$\Ph$.


Cuando $K$ es un subgrupo central de $G$, (i.e. $K$ está contenido en el centro
de $G$) y $\Phi$ es una función esférica, tenemos
$$\Phi(x)\Phi(y)=\int_K \chi_\delta(k^{-1})\Phi(xky)dk=\int_K
\chi_\delta(k^{-1})\Phi(k)\Phi(xy)dk=\Phi(xy),$$
para todo $x,y\in G$. En otras palabras, $\Phi$ es una representación de $G$.
Por lo tanto si tomamos $K=\{e\}$, las funciones esféricas de $G$ son
precisamente las representaciones de dimensión finita de $G$, y si $G$ es
abeliano las funciones esféricas son las representaciones de dimensión finita de
$G$ tales que \ref{propesf} (ii) se satisface.

Otro caso extremo ocurre cuando $G$ es compacto y $K=G$. En este
caso las funciones esféricas son también las representaciones de
dimensión finita de $G$, con todas sus sub\-re\-pre\-sentaciones
equivalentes.

Sea $\ph$ una solución continua a valores complejos de la
ecuación \eqref{defclasica}. Si $\ph$ no es i\-dén\-ti\-ca\-men\-te cero
entonces $\ph(e)=1$. (confrontar \cite{He00}, página 399). Este resultado se
generaliza de la siguiente forma:

Diremos que una función
$\Ph:G\longrightarrow \End(V)$ es {\em irreducible} \index{Función
esférica !irreducible} si $\Ph(g)$, $g\in G$, es una familia
irreducible de transformaciones de $V$ en $V$. Entonces tenemos

\begin{prop}[\cite{T77}]
Sea $\Ph$ una solución continua con valores en $\End(V)$ de la ecuación ii) en la Definición
\ref{defesf}. Si $\Ph$ es irreducible
entonces $\Ph(e)=I$.
\end{prop}

Las funciones esféricas de tipo $\delta$ aparecen de forma natural
considerando representaciones de $G$. Si $g\mapsto U(g)$ es una
representación continua de $G$, digamos en un espacio vectorial
topológico $E$ completo, localmente convexo y Hausdorff, entonces
$$P(\delta)=\int_K \chi_\delta(k^{-1})U(k)\, dk$$ es una
proyección continua de $E$ en $P(\delta)E=E(\delta)$; $E(\delta)$
consiste de aquellos vectores en $E$, para los cuales el espacio
vectorial generado por su $K$-órbita es de dimensión finita y se
descompone en subrepresentaciones irreducibles de $K$ de tipo
$\delta$. Si $E(\delta)$ es de dimensión finita y no nulo, la función
$\Ph:G\longrightarrow \End(E(\delta))$ definida por
$\Ph(g)a=P(\delta)U(g)a$, $g\in G, a\in E(\delta)$ es una función
esférica de tipo $\delta$. De hecho, si $a\in E(\delta)$ tenemos
\begin{align*}
\Ph(x)\Ph(y)a&= P(\delta)U(x)P(\delta)U(y)a=\int_K
\chi_\delta(k^{-1})
P(\delta)U(x)U(k)U(y)a\, dk\\
&=\left(\int_K\chi_\delta(k^{-1})\Ph(xky)\, dk\right) a.
\end{align*} Si la representación $g\mapsto U(g)$ es
topológicamente irreducible (i.e. $E$ no tiene subespacios
ce\-rra\-dos $G$-invariantes no triviales) entonces la función
esférica asociada $\Ph$ es también irreducible.

Si una función esférica $\Phi$ es asociada a una
representación Banach de $G$ entonces es casi-acotada, \index{Función esférica
!casi acotada} en el sentido de que existe una seminorma $\rho$ en $G$ y $M\in\RR$ tal
que $\|\Phi(g)\|\le M\rho(g)$ para todo $g\in G$. Por otro lado, si
$\Phi$ es una función esférica irreducible casi-acotada en $G$,
entonces es asociada a una representación Banach topológicamente
irreducible de $G$ (ver \cite{T77}). Por lo tanto, si $G$
es compacto cualquier función esférica irreducible en $G$ es
asociada a una representación Banach de $G$, que es de dimensión
finita por el teorema de Peter-Weyl.

Denotaremos por $C_c(G)$ \index{$C_c(G)$} al álgebra, con respecto a
la convolución ``*", de funciones continuas sobre $G$ con soporte
compacto. Podemos considerar el conjunto $C_{c,\delta}(G)$
\index{$C_{c,\delta}(G)$} de aquellas $f\in C_c(G)$ que satisfacen
$\bar \chi_\delta\ast f=f\ast\bar \chi_\delta=f$. Como
$\chi_\delta\ast\chi_\delta=\chi_\delta$ (relaciones de
ortogonalidad), es claro que $C_{c,\delta}(G)$ es una subálgebra de
$C_c(G)$ y que $f\mapsto \bar \chi_\delta \ast f\ast\bar \chi_\delta
$ es una proyección continua de $C_c(G)$ en $C_{c,\delta}(G)$.
Podemos considerar $C_{c,\delta}(G)$ como un subespacio topológico
de $C_{c}(G)$.

Para toda $f\in C_c(G)$, sea $\check f$ la función definida por
$\check f(g)=f(g^{-1})$, entonces
$$(f\ast g)\check{}=\check g\ast \check f.$$

\begin{prop}\label{proposicion_Esfe_rep_Cc}
Sea $\Phi:G\longrightarrow \End(V)$ una función continua tal que
$\chi_\delta\ast\Phi=\Phi\ast\chi_\delta=\Phi$. Entonces $\Phi$
satisface la ecuación integral ii) de la Definición \ref{defesf} si
y sólo si la aplicación
$$\Phi:f\mapsto\int_G f(g)\Phi(g)dg,$$
es una representación de $C_{c,\delta}(G)$.
\end{prop}
\begin{proof}[\it Demostración]
Sean $f$ y $h$ dos funciones en $C_{c,\delta}(G)$, entonces
$$\Phi(f)=\int_G f(g)\Phi(g)dg=(\Phi\ast\check f)(e).$$
Por lo tanto
\begin{equation}
\label{ecu1}
\begin{split}
\Phi(\overline \chi_\delta\ast f \ast \overline \chi_\delta)&=(\Phi\ast
(\overline \chi_\delta\ast f \ast \overline \chi_\delta)\check{})(e)=(\Phi\ast
\chi_\delta\ast \check f \ast \chi_\delta)(e)\\
&=(\Phi \ast \check f \ast \chi_\delta)(e)=( \chi_\delta\ast \Phi \ast\check
f)(e)=(\Phi \ast\check f)(e)=\Phi(f).
\end{split}
\end{equation}
Hemos usado que $\overline \chi_\delta=\check{\chi}_\delta$. Ahora
\begin{equation}
\label{ecu2}
\begin{split}
\Phi((\overline \chi_\delta\ast f \ast \overline \chi_\delta)\ast (\overline
\chi_\delta\ast h \ast \overline \chi_\delta)&=\Phi(f\ast \overline \chi_\delta
\ast h)=\int_G (f\ast\overline \chi_\delta \ast h)(y)\Phi(y)dy\\
&=\int_G\int_G (f\ast \overline \chi_\delta)(x)h(x^{-1}y)\Phi(y)dx\,dy\\
&=\int_G\int_G\int_K f(xk^{-1})\overline \chi_\delta(k)h(y)\Phi(xy) dk\,dx\,dy\\
&=\int_G\int_G f(x)h(y)(\int_K \chi_\delta(k^{-1})\Phi(xky)dk)dx\,dy.
\end{split}
\end{equation}
Por otro lado
\begin{equation}
\label{ecu3}
\begin{split}
\Phi(\overline \chi_\delta\ast f \ast \overline \chi_\delta)\Phi(\overline
\chi_\delta\ast h \ast \overline \chi_\delta)=\Phi(f)\Phi(h)=\int_G\int_G
f(x)h(y)\Phi(x)\Phi(y) dx\, dy.
\end{split}
\end{equation}
Teniendo en cuenta \eqref{ecu2} y \eqref{ecu3}, la proposición sigue inmediatamente.
\end{proof}

Denotamos por $I_c(G)$ \index{$I_{c,\delta}(G)$}al conjunto de
funciones $f\in C_c(G)$ que son $K$-centrales, i.e. inva\-rian\-tes
por $g\mapsto kgk^{-1}$. Observemos que $I_c(G)$ es una subálgebra
de $C_c(G)$ y que el operador $$f\mapsto f^0(g)=\int_K
f(kgk^{-1})dk,$$ es una proyección continua (en la topología
inductiva) de $C_c(G)$ en $I_c(G)$. Podemos definir
$I_{c,\delta}(G)=I_c(G)\cap C_{c,\delta}(G)$, i.e.
\begin{align*}
I_{c,\delta}(G)=\{f\in C_c(G):\bar \chi_\delta\ast f=f,\text{ y
}f(kxk^{-1})\text{ para todo }x\in G,k\in K\}.
\end{align*}
Esta es también una subálgebra de $C_c(G)$ y $f\mapsto f^{0}$ lleva
$C_{c,\delta}(G)$ en $I_{c,\delta}(G)$. Si $f\in I_c(G)$ y  $ \bar
\chi_\delta \ast f=f$, entonces también $f=f \ast \bar \chi_\delta$;
esto significa que la aplicación $f\mapsto  \bar \chi_\delta \ast f
$ es una proyección continua de $I_c(G)$ en $I_{c,\delta}(G)$.

\medskip

En \cite{T77} se puede encontrar una demostración de la siguiente
proposición.
\begin{prop}
\label{icdeltaconmutativa}
Las siguientes propiedades son equivalentes:
\begin{itemize}
  \item[i] $I_{c,\delta}(G)$ es conmutativa.
  \item[ii] Toda función esférica irreducible de tipo $\delta$ es de altura $1$.
  \item[iii] $I_{c,\delta}(G)$ es el centro de $C_{c,\delta}(G)$.
\end{itemize}
\end{prop}

{}De ahora en adelante asumimos que $G$ es un grupo de Lie conexo.
Se puede probar que cualquier función esférica
$\Ph:G\longrightarrow \End(V)$ es diferenciable ($C^\infty$), y
además analítica. Sea $D(G)$ el álgebra de todos los
o\-pe\-ra\-dores diferenciales invariantes a izquierda en $G$ y sea
$D(G)^K$ la subálgebra de todos los operadores en $D(G)$ que son
invariantes por traslación a derecha por elementos de $K$.

\smallskip
En la siguiente proposición $(V,\pi)$ será una representación
de dimensión finita de $K$ tal que cualquier subrepresentación
pertenece a la misma clase $\delta\in\hat K$. \begin{prop}[\cite{T77},\cite{GV88}]
\label{defeq} Una función
$\Ph:G\longrightarrow \End(V)$ es una función esférica de tipo
$\delta$ si y sólo si
\begin{enumerate}
\item[i)] $\Ph$ es analítica.
\item[ii)] $\Ph(k_1gk_2)=\pi(k_1)\Ph(g)\pi(k_2)$, para todo $k_1,k_2\in K$,
$g\in G$, y $\Phi(e)=I$.
\item[iii)] $[D\Ph ](g)=\Ph(g)[D\Ph](e)$, para todo $D\in D(G)^K$, $g\in G$.
\end{enumerate}
\end{prop}

\begin{proof}[\it Demostración]
Si $\Phi:G\longrightarrow\End(V)$ es una función esférica de tipo $\delta$ entonces $\Phi$ satisface iii) (ver Lema 4.2 en \cite{T77}) y $\Phi$ es analítica (ver Proposición 4.3 en \cite{T77}). Recíprocamente, si $\Phi$ satisface i), ii) and iii), entonces $D\mapsto [D\Phi](e)$ es una representación de $D(G)^K$ y por lo tanto $\Phi$ satisface la ecuación integral ii) en Definición \ref{defesf}, ver Proposición 4.6 en \cite{T77}.
\end{proof}

Más aún, tenemos que los autovalores $[D\Phi](e)$, $D\in D(G)^K$ caracterizan las funciones esféricas $\Phi$ como se plantea en la siguiente proposición.

\begin{prop}[Observación 4.7 en \cite{T77}]\label{unicidad}
 Sean $\Phi,\Psi:G\longrightarrow \End(V)$ dos funciones esféricas en un grupo de Lie $G$ del mismo tipo $\delta\in \hat K$. Entonces $\Phi=\Psi$ si y sólo si $(D\Phi)(e)=(D\Psi)(e)$ para todo $D\in D(G)^K$.
\end{prop}

Observemos que si $\Ph:G\longrightarrow \End(V)$ es una función
esférica entonces $\Ph:D\mapsto [D\Ph](e)$ transforma $D(G)^K$ en
$\End_K(V)$ ($\End_K(V)$ denota el espacio de las transformaciones
lineales de $V$ en $V$ que conmutan con $\pi(k)$ para todo $k\in K$)
definiendo una representación de dimensión finita del álgebra
asociativa $D(G)^K$. Además la función esférica es \-i\-rre\-du\-ci\-ble
si y sólo si la representación $\Ph: D(G)^K\longrightarrow
\End_K (V)$ es irreducible.
En efecto, si $W<V$ es $\Phi(G)$-invariante, entonces claramente $W$ es invariante como $(D(G)^K,K)$-módulo. Por lo tanto, si $\Phi:D(G)^K\longrightarrow\End_K(V)$
es irreducible entonces la función esférica $\Phi$ es irreducible. Recíprocamente, si $\Phi:D(G)^K\longrightarrow\End_K(V)$ no es i\-rre\-du\-ci\-ble, entonces e\-xis\-te un subespacio propio $W<V$ que es $(D(G)^K,K)$-invariante. Sea $P:V\longrightarrow W$ una $K$-proyección. Consideremos las siguientes funciones: $P\Phi P$ y $\Phi P$. Ambas son analíticas y $(P\Phi P)(k_1gk_2)=\pi(k_1)(P\Phi P)(g)\pi(k_2)$ para todo $g\in G$ y $k_1,k_2\in K$. Más aún, si $D\in D(G)^K$ entonces $[D(P\Phi P)](e)=P[D(\Phi)](e)P=[D(\Phi)](e)P=[D(\Phi P)](e)$. Por lo tanto, usando la Observación 4.7 en \cite{T77} tenemos que $P\Phi P=\Phi P$. Esto implica que $W$ es $\Phi(G)$-invariante. Luego, si $\Phi$ es una función esférica irreducible, entonces $\Phi:D(G)^K\longrightarrow\End_K(V)$ es una representación irreducible.

Como consecuencia de esto tenemos:

\begin{prop}\label{height1} Las siguientes propiedades son equivalentes:
\begin{enumerate}
\item[i)] $D(G)^K$ es conmutativa.
\item[ii)] Toda función esférica irreducible de $(G,K)$ es de altura uno. \end{enumerate}
\end{prop}

\begin{lem}\label{separation} Sea $G$ un grupo de Lie lineal. Dado $D\ne0$, $D\in D(G)$, existe una representación de dimensión finita $U$ de $G$ tal que $[DU](e)\ne0$.
\end{lem}
\begin{proof}[\it Demostración] Podemos asumir que $G$ es un subgrupo de Lie de $\SL(E)$ para un cierto espacio vectorial real de dimensión finita $E$. La representación identidad de $G$ se extiende en la forma usual a una representación $U_s$ de $G$ en $E_s=\otimes^s E$. Que $U_s$ también denote la correspondiente representación del álgebra universal envolvente $\mathcal U(\lieg)$ del álgebra de Lie $\lieg$ de $G$. Entonces, como Harish-Chandra mostró (ver $\S$2.3.2 de \cite{W72}), existe $s\in \NN$ tal que $U_s(D)\ne0$. Finalmente,  usando el isomorfismo canónico $\mathcal U(\lieg)\simeq D(G)$, obtenemos $[DU_s](e)=U_s(D)\ne0$.
\end{proof}

\begin{proof}[Demostración de Proposición \ref{height1}]
 i) $\Rightarrow$ ii). Si $\Phi$ es una función esférica irreducible entonces  $\Phi:D(G)^K\longrightarrow \End_K(V)$ es una representación irreducible. Entonces, $\End_K(V)\simeq\CC$ lo cual es equivalente a $\Phi$ que sea de altura uno.

ii) $\Rightarrow$ i). Si $\Phi$ es una función esférica de altura uno y $D\in D(G)^K$, entonces $[D\Phi](e)=\lambda I$ con $\lambda \in \CC$. Entonces, si $D_1, D_2 \in D(G)^K$ tenemos que
$$[(D_1D_2)\Phi](e)=[D_1\Phi](e)[D_2\Phi](e)=[(D_2D_1)\Phi](e).$$

Por otra parte, tenemos que las funciones esféricas irreducibles de $(G,K)$ separan los elementos de $D(G)^K$. De hecho, si $D\ne0$, $D\in D(G)^K$, por Lema \ref{separation} existe una representación de dimensión finita $U$ de $G$ tal que $[DU](e)\ne0$. Por hipótesis podemos asumir que $U$ es irreducible. Sea $U=\oplus_{\delta\in \hat K}U_\delta$ la descomposición de $U$ en componentes $K$-isotípicas y sea $P_\delta$ la correspondiente proyección de $U$ en $U_\delta$. Luego, existe $\delta\in \hat K$ tal que  $[D(P_\delta UP_\delta)](e)\ne0$. Entonces, la correspondiente función esférica $\Phi_\delta$ es irreducible y $[D\Phi_\delta](e)\ne0$. Por lo tanto $D_1D_2=D_2D_1$.
\end{proof}

En este trabajo, el par $(G,K)$ es $(\SO(n+1),\SO(n))$. En este caso es conocido que $D(G)^K$ es
abeliana; en efecto
$$D(G)^K \cong D(G)^G\otimes D(K)^K$$ (confrontar \cite{Co}, \cite{K89}), donde $D(G)^G$ (respectivamente $D(K)^K$) denota
la subálgebra de todos los opera\-dores en $D(G)$ (respectivamente $D(K)$)
que son invariantes por todas las traslaciones a derecha de $G$
(respectivamente $K$). Por lo tanto, tenemos que todas las funciones esféricas irreducibles son de altura uno.

Por otra parte, en el Capítulo \ref{S3} se estudia particularmente el caso $(G,K)=(\SO(4),\SO(3))$; un famoso teorema de Harish-Chandra dice que
$D(G)^G$ es un álgebra de polinomios en dos
generadores algebraicamente independientes $\Delta_1$ y $\Delta_2$.
Entonces encontrar todas las funciones
esféricas de tipo $\delta\in \hat K$ es equivalente a tomar
cualquier representación irreducible $(V,\pi)$ de $K$ en la clase
$\delta$ y a determinar todas las funciones analíticas
$\Ph:G\longrightarrow \End(V)$ tal que \begin{enumerate}
\item [(1)] $\Ph(k_1gk_2)=\pi(k_1)\Ph(g)\pi(k_2)$, para todo $k_1,k_2\in K$,
$g\in G$.

\smallskip
\item [(2)] $[\Delta_j\Ph](g)=\Ph(g)[\Delta_j\Ph](e)$, $j=1,2$.
\end{enumerate}

Pues, para cualquier $\Delta$ en $D(K)^K$ si $\Ph$
satisface (1) tenemos
$$[\Delta\Ph](g)=\Ph(g)\dot\pi(\Delta)=\Ph(g)[\Delta\Ph](e).$$
En este caso $\dot\pi:\liek_\CC\longrightarrow \End(V_\pi)$ denota
la derivada de la representación $\pi$ de $K$.  También
denotamos con $\dot\pi$ la representación de $D(K)$ en
$\End(V_\pi)$ inducida por $\dot\pi$.

Por lo tanto una función analítica $\Ph$ que satisface (1) y
(2) verifica las condiciones i), ii) y iii) de la Proposición
\ref{defeq}, y por lo tanto es una función esférica.

 \chapter{La Esfera Tridimensional}\label{S3}
\begin{flushright}{\it
``Hay que desconfiar siete veces del cálculo y cien veces del matemático."}\\ Proverbio indio.
\end{flushright}

\

En este capítulo determinamos todas las funciones esféricas irreducibles
$\Phi$ de cualquier $K $-tipo asociadas al par
$(G,K)=(\SO(4),\SO(3))$. Para esto asociamos a
$\Phi$ una función vectorial $H=H(u)$ de una variable real $u$,
la cual es analítica en $u=0$ y cuyas componentes son soluciones de dos sistemas acoplados de ecuaciones diferenciales ordinarias. A través de una apropiada conjugación que involucra a los polinomios de Hahn conseguimos desacoplar uno de estos sistemas, que luego llevamos a un sistema desacoplado de ecuaciones hipergeométricas, encontrando entonces como solución a la función vectorial $P=P(u)$, cuyas entradas son polinomios de Gegenbauer.
Posteriormente, identificamos aquellas soluciones simultáneas y usamos la teoría de representaciones de $\SO(4)$ para caracterizar todas las funciones esféricas irreducibles.
Las representaciones irreducibles de $\SO(3)$ son $\{\pi_\ell\}_{\ell\in2\NN_0}$, donde la dimensión del espacio donde se representa $\pi$ es $\ell+1$, entonces empaquetamos las funciones polinomiales $P=P(u)$ correspondientes a las funciones esféricas irreducibles de un $K$-tipo fijo $\pi_\ell$ cuidadosamente en una sucesión de polinomios matriciales $\{P_w\}_{w\ge0}$ de tamaño $(\ell+1)\times(\ell+1)$. Finalmente probamos que  $\widetilde P_w={P_0}^{-1}P_w$ es una sucesión de polinomios ortogonales con respecto a un peso matricial $W$. Más aún, probamos que $W$ admite un operador diferencial simétrico de segundo orden $\widetilde D$ y un operador diferencial simétrico de primer orden $\widetilde E$.

\section{Preliminares}\label{sec:prelim}

\subsection{Los Grupos $G$ y $K$}\label{GyK}
\

La esfera tridimensional $S^3$ puede realizarse como el espacio homogéneo $G/K$, con $G=\SO(4)$ y $K={\SO(3)}$,
con la identificación usual de $\SO(3)$ como subgrupo de $\SO(4)$: para cada $k$ en $K$, sea $k=\left(\begin{smallmatrix} k &0\\ 0 & 1 \end{smallmatrix}\right) \in G $.

Además, tenemos una descomposición $G=KAK$, donde $A$ es el subgrupo de Lie de $G$ de elementos de la forma
$$a(\theta)= \left(\begin{matrix} \cos \theta&0& 0&
\sin \theta\\ 0&1&0&0	\\ 0&0&1&0\\ -\sin \theta&0& 0&\cos
\theta\end{matrix}\right)\, ,	\qquad \theta\in \RR.$$

\smallskip
Es sabido que existe un morfismo de Lie que es cubrimiento doble $\SO(4)\longrightarrow \SO(3)\times \SO(3)$, en particular $\mathfrak{so}(4)\simeq \mathfrak{so}(3)\oplus \mathfrak{so}(3)$. Explícitamente, se obtiene en la siguiente manera:
 Sea  $q:\SO(4)\longrightarrow \GL\big(\Lambda^2(\RR^4)\big)$ el homomorfismo de Lie definido por
 $$q(g)(e_i \wedge e_j)=g(e_i)\wedge g(e_j)\,,\qquad g\in \SO(4)\,,\qquad 1\leq i<j\leq4,$$
 donde $\{e_j\}_{j =1}^4$ es la base canónica $\RR^4$.
Sea $\dot{q}:\so(4)\longrightarrow \mathfrak{gl}\big(\Lambda^2(\RR^4)\big)$
la  correspondiente derivada.

Observemos que $\Lambda^2(\RR^4)$ es reducible como $G$-módulo. De hecho, tenemos la siguiente des\-com\-po\-si\-ción en $G$-módulos irreducibles,
 $\Lambda^2(\RR^4)=V_1 \oplus V_2$, donde
  $$V_1=\text{span}\{e_1 \wedge e_4+e_2 \wedge e_3 , e_1 \wedge e_3-e_2 \wedge e_4 , -e_1 \wedge e_2-e_3 \wedge e_4\},$$
$$V_2=\text{span}\{e_1 \wedge e_4-e_2 \wedge e_3 , e_1 \wedge e_3+e_2 \wedge e_4 , -e_1 \wedge e_2+e_3 \wedge e_4\}.$$
Sean $P_{1}$ y $P_{2}$ las proyecciones canónicas en los subespacios $V_1$ y $V_2$, respectivamente.
Las funciones definidas por
$$a(g)=P_{1}\,q(g)_{\mid_{V_1}}, \qquad b(g)=P_{2}\,q(g)_{|_{V_2}}, $$
son homomorfismos de Lie de $\mathrm{SO}(4)$ sobre $\mathrm{SO}(V_1)\simeq \mathrm{SO}(3)$ y $\mathrm{SO}(V_2)\simeq \mathrm{SO}(3)$, respectivamente.
Entonces, en una base apropiada tenemos que para cada $g\in\SO(4)$ y para todo $X\in\so(4)$
\begin{equation}\label{funcionq}
q(g)=\left(\begin{matrix}  a(g) &0\\ 0 & b(g) \end{matrix}\right), \qquad
\dot{q}(X)=\left(\begin{matrix} \dot a(X) & 0\\ 0 & \dot b(X) \end{matrix}\right).
\end{equation}

Por lo tanto podemos considerar $q$ como un homomorfismo de $\mathrm{SO}(4)$ sobre $\mathrm{SO}(3)\times \mathrm{SO}(3)$
con núcleo $\{I,-I\}$.
Aparte, se prueba que $a(g)=b(g)$ si y solo si $g\in K$.

\subsection{La Estructura del Álgebra de Lie}
\

Una base de $\lieg=\mathfrak{so}(4)$ sobre $\RR$ está dada por
\begin{align*}
  Y_1&=\left( \begin{smallmatrix} 0&1&0&0 \\-1&0&0&0 \\0&0&0&0 \\0&0&0&0 \end{smallmatrix}\right),  &
Y_2&=\left( \begin{smallmatrix} 0&0&1&0\\0&0&0&0 \\-1&0&0&0 \\0&0&0&0 \end{smallmatrix}\right), &
Y_3&=\left( \begin{smallmatrix} 0&0&0&0 \\0&0&1&0  \\0&-1&0&0 \\0&0&0&0 \end{smallmatrix}\right),  \\
 Y_4&=\left( \begin{smallmatrix} 0&0&0&1 \\0&0&0&0 \\0&0&0&0 \\-1&0&0&0 \end{smallmatrix}\right),&
Y_5&=\left( \begin{smallmatrix} 0&0&0&0 \\0&0&0&1 \\0&0&0&0 \\0&-1&0&0 \end{smallmatrix}\right), &
Y_6&=\left( \begin{smallmatrix} 0&0&0&0 \\0&0&0&0  \\0&0&0&1 \\0&0&-1&0 \end{smallmatrix}\right).
\end{align*}

Consideremos los siguientes vectores
\begin{align*}
Z_1&=\frac{1}{2}\left(Y_3+Y_4\right), &
Z_2&=\frac{1}{2}\left(Y_2-Y_5\right), &
Z_3&=\frac{1}{2}\left(Y_1+Y_6\right),    \\
Z_4&=\frac{1}{2}\left(Y_3-Y_4\right), &
Z_5&=\frac{1}{2}\left(Y_2+Y_5\right), &
Z_6&=\frac{1}{2}\left(Y_1-Y_6\right).
\end{align*}

Se prueba que estos vectores definen una base de $\so(4)$ adaptada a la descomposición $\so(4)\simeq \so(3)\oplus\so(3)$, i.e.
$\{Z_4,Z_5,Z_6 \}$ es una base del primer sumando y $\{Z_1,Z_2,Z_3 \}$ es una base del segundo.

\
\smallskip
El álgebra $D(G)^G$ está generada por los elementos algebraicamente independientes
\begin{equation}\label{deltas}
 \Delta_1=-Z_4^2-Z_5^2-Z_6^2\, , \qquad \qquad
 \Delta_2= -Z_1^2-Z_2^2-Z_3^2,
\end{equation}
los cuales son los Casimires del primer y segundo $\mathfrak{so}(3)$ respectivamente. El Casimir de $K$ será denotado por $\Delta_K$, y está dado por $-Y_1^2-Y_2^2-Y_3^2$.

Llamaremos $\liek $ al álgebra de Lie de $K$, observemos que su complexificación es isomorfa a $\mathfrak{sl}(2,\CC)$. Si definimos
\begin{equation}\label{efh} e= \left( \begin{smallmatrix} 0&i&-1 \\-i&0&0 \\1&0&0 \end{smallmatrix}\right),  \qquad
  f= \left( \begin{smallmatrix} 0&i&1 \\-i&0&0 \\-1&0&0 \end{smallmatrix}\right) , \qquad
  h= \left( \begin{smallmatrix} 0&0&0 \\0&0&-2i \\0&2i&0 \end{smallmatrix}\right),
\end{equation}
tenemos que $\{e,f,h\}$ es un ${s}$-triple en $\liek_\CC$, i.e. $$[e,f]=h,\qquad[h,e]=2e,\qquad[h,f]=-2f .$$

Como subálgebra de Cartan $\mathfrak h_\CC$ de $\mathfrak{so}(4,\CC)$ tomamos la complexificación de la subálgebra abeliana maximal de $\so(4)$ constituida por todas las matrices de la forma
\begin{align*}
H=  \begin{pmatrix}
    0   &x_1  &0&0
\\ -x_1&0     & 0    &0
\\ 0    &0     &0     &x_2
\\ 0    &0     &-x_2 &0
  \end{pmatrix}.
\end{align*}

\noindent Sea $\varepsilon_j\in \mathfrak h_\CC^*$  dada por $\varepsilon_j(H)=-i x_j$
para $j=1,2$. Entonces
$$\Delta(\mathfrak g_\CC,\mathfrak h_\CC)=\vzm{\pm(\varepsilon_1\pm \varepsilon_2)}{\varepsilon_1, \varepsilon_2\in \mathfrak h_\CC^*},$$
y escogemos como raíces positivas aquellas en el conjunto $\Delta^+(\mathfrak g_\CC,\mathfrak h_\CC)=\{\varepsilon_1-\varepsilon_2,\varepsilon_1+\varepsilon_2\}.$

Definimos
\begin{equation*}
\begin{split}
X_{\varepsilon_1+\varepsilon_2}=  \begin{pmatrix}
   0 & 0  & 1 & -i
\\ 0 & 0  & -i & -1
\\ -1 & i  & 0 & 0
\\ i & 1  & 0 & 0
  \end{pmatrix},\,\,\,
X_{\varepsilon_1-\varepsilon_2}&=  \begin{pmatrix}
   0 & 0  & 1 & i
\\ 0 & 0  & -i & 1
\\ -1 & i  & 0 & 0
\\ -i & -1  & 0 & 0
  \end{pmatrix},\\
X_{-\varepsilon_1+\varepsilon_2}=  \begin{pmatrix}
   0 & 0  & 1 & -i
\\ 0 & 0  & i & 1
\\ -1 & -i  & 0 & 0
\\ i & -1  & 0 & 0
  \end{pmatrix},
X_{-\varepsilon_1-\varepsilon_2}&=  \begin{pmatrix}
   0 & 0  & 1 & i
\\ 0 & 0  & i & -1
\\ -1 & -i  & 0 & 0
\\ -i & 1  & 0 & 0
  \end{pmatrix}.
\end{split}
\end{equation*}

\noindent Entonces, para cada $H$ en $\mathfrak{h}_\CC$ tenemos que
 $$ [H,X_{\pm(\varepsilon_1\pm\varepsilon_2)}]=\pm(\varepsilon_1\pm\varepsilon_2)(H)X_{\pm(\varepsilon_1\pm\varepsilon_2)}.$$
Luego, cada $X_{\pm(\varepsilon_1\pm\varepsilon_2)}$ perteneces al espacio-raíz $\mathfrak g_{\pm(\varepsilon_1\pm\varepsilon_2)}$.

Entonces, en términos de la estructura de raíces de $\mathfrak{so}(4,\CC)$, $\Delta_1$ y $\Delta_2$ se escriben
\begin{equation} \label{Delta23}
\begin{split}
\Delta_1&=-Z_6^2+iZ_6-(Z_5+iZ_4)(Z_5-iZ_4), \\
\Delta_2&=-Z_3^2+iZ_3-(Z_2+iZ_1)(Z_2-iZ_1) .
\end{split}
\end{equation}
Observamos que $(Z_5-iZ_4)=X_{\varepsilon_1-\varepsilon_2}\in \mathfrak g_{\varepsilon_1-\varepsilon_2}$ y
 $(Z_2-iZ_1)=X_{\varepsilon_1+\varepsilon_2}\in \mathfrak g_{\varepsilon_1+\varepsilon_2}$
y $ Z_3,  Z_6 \in \mathfrak h _\CC$.

\subsection{Representaciones Irreducibles de $G$ y $K$}\label{representaciones}
\

Primero consideremos $\SU (2)$. Es sabido que las representaciones irreducibles de dimensión finita de $\SU(2)$ son, salvo equivalencia, $\{(\pi_\ell,V_\ell)\}_{\ell\geq0}$, donde $V_\ell$ es el espacio vectorial complejo de todas las funciones polinomiales en dos variables complejas $z_1$ y $z_2$ homogéneos de grado $\ell$, y $\pi_\ell$ es definido por
$$\pi_\ell\left(\begin{matrix}a &b \\c & d\end{matrix}\right)\, P\left(\begin{matrix} z_1\\z_2\end{matrix}\right)=
P\left(\left(\begin{matrix}a &b \\c & d\end{matrix}\right)^{-1}\left(\begin{matrix} z_1\\z_2\end{matrix}\right)\right),
 \qquad \text{para } \left(\begin{matrix}a &b \\c & d\end{matrix}\right) \in \SU(2).$$
Entonces, como existe un homomorfismo de Lie de $\SU(2)$ sobre $\SO(3)$ con kernel $\{\pm I\}$, las representaciones irreducibles de $\SO(3)$ corresponden a aquellas
representaciones $\pi_\ell$ de $\SU(2)$ con $\ell \in 2\NN_0$. Por lo tanto, tenemos $\hat \SO(3)=\{[\pi_\ell]\}_{\ell\in2\NN_0}$,
 más aún, si $\pi=\pi_\ell$ es una tal representación irreducible de
$\mathrm{SO}(3)$, se sabe (ver \cite{H72}, página 32) que existe una base $\mathcal B=\vz{v_j}_{j=0}^\ell$ de $V_\pi$
tal que la correspondiente representación $\dot\pi$ de la complexificación $\so(3)$ está dada por
\begin{equation*}
\begin{split}
&\dot\pi(h)v_j=(\ell-2j)v_j,\\
&\dot\pi(e)v_j=(\ell-j+1)v_{j-1}, \quad (v_{-1}=0),\\
&\dot\pi(f)v_j=(j+1)v_{j+1}, \quad (v_{\ell+1}=0).
\end{split}
\end{equation*}

{ Es conocido (ver \cite{V92}, página 362) que una representación irreducible $\tau\in \hat\SO(4)$ tiene peso máximo de la forma $\eta = m_1\varepsilon_1+m_2\varepsilon_2$,}
 donde  $m_1$ y $m_2$ son enteros tales que $m_1\geq |m_2|$. Más aún, la representación $\tau=\tau_{(m_1,m_2)}$, restringida a $\SO(3)$, contiene la representación $\pi_\ell$  si y solo si
 $$m_1\geq\tfrac\ell2\geq |m_2|.$$

\subsection{$K$-Órbitas en $G/K$}
\

El grupo $G=\SO(4)$ actúa en una manera natural sobre la esfera $S^3$.
Esta acción es transitiva y $K$ es el subgrupo de isotropía del polo norte
$e_4=(0,0,0,1)\in S^3$. Por lo tanto, $ S^3\simeq G/K.$
Más aún, la $G$-acción en $S^3$
corresponde a la acción inducida por multiplicación a izquierda en $G/K$.

En el hemisferio norte de $S^3$
 $$(S^3)^{+}=\vzm{x=(x_1,x_2,x_3,x_4)\in S^3}{x_4>0},$$ consideramos el sistema de coordenadas
{ $p:(S^3)^{+}\longrightarrow \RR^3$ dado por la proyección central de la esfera sobre su plano tangente en el polo norte (Ver Figura \ref{fig1}):}
\begin{equation}\label{pfunction}
p(x)=\left(\frac{x_1}{x_4},\frac{x_2}{x_4},\frac{x_3}{x_4}\right)=(y_1,y_2,y_3).
 \end{equation}

\begin{figure}
 \centering
\includegraphics{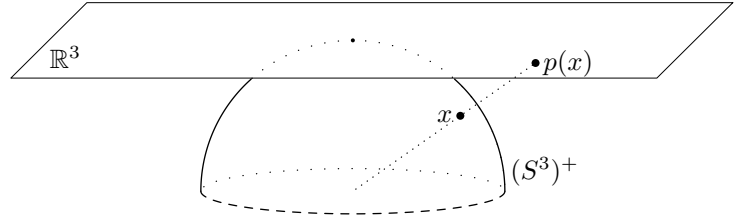}
 \caption{La proyección central $p$.}
\label{fig1}
\end{figure}

{ Coordenadas homogéneas fueron usadas antes en el caso de plano proyectivo complejo, ver \cite{GPT02a}.}
El mapa coordenado $p$ lleva las $K$-órbitas en $(S^3)^+$  a las $K$-órbitas en $\RR^3$, las cuales son las esferas
 $$S_r=\vzm{(y_1,y_2,y_3)\in \RR^3}{|| y||^2=|y_1|^2+|y_2|^2+|y_3|^2=r^2}, \qquad 0\leq r <\infty.$$
Luego, el intervalo $[0,\infty)$ parametriza el conjunto de $K$-órbitas de $\RR^3$.

\subsection{La Función Auxiliar $\Phi_\pi$}\label{auxiliar}
\

Similarmente a \cite{GPT02a}, para determinar todas las funciones esféricas irreducibles $\Phi$ de tipo $\pi=\pi_\ell\in \hat K$ introducimos una función auxiliar $\Phi_\pi: G\longrightarrow \End(V_\pi)$.
En este caso la definimos de la siguiente manera
\begin{equation*}
\Phi_\pi(g)= \pi(a(g)), \qquad g\in G,
\end{equation*}
donde $a$ es el homomorfismo de Lie de $\SO(4)$ a $\SO(3)$ dado  en \eqref{funcionq}.
Es claro que $\Phi_\pi$ es una representación irreducible de $\SO(4)$ y por lo tanto una función esférica de tipo $\pi$ (ver Definición \ref{defesf}).

\section{Los Operadores Diferenciales $D$ y $E$}
\

Determinar todas las funciones esféricas irreducibles en $G$ de tipo  $\pi\in \hat K$ es equivalente a determinar todas las funciones analíticas
$\Phi:G\longrightarrow \End(V_\pi)$ tales que
\begin{enumerate}
\item[i)] $\Phi(k_1gk_2)=\pi(k_1)\Phi(g)\pi(k_2)$, para
todo $k_1,k_2\in K$, $g\in G$, y $\Phi(e)=I$.
\item[ii)] $[\Delta_1\Phi ](g)=\widetilde\lambda\Phi(g)$, $[\Delta_2\Phi ](g)=\widetilde\mu\Phi(g)$ para todo
$g\in G$ y para ciertos $\widetilde\lambda,\widetilde\mu\in \CC$.
\end{enumerate}

En lugar de estudiar de manera directa la función esférica
$\Phi$ de tipo $\pi$, usamos una función auxiliar $\Phi_\pi$ para concentrarnos en la función $H:G\longrightarrow \End(V_\pi)$ definida por
\begin{equation}\label{defHg}
H(g)=\Phi(g)\Phi_\pi (g)^{-1}.
\end{equation}
Observemos que $H$ está bien definida en $G$ dado que $\Phi_\pi$ es una representación de $G$. Esta función $H$, asociada a la función esférica $\Phi$, satisface
\begin{enumerate}
\item [i)] $H(e)=I$.
\item [ii)] $ H(gk)=H(g)$, para todo $g\in G, k\in K$.
\item [iii)] $H(kg)=\pi(k)H(g)\pi(k^{-1})$, para todo
$g\in G, k\in K$.
\end{enumerate}

\smallskip
El hecho de que $\Phi$ sea autofunción de $\Delta_1 $ y $\Delta_2$ convierte a $H$ en una autofunción de ciertos operadores $D$ y $E$ en $G$ que determinaremos a continuación.
Definamos
\begin{align}
\label{Ddefuniv} D(H)&=Y_4^2(H)+Y_5^2(H)+Y_6^2(H),\\
\label{Edefuniv}
E(H)&=\big(-Y_4(H)Y_3(\Phi_\pi)+ Y_5(H)Y_2(\Phi_\pi)-Y_6(H)Y_1(\Phi_\pi) \big)\Phi_\pi^{-1}.
\end{align}

\smallskip

\begin{prop}\label{relacionautovalores}
 Para cualquier $H\in C^{\infty}(G)\otimes
\End(V_\pi)$ invariante a derecha por $K$, la función $\Phi=H\Phi_\pi$
satisface  $\Delta_1 \Phi=\widetilde\lambda \Phi$ y
$\Delta_2\Phi=\widetilde\mu \Phi$
 si y solo si $H$ satisface $DH=\lambda H$ y $E H=\mu H$, con
$$\lambda = -4\widetilde\lambda, \quad \mu = -\tfrac14\ell(\ell+2)+\widetilde\mu-\widetilde\lambda.$$
\end{prop}
\begin{proof}[\it Demostración]
Primero observemos que $Z_4(\Phi_\pi)=Z_5(\Phi_\pi)=Z_6(\Phi_\pi)=0$, ya que $\Phi_\pi$ es una representación de $G$ y $\dot a(Z_j)=0$ para $j=4,5,6$. En efecto,
\begin{align*}
 Z_j(\Phi_\pi)(g)&
 =\left.\tfrac{d}{dt} \right|_{t=0} \left[\Phi_\pi(g) \Phi_\pi(\exp tZ_j) \right] = \Phi_\pi(g)  \dot\pi(\dot a (Z_j))
 =0.
\end{align*}

Por otra parte, como $H$ es invariante a derecha por $K$, tenemos que  $Y_1(H)=Y_2(H)=Y_3(H)=0$.
Como $[Y_3,Y_4]=0$,  $[Y_2,Y_5]=0$ y  $[Y_1,Y_6]=0$, tenemos que $Z_j^2(H)=\frac 14 Y_j^2(H)$, para $j=4,5,6$.
Por lo tanto, obtenemos
\begin{align*}
\Delta_1(H\Phi_\pi)&=-\sum_{j=4}^{6}Z_j^2(H)\,\Phi_\pi=-\frac{1}{4}\sum_{j=4}^{6}Y_j^2(H)\,\Phi_\pi = -\frac{1}{4}D(H) \Phi_\pi.
\end{align*}

Además, tenemos
\begin{align*}
\Delta_2 (H\Phi_\pi)=-\sum_{j=1}^{3}\big(Z_j^2(H)\,\Phi_\pi+2Z_j(H)Z_j(\Phi_\pi)+HZ_j^2(\,\Phi_\pi)\big),
 \end{align*}
Observemos que $Z_1(H)= \frac 12 Y_4(H)$. Ya que $Z_1=Y_3-Z_4$, tenemos
$Z_1(\Phi_\pi)=Y_3(\Phi_\pi)$ y $Z_1^2(\Phi_\pi)=Y_3^2(\Phi_\pi)$. Resultados similares se obtienen para $Z_2$ y $Z_3$.
Por lo tanto,
\begin{align*}
\Delta_2 (H\Phi_\pi)&=-\big(Z_1^2(H)\,\Phi_\pi+Y_4(H)Y_3(\Phi_\pi)+HY_3^2(\,\Phi_\pi)\big) \\
 &\quad -\big(Z_2^2(H)\,\Phi_\pi-Y_5(H)Y_2(\Phi_\pi)+HY_2^2(\,\Phi_\pi)\big) \\
 &\quad -\big(Z_3^2(H)\,\Phi_\pi+Y_6(H)Y_1(\Phi_\pi)+HY_1^2(\,\Phi_\pi)\big) \\
 &=-\frac14 D(H) \,\Phi_\pi +E(H)\,\Phi_\pi + H \Delta_K(\Phi_\pi)\\
 &=-\frac14 D(H)\,\Phi_\pi +E(H)\,\Phi_\pi + H \Phi_\pi\dot\pi(\Delta_K).
 \end{align*}

Como $\Delta_K\in D(G)^K$, el Lema de Schur nos dice que $\dot\pi(\Delta_K)=cI$.
Ahora tenemos $\Delta_1(H\Phi_\pi)=\widetilde\lambda H\Phi_\pi$ y $\Delta_2(H\Phi_\pi)=\widetilde\mu H\Phi_\pi$ si y solo si
$D(H)=\lambda H$ y $E(H)=\mu H$, donde
$$\widetilde\lambda=-\frac{1}{4}\lambda  \qquad \text{ y } \qquad \widetilde\mu=c+\widetilde\lambda+\mu .  $$

Para calcular la constante $c$ tomamos un vector peso máximo $v \in V_\pi$, y escribimos $Y_1$, $Y_2$, $Y_3$ en términos de la base $\{e,f,g\}$ introducida en \eqref{efh}. Se tiene entonces que
\begin{align*}
\dot\pi (\Delta_K) v &=\dot\pi\left(-\big(\tfrac{-i}{2}(e+f)\big)^2-\big(\tfrac{-1}{2}(e-f)\big)^2-\big(\tfrac{i}{2}h\big)^2\right)  v \\
&= \frac{-1}{4}\dot\pi\left( - 2 e f - 2 f e -h^2 \right)v = \frac{1}{4}\dot\pi\left( 2(f e+h)+2 f e +h^2\right) v \\
&= \frac{1}{4}\left( 2\ell+\ell^2\right)v = \frac{\ell(\ell+2)}{4} \,v\,.
\end{align*}
Luego, $c={\ell(\ell+2)}/{4}$ completando la prueba de la proposición.
\end{proof}

{ \begin{remark}\label{DyEconmutan}
Observemos que los operadores diferenciales $D$ y $E$ conmutan. De hecho, de la Proposición \ref{relacionautovalores} tenemos que
  \begin{align*}
    D(H)& = -4\Delta_1(H\Phi_\pi)\Phi_\pi^{-1},\\
    E(H)& = \Delta_2(H\Phi_\pi) \Phi_\pi^{-1}+\tfrac 14 D(H) \Phi_\pi^{-1}-\tfrac{\ell(\ell+2)}2 H,
  \end{align*}
y además $\Delta_1$ y  $\Delta_2$ conmutan por estar en el centro del álgebra $D(G)$.
\end{remark}
}

\subsection{Reducción a $G/K$}\label{redG/K}
\

El cociente $G/K$ es la esfera $S^3$; es más, el difeomorfismo canónico está dado por $gK\mapsto (g_{14},g_{24}, g_{34}, g_{44})\in S^3$.

La función $H$ asociada a la función esférica $\Phi$ es invariante a derecha por $K$;
entonces, puede considerarse como función en $S^3$, y la seguiremos llamando $H$. Los operadores diferenciales $D$ y $E$ introducidos en \eqref{Ddefuniv} y \eqref{Edefuniv}
 definen operadores diferenciales en $S^3$.

\begin{lem} Los operadores diferenciales $D$ y $E$ en $G$ definen operadores diferenciales $D$
y $E$ actuando en $C^\infty( S^3)\otimes\End(V_\pi)$.
\end{lem}
\begin{proof}[\it Demostración]
Lo único que necesitamos demostrar es que $D$ y $E$ preservan el subespacio  $C^{\infty}(G)^K\otimes\End(V_\pi)$.

Dadas una función esférica irreducible $\Phi$ de tipo $\pi$ y la función $\Phi_\pi$ introducida en la Subsección \ref{auxiliar}, sea $H(g)=\Phi(g)\,\Phi_\pi^{-1}(g)$ y consideremos la función $r_k(g)=gk^{-1}$. Entonces
\begin{multline*}
 r_k^*(DH)(g)=r_k^*(\Delta(H\Phi_\pi))(g)\pi(k)^{-1}\Phi_\pi^{-1}(g)=\\
r_k^*(\Delta)(r_k^*(H\Phi_\pi))(g)\pi(k)^{-1}\Phi_\pi^{-1}(g)=\Delta(H\Phi_\pi)(g)\Phi_\pi^{-1}(g)=DH(g),
\end{multline*}
demostrando que $DH$ es $K$-invariante a derecha.
\end{proof}

Ahora daremos las expresiones de los operadores $D$ y $E$ en el sistema de coordenadas
{ $p:(S^3)^{+}\longrightarrow \RR^3$ introducido en \eqref{pfunction} y dado por}
\begin{equation*}
p(x)=\left(\frac{x_1}{x_4},\frac{x_2}{x_4},\frac{x_3}{x_4}\right)=(y_1,y_2,y_3).
 \end{equation*}
\begin{remark}\label{observacion} Dado $g\in G$ tal que $gK\in(S^3)^+$, si tomamos $p(g)=y=(y_1,y_2,y_3)$, tenemos que
$$g_{44}^{-2}=\|y\|^2+1 ,\qquad \frac{g_{14}}{g_{44}}=y_1,\qquad \frac{g_{24}}{g_{44}}=y_2,\qquad \frac{g_{34}}{g_{44}}=y_3.$$
\end{remark}

\begin{lem} \label{uvw} Dada $H\in C ^\infty(\RR^3)$, si denotamos también por  $H$ a la función definida por $H(g)=H(p(g)),$ para todo $g$ tal que $gK\in (S^3)^+$, se tiene
$$
(Y_4H)(g)= \frac{g_{11}g_{44}-g_{14}g_{41}}{g_{44}^2} H_{y_1}+\frac{g_{12}g_{44}
           -g_{14}g_{42}}{g_{44}^2}H_{y_2}+\frac{g_{13}g_{44}-g_{14}g_{43}}{g_{44}^2}        H_{y_3},
$$
$$
(Y_5H)(g)= \frac{g_{21}g_{44}-g_{24}g_{41}}{g_{44}^2} H_{y_1} +\frac{g_{22}g_{44}
              -g_{24}g_{42}}{g_{44}^2}H_{y_2}+\frac{g_{23}g_{44}-g_{24}g_{43}}{g_{44}^2}        H_{y_3},
$$
$$
(Y_6H)(g)= \frac{g_{31}g_{44}-g_{34}g_{41}}{g_{44}^2} H_{y_1}+ \frac{g_{32}g_{44}
              -g_{34}g_{42}}{g_{44}^2}H_{y_2}+\frac{g_{33}g_{44}-g_{34}g_{43}}{g_{44}^2}        H_{y_3}.
$$
\end{lem}

\begin{proof}[\it Demostración.]
Tenemos
$$p(g \exp tY_4)=\left( u_4(t) , v_4(t), w_4(t)   \right) $$
$$                             =\left( \frac{g_{11}\sin(t)+g_{14}\cos(t)}{g_{41}sin(t)+g_{44}cos(t)} ,
                                    \frac{g_{21}\sin(t)+g_{24}\cos(t)}{g_{41}sin(t)+g_{44}cos(t)} ,
                                      \frac{g_{31}\sin(t)+g_{34}\cos(t)}{g_{41}sin(t)+g_{44}cos(t)}  \right )                                .$$

\bigskip
 Entonces si denotamos por $u_4'$, $v_4'$ y $w_4'$ a las respectivas derivadas de $u_4$, de $v_4$ y de $w_4$ en $t=0$, tenemos

$$(Y_4H)(g)=\left(\frac {d} {dt} H(p(g \exp tY_5))\right)=  H_{y_1}  u_4'+ H_{y_2}v_4' + H_{y_3}w_4'.$$
 Además $$u_4'= \frac{g_{11}g_{44}-g_{14}g_{41}}{g_{44}^2} ,\qquad
                v_4'=\frac{g_{21}g_{44} -g_{24}g_{41}}{g_{44}^2}, \qquad
              w_4' = \frac{g_{31}g_{44}-g_{34}g_{41}}{g_{44}^2}.
              $$
 Similarmente, para $Y_5$ se tiene
$$p(g \exp tY_5) = \left( u_5(t) , v_5(t), w_5(t)   \right) ,$$
entonces
$$(Y_5H)(g)=\left(\frac {d} {dt} H(p(g \exp tY_5))\right)=  H_{y_1}  u_5'+ H_{y_2}v_5' + H_{y_3}w_5',$$
con              $$u_5'= \frac{g_{12}g_{44}-g_{14}g_{42}}{g_{44}^2}, \qquad
                v_5'=\frac{g_{22}g_{44} -g_{24}g_{42}}{g_{44}^2}, \qquad
              w_5' = \frac{g_{32}g_{44}-g_{34}g_{42}}{g_{44}^2}.
              $$

Y para $Y_6$ se tiene
$$p(g \exp tY_6) = \left( u_6(t) , v_6(t), w_6(t)   \right) .$$
Entonces
$$(Y_6H)(g)=\left(\frac {d} {dt} H(p(g \exp tY_6))\right)=  H_{y_1}  u_6'+ H_{y_3}v_6' + H_{y_3}w_6',$$
con              $$u_6'= \frac{g_{13}g_{44}-g_{14}g_{43}}{g_{44}^3}, \qquad
                v_6'=\frac{g_{23}g_{44} -g_{24}g_{43}}{g_{44}^3} ,\qquad
              w_6' = \frac{g_{33}g_{44}-g_{34}g_{43}}{g_{44}^3}.
              $$
\end{proof}

\begin{prop}\label{D3}
Para cualquier $H\in C^\infty( \RR^3)\otimes
\End(V_\pi)$ tenemos
\begin{align*}
D(H) (y) =&  (1+\left\|y\right\|^2)\Big((y_1^2+1) H_{y_1y_1}+(y_2^2+1) H_{y_2y_2}+(y_3^2+1) H_{y_3y_3}\\
 &+ 2(y_1y_2 H_{y_1y_2}+y_2y_3 H_{y_2y_3}+y_1y_3 H_{y_1y_3}) + 2(y_1 H_{y_1}+y_2 H_{y_2}+y_3 H_{y_3})
\Big).
\end{align*}
\end{prop}

\begin{proof}[\it Demostración]

Recordemos que  $$D(H)=(Y_4^2+Y_5^2+Y_6^2)(H).$$
Empezaremos por calcular $Y_4^2(H)(g)$, para $g$ tal que $gK\in(S^3)^+$:

$$Y_4^2(H)(g)=\left(\frac{d}{ds}\frac{d}{dt}H(p(g \exp((s+t)Y_5)))  \right)_{s=t=0}.$$

Sean $u_4$, $v_4$ y $w_4$ definidos por
\begin{align*}
  &(u_4(t,s),v_4(t,s),w_4(t,s))=        p(g (\exp(s+t)Y_4)) = \\
 &\left(\frac{g_{11} \sin(s+t)+g_{14}\cos(s+t)}{g_{41}sin(s+t)+g_{44}cos(s+t)} ,\frac{g_{21}\sin(s+t)+g_{24}\cos(s+t)}{g_{41}sin(s+t)+g_{44}cos(s+t)} ,
                                    \frac{g_{31}\sin(s+t)+g_{34}\cos(s+t)}{g_{41}sin(s+t)+g_{44}cos(s+t)} \right).
\end{align*}
Por lo tanto:
\begin{multline*}
 Y_4^2(H)(g)= \biggr(\frac{d}{ds}\frac{d}{dt}  H(u_4(t,s),v_4(t,s),w_4(t,s))\biggr)_{s=t=0} \\
= H_{y_1y_1}u_4'^2+H_{y_2y_2}v_4'^2+H_{y_3y_3}w_4'^2+H_{y_1}u_4''+H_{y_2}v_4''+H_{y_3}w_4'' .
\end{multline*}
Similarmente
\begin{multline*}
 Y_5^2(H)(g)= \biggr(\frac{d}{ds}\frac{d}{dt}  H(u_5(t,s),v_5(t,s),w_5(t,s))\biggr)_{s=t=0} \\
= H_{y_1y_1}u_5'^2+H_{y_2y_2}v_5'^2+H_{y_3y_3}w_5'^2+H_{y_1}u_5''+H_{y_2}v_5''+H_{y_3}w_5'' ,
\end{multline*}
\begin{multline*}
Y_6^2(H)(g)= \biggr(\frac{d}{ds}\frac{d}{dt}  H(u_6(t,s),v_6(t,s),w_6(t,s))\biggr)_{s=t=0} \\
= H_{y_1y_1}u_6'^2+H_{y_2y_2}v_6'^2+H_{y_3y_3}w_6'^2+H_{y_1}u_6''+H_{y_2}v_6''+H_{y_3}w_6''.
\end{multline*}
con
  \begin{align*}
& (u_5(t,s),v_5(t,s),w_5(t,s))= \\
&\left(\frac{g_{12} \sin(s+t)+g_{14}\cos(s+t)}{g_{42}sin(s+t)+g_{44}cos(s+t)} ,\frac{g_{22}\sin(s+t)+g_{24}\cos(s+t)}{g_{42}sin(s+t)+g_{44}cos(s+t)} ,                                    \frac{g_{32}\sin(s+t)+g_{34}\cos(s+t)}{g_{42}sin(s+t)+g_{44}cos(s+t)} \right),
\end{align*}
  \begin{align*}
 &(u_6(t,s),v_6(t,s),w_6(t,s))=\\
&\biggr(\frac{g_{13} \sin(s+t)+g_{14}\cos(s+t)}{g_{43}sin(s+t)+g_{44}cos(s+t)} ,\frac{g_{23}\sin(s+t)+g_{24}\cos(s+t)}{g_{43}sin(s+t)+g_{44}cos(s+t)} ,
                                    \frac{g_{33}\sin(s+t)+g_{34}\cos(s+t)}{g_{43}sin(s+t)+g_{44}cos(s+t)} \biggr).
 \end{align*}
Con lo cual tenemos

\begin{align*}D(H)=&   H_{y_1y_1}  \sum_{j=4}^{6}u_j'^2  +  H_{y_2y_2}\sum_{j=4}^{6}v_j'^2     +  H_{y_3y_3}\sum_{j=4}^{6}w_j'^2
                  +   H_{y_1}     \sum_{j=4}^{6}u_j''   +  H_{y_2}\sum_{j=4}^{6}v_j''         +  H_{y_3}\sum_{j=4}^{6}w_j'' \\
                  &+   H_{y_1y_2}  \sum_{j=4}^{6}u_j'v_j'+  H_{y_1y_3}  \sum_{j=4}^{6}u_j'w_j' +  H_{y_3y_2}  \sum_{j=4}^{6}w_j'v_j'.
\end{align*}

Además en el desarrollo de la demostración del Lema \ref{uvw} tenemos las expresiones de $u_j'$, $v_j'$ y $w_j'$, para $1\leq j \leq 3$. Teniendo en cuenta el hecho de que las filas son ortonormales entre sí (al igual que las columnas), resulta:

\begin{align*}
\sum_{j=4}^{6}u_j'^2&=\frac{g_{44}^2+g_{14}^2}{g_{44}^4}  ,&  \sum_{j=4}^{6}v_j'^2&=\frac{g_{44}^2+g_{24}^2}{g_{44}^4}  ,&   \sum_{j=4}^{6}w_j'^2 &=\frac{g_{44}^2+g_{34}^2}{g_{44}^4},\\
\sum_{j=4}^{6}u_j'v_j'&=\frac{g_{14}g_{24}}{g_{44}^4} ,&  \sum_{j=4}^{6}u_j'w_j'&=\frac{g_{14}g_{34}}{g_{44}^4},&
    \sum_{j=4}^{6}w_j'v_j'&=\frac{g_{34}g_{24}}{g_{44}^4}.
\end{align*}
O sea
\begin{align*}
\sum_{j=4}^{6}u_j'^2&=(1+y_1^2)(1+|y|^2)  ,&   \sum_{j=4}^{6}v_j'^2&=(1+y_2^2)(1+|y|^2)  ,&    \sum_{j=4}^{6}w_j'^2 &=(1+y_3^2)(1+|y|^2),\\
\sum_{j=4}^{6}u_j'v_j'&=y_1y_2(1+|y|^2) ,&   \sum_{j=4}^{6}u_j'w_j'&=y_1y_3(1+|y|^2),&     \sum_{j=4}^{6}w_j'v_j'&=y_3y_2(1+|y|^2) .
\end{align*}
Por otro lado, se obtiene que

$$u_j''= 2\frac{g_{14}}{g_{44}^3}   ,\qquad  v_j''=2\frac{g_{24}}{g_{44}^3}  ,\qquad    w_j''=2\frac{g_{34}}{g_{44}^3} .$$
Y así queda
$$u_j''= 2y_1(1+|y|^2)   ,\qquad  v_j''=2y_2(1+|y|^2)  ,\qquad    w_j''=2y_2(1+|y|^2) .$$
Y la proposición se ha demostrado.
\end{proof}

\begin{prop}\label {E3}
Para cualquier $H\in C^\infty( \RR^3)\otimes
\End(V_\pi)$
tenemos
\begin{multline*}
E(H)(y)=H_{y_1}\dot\pi \left(\begin{smallmatrix} 0 &-y_2-y_1y_3 &-y_3+y_1y_2 \\y_2+y_1y_3 &0 &-1-y_1^2 \\ y_3-y_1y_2&1+y_1^2 &0 \end{smallmatrix}\right)
 +H_{y_2}\dot\pi \left(\begin{smallmatrix} 0 &-y_2y_3+y_1 &1+y_2^2 \\y_2y_3-y_1 & 0 & -y_3-y_1 y_2 \\-1-y_2^2 & y_3+y_1y_2&0 \end{smallmatrix}\right)
\\+H_{y_3}\dot\pi \left(\begin{smallmatrix} 0 &-1-y_3^2 &y_1+y_2y_3 \\1+y_3^2 & 0 & y_2-y_1 y_3 \\ -y_1-y_2y_3 & -y_2+y_1y_3 &0 \end{smallmatrix}\right).
\end{multline*}
\end{prop}
\begin{proof}[\it Demostración]

Recordemos $$E(H)=(-Y_4(H)Y_3(\Phi_\pi)+ Y_5(H)Y_2(\Phi_\pi)-Y_6(H)Y_1(\Phi_\pi) )\Phi_\pi^{-1}.$$
 Usando el Lema \ref{uvw} tenemos
\begin{multline*}
E(H)=H_{y_1}\left(-u_4'Y_3(\Phi_\pi)+u_5'Y_2(\Phi_\pi)-u_6'Y_1(\Phi_\pi)\right)\\
    +H_{y_2}\left(-v_4'Y_3(\Phi_\pi)+v_5'Y_2(\Phi_\pi)-v_6'Y_1(\Phi_\pi)\right)
    +H_{y_3}\left(-w_4'Y_3(\Phi_\pi)+w_5'Y_2(\Phi_\pi)-w_6'Y_1(\Phi_\pi)\right).
\end{multline*}

Ahora, considerando el hecho de que
\begin{align*}
Y_j(\Phi_\pi)(g)=&\frac{d}{dt}_{t=0}(\Phi\pi(g\exp t Y_j))= \frac{d}{dt}_{t=0} \pi \circ \tau \circ q(g\exp t Y_j) )\\
=&\left( \frac{d}{dt}_{t=0} \pi \circ \tau \circ q(g\exp t Y_jg^T)  \right)\Phi_\pi(g)=\biggr(\dot\pi \circ \dot\tau \circ \dot q(gY_jg^T)  \biggr)\Phi_\pi(g) ,
\end{align*}
y que $g_{ij} (-g_{44})^{i+j}= \det g(i/j)$, tenemos que
 \begin{align*}
-u_4'Y_3(\Phi_\pi)+u_5'Y_2(\Phi_\pi)-u_6'Y_1(\Phi_\pi) =\frac{1}{g_{44}^2} \left(\begin{smallmatrix} 0 &-g_{24}g_{44}-g_{14}g_{34} &-g_{34}g_{44}+g_{14}g_{24} \\g_{24}g_{44}+g_{14}g_{34} &0 &-g_{44}^2-g_{14}^2 \\ g_{34}g_{44}-g_{24}g_{14}&g_{44}^2+g_{14}^2 &0 \end{smallmatrix}\right),\\
 -v_4'Y_3(\Phi_\pi)+v_5'Y_2(\Phi_\pi)-v_6'Y_1(\Phi_\pi) =\frac{1}{g_{44}^2}\left(\begin{smallmatrix} 0 &-g_{24}g_{34}+g_{14}g_{44} &g_{44}^2+g_{24}^2 \\g_{24}g_{34}-g_{14}g_{44} & 0 & -g_{34}g_{44}-g_{14} g_{24} \\-g_{44}^2-g_{24}^2 & g_{34}g_{44}+g_{14}g_{24}&0 \end{smallmatrix}\right),\\
 -w_4'Y_3(\Phi_\pi)+w_5'Y_2(\Phi_\pi)-w_6'Y_1(\Phi_\pi) =\frac{1}{g_{44}^2}\left(\begin{smallmatrix} 0 &-g_{44}^2-g_{34}^2 &g_{14}g_{44}+g_{24}g_{34} \\g_{44}^2+g_{34}^2 & 0 & g_{24}g_{44}-g_{14} g_{34} \\ -g_{14}g_{44}-g_{24}g_{34} & -g_{24}g_{44}+g_{14}g_{34} &0 \end{smallmatrix}\right).
 \end{align*}
Entonces por la Observación \ref{observacion}, la proposición está probada.
\end{proof}

\smallskip

\subsection{Reducción a una Variable}\label{onevariable}
\

Estamos interesados en considerar los operadores diferenciales $D$ y $E$ dados en las Proposiciones \ref{D3} y \ref{E3}
aplicados a funciones $H\in C^\infty(\RR^3)\otimes
\End(V_\pi)$ tales que  $$H(ky)=\pi(k)H(y)\pi(k)^{-1}, \qquad  \quad \text {para todo }k\in K , y\in\RR^3.$$
{ Por lo tanto, la función $H=H(y)$ es determinada por su restricción a la sección de $K$-órbitas en $\RR^3$. Recordemos que las $K$-órbitas en $\RR^3$ son las esferas
$$S_r=\vzm{(y_1,y_2,y_3)\in \RR^3}{|| y||^2=|y_1|^2+|y_2|^2+|y_3|^2=r^2}, \qquad 0\leq r <\infty.$$
En cada órbita $S_r$ elegimos el punto $(r,0,0) \in \RR^3$ como representante.

Esto nos permite encontrar operadores diferenciales ordinarios $\widetilde
D$} y $\widetilde E$ definidos en el intervalo
$(0,\infty)$ tales que  $$ (D\,H)(r,0,0)=(\widetilde D\widetilde H)(r),\qquad
(E\,H)(r,0,0)=(\widetilde E\widetilde H)(r),$$ donde { $\widetilde
H(r)=H(r,0,0)$.
}

{
\begin{remark}\label{DyEconmutan2}
  Notar que los operadores diferenciales $\widetilde D$ y $\widetilde E$ conmutan dado que ellos son las restricciones de los operadores diferenciales $D$ y $E$, los cuales conmutan (ver Observación \ref{DyEconmutan}).
\end{remark}
}
Para conseguir expresiones explícitas de los operadores diferenciales $\widetilde D$ y $\widetilde E$ comenzando desde las Proposiciones \ref{D3} y \ref{E3} necesitamos
computar una cantidad de derivadas parciales de segundo orden de la función  $H:\RR^3\longrightarrow \End(V_\pi)$ en los puntos $(r,0,0)$, para $r>0$.
Dado $y=(y_1,y_2,y_3)\in \RR^3$ en un entono de $(r,0,0)$, $r>0$, necesitamos una función de $K=\SO(3)$ que lleve el punto $y$ al meridiano $\{(r,0,0)\, : r>0\}$. Una buena elección es la siguiente función:
\begin{equation}\label{matrixA}
A(y)=\frac{1}{\left\|y\right\|} \left( \begin{matrix} y_1 &-y_2 & -y_3 \\y_2 & \frac{-y_2^2}{\left\|y\right\|+y_1}+ \left\|y\right\| & \frac{-y_2 y_3}{\left\|y\right\|+y_1} \\y_3 & \frac{-y_2 y_3}{\left\|y\right\|+y_1} & \frac{-y_3^2}{\left\|y\right\|+y_1}+ \left\|y\right\| \end{matrix}\right).
\end{equation}
Entonces
$$y=A(y)(\left\|y\right\|,0,0)^t.$$

Es fácil verificar que $A(y)$ es una matriz en $\SO(3)$ y que está bien definida en $\RR^3\smallsetminus\vzm{(y_1,0,0)\in\RR^3}{y_1\leq0}$.

Consideremos los siguientes elementos en $\liek$
\begin{equation}\label{As}
A_1=E_{21}-E_{12},\qquad A_2=E_{31}-E_{13},\qquad A_3=E_{32}-E_{23}.
\end{equation}

\begin{prop}\label{derx1}
Para $r>0$ tenemos
\begin{align*}
\frac{\partial H}{\partial y_1}(r,0,0)=\frac{d\widetilde H}{dr}(r) ,\quad
\frac{\partial H}{\partial y_2}(r,0,0)= \frac1r\left[\dot \pi \left(A_1\right) , \widetilde H(r) \right], \quad
\frac{\partial H}{\partial y_3}(r,0,0)= \frac1r\left[\dot \pi \left(A_2\right) , \widetilde H(r) \right].
\end{align*}
\end{prop}
\begin{proof}[\it Demostración]
Hallamos que $$H(y_1,y_2,y_3)=H \,(A(y) (\|y\|,0,0))= \pi(A(y)) \,\,\widetilde H(\|y\|)\,\, \pi(A(y))^{-1}.$$
Entonces,
\begin{equation}\label{H_{yj}}
H_{y_j}=\frac{d (\pi\circ A)}{dy_j} \widetilde H(\|y\|)  (\pi\circ A)+(\pi\circ A^{-1}) \frac{d (\widetilde H(\|y\|))}{dy_j}  (\pi\circ A^{-1})\\
              +(\pi\circ A) \widetilde H(\|y\|) \frac{d (\pi\circ A^{-1})}{dy_j}                                     .
\end{equation}
Además
$$\frac{d (\widetilde H(\|y\|))}{dy_j}=\left(\frac{d (\widetilde H)}{dr}\circ \|y\|\right)\,\,\,\, \frac{d(\|y\|)}{dy_j}.
$$
Notemos que $A(r,0,0)=I$ y que en $(r,0,0)$ nos encontramos con
$$\frac{d(\|y\|)}{dy_j}=\delta_{1j} \, ,\qquad \frac{d (\pi\circ A)}{dy_j}= \dot\pi \left(\frac{dA}{dy_j}\right) \,, \qquad   \frac{d (\pi\circ A^{-1})}{dy_j}=- \dot\pi \left(\frac{dA}{dy_j}\right), $$
Como $ \frac{dA}{dy_1} ,  \frac{dA}{dy_2}  $ y $ \frac{dA}{dy_3}  $ en $(r,0,0)$ son la matriz nula, $\tfrac1rA_2 $ y $\tfrac1rA_3 $ respectivamente, la proposición sigue.
\end{proof}

\begin{prop}\label{derx1x1}
 Para $r>0$ tenemos
\begin{align*}
\frac{\partial² H}{\partial y_1²}(r,0,0) =&\frac{d^2\widetilde H}{dr^2}(r),\displaybreak[0]\\
\frac{\partial² H}{\partial y_2²}(r,0,0)=&\frac{1}{r²}\bigg( r\frac{d\widetilde H}{dr}+\dot \pi \left(A_1\right)^2 \widetilde H(r)+\widetilde H(r) \dot \pi \left(A_1\right)^2
-2  \dot \pi \left(A_1\right) \widetilde H(r)\dot \pi \left(A_1\right)\bigg),\\
\frac{\partial² H}{\partial y_3²}(r,0,0)=&\frac{1}{r²} \bigg(r\frac{d\widetilde H}{dr}+\dot \pi \left(A_2\right)^2 \widetilde H(r)+\widetilde H(r) \dot \pi \left(A_2\right)^2
-2  \dot \pi \left(A_2\right) \widetilde H(r)\dot \pi \left(A_2\right)\bigg).
\end{align*}
 \end{prop}

Esta proposición es consecuencia directa de los siguientes dos lemas, cuyas pruebas son análogas a las del caso de plano proyectivo complejo considerado en \cite{GPT02a} (ver las Proposiciones 13.2 y 13.3 en ese trabajo).

\begin{lem}\label{y_jy_j} En $(r,0,0)\in\RR^3$ se tiene
\begin{multline*}
\frac{d^2H}{dy_j^2}(r,0,0)=\frac{d^2(\pi\circ A)}{dy_j^2} \widetilde H(r)+\widetilde H(r)\frac{d^2(\pi\circ A^{-1})}{dy_j^2}+\delta_{j1}\frac{d^2\widetilde H}{dy_j^2}+\frac{1}{r} (1-\delta_{j1})\frac{d\widetilde H}{dr}
\\-2\dot\pi\left(\frac{dA}{dy_j}\right)\widetilde H \left(r\right)\dot\pi\left(\frac{dA}{dy_j}\right).
\end{multline*}
\end{lem}

\begin{proof}[\it Demostración]
Primero recordemos que en $(r,0,0)$ $A_{y_1}=0$ y observemos que
$$\frac{d^2(H(|y|,0,0))}{dy_j^2}=\frac{d^2\widetilde H}{dy_j^2}\left(\frac{d|y|}{dy_j}\right)^2+  \frac{d\widetilde H}{dr} \frac{d^2|y|}{dr^2}
=\delta_{j1}\frac{d^2\widetilde H}{dy_j^2}+\frac{1}{r}(1-\delta_{1j})\frac{d\widetilde H}{dr}. $$
Ahora derivando \eqref{H_{yj}} con respecto a $y_j$ y evaluando en $(r,0,0)\in\RR^3$ se tiene

$$\frac{d^2H}{dy_j^2}(r,0,0)=\frac{d^2\pi\circ A}{dy_j^2} \widetilde H(r)+\widetilde H(r)\frac{d^2\pi\circ A^{-1}}{dy_j^2}+\delta_{j1}\frac{d^2\widetilde H}{dy_j^2}
+\frac{1}{r}(1-\delta_{1j})\frac{d\widetilde H}{dr}-2\dot\pi(\frac{dA}{dy_j})\widetilde H (r)\dot\pi(\frac{dA}{dy_j}).$$
Como $ \frac{dA}{dy_1} ,  \frac{dA}{dy_2}  $ y $ \frac{dA}{dy_3}  $ en $(r,0,0)$ son la matriz nula, $\tfrac1rA_2 $ y $\tfrac1rA_3 $ respectivamente, el lema sigue.
\end{proof}

\begin{lem}En $(r,0,0)\in\RR^3$ valen
\begin{equation}\label{A}
\frac{d^2(\pi\circ A)}{dy_j^2}=\dot\pi\left(\frac{d^2A}{dy_j^2}\right)-\dot\pi\left(\frac{dA}{dy_j}\frac{dA}{dy_j}\right) + \dot\pi\left(\frac{dA}{dy_j}\right)\dot\pi\left(\frac{dA}{dy_j}\right),
\end{equation}
\begin{equation}\label{B}
\frac{d^2(\pi\circ A^{-1})}{dy_j^2}=-\dot\pi\left(\frac{d^2A}{dy_j^2}\right)+\dot\pi\left(\frac{dA}{dy_j}\frac{dA}{dy_j}\right) + \dot\pi\left(\frac{dA}{dy_j}\right)\dot\pi\left(\frac{dA}{dy_j}\right).
\end{equation}
\end{lem}
\begin{proof}[\it Demostración]
Para $y_2$ y $y_3$ suficientemente pequeños consideramos
$$X(y)=\log(A(y))=B(y)-\frac{B(y)^2}{2}+\frac{B(y)^3}{3}-\dots,$$
donde $B(y)=A(y)-I$. Luego
$$\pi(A(y))=\pi(\exp(X(y)))=\exp\dot\pi(X(y))=\sum_{j\geq 0} \frac{\dot\pi(X(y))^j}{j \, !}          ,$$
ahora, si derivamos con respecto a $y_j$ obtenemos
\begin{align*}
\frac{d(\pi \circ A)}{dy_j} =& \dot\pi \left(\frac{dX}{dy_j}\right)+\frac{1}{2}\dot\pi\left(\frac{dX}{dy_j}\right)\dot\pi(X) + \frac{1}{2} \dot\pi(X) \left(\frac{dX}{dy_j}\right) \\
&+ \frac{1}{3\,!}\dot\pi \left(\frac{dX}{dy_j}\right) \dot\pi(X)^2+\frac{1}{3} \dot\pi(X) \dot\pi \left(\frac{dX}{dy_j}\right) \dot\pi(X)
+\frac{1}{3!} \dot\pi(X)^2\dot\pi \left(\frac{dX}{dy_j}\right) + \dots .
\end{align*}
Notemos que $X(r,0,0)=0$, así que si derivamos y evaluamos en $(r,0,0)$ nos queda
$$\frac{d^2(\pi \circ A)}{dy_j^2} = \dot\pi \left(\frac{d^2X}{dy_j^2}\right)+\dot\pi\left(\frac{dX}{dy_j}\right)\dot\pi\left(\frac{dX}{dy_j}\right).
$$

Como $B(r,0,0)=0$ tenemos
\begin{align*}
\frac{dX}{dy_j}(r,0,0)=&\frac{dB}{dy_j}(r,0,0)=\frac{dA}{dy_j}(r,0,0),\\
\frac{d^2X}{dy_j^2}(r,0,0)=&\frac{d^2A}{dy_j^2}(r,0,0)-\frac{dA}{dy_j}(r,0,0)\frac{dA}{dy_j}(r,0,0),
\end{align*}
por lo tanto \eqref{A} queda probado.

Para \eqref{B} observemos que
$$ \frac{d(\pi\circ A^{-1})}{dy_j}=-(\pi\circ A^{-1}) \frac{d(\pi\circ A)}{dy_j} (\pi\circ A^{-1}),$$
por lo tanto en $(r,0,0)\in\RR^3$ queda
$$ \frac{d^2(\pi\circ A^{-1})}{dy_j^2}=-\frac{d(\pi\circ A^{-1})}{dy_j}  \frac{d(\pi\circ A)}{dy_j} -  \frac{d^2(\pi\circ A)}{dy_j^2}- \frac{d(\pi\circ A)}{dy_j} \frac{d(\pi\circ A^{-1})}{dy_j}. $$
Entonces se ha probado también \eqref{B}.

\end{proof}

Ahora podemos obtener las expresiones explícitas de los operadores diferenciales  $\tilde D$ y $\tilde E$.

\begin{thm}\label{D} Para  $r>0$ tenemos
\begin{align*}\widetilde D(\widetilde H)(r)&=(1+r^2)^2\frac{d^2 \widetilde H}{dr^2}
+2\frac{(1+r^2)^2}{r} \frac{d\widetilde H}{dr}\displaybreak[0] \\
 & \quad +\frac{(1+r^2)}{r^2} \left( \dot\pi(A_1)^2 \,\widetilde H(r)\,+\widetilde H(r)\dot\pi(A_1)^2 -2\dot\pi(A_1)\widetilde H(r)\dot\pi(A_1)\right)\\
 & \quad +\frac{(1+r^2)}{r^2} \left( \dot\pi(A_2)^2 \,\widetilde H(r)\,+\widetilde H(r)\dot\pi(A_2)^2 -2\dot\pi(A_2)\widetilde H(r)\dot\pi(A_2)\right).
\displaybreak[0]
\end{align*}
\end{thm}

\begin{proof}[\it Demostración]
  Como $\widetilde D( \widetilde H) (r)=D(H) (r,0,0)$, por la Proposición \ref{D3} tenemos
 \begin{align*}
\widetilde D (\widetilde H )(r) =  (1+r^2)\Big((1+r^2)H_{y_1y_1}+H_{y_2y_2}+H_{y_3y_3}+ 2rH_{y_1} \Big).	
\end{align*}
Usando las Proposiciones \ref{derx1} y \ref{derx1x1} llegamos a que
\begin{align*}
\widetilde D (\widetilde H )(r) =&  (1+r^2)\Bigg[(1+r^2)\frac{d^2\widetilde H}{dr^2}(r)+2r\frac{d\widetilde H}{dr}(r) +\frac2r\frac{d\widetilde H}{dr}\\
&+\frac{1}{r²}\bigg(\dot \pi \left(A_1\right)^2 \widetilde H(r)+\widetilde H(r) \dot \pi \left(A_1\right)^2 -2  \dot \pi \left(A_1\right) \widetilde H(r)\dot \pi \left(A_1\right)\bigg)\\
&+\frac{1}{r²} \bigg(\dot \pi \left(A_2\right)^2 \widetilde H(r)+\widetilde H(r) \dot \pi \left(A_2\right)^2-2  \dot \pi \left(A_2\right) \widetilde H(r)\dot \pi \left(A_2\right)\bigg)
\Bigg].
\end{align*}
Ahora el teorema sigue fácilmente.

\end{proof}

\begin{thm}\label{E}    Para  $r>0$ tenemos
\begin{align*}
\widetilde E(\widetilde H)(r)& =
\frac{d\widetilde H}{dr}(1+r²)\dot\pi(A_3)
-\frac1r\left[\dot \pi (A_1) , \widetilde H(r) \right ]  \dot\pi\left(rA_1+A_2\right)+\frac1r\left[\dot \pi (A_2) , \widetilde H(r) \right ]   \dot\pi (A_1-rA_2).
\end{align*}
\end{thm}
\begin{proof}[\it Demostración]

Como $\widetilde E( \widetilde H) (r)=E(H) (r,0,0)$, por la Proposición \ref{E3} tenemos

\begin{align*}
&\widetilde E(\widetilde H)(r)=H_{y_1}\dot\pi \left(\begin{smallmatrix} 0 &0&0 \\0 &0 &-1-r^2 \\ 0&1+r^2 &0 \end{smallmatrix}\right) +H_{y_2}\dot\pi \left(\begin{smallmatrix} 0 &r &1 \\-r & 0 & 0 \\-1 & 0&0 \end{smallmatrix}\right)
+H_{y_3}\dot\pi \left(\begin{smallmatrix} 0 &-1 &r \\1 & 0 & 0 \\ -r & 0 &0 \end{smallmatrix}\right).
\end{align*}
Ahora, por la Proposición \ref{derx1} llegamos a que

\begin{align*}
\widetilde E(H)(r)=&\frac{d\widetilde H}{dr}(1+r^2)\dot\pi \left(A_3\right)
-\frac1r\left[\dot \pi \left(A_1\right) , \widetilde H(r) \right ]\dot\pi \left(rA_1+A_2\right)
\\&+\frac1r\left[\dot \pi \left(A_2\right) , \widetilde H(r) \right ]\dot\pi \left(A_1-rA_2\right),
\end{align*}
que es el enunciado del teorema.
\end{proof}

\medskip

Los Teoremas \ref{D} y \ref{E} están dados en términos de transformaciones lineales. Ahora daremos el correspondiente enunciado en términos de matrices escogiendo una apropiada base de $V_\pi$.
 Tomamos el $\mathfrak{sl}(2)$-triple $\{e,f,h\}$ en $\mathfrak k_\CC \simeq \mathfrak{sl}(2,\CC)$ introducido en \eqref{efh}.

 Si  $\pi=\pi_\ell$ es la única representación irreducible de
$\mathrm{SO}(3)$ con peso máximo $\ell/2$, como ya mencionamos en la Subsección \ref{representaciones}, existe una base $\mathcal B=\vz{v_j}_{j=0}^\ell$ de $V_\pi$
tal que
\begin{equation}\label{basedeVpi}
\begin{split}
&\dot\pi(h)v_j=(\ell-2j)v_j,\\
&\dot\pi(e)v_j=(\ell-j+1)v_{j-1}, \quad (v_{-1}=0),\\
&\dot\pi(f)v_j=(j+1)v_{j+1}, \quad (v_{\ell+1}=0).
\end{split}
\end{equation}

\begin{prop}\label{Hdiagonal}
La función $\widetilde H$ asociada a una función esférica irreducible $\Phi$ de tipo $\pi\in\hat K$ diagonaliza simultáneamente en la base $\mathcal B=\vz{v_j}_{j=0}^\ell$ de $V_\pi$.
\end{prop}
\begin{proof}[\it Demostración]
Consideremos el subgrupo $M=\vzm{m_\theta}{\theta\in \RR}$ de $K$, donde
\begin{equation}\label{Msubgrupo}
  m_\theta=\left(\begin{matrix} 1& 0&0  \\0&\cos\theta & \sin\theta\\0& -\sin\theta & \cos\theta  \\
\end{matrix}\right).
\end{equation}
Entonces, $M$ es isomorfo a $\SO(2)$ y fija los puntos $(r,0,0)$ en $\RR^3$. Además, como la función $H$ satisface $H(kg)=\pi(k)  H(g)\pi(k^{-1})$ para todo $k\in K$,
 tenemos que
\begin{align*}
\widetilde H(r)= H(r,0,0)=H(m_\theta(r,0,0)^t)=\pi(m_\theta)H(r,0,0)\pi(m_\theta^{-1})=\pi(m_\theta)\widetilde H(r)\pi(m_\theta^{-1}).
 \end{align*}
Luego, $\widetilde H(r)$ y $\pi(m_\theta)$ conmutan para cada $r$ en $\RR$ y cada $m_\theta$ en $M$.

Por otra parte, notemos que $m_\theta=\exp (\theta\frac i2h)$ y entonces $\pi(m_\theta)=\exp (\dot\pi(\theta\frac i2h))$, pero
de (\ref{basedeVpi}) sabemos que $\dot\pi(h)$ diagonaliza y que sus autovalores tienen multiplicidad uno. Por lo tanto, la función $\widetilde H(r)$ diagonaliza simultáneamente en la base $\mathcal B=\vz{v_j}_{j=0}^\ell$ de $V_\pi$.
\end{proof}

Ahora introducimos las funciones coordenadas $\widetilde h_j(r)$ a través de
\begin{equation}\label{hj}
 \widetilde H(r)v_j= \widetilde h_j(r)v_j,
\end{equation}
e identificamos $\widetilde H$ con  el vector columna
\begin{equation}\label{Ht}
\widetilde H(r)=(\widetilde h_0(r),\dots,\widetilde h_\ell(r))^t.
\end{equation}

\begin{cor}\label{sistema1}
Las funciones  $\widetilde H(r)$, $0<r<\infty$, satisfacen
$(\widetilde D\widetilde H)(r)=\lambda \widetilde H(r)$ si y solo si \begin{align*}
\textstyle(1+r^2)^2\widetilde  h_j'' +2\textstyle\frac{(1+r^2)^2}{r}\widetilde h_j'+\textstyle\frac{1+r^2}{r^2} (j+1)(\ell-j)(\widetilde h_{j+1}-\widetilde h_j)
 +\textstyle\frac{1+r^2}{r^2}j(\ell-j+1)(\widetilde h_{j-1}-\widetilde h_{j}) = \lambda \widetilde h_j ,
\end{align*}
para todo $j=0,\dots, \ell$.
\end{cor}
\begin{proof}[\it Demostración]
Usando la base $\mathcal B=\vz{v_j}_{j=0}^\ell$ de $V_\pi$ (ver (\ref{basedeVpi})) y escribiendo las matrices $A_1$ y $A_2$ en términos
 del  $\mathfrak{sl}(2)$-triple $\{e,\,f,\,h\}$, ver (\ref{efh}),
$$A_1=E_{21}-E_{12}=\, \frac i2({e+f})\; ,
\qquad A_2=E_{31}-E_{13}=\frac12({e-f}),$$
tenemos que el Teorema \ref{D} dice que $(\widetilde D\widetilde H)(r)=\lambda \widetilde H(r)$ si y solo si

\begin{align*}
\lambda \widetilde H(r)&\,v_j=
(1+r^2)^2 \widetilde H''(r)\,v_j +2\frac{(1+r^2)^2}{r} \widetilde H'(r)\displaybreak[0]\,v_j \\
 & -\frac{(1+r^2)}{4r^2} \left( \dot\pi({e+f})^2 \,\widetilde H(r)\,+\widetilde H(r)\dot\pi({e+f})^2 -2\dot\pi({e+f})\widetilde H(r)\dot\pi({e+f})\right)\,v_j\\
 & +\frac{(1+r^2)}{4r^2} \left( \dot\pi({e-f}{})^2 \,\widetilde H(r)\,+\widetilde H(r)\dot\pi({e-f}{})^2 -2\dot\pi({e-f}{})\widetilde H(r)\dot\pi({e-f}{})\right)\,v_j,
\displaybreak[0]
\end{align*}
para $0\leq j\leq\ell$.
Como $[e,f]=h$, tenemos que esto es equivalente a
\begin{align*}
\lambda \widetilde H(r)v_j=\; &
(1+r^2)^2 \widetilde H''(r)\,v_j +2\frac{(1+r^2)^2}{r} \widetilde H'(r)\displaybreak[0]\,v_j &\\
 & -\frac{(1+r^2)}{2r^2} \bigg[ ( \dot\pi({h})+2\dot\pi({f})\dot\pi({e}) \, )\widetilde H(r)\,v_j\,+\widetilde H(r) (\dot\pi({h})+2\dot\pi({f})\dot\pi({e}) )\,v_j\\
&-2 (\dot\pi({e})\widetilde H(r)\dot\pi({f})+\dot\pi({f})\widetilde H(r)\dot\pi({e}) )\,v_j\bigg],
\end{align*}
para $0\leq j\leq\ell$.
Ahora, usando (\ref{basedeVpi}), obtenemos que $(\widetilde D\widetilde H)(r)=\lambda \widetilde H(r)$ si y solo si
\begin{align*}
\lambda \widetilde h_j(r)\,v_j= &(1+r^2)^2 \widetilde h_j''(r)\,v_j +2\frac{(1+r^2)^2}{r} \widetilde h_j'(r)\displaybreak[0]\,v_j  -\frac{(1+r^2)}{2r^2} \bigg[( (\ell-2j)\\&+2j(\ell-j+1) \, )\widetilde h_j(r)\,v_j\,+\widetilde h_j(r) ((\ell-2j)+2j(\ell-j+1) )\,v_j      \\
           &  -2 ((\ell-j)\widetilde h_{j+1}(r)(j+1)+j\widetilde h_{j-1}(r)(\ell-j+1) )\,v_j\bigg],
\end{align*}
para $0\leq j\leq\ell$.

Se ve fácil que este es el resultado que buscamos.
 \end{proof}

\begin{cor}\label{sistema2}
Las funciones $\widetilde H(r)$, $0<r<\infty$, satisfacen
$(\widetilde E\widetilde H)(r)=\mu \widetilde H(r)$ si y solo si
\begin{align*}
 -i(\ell-2j)\frac{1+r^2}{2}\widetilde h_{j}'+\frac{i}{2r}\biggr( (j+1)(\ell-j)(\widetilde h_{j+1}-\widetilde h_j)-j(\ell-j+1)(\widetilde h_{j-1}-\widetilde h_{j})\biggr)& \\
+\textstyle\frac{1}{2}\biggr( (j+1)(\ell-j)(\widetilde h_{j+1}-\widetilde h_j)+j(\ell-j+1)(\widetilde h_{j-1}-\widetilde h_{j})\biggr)
&=\mu \widetilde h_j,
\end{align*}
para todo $j=0,\dots, \ell$.
\end{cor}
\begin{proof}[\it Demostración]
Procedemos de manera similar a la demostración del Corolario \ref{sistema1}.
Usando el $\mathfrak{sl}(2)$-triple $\{e,\,f,\,h\}$ y las matrices $A_1$, $A_2$ y $A_3$ (ver \eqref{As}), del Teorema \ref{E} tenemos que $(\widetilde E\widetilde H)(r)=\mu \widetilde H(r)(r)$ si y solo si
\begin{align*}
\mu \widetilde H(r)\,v_j=&(1+r²)H'(r)\dot\pi \left(A_3\right)\,v_j
-\frac1r\left[\dot \pi \left(A_1\right) , \widetilde H(r) \right ]\dot\pi \left(r A_1 +A_2\right)\,v_j
\\&+\frac1r\left[\dot \pi \left(A_2\right) , \widetilde H(r) \right ]\dot\pi \left(A_1 -rA_2\right)\,v_j,
\end{align*}
para cada $v_j$ en $\mathcal B=\vz{v_j}_{j=0}^\ell$.

Como en la demostración del Teorema \ref{sistema1}, escribimos $A_1$, $A_2$ y $A_3$ en términos de $\{e,\,f,\,h\} $ (ver (\ref{efh})),
$$A_1= \frac i2({e+f})\; ,\qquad A_2=\frac12{(e-f)}, \qquad A_3=-\frac i2h.$$
Luego, $(\widetilde D\widetilde H)(r)=\lambda \widetilde H(r)$ si y solo si
\begin{align*}
\mu \widetilde H(r)v_j=&\frac 1{4r}\left[\dot \pi \left(e+f\right) , \widetilde H(r) \right ]\dot\pi \left(r (e+f) -i(e-f)\right)v_j \\
&+\frac1{4r}\left[\dot \pi \left(e-f\right) , \widetilde H(r) \right ]\dot\pi \left(i(e+f) -r(e-f)\right)v_j-i\frac{1+r²}{2}H'(r)\dot\pi \left(h\right)v_j
,
\end{align*}
para $0\leq j\leq\ell$, y esto es equivalente a
\begin{align*}
\mu \widetilde H(r)\,v_j=&-i\frac{1+r²}{2}H'(r)\dot\pi \left(h\right)\,v_j
+\frac 1{2r}(r+i)\left[\dot \pi \left(e\right) , \widetilde H(r) \right ]  \dot\pi(f)\,v_j
\\&+\frac 1{2r}(r-i)\left[\dot \pi \left(f\right) , \widetilde H(r) \right ]  \dot\pi(e)\,v_j ,
\end{align*}
para $0\leq j\leq\ell$.

Finalmente, usando (\ref{basedeVpi}) obtenemos
\begin{align*}
\mu \widetilde h_j\,v_j=&-i\frac{1+r²}{2}\widetilde h'_j(2\ell-j)\,v_j
+\frac 1{2r}(r+i)(\ell-j)(\widetilde h_{j+1}-\widetilde h_{j})  (j+1)\,v_j
\\&+\frac 1{2r}(r-i)j( \widetilde h_{j-1}-\widetilde h_{j}) (\ell-j+1)\,v_j ,
\end{align*}
\noindent para $0\leq j\leq\ell$.
Por lo tanto, el corolario está demostrado. \end{proof}

En notación matricial, los operadores diferenciales $\widetilde D$ y $\widetilde E$ son de la forma
\begin{align*}
 \widetilde D \widetilde H&=(1+r^2)^2 \widetilde H''+2\frac{(1+r^2)^2}{r} \widetilde H'+\frac{(1+r^2)}{r^2} (C_1+C_0)\widetilde H,\\
 \widetilde E \widetilde H&= -i\frac{1+r^2}{2}A_0\widetilde H'+\frac{i}{2r}(C_1-C_0)\widetilde H+\frac{1}{2}(C_1+C_0)\widetilde H.
\end{align*}
 donde las  matrices están dadas por
\begin{equation}\label{matricesAC}
\begin{split}
A_0&=\textstyle\sum_{j=0}^\ell(\ell-2j)E_{j,j},\\
C_0&=\textstyle\sum_{j=1}^\ell j(\ell-j+1)(E_{j,j-1}-E_{j,j}),\\
C_1&=\textstyle\sum_{j=0}^{\ell-1}(j+1)(\ell-j)(E_{j,j+1}-E_{j,j}).
\end{split}
\end{equation}

  Cuando $\ell=0$, estamos en el caso escalar y las matrices $C_0$, $C_1$ y $A_0$ son cero. Es conocido que las funciones esféricas zonales en la esfera $S^3$ están dadas, en la variable apropiada $x$, en términos de polinomios de  Gegenbauer $C_n^\nu(x)$ con $\nu=1$ y $n=0,1,2,\dots$ (ver \cite{AAR}, página 302). Por lo tanto, en alguna variable variable $x$, las funciones $\widetilde H$ deberían satisfacer una ecuación diferencial de la forma
$$ (1-x^2)y'' - 3x y+n(n+2) y=0.$$
Esto sugiere el siguiente cambio de variables
\begin{equation}\label{variableu}
u=\tfrac {1} {\sqrt{1+r^2}}, \qquad u\in(0,1].
\end{equation}

\begin{remark}
Es importante notar que si $g=ka(\theta)k'$, con $k,k'\in K$, $a(\theta)\in A$ y $gK=(x_1,x_2,x_3,x_4)\in(S^3)^+$, entonces
$$u=\cos(\theta),$$
ya que
$$u(g)=\tfrac {1} {\sqrt{1+r^2}}=\tfrac {1} {\sqrt{1+y_1²+y_2²+y_3²}}=x_4=\cos(\theta).$$
\end{remark}

Sea
\begin{equation}\label{Hu}
H(u)=\widetilde H\left(\tfrac{\sqrt{1-u^2}}{u}\right) \text{ y } h_j(u)=\widetilde h_j\left(\tfrac{\sqrt{1-u^2}}{u}\right).
 \end{equation}
Bajo este cambio de variables, los operadores diferenciales
$\widetilde D$ y $\widetilde E$ se convierten en dos nuevos operadores diferenciales $D$
y $E$.
Y sus respectivas expresiones son,
\begin{align}
\label{opDH}  DH&=(1-u^2) \frac{d^2H}{du^2} - 3u \frac{dH}{du}+\frac 1{1-u^2}(C_0+C_1)H,\\
\label{opEH}  EH& = \frac i2 \sqrt{1-u^2}A_0\frac{dH}{du}+\frac i2 \frac u{\sqrt{1-u^2}} (C_1-C_0)H+\frac 12 (C_0+C_1)H.
\end{align}

En este punto hay un cierto abuso de notación, ya que $D$ y $E$
fueron usadas antes para denotar operadores en $\RR^3$.
{ \begin{remark}\label{DyEconmutan3}
 Claramente de la Observación \ref{DyEconmutan2} tenemos que los operadores diferenciales $D$ y $E$ conmutan.
\end{remark}
}

 \section{Autofunciones de $D$}\label{DH}
\

Estamos interesados en determinar las funciones
$H:(0,1)\longrightarrow \CC^{\ell+1}$ que son autofunciones del operador diferencial
$$DH=(1-u^2) \frac{d^2H}{du^2} - 3u \frac{dH}{du}+\frac 1{1-u^2}(C_0+C_1)H,$$
$u\in (0,1)$.

Es sabido que tales autofunciones son funciones analíticas en el intervalo
$(0,1)$ y que la dimensión del correspondiente autoespacio
es $2(\ell+1)$.

\medskip
La ecuación $DH=\lambda H$ es un sistema acoplado de $\ell+1$ ecuaciones diferenciales de segundo orden en las componentes $(h_0, \dots , h_\ell)$ de
$H$, ya que la  matriz $(\ell+1)\times (\ell+1)$ $C_0+C_1$ no es diagonal. Pero por fortuna la matriz $C_0+C_1$ es simétrica, y por lo tanto diagonalizable.
Ahora citamos de \cite{GPT02b} la Proposición 5.1.

\begin{prop}\label{Hahn} La matriz $C_0+C_1$ es diagonalizable.
Más aún, los autovalores son  $-j(j+1)$ para $0\le j\le \ell$
y los autovectores correspondientes están dados por
$u_j=(U_{0,j},\dots,U_{\ell,j})$ donde
\begin{equation*}
U_{k,j}= \lw{3}F_2\left( \begin{smallmatrix}
-j,\;-k,\;j+1 \\ 1,\;-\ell \end{smallmatrix}; 1 \right),
\end{equation*}
un caso particular de los polinomios de Hahn.
\end{prop}

Por lo tanto, si se define $\check H(u)=U^{-1}H(u)$, tenemos que $DH=\lambda H$ es equivalente a
$$(1-u^2) \frac{d^2\check H}{du^2} - 3u \frac{d\check H}{du}-\frac 1{1-u^2}V_0\check H=\lambda \check H,$$
donde $V_0=\sum_{j=0}^{\ell-1} j(j+1)E_{j,j}$.

De esta forma obtenemos que  $DH=\lambda H$ si y solo si  la $j$-ésima componente
$\check h_j(u)$ de $\check H(u)$, para  $0\leq j\leq \ell$, satisface
\begin{equation}\label{eqtilde}
(1-u^2) \,\check h_j''(u)-3\,u \check
h_j'(u)-j(j+1)\frac{1}{(1-u^2)}\check h_j(u)-\lambda \check h_j(u)=0.
\end{equation}

Si escribimos $\lambda=-n(n+2)$ con $n\in\CC$ y
$\check{h_j}(u)=(1-u^2)^{ j /2 }p_j(u)$, entonces para $0<j<\ell$, $p_j(u)$ satisface
\begin{equation}\label{eqfj2}
(1-u^2) p_j''(u)-(2j+3)\,u p_j'(u)+ (n-j)(n+j+2) p_j(u) =0.
\end{equation}

Haciendo un nuevo cambio de variables $s=(1-u)/2,\,s\in\left[0,\frac12 \right)$, y definiendo $\tilde p_j(s)=p_j(u)$ tenemos
\begin{equation}\label{eqfj3}
s(1-s) \tilde p_j''(s)+( j +\tfrac 32-(2j+3)\,s) \tilde p_j'(s)+ (n-j)(n+j+2) \tilde p_j
=0,
\end{equation}
 para $0<j<\ell$. Esta es una ecuación hipergeométrica de parámetros
$$a= -n+j\, , \qquad  b=n+j+2\, , \qquad c= j+\tfrac 32.$$
Luego, cada solución $\tilde p_j(s)$ de \eqref{eqfj3} para $0<s<\tfrac12$ es
una combinación lineal de
$$  {}_2\!F_1\left(\begin{smallmatrix}-n+j,\,n+j+2\\  j+3/2 \end{smallmatrix}; s\right) \,  \qquad \text{ y } \qquad
s^{-j-1/2}\, {}_2\!F_1\left(\begin{smallmatrix} -n-1/2 ,\,n+ 3/2\\  -j+1/2 \end{smallmatrix}; s\right).  $$

Entonces, para $0\leq j\leq \ell$,
cualquier solución $\check h_j(u)$ de \eqref{eqtilde}, para $0<u<1$, es de la forma
\begin{equation}\label{Hcheck}
\begin{split}
\check h_j(u) =& \,a_j (1-u^2)^{j/2} \, {}_2\!F_1\left(\begin{smallmatrix}-n+j,\,n+j+2\\  j+3/2 \end{smallmatrix}; \tfrac {1-u}2\right)\\&  + b_j \, (1-u^2)^{-(j+1)/2}
 {}_2\!F_1\left(\begin{smallmatrix} -n-1/2 ,\,n+ 3/2\\  -j+1/2 \end{smallmatrix}; \tfrac{1-u}2\right),
 \end{split}
\end{equation}
para algún $a_j,b_j\in \CC$.

\medskip

Por lo tanto,  hemos demostrado el siguiente teorema.

\begin{thm}\label{Dhyp0}
Sea $H(u)$ una autofunción de $D$ con autovalor $\lambda=-n(n+2)$, $n\in \CC$. Entonces, $H$ es de la forma
  \begin{equation*}
    H(u)= UT(u)P(u) + US(u) Q(u)
  \end{equation*}
   donde $U$ es la matriz definida en la Proposición \ref{Hahn},
   $$  T(u)=\sum_{j=0}^\ell (1-u^2)^{j/2}E_{jj},\qquad  S(u)=\sum_{j=0}^\ell (1-u^2)^{-(j+1)/2}E_{jj}, $$
   $ P=(p_0, \dots, p_\ell)^t$  y $ Q=(q_0, \dots, q_\ell)^t$  son las funciones vectoriales dadas por
  \begin{align*}
  p_j(u) = a_j \, {}_2\!F_1\left(\begin{smallmatrix}-n+j,\,n+j+2\\  j+3/2 \end{smallmatrix}; \tfrac{1-u}2\right),\qquad  q_j(u) = b_j \, {}_2\!F_1\left(\begin{smallmatrix}-n-1/2,\,n+3/2\\  -j+1/2 \end{smallmatrix};\tfrac{1-u}2\right),
\end{align*}
donde  $a_j$ y $b_j$ son números complejos arbitrarios para $j=0,1,\ldots,\ell$.
\end{thm}

Regresando a nuestro problema de determinar todas las funciones esféricas irreducibles $\Phi$, recordemos que  $\Phi(e)=I$; entonces, la fución asociada $H\in C^\infty(\RR^3)\otimes\End(V_\pi)$ satisface $H(0,0,0)=I$.
En la variable $r\in \RR$, tenemos que $\lim_{r\to 0^+}\widetilde H(r)=I$.
Por lo tanto nos interesan aquellas autofunciones de $D$ tales que
$$\lim_{u\to 1^-} H(u)= (1,1,\dots, 1) \in \CC^{\ell+1}. $$

Del Teorema \ref{Dhyp0} observamos que
$$\lim_{u\to 1^-}  P(u)=(a_0, a_1,\dots , a_\ell) \qquad \text{ y } \qquad \lim_{u\to 1^-}  Q(u)=(b_0, b_1,\dots , b_\ell).$$
Más aún, la función matricial $T(u)$ tiene límite cuando $u\to 1^-$, mientras que $S(u)$ no. Por lo tanto una autofunción $H$ de $D$ tiene límite cuando $u\to 1^-$ si y solo si
el límite de $Q(u)$ cuando $u\to 1^-$ es $(0,\dots, 0)$. En tal caso tenemos que
\begin{equation}\label{limiteH}
  \lim_{u\to 1^-} H(u)= \lim_{u\to 1^-} UT(u)P(u) =U \,(a_0,0, \dots, 0)^t= a_0 (1,\dots, 1)^t.
\end{equation}

De esta forma hemos probado el siguiente resultado.

\begin{cor}\label{Dhyp}
Sea $H(u)$ una autofunción de $D$ con autovalor $\lambda=-n(n+2)$, $n\in \CC$, tal que $\lim_{u \to 1^- }{H(u)}$ existe.
 Entonces, $H$ es de la forma
  \begin{equation*}
    H(u)= UT(u)P(u)
  \end{equation*}
   con $U$ la matriz definida en la Proposición \ref{Hahn}, $  T(u)=\displaystyle \sum_{j=0}^\ell (1-u^2)^{j/2}E_{jj}, $ y
$ P=(p_0, \dots, p_\ell)^t$  es la función vectorial dada por
  \begin{align*}
  p_j(u)& = a_j \, {}_2\!F_1\left(\begin{smallmatrix}-n+j,n+j+2\\  j+3/2 \end{smallmatrix}; \tfrac{1-u}{2}\right), \qquad  0\le j \le \ell,
\end{align*}
donde  $a_j$ son números complejos arbitrarios para $j=1,2,\ldots,\ell$.\\
 También tenemos que $\lim_{u \to 1^- }{H(u)}=a_0(1,1,\ldots,1)^t$. Particularmente, si $H(u)$ está asociada a una función esférica irreducible, entonces $a_0=1$.
 \end{cor}

\section{Autofunciones Simultáneas de $D$ y $E$}\label{DE}
\

En esta sección estudiaremos las soluciones simultáneas de $DH(u)=\lambda H(u)$ y $EH(u)=\mu H(u)$, $0<u<1$.

Introducimos la función matricial $P(u)$, definida a partir de $H(u)$  a través de
\begin{equation}\label{HP}
  H(u)=U\, T(u)\,P(u),
\end{equation}
donde  $U$ es la matriz definida en la Proposición \ref{Hahn}
y $ T(u)=\sum_{j=0}^\ell (1-u^2)^{j/2}E_{jj} $.

El hecho de que $H$ sea una autofunción de los operadores diferenciales $D$ y $E$
hace a $P$ una autofunción de los operadores diferenciales
\begin{equation}\label{barras}
\overline D =\left(UT(u)\right)^{-1} D \left(UT(u)\right) \quad \text{ y } \quad \overline E =\left(UT(u)\right)^{-1} E \left(UT(u)\right),
\end{equation}
 con los mismos autovalores $\lambda$ y $\mu$, respectivamente.

\smallskip
 Las expresiones explícitas de $\overline D$ y $\overline E$ serán dadas en el Teorema \ref{puntos}; primero recordaremos algunas propiedades de los polinomios de Hahn.

Para $\alpha$ y $\beta >-1$, números reales,  y para un natural $N$ los polinomios de Hahn
 $Q_n(x)=Q_n(x;\alpha, \beta, N)$ están definidos como
$$Q_n(x)=\lw{3}F_2\left( \begin{smallmatrix}
-n,\;-x,\;n+\alpha+\beta+1 \\ \alpha+1,\;-N+1 \end{smallmatrix}; 1 \right),
\qquad \text{ para } n=0,1,\dots , N-1.$$
Si uno toma $\alpha=\beta=0$, $x=j$, $N=\ell+1$, $n=k$,  se obtiene
$$ U_{j,k}= Q_k(j)= \lw{3}F_2\left( \begin{smallmatrix}
-k,\;-j,\;k+1 \\ 1,\;-\ell \end{smallmatrix}; 1 \right).$$

Estos polinomios de Hahn son ejemplos de polinomios ortogonales y por lo tanto
satisfacen una relación de recurrencia de tres términos en la variable $j$, ver  \cite{AAR} ecuación (d) en página 346,
 \begin{equation}\label{Hahnrec_j}
\begin{split}
   \big ( j(\ell-j+1)+ &(j+1)(\ell-j)-k(k+1)\big)  U_{j,k}\\ &= j(\ell-j+1) U_{j-1,k}+ (j+1)(\ell-j)U_{j+1,k}.
   \end{split}
 \end{equation}
Además, los polinomios de Hahn satisfacen una relación de recurrencia de tres términos en la variable  $k$, ver  \cite{AAR} ecuación (c) en página 346,
\begin{equation}\label{Hahnrec_k}
(\ell-2j) U_{j,k}
=  \tfrac{k(\ell+k+1)}{2k+1} U_{j,k-1} +\tfrac{(k+1)(\ell-k)}{2k+1}U_{j,k+1}.
\end{equation}

Karlin y McGregor en \cite{KM61} también probaron que los polinomios de Hahn satisfacen una relación de recurrencia de primer orden que combina las variables $j$ y $k$ (ver también \cite{RS89}, ecuación (36)):
\begin{equation}\label{Hahnrec_j_k}
\begin{split}
  \big( k(\ell-j)- & k(k+j+1)+ 2(j+1)(\ell-j) \big) U_{j,k}\\&= 2(j+1)(\ell-j) U_{j+1,k} - k(k+\ell+1) U_{j,k-1}.
\end{split}
\end{equation}

\smallskip
Nosotros necesitaremos el siguiente lema técnico.

\begin{lem}\label{Hahnproperties}
Sea $U=\left(U_{j,k}\right) $ la matriz definida por
 \begin{equation}\label{Ucolumnas}
U_{j,k}= \lw{3}F_2\left( \begin{smallmatrix} -k,\;-j,\;k+1 \\
1,\;-\ell \end{smallmatrix}; 1 \right),
\end{equation}
y sean $A_0$, $C_0$ y $C_1$ las matrices dadas en \eqref{matricesAC}.
Entonces,
   \begin{align}
\nonumber    U^{-1} A_0 U & =     Q_0+Q_1,\\
 \label{UU-1}  U^{-1} (C_1+C_0) U & = -V_0,\\
  \nonumber  U^{-1} (C_1-C_0) U & = Q_1 J-Q_0(J+1),
   \end{align}
donde
\begin{align*}
V_0& = \sum_{j=0}^{\ell-1} j(j+1)E_{j,j} \,,&
J&=\sum_{j=0}^\ell j E_{jj} \,,\\
Q_0 & = \sum_{j=0}^{\ell-1}  \tfrac{(j+1)(\ell+j+2)}{2j+3} E_{j,j+1}\,,& Q_1&= \sum_{j=1}^{\ell} \tfrac{j(\ell-j+1)}{2j-1}E_{j,j-1} \,.
\end{align*}
\end{lem}

\begin{proof}[\it Demostración]
   Probar que $U^{-1}A_0 U=Q_0+Q_1$ es equivalente a verificar que
\begin{equation*}
A_0 U=U(Q_0+Q_1).
\end{equation*}
Mirando la entrada $(j,k)$, para $j,k=0, \dots, \ell$, obtenemos que
\begin{equation*}
(\ell-2j) U_{j,k}
= U_{j,k-1} \tfrac{k(\ell+k+1)}{2k+1} +U_{j,k+1}
\tfrac{(k+1)(\ell-k)}{2k+1}.
\end{equation*}
Ésta es una relación de recurrencia de tres términos  en la variable $k$ dada en \eqref{Hahnrec_k}.

Observemos que $U^{-1} (C_1+C_0) U  = -V_0$ es una consecuencia directa de la Proposición \ref{Hahn}. Además, se obtiene de manera directa al considerar la entrada $(j,k)$  de $(C_1+C_0)U=-UV_0$ y usando la relación  de recurrencia \eqref{Hahnrec_j}.

Ahora tenemos que probar que
\begin{equation*}\label{black}
U^{-1}(C_1-C_0)U=-Q_0 (J+1)+Q_1J.
\end{equation*}
Usando que   $ (C_0+C_1)U=-UV_0$, esto equivalente a probar que
\begin{equation*}
-2C_0U=U(-Q_0 (J+1)+Q_1J+V_0);
\end{equation*}
por consiguiente, si miramos la entrada $(j,k)$, lo que necesitamos verificar es
\begin{align*}
-2{(C_0)}_{j,j}U_{j,k}-2{(C_0)}_{j,j-1}U_{j-1,k}&\\=-U_{j,k-1}
{(Q_0)}_{k-1,k} &(J+1)_{k,k}+U_{j,k+1} {(Q_1)}_{k+1,k} J_{k,k}+U_{j,k}
{(V_0)}_{k,k},
\end{align*}
o, equivalentemente, tenemos que probar que
\begin{equation}\label{T}
\begin{split}
2j& (\ell-j+1)U_{j,k}-2j(\ell-j+1)U_{j-1,k}\\
&=-\tfrac{k(\ell+k+1)(k+1)}{2k+1} U_{j,k-1}  +
\tfrac{k(k+1)(\ell-k)}{2k+1} U_{j,k+1}+ k(k+1) U_{j,k}.
\end{split}
\end{equation}

Aplicando la relación de recurrencia \eqref{Hahnrec_k}, podemos escribir  $ U_{j,k+1}$ en términos de $U_{j,k}$ y $U_{j,k-1}$. Por lo tanto, la identidad (\ref{T}) se vuelve
\begin{align*}
\big(2j(\ell-j+1)-k(\ell-2j)& -k(k+1)\big) U_{j,k} =2j (\ell-j+1)U_{j-1,k}-k(\ell+k+1)U_{j,k-1},
\end{align*}

Finalmente, usamos \eqref{Hahnrec_j} para escribir $ U_{j-1,k}$ en términos de $U_{j+1,k}$ y $U_{j,k}$ y obtener
\begin{align*}\big( k(\ell-j)- & k(k+j+1)+ 2(j+1)(\ell-j) \big) U_{j,k}= 2(j+1)(\ell-j) U_{j+1,k} - k(k+\ell+1) U_{j,k-1},
\end{align*}
que es exactamente la identidad en \eqref{Hahnrec_j_k}, y esto concluye la prueba del lema.
\end{proof}

 \begin{thm}\label{puntos}
 Los operadores $\overline D$ y $\overline E$ definidos en \eqref{barras} están dados por
 \begin{align*}
  \overline D P &= (1-u^2) P'' -u CP'-V P, \mbox{} \\
  \overline EP & = \tfrac i 2\left( (1-u^2)Q_0+Q_1 \right) P'-\tfrac i2 u MP-\tfrac 12 V_0 P,
\end{align*}
donde
\begin{align*}
  \begin{alignedat}{2}
&C = \sum_{j=0}^\ell (2j+3)E_{jj},
& V &=  \sum_{j=0}^\ell j(j+2)E_{jj}, \\
& Q_0 = \sum_{j=0}^{\ell-1}  \tfrac{(j+1)(\ell+j+2)}{2j+3} E_{j,j+1},
& \qquad Q_1&= \sum_{j=1}^{\ell} \tfrac{j(\ell-j+1)}{2j-1}E_{j,j-1},\\
&M = \sum_{j=0}^{\ell-1} (j+1)(\ell+j+2) E_{j,j+1},
& V_0 &= \sum_{j=0}^{\ell-1} j(j+1)E_{jj}.
  \end{alignedat}
  \end{align*}
\end{thm}

\begin{proof}[\it Demostración]
Sea $H=H(u)=U T(u) P(u)$. Comenzamos por calcular $D(H)$ para el operador diferencial $D$ introducido en \eqref{opDH}.
\begin{align*}
D H  =& (1-u^2)UTP''+\big (2(1-u^2)UT'-3uUT\big)P'\\
&+\big((1-u^2)UT''-3uUT'+\tfrac {1}{1-u^2}(C_0+C_1)UT\big)P\\
=&UT\Big( (1-u^2)P''+\big(2(1-u^2)T^{-1}T'-3u\big)P'\\
 &+\Big( (1-u^2)T^{-1}T''-3uT^{-1}T' +\frac{1}{1-u^2}T^{-1} U^{-1}(C_0+C_1)UT \Big)P \Big).
\end{align*}
Como $T$ es una matriz diagonal, fácilmente calculamos
\begin{align*}
T^{-1}(u)T'(u) & = -\tfrac u {(1-u^2)} \sum_{j=0}^\ell j \, E_{jj},  \quad
  T^{-1}T''(u) = \tfrac 1{(1-u^2)^2} \sum_{j=0}^\ell j ((j-1)u^2-1)\, E_{jj}.
\end{align*}
 Además, de \eqref{UU-1} tenemos que
$U^{-1}(C_0+C_1)U = -V_0$. Como $V_0$ es es una matriz diagonal, conmuta con $T$ y tenemos
\begin{align*}
(1-u^2)T^{-1} & T''-3uT^{-1}T'  +\frac{1}{1-u^2}T^{-1}U^{-1}(C_0+C_1)UT  \\
&= \frac 1{(1-u^2)} \sum_{j=0}^\ell \big ( j(j-1)u^2-j +3ju^2-j(j+1)\big ) E_{jj} =-V.
 \end{align*}

\smallskip
Ahora, para el operador diferencial $E$ introducido en \eqref{opEH}, calculamos $E(H)$ con $H(u)=U T(u) P(u)$.
\begin{align*}
 EH &= \frac i2
\sqrt{1-u^2}A_0 UTP'\\
&\quad + \bigg(\frac i2 \sqrt{1-u^2}A_0 UT'+\frac i2 \frac u{\sqrt{1-u^2}}
(C_1-C_0)UT+\frac 12 (C_0+C_1)UT\bigg) P \displaybreak[0]\\
&=UT\bigg
( \frac i2 \sqrt{1-u^2}T^{-1}U^{-1}A_0 UTP'+ \bigg(\frac i2 \sqrt{1-u^2}T^{-1}U^{-1}A_0 UT'\\
&\quad +\frac i2 \frac u{\sqrt{1-u^2}} T^{-1}U^{-1}(C_1-C_0)UT+\frac 12
T^{-1}U^{-1}(C_0+C_1)UT\bigg)P\bigg).
\end{align*}

Por el Lema \ref{Hahnproperties} tenemos que
$U^{-1}A_0 U=Q_0+Q_1$.
Usando $T=\sum_{j=0}^\ell (1-u^2)^{j/2} E_{jj}$ obtenemos
\begin{align*}
\sqrt{1-u^2} \,T^{-1}U^{-1}A_0 UT&=\sqrt{1-u^2}\, T^{-1}(Q_0+Q_1)T=(1-u^2)Q_0+Q_1.
\end{align*}
Por  \eqref{UU-1} y el hecho de que $T$ es diagonal,  tenemos que
$T^{-1}U^{-1}(C_0+C_1)UT = -V_0$.
Entonces, solo resta probar que
\begin{equation}\label{ecuac}
\sqrt{1-u^2}T^{-1}U^{-1}A_0 UT'\\
+ \frac u{\sqrt{1-u^2}} T^{-1}U^{-1}(C_1-C_0)UT=-uM.
\end{equation}
Como $T'(u)=\tfrac{-u}{1-u^2}J T(u) $, donde  $J=\sum_{j=0}^\ell j E_{jj}$, tenemos que demostrar que
\begin{equation}\label{aux}
T^{-1}\big (U^{-1}A_0 U J   - U^{-1}(C_1-C_0)U \big )T =\sqrt {1-u^2} M.
\end{equation}
Gracias al  Lema \ref{Hahnproperties} tenemos que
\begin{align*}
U^{-1}A_0 U J  - U^{-1}(C_1-C_0)U & = (Q_1+Q_0)J - Q_1J+ Q_0(J+1) = Q_0(2J+1) \\
& = \sum_{j=0}^{\ell-1} (j+1)(\ell+j+2) E_{j,j+1}= M.
\end{align*}
Como $T= \sum_{j=0}^\ell (1-u^2)^{j/2} E_{jj} $, \eqref{aux} es satisfecha y esto completa la demostración del teorema. \end{proof}

La función $P$ es una autofunción del operador diferencial $\overline D$ si y solo si la función $ H=U T(u) P(u)$
es una autofunción del operador diferencial $D$.
Del Teorema \ref{Dhyp0} tenemos la explícita expresión de la función $P(u)=(p_0(u), \dots, p_\ell(u))^t$,
\begin{align*}
 p_j(u) =a_j\, {}_2\!F_1\left(\begin{smallmatrix}-n+j,\,n+j+2\\  j+3/2 \end{smallmatrix}; \tfrac {1-u}2\right)  + b_j \, (1-u^2)^{-(j+1/2)}
 {}_2\!F_1\left(\begin{smallmatrix} -n-1/2 ,\,n+ 3/2\\  -j+1/2 \end{smallmatrix}; \tfrac{1-u}2\right),
 \end{align*}
donde $a_j$ y $b_j$ están en $\CC$, para $0\le j\leq \ell$.

Como nosotros estamos interesados en determinar las funciones esféricas irreducibles del par $(G,K)$,
necesitamos estudiar las autofunciones simultáneas de $D$ y $E$ tales que existe el límite, y sea finito, de la función $H$ cuando $u\to1^{-}$.

Del Teorema \ref{Dhyp0} tenemos que
$\lim_{u\to 1^-} H(u)$ es finito si y solo si
  $$\lim_{u\to 1^-} b_j \, (1-u^2)^{-(j+1)/2}
 {}_2\!F_1\left(\begin{smallmatrix} -n-1/2 ,\,n+ 3/2\\  -j+1/2 \end{smallmatrix}; \tfrac{1-u}2\right)$$
existe y es finito para todo $0\leq j \leq \ell$. Esto es cierto si y solo si $b_j=0$ para todo $0\leq j\leq \ell$. Por lo tanto, $\lim_{u\to 1^-} H(u)$ es finito si y solo si
$\lim_{u\to 1^-} P(u)$ es finito.

Por el Corolario \ref{Dhyp} sabemos que una autofunción $P=P(u)$ de $\overline D$ en el intervalo $(0,1)$ tiene límite finito cuando $u\to1^{-}$ si y solo si $P$ es analítica en $u=1$.
Consideremos ahora el siguiente espacio vectorial de funciones sobre $\CC^{\ell+1}$,
$$W_\lambda=\vzm{P=P(u) \text{ analítica en }(0,1]}{\overline D P=\lambda P }.$$

Una función $P\in W_\lambda$ es caracterizada por $P(1)= (a_0, \dots , a_\ell)$. Entonces, la dimensión de $W_\lambda$ es $\ell+1$ y el isomorfismo  $W_\lambda\simeq\CC^{\ell+1}$  está dado por
$$\nu:W_\lambda\longrightarrow\CC^{\ell+1},\qquad P\mapsto P(1). $$

{  Los operadores diferenciales $ \overline D$ y $\overline E$ conmutan
 porque los operadores diferenciales $D$ y $E$ conmutan (ver la  Observación \ref{DyEconmutan3}).}

\begin{prop}\label{Llambda} El espacio lineal $W_\lambda$
es estable por el operador diferencial  $\overline E$ y la restricción de este operador a $W_\lambda$ es una función lineal . Más aún,  el siguiente es un diagrama conmutativo
\begin{equation}\label{diagrama1}
\begin{CD}
W_\lambda @ >\overline E >>W_\lambda \\ @ V \nu VV @ VV \nu V \\
\CC^{\ell+1} @> L(\lambda) >> \CC^{\ell+1}
\end{CD}
\end{equation}
donde $L(\lambda)$ es la matriz $(\ell+1)\times (\ell+1)$
\begin{equation*}
  \begin{split}
    L(\lambda) & = -\tfrac i2 Q_1C^{-1}(V+\lambda)-\tfrac i2 M-\tfrac 12 V_0 \\
    &=- i \sum_{j=1}^\ell \tfrac{j(\ell-j+1)\big((j-1)(j+1)+\lambda \big)}{2(2j-1)(2j+1)} E_{j,j-1} -i \sum_{j=0}^{\ell-1}  \tfrac{(j+1)(\ell+j+2)}2 E_{j,j+1} - \sum_{j=0}^\ell   \tfrac{j(j+1) }2E_{jj}.
  \end{split}
\end{equation*}
\end{prop}
\begin{proof}[\it Demostración]
  El operador diferencial $\overline E$ lleva  funciones analíticas a funciones analíticas, pues sus coeficientes son polinomios, ver Teorema \ref{puntos}. Una función $P\in W_\lambda$ es  analítica, luego $\lim_{u\to 1^-} \overline EP(u)$ es finito.
 Por otra parte, como $\overline D$ y $\overline E$ conmutan, el operador diferencial $\overline E$ preserva los autoespacios de $\overline D$. Esto prueba que $W_\lambda$ es estable por la acción de $\overline E$. En particular, $\overline E$ se restringe como una función lineal $L(\lambda)$ en $W_\lambda$ que ahora determinaremos.

  Del Teorema \ref{puntos} tenemos
  $$\nu(\overline E(P))=(\overline E P )(1)=\tfrac i2 Q_1 P'(1)-\tfrac i2 MP(1)-\tfrac 12 V_0P(1).$$
   Pero ahora podemos obtener  $P'(1)$ en términos de $P(1)$. De hecho, si evaluamos $\overline DP=\lambda P$ en $u=1$ tenemos $$P'(1)= -C^{-1} (V+\lambda) P(1).$$
   Notar que $C$ es una matriz inversible.
Luego,
  \begin{align*}
  \nu(\overline E(P))&=-\tfrac i2 Q_1 C^{-1} (V+\lambda) P(1)-\tfrac i2 MP(1)-\tfrac 12 V_0P(1)=L(\lambda) P(1)= L(\lambda) \,\nu(P).
  \end{align*}
Esto completa la prueba de la proposición.
\end{proof}

 \begin{remark}

Si $\lambda=-n(n+2)$, con $n\in\CC$,  se  tiene
\begin{equation}\label{matrixL}
  \begin{split}
    L(\lambda)& = i \sum_{j=1}^\ell \tfrac{j(\ell-j+1)(n-j+1)(n+j+1)}{2(2j-1)(2j+1)} E_{j,j-1} -i \sum_{j=0}^{\ell-1}  \tfrac{(j+1)(\ell+j+2)}2 E_{j,j+1}
   -  \sum_{j=0}^\ell  \tfrac {j(j+1)}2 E_{jj}.
  \end{split}
\end{equation}

 \end{remark}

\begin{cor}\label{corttr}
  Todos los autovalores $\mu$ de $L(\lambda)$  tienen multiplicidad geométrica uno, esto es, todo autoespacio es unidimensional.
\end{cor}
\begin{proof}[\it Demostración]
Un vector $a=\left(a_0,a_1,\dots,a_\ell\right)^t$ es un autovector de
$L(\lambda)$ de autovalor $\mu$, si y solo si
$\left\{a_j\right\}_{j=0}^\ell$ satisface la siguiente relación de recurrencia de tres términos
\begin{equation}\label{threetermI}
i \, \tfrac{j(\ell-j+1)(n-j+1)(n+j+1)}{2(2j-1)(2j+1)}  \,\, a_{j-1} - \tfrac {j(j+1)}2 \,\,a_j
 - i   \tfrac{(j+1)(\ell+j+2)}2 \,\,a_{j+1}  = \mu \, a_j,
\end{equation}
para $j=0,\dots,\ell-1$ (donde interpretamos $a_{-1}=0$), y
\begin{equation}\label{threetermclosing}
  i \,\tfrac{\ell(n-\ell+1)(n+\ell+1)}{2(2\ell-1)(2\ell+1)} \,\, a_{\ell-1} - \tfrac {\ell(\ell+1)}2 \,\,a_\ell=\mu a_\ell.
\end{equation}

De estas ecuaciones vemos que el vector $a$ está determinado por  $a_0$,
lo cual prueba que la multiplicidad geométrica del autovalor $\mu$ de $L(\lambda)$ es uno.
\end{proof}

{\begin{remark}
Los valores de $\mu$ para los que las ecuaciones  \eqref{threetermI} y \eqref{threetermclosing} tienen solución $\{a_j\}_{j=0}^\ell$ son exactamente los autovalores de la matriz $L(\lambda)$.
Las ecuaciones \eqref{threetermI}, para $j=0, \dots , \ell-1,$ son empleadas para definir $a_1, \dots , a_\ell$ comenzando con algún $a_0\in \CC$. La ecuación \eqref{threetermclosing} es una condición extra (una ``condición de cierre") que los coeficientes $a_j$ deben satisfacer para que $a=(a_0, \dots , a_\ell)$ sea un autovector de $L(\lambda)$ de autovalor $\mu$.
\end{remark}}

Finalmente alcanzamos el principal resultado de esta sección, que es la caracterización de las autofunciones simultáneas $H$ de los operadores diferenciales $D$ y $E$ en $(0,1)$ que son continuas en $(0,1]$. Recordemos que las funciones esféricas irreducibles del par $(G,K)$ nos llevan a tales funciones $H$.

\begin{cor}\label{DEhyp}
Sea $H(u)$ una autofunción simultánea  de $D$ y $E$ en $(0,1)$, continua en $(0,1]$, con respectivos autovalores $\lambda=-n(n+2)$, $n\in\CC$, y $\mu$.  Entonces, $H$ es de la forma
  \begin{equation*}
    H(u)= UT(u)P(u)
  \end{equation*}
   con $U$ la matriz dada en \eqref{Ucolumnas}, $  T(u)=\displaystyle \sum_{j=0}^\ell (1-u^2)^{j/2}E_{jj}, $ y
   $ P=(p_0, \dots, p_\ell)^t$  es la función vectorial dada por
  \begin{align*}
  p_j(u)& = a_j \, {}_2\!F_1\left(\begin{smallmatrix}-n+j,n+j+2\\  j+3/2 \end{smallmatrix}; \tfrac{1-u}2\right)
\end{align*}
 donde $\{a_j\}_{j=0}^\ell$ satisface las relaciones de recurrencia \eqref{threetermI} y  \eqref{threetermclosing}.
Además tenemos que  $H(1)=a_0(1,1,\ldots,1)^t$. En particular,  si $H(u)$ es asociada a una función esférica tenemos que $a_0=1$.
\end{cor}

\begin{remark}
La condición $H(1)=  (1,\dots , 1)^t$  implica que  $P(1)$ es un vector cuya primera entrada es igual a 1.
\end{remark}

\medskip

En $S^3$, el conjunto
$$\vzm{x_\theta=(\sqrt{1-\theta^2},0,0,\theta)}{\theta\in[-1,1]}$$ parametriza todas las $K$-órbitas. Notar que para $\theta>0$ tenemos que $x_\theta\in (S³)^+$, y $p(x_\theta)=(\frac{\sqrt{1-\theta^2}}{\theta},0,0)$. Por lo tanto, en términos de la variable $r\in[0,\infty)$ tenemos que $r=\frac{\sqrt{1-\theta^2}}{\theta}$, y entonces, en términos de la variable $u\in(0,1]$ tenemos $u=\frac1{\sqrt{1+r²}}= \theta$.
Luego, dada una función esférica irreducible $\Phi$ de tipo $\pi\in\hat K$, si consideramos la función asociada  $H:S^3\longrightarrow\End(V_\pi)$ definida por
$$H(g\,(0,0,0,1)^t)=\Phi(g)\Phi^{-1}_\pi(g), \qquad g\in G,$$
tenemos que
\begin{equation}\label {u}
 H(\sqrt{1-u^2},0,0,u)=\text{diag}\{H(u)\}=\text{diag}\{UT(u)P(u)\},
\end{equation}
donde $H(u)$, $u\in(0,1]$, es la función vectorial dada en el Corolario \ref{DEhyp} y $\text{diag}\{H(u)\}$ denota la función que toma como valores a matrices diagonales cuya entrada $(k,k)$ es igual a la $k$-ésima  entrada de la función vectorial $H(u)$.

\section{Autovalores de las Funciones Esféricas}
\label{sec:RepresentationTheory}
\

El propósito en esta sección es usar la teoría de representaciones de  $G$ para calcular los autovalores de una función esférica irreducible $\Phi$ correspondiente a los
operadores diferenciales $\Delta_1$ y $\Delta_2$. De estos autovalores obtendremos los autovalores de la función $H$ como autofunción de $D$ y $E$.

 Como dijimos en la sección \ref{sec:prelim},  existe una correspondencia unívoca entre funciones esféricas irreducibles
de $(G,K)$ de tipo $\delta \in \hat K $ y representaciones irreducibles de dimensión finita
 de $G$ que contienen al $K$-tipo $\delta$.
 De hecho, cada función esférica irreducible $\Phi$ de tipo $\delta \in \hat K$ es de la forma
\begin{equation}\label{sphasprojec}
  \Phi(g)v= P(\delta)\tau(g)v, \qquad \quad  g\in G, \qquad v\in P(\delta)V_\tau,
\end{equation}
 donde $(\tau , V_\tau)$ es una representación irreducible finito
dimensional  de $G$, que contiene el $K$-tipo $\delta$, y $P(\delta)$ es la
proyección de $V_\tau$ sobre la componente $K$-isotípica de tipo $\delta$.

Las representaciones irreducibles de dimensión finita $\tau$ de $G=\SO(4)$ están parametrizadas por un par de enteros
$(m_1, m_2)$ tales que  $$m_1\geq |m_2|,$$ mientras que las representaciones irreducibles de dimensión finita $\pi_\ell$ de $K=\SO(3)$ están parametrizadas por $\ell\in 2\NN_0$.

Las representaciones $\tau_{(m_1,m_2)}$ restringidas a $\SO(3)$ contienen la representación $\pi_\ell$  si y solo si
{ $m_1\geq\ell/2\geq |m_2|.$}
Por lo tanto, la clase de equivalencia de funciones esféricas irreducibles de $(G,K)$ de
tipo $\pi_\ell$ están parametrizadas por el conjunto de todos los pares $(m_1,m_2)\in \ZZ^2$
{ tales que   $$m_1\geq\tfrac\ell2\geq |m_2|.$$
Denotamos por
\begin{equation*}
\Phi_\ell^{(m_1,m_2)}, \qquad  \text{ con }  \quad  m_1\geq\tfrac\ell2\geq |m_2|,
\end{equation*}  a la función esférica de tipo $\pi_\ell$
asociada a la representación $\tau_{(m_1,m_2)}$ de $G$.}

\begin{thm}\label{param2}
La función esférica $\Phi_\ell^{(m_1,m_2)}$ satisface
\begin{align*}
\Delta_1\Phi_\ell^{(m_1,m_2)}&= \tfrac14(m_1-m_2)(m_1-m_2+2) \Phi_\ell^{(m_1,m_2)},\\
\Delta_2\Phi_\ell^{(m_1,m_2)}&= \tfrac14(m_1+m_2)(m_1+m_2+2) \Phi_\ell^{(m_1,m_2)}.
\end{align*}
\end{thm}
\begin{proof}[\it Demostración]
Comenzamos por observar que el autovalor de cualquier función esférica irreducible $\Phi$ correspondiente a un operador diferencial $\Delta \in D(G)^G$, dado por
$[\Delta\Phi](e)$, es un múltiplo escalar de la identidad.
Como $\Delta_1$ y $\Delta_2$ están en $D(G)^G$, tenemos que
$$[\Delta_1\Phi_\ell^{(m_1,m_2)}](e)= \dot \tau_{(m_1,m_2)}(\Delta_1) \quad \text{ y } \quad
  [\Delta_2\Phi_\ell^{(m_1,m_2)}](e)= \dot \tau_{(m_1,m_2)}(\Delta_2).$$

Estos escalares pueden ser calculados mirando la acción de $\Delta_1$ y $\Delta_2$ en un vector peso máximo $v$ de la representación $\tau_{(m_1,m_2)}$,
 cuyo peso máximo es de la forma $m_1\varepsilon_1+ m_2\varepsilon_2$.

Recordemos que \begin{align*}
\Delta_1= (iZ_6)^2+iZ_6-(Z_5+iZ_4)(Z_5-iZ_4),\\
\Delta_2= (iZ_3)^2+iZ_3-(Z_2+iZ_1)(Z_2-iZ_1).
\end{align*}
Como $(Z_5-iZ_4)$ y $(Z_2-iZ_1)$ son vectores raíces positivas $Z_6, Z_3\in \lieh_\CC$, tenemos
\begin{align*}
\dot\tau_{(m_1,m_2)}(\Delta_1)&v = \dot\tau_{(m_1,m_2)}(iZ_6)^2v+\,\dot\tau_{(m_1,m_2)} (iZ_6)v
 = \tfrac14(m_1-m_2)(m_1-m_2+2)v ,\\
\dot\tau_{(m_1,m_2)}(\Delta_2)&v =\dot\tau_{(m_1,m_2)}(iZ_3)^2v+\,\dot\tau_{(m_1,m_2)} (iZ_3)v
= \tfrac14(m_1+m_2)(m_1+m_2+2)v .
\end{align*}
Esto completa la demostración del teorema.
\end{proof}

Ahora daremos los autovalores de la función $H$ asociada a una función esférica irreducible,
 correspondiente a los operadores diferenciales $D$ y $E$.

\begin{cor}\label{autovs} La función $H$ asociada a la función esférica
 $\Phi_\ell^{(m_1,m_2)}$  satisface $D H=\lambda H$ y $E H=\mu H$ con
\begin{equation*}
  \lambda= -(m_1-m_2)(m_1-m_2+2)  , \qquad
  \mu=-\tfrac {\ell(\ell+2)} 4 + (m_1+1)m_2.
\end{equation*}
\end{cor}
\begin{proof}[\it Demostración]
Sea $\Phi=\Phi_\ell^{(m_1,m_2)}$.
Por la  Proposición \ref{relacionautovalores} tenemos que $\Delta_1 \Phi=\widetilde \lambda \Phi$ y $\Delta_2 \Phi=\widetilde \mu \Phi$
si y solo si $DH=\lambda H$ y $EH=\mu H$, donde
la relación entre los autovalores de $H$ y  $\Phi$ es
  $$\lambda=-4\widetilde \lambda, \qquad \qquad \mu=-\tfrac 14 \ell(\ell+2)+\widetilde \mu-\widetilde \lambda.$$
Ahora el enunciado sigue fácilmente del Teorema \ref{param2}.
\end{proof}

{
\begin{cor}\label{autovsP}
  La función $P$ asociada a la función esférica $\Phi^{(m_1,m_2)}_\ell$, definida por $H(u)= UT(u) P(u)$ (ver \eqref{HP}), satisface $\overline D P=\lambda P$ y $\overline E P=\mu P$ con
  \begin{equation*}
  \lambda= -(m_1-m_2)(m_1-m_2+2)  , \qquad
  \mu=-\tfrac {\ell(\ell+2)} 4 + (m_1+1)m_2.
\end{equation*}
\end{cor}

\begin{remark} \label{ninteger}
Observemos que acabamos de demostrar que el autovalor $\lambda$ puede escribirse de la forma
\begin{equation*}
\lambda=-n(n+2), \qquad \quad \text{con }  \quad n\in\NN_{0}.
\end{equation*}
\end{remark}
}
\smallskip
\begin{prop}\label{parimpar}
 Si $\Phi$ es una función esférica irreducible de
$\big(\SO(4),\SO(3)\big)$, luego $\Phi(-e)=\pm I$. Más aún, si
$\Phi=\Phi_\ell^{(m_1,m_2)}$ y $g\in \SO(4)$, entonces
$$\Phi(-g)= \begin{cases}
      \Phi(g)&    \text{ si }\; m_1+m_2\equiv 0 \mod(2) \\
      -\Phi(g)&   \text{ si } \; m_1+m_2\equiv 1 \mod(2).
\end{cases}
$$
\end{prop}
\begin{proof}[\it Demostración]
Como mencionamos en la Sección \ref{sec:prelim}, cada función esférica irreducible del par $\big(\SO(4),\SO(3)\big)$ de tipo $\pi$, es de la forma
$\Phi(g)= P_\pi \tau(g)$, donde
$\tau \in \hat\SO(4)$  contiene el $K$-tipo $\pi$  y  $P_\pi$ es la proyección sobre la componente $\pi$-isotípica de $V_\tau$.

{ Sea  $\eta$ el peso máximo de $\tau=(m_1,m_2) \in \hat\SO(4)$, i.e. $\eta=m_1\varepsilon_1+m_2\varepsilon_2$ (ver Subsección \ref{representaciones}).}
Tenemos que  $$-e=\exp \left(\begin{smallmatrix} 0&\pi &0&0 \\-\pi& 0&0&0 \\0&0&0&\pi \\0&0&-\pi&0\end{smallmatrix}\right);$$
por consiguiente, si  $v$ es un vector peso máximo en $V_\tau$ de peso $\eta$, tenemos
$$\tau (-e)v
=e^{-\pi(m_1+m_2)i} v=\pm v.$$
Como  $\tau(-e)$ conmuta con $\tau(g)$ para todo $g\in \SO(4)$, por el Lema de Schur $\tau(-e)$ es un múltiplo de la identidad. Luego,
$$\tau(-e)= \left\{
    \begin{array}{ll}
       I ,&   \hbox{si $m_1+m_2\equiv 0 \mod(2)$} \\
      -I ,&   \hbox{si $m_1+m_2\equiv 1 \mod(2)$}
    \end{array}.
  \right.
$$
Por lo tanto,
\begin{align*}
\Phi(-g)&= P_\pi \tau(-g) =P_\pi \tau(-e)\tau(g) =\tau(-e)\Phi(g).
\end{align*}
Luego, la proposición está demostrada.
\end{proof}

\section{La Función $P$}\label{ThefunctionP}
\

{
 En las secciones previas nos interesamos en estudiar las funciones esféricas irreducibles $\Phi$ de un  $K$-tipo $\pi=\pi_\ell$.  Esto se realizó asociando a cada función
 $\Phi$  a una función $H$ con valores en  $\CC^{\ell+1}$, la cual es una autofunción simultánea  de los operadores diferenciales $D$ y $E$ en $(0,1)$, dados en \eqref{opDH} y \eqref{opEH}, continua en $(0,1]$ y tal que $H(1)= (1,\dots , 1)$.
  Esta función  $H$
  es de la forma
  \begin{equation*}
    H(u)= UT(u)P(u),
  \end{equation*}
donde $U$ es la matriz constante dada en \eqref{Ucolumnas}  y $  T(u)=\sum_{j=0}^\ell (1-u^2)^{j/2}E_{jj} $.

De esta forma tenemos asociada cada función esférica irreducible $\Phi$  a una función $P(u)$, analítica en $(0,1]$, que es una autofunción simultánea
de los operadores diferenciales $\overline D$ y $ \overline E$ explícitamente dados en el Teorema \ref{puntos} como
\begin{align*}
  \overline D P &= (1-u^2) P'' -u CP'-V P, \mbox{} \\
  \overline EP & = \tfrac i 2\left( (1-u^2)Q_0+Q_1 \right) P'-\tfrac i2 u MP-\tfrac 12 V_0 P.
\end{align*}

Del Corolario \ref{Dhyp} tenemos que
una autofunción $P=(p_0, \dots , p_\ell)^t$ de $\overline D$, a valores vectoriales,  con autovalor $\lambda=-n(n+2)$ y tal que $\lim_{u\to 1^{-}} P(u)$ existe está dada por
  \begin{align*}
  p_j(u)& = a_j \, {}_2\!F_1\left(\begin{smallmatrix}-n+j,\,n+j+2\\  j+3/2 \end{smallmatrix}; \tfrac{1-u}{2}\right)\, ,  \qquad  0\le j \le \ell,
\end{align*}
donde  $a_j$ son números complejos arbitrarios para $j=1,2,\ldots,\ell$.

\smallskip

Introduzcamos el siguiente espacio vectorial de funciones a $\CC^{\ell+1}$, definido para $n\in \NN_0$ y $\mu\in \CC$,
\begin{equation*}
{\mathcal  V}_{n, \mu} =\{P=P(u) \text{ analítica en }(0,1] \, : \, \overline D P=-n(n+2) P, \,\overline E P=\mu P \}.
\end{equation*}

Observemos que $\mathcal V_{n,\mu}\neq 0$
si y solo si $\mu $ es un autovalor de la matriz $L(\lambda)$ dada en \eqref{matrixL}, con $\lambda=-n(n+2)$.
Nos interesa considerar solo los casos $\mathcal V_{n,\mu}\neq 0$.

\smallskip
Del Corolario \ref{DEhyp} tenemos que una función $P\in \mathcal V_{n,\mu}$
es de la forma $P=(p_0, \dots , p_\ell)^t$, donde
  \begin{align}\label{formadeP}
  p_j(u)& = a_j \, {}_2\!F_1\left(\begin{smallmatrix}-n+j,\,n+j+2\\  j+3/2 \end{smallmatrix}; \tfrac{1-u}{2}\right)\, ,  \qquad  0\le j \le \ell,
\end{align}
y los coeficientes
 $\{a_j\}_{j=0}^\ell$ satisfacen las relaciones de recurrencia \eqref{threetermI}.

Observemos que la ecuación \eqref{threetermclosing} es  automáticamente satisfecha ya que    $\mu$ es un autovalor de la matriz $L(\lambda)$ dada en en \eqref{matrixL}.

Si la función $P$ está  asociada a una función esférica irreducible  tenemos que $a_0=1$, puesto que
la condición $H(1)= (1,\dots , 1)$ implica que $P(1)$  es un vector cuya primer entrada es  $1$,  ver \eqref{limiteH}.

\begin{prop}\label{losaj}
 Si $P\in \mathcal V_{n,\mu}$ entonces $P$ es una función polinomial.
\end{prop}

\begin{proof}[\it Demostración]
Sea $P(u)=(p_0(u),\dots,p_\ell(u))\in \CC^{\ell+1}$. Por el  Corolario \ref{DEhyp} tenemos que las entradas de la función
$P$  están dadas por
 \begin{equation}\label{geg}
   p_j(u)= a_j\,    {}_2\!F_1\left(\begin{smallmatrix}-n+j,n+j+2\\  j+3/2 \end{smallmatrix}; \tfrac{1-u}{2}\right),
\end{equation}
 donde los coeficientes
 $\{a_j\}_{j=0}^\ell$ satisfacen la relación de recurrencia \eqref{threetermI}, para algún  autovalor $\mu$  de $L(\lambda)$.

Para $0\leq j\leq n$,  la función $p_j(u)$  es una función polinomial, mientras que para $n<j\leq \ell$ la serie definida por la función hipergeométrica no es finita.
Luego, en este caso tenemos que $p_j$ es un polinomio si y solo si el coeficiente $a_j$ es cero.

De la expresión de  $\overline E$ en el  Teorema \ref{puntos}, para $0\leq j\leq \ell$ tenemos
 \begin{equation}\label{Epj}
 \begin{split}
 \mu p_j   = \tfrac i 2  & \left( (1-u^2)\tfrac{(j+1)(\ell+j+2)}{2j+3} p'_{j+1}+\tfrac{j(\ell-j+1)}{2j-1}p'_{j-1} \right)\\
&-\tfrac i2 u (j+1)(\ell+j+2) p_{j+1}-\tfrac 12 j(j+1)p_j,
\end{split}
\end{equation}
donde interpretamos  $p_{-1}=p_{\ell+1}=0$.

Por inducción en $j$, supongamos que  $p_j$ y $p_{j-1}$ son funciones polinomiales y sea $p_{j+1}(u) = \sum _{k\geq 0} b_k u^k$.
Entonces,
\begin{align*}
  p(u)& =\tfrac{1}{2j+3} (1-u^2) p'_{j+1}(u)- u p_{j+1}(u) = \tfrac 1{2j+3}\sum_{k\geq 0} \big( (k+1) b_{k+1} -  (k+2j+2)b_{k-1}\big) u^{k}
\end{align*}
es también una función polinomial en $u$, (donde como es usual denotamos $b_{-1}=0$). Sea $m=\gr(p)$.
Luego, para $k> m$
tenemos
$$b_{k+1}= \tfrac{(k+2j+1)}{k+2}\, b_{k-1},$$
entonces $|b_{k+1}|\ge|b_{k-1}|$ para $k> m$. Como $\lim_{u \to 1^- }{p_{j+1}(u)}$ existe, tenemos que $b_k=0$ para $k> m$.
Por lo tanto, $p_{j+1}$ es un  polinomio.
Además, concluimos que si $n< \ell$ la $j$-ésima entrada de $P$ es $p_j=0$ para $n<j\leq \ell$.
\end{proof}

\begin{cor}\label{Ppolynomial}
  La función $P(u)$ asociada a una función esférica irreducible $\Phi$ de tipo $\pi_\ell$ es una función  polinomial que toma valores en $\CC^{\ell+1}$.
\end{cor}

\begin{proof}[\it Demostración] Solo debemos recordar que la función $P$ asociada a la función esférica irreducible $\Phi_\ell^{(m_1,m_2)}$ de tipo $\pi_\ell$ pertenece a $\mathcal V_{n,\mu}$, con
 $n=m_1-m_2\in \NN_0$ y $\mu=-\tfrac {\ell(\ell+2)} 4 + (m_1+1)m_2$ (ver  Corolario \ref{autovsP}).
\end{proof}

\section{Desde $P$ a $\Phi$}\label{depafi}
\

\subsection{Correspondencia entre Polinomios y Funciones Esféricas}
\

En esta subsección probaremos que, dado $\pi_\ell\in\hat K$, hay una correspondencia biunívoca
entre los polinomios a valores vectoriales $P(u)$  autofunciones de $\overline D$ y $ \overline E$ tales que la primera entrada del vector $P(1)$ es igual a  $1$ y las funciones esféricas irreducibles $\Phi$ de tipo $\pi_\ell$.

\begin{thm}
Hay una correspondencia uno a uno entre las  funciones esféricas irreducibles $\Phi$ de tipo $\pi_\ell\in\hat K$, $\ell\in2\NN_0,$
 y las funciones $P$ en $\mathcal V_{n,\mu}\neq 0$ tales que  $P(1)=(1, a_1,\dots , a_\ell)$.
\end{thm}
\begin{proof}[\it Demostración] Dada una función esférica irreducible de tipo $\pi_\ell$, tenemos ya demostrado que la función $P$ asociada a ella pertenece al espacio $\mathcal V_{n,\mu}$ y que la primera entrada $P(1)$ es igual a uno.

Las clases de equivalencia de funciones esféricas irreducibles de $(G,K)$ de
tipo $\pi_\ell$ están parametrizadas por el conjunto de todos lo pares  $(m_1,m_2)\in \ZZ^2$
tales que
 $$m_1\geq\tfrac\ell2\geq |m_2|.$$

Cada función esférica irreducible $\Phi^{(m_1,m_2)}_{\ell}$ corresponde a una autofunción vectorial $P_\ell^{(m_1,m_2)}$ de los operadores $\overline D$ y $\overline E$ cuyos autovalores,
 de acuerdo al Corolario \ref{autovs}, son  respectivamente
\begin{equation*}
  \begin{split}
  \lambda_\ell^{(m_1,m_2)}&= -(m_1-m_2)(m_1-m_2+2)  ,\\
  \mu_\ell^{(m_1,m_2)}&=-\tfrac {\ell(\ell+2)} 4 + (m_1+1)m_2.
  \end{split}
\end{equation*}

Fácilmente uno puede ver que para diferentes pares  $(m_1,m_2)$
los pares de  autovalores $(\lambda_\ell^{(m_1,m_2)}, \mu_\ell^{(m_1,m_2)})$ son distintos.
Entonces, cada autofunción $P^{(m_1,m_2)}_{\ell}$ está asociada a una única función esférica irreducible $\Phi^{(m_1,m_2)}_{\ell}$.

Por otra parte, por \eqref{formadeP} sabemos que  $P\in \mathcal V_{n,\mu}$
si y solo si $P(u)=(p_0(u),\dots,p_\ell(u))^t$ es de la forma
\begin{align*}
  p_j(u)& = a_j \, {}_2\!F_1\left(\begin{smallmatrix}-n+j,n+j+2\\  j+3/2 \end{smallmatrix}; \tfrac{1-u}2\right),\quad \text{ para todo $0\leq j\leq \ell$},
\end{align*}
 donde $\{a_j\}_{j=0}^\ell$ satisface  \eqref{threetermI}.
Luego,  $a=(a_0,\dots,a_\ell)^t$ es un autovector de la matriz $L(\lambda)$ con autovalor $\mu$.
 En particular no hay más de $\ell+1$ autovectores linealmente independientes.
Si $P\in \mathcal V_{n,\mu}$  entonces  $P$ es una función polinomial; por lo tanto, cuando $n< \ell$ tenemos que
 $$a_j=0, \qquad \text{  para $n<j\leq \ell$} .$$
Entonces, los autovalores de $L(\lambda)$  viven en un subespacio de dimensión $n+1$ y, por lo tanto, hay a lo sumo $n+1$ autovectores linealmente independientes.
Del Corolario \ref{corttr} tenemos que cada autoespacio de $L(\lambda)$ es unidimensional.
Por lo tanto concluimos que, salvo escalares, no hay mas de $\min \{\ell+1, n+1\}$
autovectores de $L(\lambda)$.

Luego, es suficiente demostrar que para cada  $\lambda=-n(n+2)$, con $n\in \NN_0$, hay exactamente  $\min\{\ell+1, n+1\}$ funciones esféricas irreducibles de tipo $\pi_\ell\in K$.

Es fácil verificar (ver Figura \ref{fig2}) que hay exactamente
 $\min\{\ell+1,n+1\}$ pares $(m_1,m_2)\in \ZZ \times \ZZ$, satisfaciendo
\begin{equation}\label{cond}
m_1\geq\tfrac\ell2\geq |m_2|\quad \text{  y }  \quad m_1-m_2=n.
\end{equation}
\begin{figure}
 \centering
\includegraphics{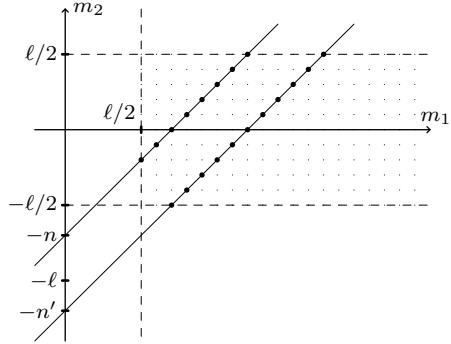}
 \caption[Fig]{Pares $(m_1,m_2)$ que contienen el $K$-tipo $\ell$.}
\label{fig2}
\end{figure}
Esto concluye la demostración del teorema.
\end{proof}

  Si tomamos $n=m_1-m_2$ y $k=\ell/2-m_2$ en Corolario \ref{autovs} tenemos que para una autofunción $P(u)$ de $D$ y $E$, asociada
a una función esférica irreducible de tipo $\pi_\ell\in\hat K$,
 los respectivos autovalores son de la forma
\begin{equation*}
  \begin{split}
    \lambda&= -n(n+2)  ,\qquad  {  \mu=-\tfrac\ell2\left(\tfrac\ell2+1\right)+\left(n-k+\tfrac\ell2+1\right)\left(\tfrac\ell2-k\right)},
  \end{split}
\end{equation*}
\noindent con $0\le n$ y $0\le k\le\min(n,\ell)$.

\smallskip Ahora enunciamos el teorema principal de esta sección.

\begin{thm}\label{main}
Existe una correspondencia uno a uno entre las funciones esféricas \-irre\-du\-ci\-bles de tipo $\pi_\ell\in \hat K$ y las
funciones polinomiales vectoriales $P(u)=(p_0(u),\dots,p_\ell(u))^t$ con
  \begin{align*}
  p_j(u)& = a_j \, {}_2\!F_1\left(\begin{smallmatrix}-n+j,n+j+2\\  j+3/2 \end{smallmatrix}; \tfrac{1-u}2\right),
\end{align*}
 donde $n\in{\NN_0}$,  $a_0=1$ y $\{a_j\}_{j=0}^\ell$ satisface la relación de recurrencia
\begin{equation*}
i \,\tfrac{j(\ell-j+1)(n-j+1)(n+j+1)}{2(2j-1)(2j+1)} \,\, a_{j-1} - \tfrac {j(j+1)}2 \,\,a_j
 - i   \tfrac{(j+1)(\ell+j+2)}2 \,\,a_{j+1} = \mu \, a_j,
\end{equation*}
para $0\leq j\leq \ell-1$, y  $\mu$ de la forma
 $$\mu=-\tfrac\ell2\left(\tfrac\ell2+1\right)+\left(n+k-\tfrac\ell2+1\right)\left(k-\tfrac\ell2\right),$$
para $k\in\ZZ$, $0\le k\le \min\{n,\ell\}$.
\end{thm}

\subsection{Reconstrucción de una Función Esférica Irreducible}\label{reconstruccion}
\

Fijado $\ell\in2{\NN_0}$ sabemos que una función $P=P(u)$ como en el Teorema \ref{main} está asociada a una única función esférica irreducible $\Phi$ de tipo $\pi_\ell\in\hat K$. Ahora mostramos explícitamente como construir la función $\Phi$ a partir de tal $P$.
Recordemos que $P$ es una función polinomial.

Definamos la función vectorial $H(u)=UT(u)P(u)=(h_0(u),\dots,h_\ell(u))$, $u\in[-1,1]$, con $U$ y $T(u)$ como en el Corolario \ref{DEhyp}
y que $\text{diag}\{H(u)\}$ denote la función matricial diagonal cuya $kk$-entrada es igual a la $k$-ésima  entrada de la función vectorial $H(u)$.

  Por otra parte, si consideramos la función  $H:S^3\longrightarrow\End(V_\pi)$ asociada a la función esférica irreducible
$\Phi$, por el Corolario \ref{DEhyp} y (\ref{u}) sabemos que  para $u\in(0,1)$
$$H(\sqrt{1-u²},0,0,u)=H(u).$$
Por lo tanto, como ambas funciones en la igualdad de arriba son analíticas  en $(-1,1)$ y  continuas  en $[-1,1]$, tenemos que
$$H(\sqrt{1-u²},0,0,u)=H(u),$$
para todo $u\in[-1,1]$.

Como $H(kx)=\pi_\ell(k)H(x)\pi_\ell^{-1}(k)$ para cada $x\in S^3$ y $k\in K$, hemos encontrado el valor explícito de la función $H$ en
la esfera $S^3$. Entonces, definimos la función $H:G\longrightarrow\End(V_{\pi_\ell})$ por $$H(g)=H(gK),\qquad g\in G.$$
Finalmente, tenemos que la función esférica irreducible $\Phi$ es de la forma
$$\Phi(g)=H(g)\Phi_{\pi_\ell}(g),\qquad g\in G,$$
donde $\Phi_{\pi_\ell}$ es la función esférica auxiliar introducida en la Subsección \ref{auxiliar}.

\section{Hipergeometrización}\label{hyper}
\

{
En esta sección, para una $\ell\in2{\NN_0}$ fija construiremos una sucesión de polinomios matriciales $P_w$
fuertemente relacionados a funciones esféricas irreducibles de tipo $\pi_\ell\in\hat K$.

 Dado un entero no negativo $w$ y $k =0,1,2,\dots ,\ell$,  los enteros $m_1=w+ \ell/ 2$ y $m_2=-k+ \ell/ 2$ satisfacen
 $$ w+\tfrac \ell 2 \geq \tfrac  \ell 2 \geq  \left| -k+\tfrac \ell 2\right| . $$
 Entonces, podemos considerar
 $$\Phi^{({w}+\ell/2,-{k}+\ell/2)}_{\ell},$$
 la función esférica de tipo $\pi_\ell\in\hat K$ asociada a la $G$-representación $\tau_{(m_1, m_2)}$.

Además consideremos la función matricial $P_w= P_w(u)$,
cuya $k$-ésima columna  ($k =0,1,2,\dots ,\ell$) está dada por la función polinomial $P$ con valores en $\CC^{\ell+1}$,  asociada a
$\Phi^{(m_1,m_2)}_\ell$ con $m_1=w+ \ell/ 2$ y $m_2=-k+ \ell/ 2$.

Por el  Corolario \ref{autovsP}, tenemos que la $k$-ésima columna de $P_w$ es una
autofunción de los operadores $\overline D$ y $\overline E$ con autovalores
$\lambda_w(k)=-(w+k)(w+k+2)$ y $\mu_w(k)=w(\tfrac\ell2-k)-k(\tfrac\ell2+1)$
 respectivamente.

 Explícitamente, tenemos  que la entrada $(j,k)$ de la matriz $P_w$ está dada por
\begin{equation}\label{pes}
[ P_w(u)] _{j,k}=a_j^{w,k} \,\, {}_2\!F_1\left(\begin{smallmatrix}-w-k+j,w+k+j+2\\  j+3/2 \end{smallmatrix}; (1-u)/2\right),
\end{equation}}
donde $ {a}^{w,k}_{0}=1$ para todo $k$ y $\{a_j^{w,k}\}_{j=0}^\ell$ satisface
\begin{equation}
\begin{split}\label{ttr}
& i \tfrac{j(\ell-j+1)(w+k-j+1)(w+k+j+1)}{2(2j-1)(2j+1)} \,\, {a}^{w,k}_{{j-1}} - \tfrac {j(j+1)}2 \,\,{a}^{w,k}_{{{j}}}
 - i   \tfrac{(j+1)(\ell+j+2)}2 \,\,{a}^{w,k}_{{{j+1}}} \\
&= \left(w(\tfrac\ell2-k) -k(\tfrac\ell2+1)\right){a}^{w,k}_{{j}}.
\end{split}
\end{equation}

Por la Proposición \ref{Ppolynomial}, con $n=w+k$, tenemos que $[ P_w(u)] _{j,k}$ es un polinomio en $u$. Por lo tanto, tenemos los siguientes resultados.

\begin{prop}\label{Pweigenfunction}
  El polinomio matricial $P_w$ definido arriba satisface
  $$\overline D  P_w=P_w \Lambda_w \qquad \text{ y } \qquad \overline E P_w=P_w M_w,$$
  donde $\Lambda_w=\sum_{k=0}^\ell \lambda_w(k) E_{kk}$, $M_w=\sum_{k=0}^\ell \mu_w(k) E_{kk}$, y
$$\lambda_w(k)=-(w+k)(w+k+2) \qquad \text{y}\qquad \mu_w(k)=w(\tfrac\ell2-k)-k(\tfrac\ell2+1).$$
\end{prop}

\smallskip
Para el particular caso  $w=0$, tenemos las fórmulas explícitas para $a_j^{0,k}$.

\begin{prop} Tenemos
\begin{equation}\label{as}
\begin{split}
a_j^{0,k} &= \,\frac{(-2i)^{j}k!\,j!}{(k-j)!(2j)!}\quad\text{para } 0\leq j\leq k\leq \ell, \\
a_j^{0,k}&=0 \quad \text{para} \; 0\leq k< j\leq \ell.
\end{split}
\end{equation}
\end{prop}
\begin{proof}[\it Demostración]
 Claramente $a_0^{0,k}=1$; entonces solo necesitamos comprobar que estos $a_j^{0,k}$ sa\-tis\-fa\-cen la siguiente relación de recurrencia de tres términos:
\begin{align}\label{qq}
 i \tfrac{j(\ell-j+1)(k-j+1)(k+j+1)}{2(2j-1)(2j+1)}  {a}^{0,k}_{{j-1}} - \tfrac {j(j+1)}2 {a}^{0,k}_{{{j}}}
 - i\tfrac{(j+1)(\ell+j+2)}2 {a}^{0,k}_{{{j+1}}} = -\tfrac{k(\ell+2)}{2}{a}^{0,k}_{{{j}}}.
 \end{align}
Notemos que si los coeficiente $a_j^{0,k}$ son los dados por (\ref{as}), para $ 0\leq j\leq k\leq \ell$ tenemos
\begin{align*}
ia_{j-1}^{0,k}=-\tfrac{2j-1}{k-j+1}a_j^{0,k}\qquad  \text{y}  \qquad ia_{j+1}^{0,k}=\tfrac{k-j}{2j+1}a_j^{0,k} .
 \end{align*}
Luego, para $ 0\leq j\leq k\leq \ell$ (\ref{qq}) es equivalente
\begin{align*}
-  \tfrac{j(\ell-j+1)(k+j+1)}{2(2j+1)}  {a}^{0,k}_{{j}} - \tfrac {j(j+1)}2 {a}^{0,k}_{{{j}}}
 - \tfrac{(j+1)(\ell+j+2)(k-j)}{2(2j+1)} {a}^{0,k}_{{{j}}} = -\tfrac{k(\ell+2)}{2}{a}^{0,k}_{{{j}}},
 \end{align*}
lo que puede ser comprobado fácilmente.

\noindent Si $j=k+1$ tenemos
 \begin{align*}
 i \tfrac{(k+1)(\ell-(k+1)+1)(k-(k+1)+1)(k+(k+1)+1)}{2(2j-1)(2j+1)} \,\,  {a}^{0,k}_{k}  = 0,
 \end{align*} lo cual es cierto. Y si $j\geq k+2$ simplemente tenemos $0=0$.
Por lo tanto, los coeficientes dados en (\ref{as}) satisfacen (\ref{ttr}) y la demostración está terminada.
\end{proof}

Ahora introducimos   la función matricial $\Psi$ definida como el primer  ``paquete" de funciones esféricas $P_w$ con $w\ge0$, i.e.
 \begin{equation} \label{Fi}
 \Psi(u) =  P_0(u).
\end{equation}
{  Por  \eqref{pes} y \eqref{as} observamos que $\Psi(u)$ es una matriz triangular superior. Más aún, $\Psi(u)=(\Psi_{jk})_{jk}$ es el polinomio dado por \begin{equation}\label{psi}
\Psi _{jk}=\frac{(2j+1)(-2i)^j\, k!j!}{(k+j+1)!}C^{j+1}_{k-j}(u), \quad \text{ para $0\le j\leq k\le\ell$,}
\end{equation}}
  donde $C^{j+1}_{k-j}(u)$ es el polinomio de Gegenbauer
\begin{equation*}
C^{j+1}_{k-j}(u)=\binom{k+j+1}{k-j}\,{}_2\!F_1\left(\begin{smallmatrix}-k+j\,,\, k+j+2\\  j+3/2 \end{smallmatrix}; (1-u)/2\right).
\end{equation*}

Como la $k$-ésima columna de $\Psi$ es una autofunción de $\overline D$ y $\overline E$ con autovalores $\lambda_0(k)=-k(k+2)$ y $\mu_0(k)=-k(\tfrac\ell2+1)$ respectivamente, la función $\Psi$ satisface
\begin{equation}\label{Psiautofuncion}
  \overline D \Psi=\Psi \Lambda_0 \quad \text{ y } \quad   \overline E\Psi= \Psi M_0,
\end{equation}
donde $\Lambda_0=\sum_{k=0}^\ell \lambda_0(k) E_{kk}$ y $M_0=\sum_{k=0}^\ell \mu_0(k) E_{kk}$.

\begin{remark}
\label{psi-1} Las entradas de la diagonal de $\Psi(u)$ son polinomios no nulos, entonces tenemos que $\Psi(u)$ es inversible. \\ Más aún, la inversa $\Psi(u)^{-1}$
es también un polinomio que toma valores de matrices su\-pe\-rio\-res. Esto se puede comprobar fácilmente, por ejemplo, usando  la regla de Cramer, ya que el determinante de $\Psi(u)$ es una constante no nula.
\end{remark}

\begin{thm}\label{hyp}
Sean $\overline D$ y $\overline E$ los operadores diferenciales definidos en
Teorema \ref{puntos}, y sea $\Psi$ la  función matricial cuyas entradas son dadas por  \eqref{psi}. Sea $\widetilde
D = \Psi^{-1} \overline D \Psi$ y $\widetilde E = \Psi^{-1} \overline E
\Psi$, luego
\begin{align*}
\widetilde D F=&(1-u^2) F'' + (-u C + S_1)  F'+  \Lambda_0  F,\qquad\qquad\widetilde E F= (uR_2 + R_1)  F'+ M_0 F,
\end{align*}
para cualquier función $C^\infty$ $F$ en $(0,1)$ con valores en $\mathbb C^{\ell+1}$,
donde\begin{align*}
C &= \sum_{j=0}^\ell (2j+3)E_{jj}, & S_1 &=  {  \sum_{j=0}^{\ell-1} 2(j+1)E_{j,j+1},}\\
R_1 &= \sum_{j=0}^{\ell-1} \tfrac{(j+1)}{2} E_{j,j+1}-\sum_{j=0}^{\ell-1} \tfrac{(\ell-j)}{2} E_{j+1,j},
& R_2 &= \sum_{j=0}^{\ell} (\tfrac \ell2-j)E_{j,j},\\
 \Lambda_0 &= \sum_{j=0}^{\ell}  -j(j+2)E_{j,j}, & \qquad M_0&= {  \sum_{j=0}^{\ell}-j(\tfrac\ell2+1)E_{j,j}}.
\end{align*}
\end{thm}

\begin{proof}[\it Demostración]
Por definición tenemos
  \begin{align*}
\widetilde D\widetilde F&=(1-u^2)\widetilde F'' + \Psi^{-1}[2(1-u^2)\Psi ' - u C \Psi ] \widetilde F'+ \Psi^{-1} \left[(1-u^2) \Psi'' -u C\Psi'-V \Psi\right] \widetilde F,\\
\widetilde E\widetilde F&=\tfrac i2 \Psi^{-1}[(1-u^2)Q_0+Q_1]\Psi
\widetilde F'+ \Psi^{-1}\left[\tfrac i 2\left( (1-u^2)Q_0+Q_1 \right) \Psi'-\tfrac i2
u M\Psi-\tfrac 12 V_0 \Psi\right] \widetilde F.
\end{align*}
Usando \eqref{Psiautofuncion} observamos que
\begin{align*}
 (1-u^2) \Psi'' -u C\Psi'-V \Psi & =\overline D \Psi=\Psi\Lambda_0,\\
\tfrac i 2 \big( (1-u^2)Q_0+Q_1 \big) \Psi'-\tfrac i2 u M\Psi-\tfrac 12 V_0 \Psi & =\overline E \Psi=\Psi M_0.
\end{align*}

Para completar la prueba de este teorema usamos las siguientes propiedades de los polinomios de Gegenbauer (para las tres primeras ver \cite{KS}, página 40; para la última  ver \cite{S75}, página 83, ecuación (4.7.27))
\begin{align}
\label{a}&\frac{d C_n^\lambda}{du}(u)=2\lambda C_{n-1}^{\lambda+1}(u),\\
\label{b}&2(n+\lambda)u C_n^\lambda (u)= (n+1)C_{n+1}^{\lambda}(u)+(n+2\lambda-1)C_{n-1}^{\lambda}(u),\\
\label{c}& (1-u^2)\frac{d C_n^\lambda}{du}(u) +(1-2\lambda)u C_{n}^{\lambda}(u) =-\frac{(n+1)(2\lambda+n-1)}{2(\lambda-1)}C_{n+1}^{\lambda-1}(u),\\
\label{d}&\frac{(n+2\lambda-1)}{2(\lambda-1)} C_{n+1}^{\lambda-1}(u)=  C_{n+1}^{\lambda}(u) -  u C_{n}^{\lambda}(u).
\end{align}

Necesitamos establecer las siguientes igualdades
\begin{align}
 \label{paraD}[2(1-u^2)\Psi ' - u C \Psi ] &=\Psi (-u C + S_1),\\
 \label{paraE}\tfrac i2[(1-u^2)Q_0+Q_1]\Psi  &= \Psi (uR_2 + R_1).
 \end{align}
Como (\ref{paraD}) es a una igualdad matricial, mirando el lugar $(j,k)$
tenemos
\begin{align*}
2(1-u^2)\Psi'_{jk} - u C_{jj} \Psi_{jk}  =-\Psi_{jk} u C_{kk} &+ \Psi_{j,k-1} (S_1)_{k-1,k}.
\end{align*}
Multiplicando ambos lados por  $\frac{(k+j+1)!}{(2j+1)(-2i)^j\, k!j!}$ y usando \eqref{psi} tenemos
\begin{align*}
2(1-u^2)\frac{d }{du}C_{k-j}^{j+1} - u (2j+3)C_{k-j}^{j+1} = -u(2k+3)C_{k-j}^{j+1} +2(k+j+1)C_{k-j-1}^{j+1}.
\end{align*}
y tomando  $\lambda=j+1$ y $n=k-j$ tenemos
\begin{align*}
2(1-u^2)\frac{d C_n^\lambda}{du} - u (2\lambda+1)C_n^\lambda = -u(2(n+\lambda)+1)C_n^\lambda+2(n+2\lambda-1)C_{n-1}^\lambda.
\end{align*}
 Para ver que esta identidad se mantiene, usamos \eqref{c} para escribir $\frac{d C_n^\lambda}{du}$ en términos de $C_{n}^{\lambda} $ y $C_{n+1}^{\lambda-1}$,
y \eqref{b} para expresar $C_{n-1}^{\lambda}$ en términos de $C_{n}^{\lambda}$ y  $C_{n+1}^{\lambda}$,  luego
reconocemos la identidad \eqref{d}. Entonces tenemos probada \eqref{paraD}.

\smallskip
Ahora necesitamos verificar la igualdad matricial \eqref{paraE}. La entrada $(j,k)$ está dada por (ver Teorema \ref{puntos}
para la definición de las matrices  $Q_0$ y $Q_1$)
\begin{align*}
\tfrac i2(1-u^2)(Q_0)_{j,j+1}\Psi_{j+1,k}  +\tfrac i2(Q_1)_{j,j-1} \Psi_{j-1,k} &=\\
 u \Psi_{jk} (R_2)_{kk}+\Psi_{j,k+1} (R_1)_{k+1,k}  &+ \Psi_{j,k-1} (R_1)_{k-1,k}.
 \end{align*}
Nuevamente, multiplicando ambos lados por $\frac{(k+j+1)!}{(-2i)^j\, k!j!}$ y tomando $\lambda=j+1$ y $n=k-j$, obtenemos
\begin{align*}
(1-u^2)\frac{\lambda^2(\ell+\lambda+1)}{n+2\lambda} &C_{n-1}^{\lambda+1}- \frac{(\ell-\lambda+2)(n+2\lambda-1)}{4}C_{n+1}^{\lambda-1}=
u\left(\tfrac\ell2-(n+\lambda-1)\right)(2\lambda-1)C_{n}^{\lambda}
\\&-\frac{(\ell-n-\lambda+1)(2\lambda-1)(n+\lambda)}{2(n+2\lambda)}C_{n+1}^{\lambda}+\frac{(2\lambda-1)(n+2\lambda-1)}{2}C_{n-1}^{\lambda}.
\end{align*}

\noindent Ahora, primero usamos  (\ref {c}) combinado con \eqref{a} para escribir $C_{n-1}^{\lambda+1}$ en términos de $C_{n}^{\lambda}$ y $C_{n+1}^{\lambda-1}$, y luego usamos \eqref{b} para expresar $C_{n-1}^{\lambda}$ en términos de $C_{n}^{\lambda}$ y  $C_{n+1}^{\lambda}$. Entonces tenemos
\begin{align*}
&\tfrac{(\lambda^2-\ell\lambda-3\lambda-\ell n-2n)(2\lambda-1)}{2(n+2\lambda)}\left(u C_n^\lambda(u) + \frac{(n+2\lambda-1)}{2(\lambda-1)} C_{n+1}^{\lambda-1}(u) -C_{n+1}^\lambda(u)\right)=0,
\end{align*}
que es cierto por (\ref{d}).
\end{proof}

{ Para un dado $w\in\NN_0$  introducimos la  función matricial
\begin{equation}\label{Pwtilde}
  \widetilde P_w= \Psi^{-1} P_w,
\end{equation}
 donde  $P_w$ es el polinomio matricial introducido en  \eqref{pes} y
$\Psi$ la función con valores en matrices triangulares superiores dada en \eqref{psi}. Recordemos
que la función $\Psi^{-1}$ es una función polinomial, como notamos en  la Observación \ref{psi-1}. Por lo tanto  $\widetilde P_w$ es también una función polinomial.
El siguiente resultado es una consecuencia directa de la Proposición \ref{Pweigenfunction} y del Teorema \ref{hyp}.
\begin{cor}  \label{tildePweigenfunction}
  Los polinomios matriciales $\widetilde P_w= \Psi^{-1} P_w$  satisfacen
  $$\widetilde D  \widetilde P_w=\widetilde P_w \Lambda_w \qquad \text{ y } \qquad \widetilde  E \widetilde P_w=\widetilde P_w M_w,$$
  donde $\Lambda_w=\sum_{k=0}^\ell \lambda_w(k) E_{kk}$, $M_w=\sum_{k=0}^\ell \mu_w(k) E_{kk}$, y
  $$\lambda_w(k)=-(w+k)(w+k+2) \qquad \text{y}\qquad \mu_w(k)=w(\tfrac\ell2-k)-k(\tfrac\ell2+1).$$
\end{cor}

\section{Polinomios Ortogonales}\label{Matrix Orthogonal Polynomials}
\

El objetivo de esta sección es construir una sucesión de polinomios ortogonales matriciales clásica a partir de lo estudiado previamente. Esto significa exhibir un peso
$W$ con soporte real, una sucesión
$\{\widetilde P_w\}_{w\ge0}$ de polinomios matriciales tales que $\gr(\widetilde P_w)=w$ con el coeficiente director de
$\widetilde P_w$ no singular, ortogonal con respecto a $W$, y un operador diferencial de segundo orden (simétrico) $\widetilde D$ tal que $\widetilde D\widetilde P_w=\widetilde P_w\Lambda_w$ donde $\Lambda_w$ es una matriz diagonal. Más aún, también conseguiremos una operador diferencial de primer orden (simétrico) $\widetilde  E$ tal que $\widetilde E\widetilde P_w=\widetilde P_w M_w$, donde $M_w$ es una matriz diagonal real.

Por $\widetilde D$ (ver Teorema \ref{hyp}) obtenemos un nuevo operador diferencial $D$ al hacer el cambio de variables $s=(1-u)/2$. Entonces,
\begin{align*}
 D F=&s(1-s) F'' - \left(\frac{S_1-C}{2}+sC\right)  F'+  \Lambda_0 F.
\end{align*}

\subsection{Soluciones Polinomiales de $DF=\lambda F$}
\

Estamos interesados en estudiar las soluciones polinomiales de la ecuación $DF=\lambda F$; en particular queremos saber cuándo existen soluciones polinomiales.
Comenzamos con
\begin{align}\label{ocho}
s(1-s) F'' + \left(B-sC\right)  F'+  \left(\Lambda_0 -\lambda\right)F=0,
\end{align}
donde
\begin{align*}
B& =\frac{C-S_1}{2}=  \sum_{j=0}^\ell (j+\tfrac 32 )E_{jj} - \sum_{j=0}^{\ell-1} (j+1)E_{j,j+1}, \\
C&= \sum_{j=0}^\ell (2j+ 3 )E_{jj},   \qquad S_1=\sum_{j=0}^{\ell-1} (j+1)E_{j,j+1}, \qquad \Lambda_0  =- \sum_{j=0}^\ell j(j+2 )E_{jj}.
\end{align*}

 Esta ecuación es un caso de la ecuación hipergeométrica matricial estudiada en \cite{T03}. Como los autovalores de $B$ no están  en $-\NN_0$, la función $F$ es determinada por $F_0=F(0)$. Para $|s|<1$ está dada por
 \begin{align*}
F(s)={}_2\!H_1\left(\begin{smallmatrix}C,-\Lambda_0+\lambda\\  B \end{smallmatrix}; s\right)F_0=\sum_{j=0}^{\infty}\frac{s^j}{j!} [B;C;-\Lambda_0+\lambda]_j F_0, \qquad F_0\in \CC^{\ell+1},
\end{align*}
donde el símbolo $[B;C;-\Lambda_0+\lambda]_j$ está definido inductivamente por
\begin{align*}
[B;C;-\Lambda_0+\lambda]_0   &=1,\\
[B;C;-\Lambda_0+\lambda]_{j+1} &=
\left(B+j\right)^{-1}(j(C+j-1)-\Lambda_0+\lambda)[B;C;-\Lambda_0+\lambda]_j ,
\end{align*}
para todo $j\geq 0$.

Por lo tanto, existe una solución polinomial de \eqref{ocho} si y solo si el coeficiente
$[B;C;-\Lambda_0+\lambda]_j$ es una matriz singular
para algún $j\in {\NN_0}$.
Más aún, tenemos que hay una solución polinomial de grado
$w$ de (\ref{ocho}) si y solo si  existe $F_0\in\CC^{\ell+1}$ tal que
$[B;C;-\Lambda_0+\lambda]_{w}F_0\neq 0$  y
$$(w(C+w-1)-\Lambda_0+\lambda) F_w=0,\quad \text { donde } \quad F_w= [B;C;-\Lambda_0+\lambda]_{w}F_0.$$
La matriz
\begin{equation}\label{Mw}
  M_w=w(C+w-1)-\Lambda_0+\lambda=\sum_{j=0}^{\ell}((j+w)(j+w+2)+\lambda)E_{jj}
\end{equation}
es diagonal. Entonces, es una matriz singular si y solo si $\lambda$ es de la forma
$$\lambda_w(k)=-(k+w)(k+w+2),$$
para $0\leq k\leq\ell.$ Tenemos el siguiente resultado.
\begin{prop}\label{polynomialsol}
Dado $\lambda\in \CC$, la ecuación $DF=\lambda F$ tiene una solución polinomial si y solo si $\lambda$ es de la forma $-n(n+2)$ para
$n\in\NN_0$.
\end{prop}

\begin{remark}\label{igualdadautoval}
  Sea $w\in \NN_0$, $0\leq k \leq \ell$. El autovalor $\lambda_w(k)$ satisface
$\lambda_w(k)  =-n(n+2)$ con $n\in \NN_0$ si y solo si  $n=w+k$.  En particular,
$$\lambda_w(k)=\lambda_{w'}(k') \quad \text{si y solo si } \quad w+k=w'+k'.$$
\end{remark}

Ahora queremos estudiar en más detalle las soluciones polinomiales de $DF=\lambda F$. A\-su\-ma\-mos que $\lambda=-n(n+2)$ con $n\in \NN_0$.
Sea  $$F(s)=\sum_{i=0}^w F_i s^i$$ una una solución polinomial de grado $w$ de la ecuación $DF=\lambda F$.
Tenemos que los coeficientes $F_i$ son definidos recursivamente por
$$F_{i+1}= (B+i)^{-1} M_{i}F_i= [B;C;-\Lambda_0+\lambda]_{i}F_0,$$ donde
$M_i$ es la matriz definida en \eqref{Mw}.

La función $F$ es un polinomio de grado $w$ si y solo si  existe $F_0\in\CC^{\ell+1}$ tal que
\begin{equation}\label{condition}
  F_w = [B;C;-\Lambda_0+\lambda]_{w}F_0 \neq 0 \quad \text{ y } \quad
M_w F_w =0.
\end{equation}
 Como dijimos, la matriz $M_w$ es singular si y solo si $\lambda=\lambda_w(k)$ para algún $k$ tal que tal que $0\leq k\leq \ell$, y por consiguiente, tenemos
\begin{equation}\label{w=n-k}
  w=n-k.
\end{equation}
Además, observamos que $M_w F_w =0$ si y solo si $F_w$ está en el
subespacio  generado por $e_{k}$ (el $k$-ésimo vector de la base canónica de $\CC^{\ell+1}$).

Ahora probaremos que siempre es posible elegir un vector $F_0\in \CC^{\ell+1}$ tal que $[B;C;-\Lambda_0+\lambda]_w F_0=e_k$.
Recordemos que
$$[B;C;-\Lambda_0+\lambda]_w F_0=(B+w-1)^{-1} M_{w-1}\dotsm M_1B^{-1}M_0 F_0,$$
 y que, para $0\le i\le w$, las matrices $M_{w-i}$ son definidas por
$$M_r= \sum_{j=0}^{\ell}(\lambda_w(k)-\lambda_{w-i}(j))E_{jj}.$$
En particular el kernel de la matriz $M_{w-i}$ es  $\CC e_{k+i}$ para $0\leq i\leq \min\{w,\ell-k \}$, puesto que  $\lambda_w(k)-\lambda_{w-i}(j)=0$
si y solo si $j-i=k$ (ver Observación \ref{igualdadautoval}).

Sea $W_k$ el subespacio  $\CC^{\ell+1}$ generado por $\{e_0,e_1,\dots,e_{k}\}$.
Observamos que para cada $j\in\NN_0$ tenemos que $W_k$ es invariante por
$\left(B+j\right)^{-1}$ pues es una matriz triangular superior.
Para $j<w$, $M_j$ es diagonal con las primeras $k+1$ entradas no nulas, entonces la restricción de $M_j$ a $W_k$ es inversible.
Por lo tanto, existe  $F_0$   tal que $[B;C;-\Lambda_0+\lambda]_w F_0=e_k$.
Entonces
$$F(u)={}_2\!H_1\left(\begin{smallmatrix}C,-\Lambda_0+\lambda\\B \end{smallmatrix}; u\right)F_0$$ es polinomio vectorial de grado $w$.
Observamos que  $F_0$ es único en $ W_k $, pero no en $\CC^{\ell+1}$. De todos modos, la $k$-ésima entrada de $F$ es un polinomio de grado $w$, y todas las otras entradas son de grados menores pues el coeficiente director $F_w$ es siempre un múltiplo de $e_k$ .

\smallskip
 De esta forma, tenemos los siguientes resultados. En el primero fijamos el  autovalor $\lambda=-n(n+2)$ con $n\in \NN_0$, y en el segundo fijamos el  grado $w$ de los polinomios $F$.

\begin{prop}\label{grado}
Sea $n\in \NN_0$ y $\lambda=-n(n+2)$.
Si $P$ es una solución polinomial de $DF= \lambda F$ de grado $w$, entonces $n-\ell \leq w \leq n$.\\
Recíprocamente, para cada $w\in \NN_0$ tal que $n-\ell \leq w \leq n$, la ecuación $DF=\lambda F$ tiene una solución polinomial grado $w$.  Más aún, si $w=n-k$, $0\leq k \leq \ell$,
el coeficiente director de cualquier solución polinomial de $DF=\lambda F$ es un múltiplo de $e_k$.
\end{prop}
\begin{proof}[\it Demostración]
Por \eqref{w=n-k} tenemos que existe una solución polinomial de grado $w$ si y solo si $w=n-k$, con $0\leq k\leq \ell$.
En tal caso, tenemos probado que existe  $F_0\in \CC^{\ell+1}$ tal que  \eqref{condition} se mantiene y tenemos que $F_w$ es un múltiplo de $e_k$.
\end{proof}

\begin{prop}\label{corgrado}
Dado $w \in \NN_0$ existen exactamente $\ell +1$ valores de $\lambda$
tal que $DF=\lambda F$ tiene una solución polinomial de grado $w$, más precisamente
$$\lambda= \lambda_w(k)=-(k+w)(k+w+2), \quad 0\leq k\leq \ell.$$
Para cada $k$ el coeficiente director de cualquier solución polinomial de $DF=\lambda_w(k) F$ es un múltiplo de $e_k$, el $k$-ésimo vector en la base canónica de $\CC^{\ell+1}$.

\end{prop}

\smallskip

\subsection{Nuestra Sucesión de Polinomios Ortogonales Matriciales}\label{MVOP}
\

Los polinomios matriciales  $$\widetilde P_w(u)=\Psi(u)^{-1} P_w(u)$$ fueron introducidos en \eqref{Pwtilde}.

\begin{prop}\label{columns}
Las columnas $\{\widetilde P_w^k\}_{k=0,\dots,\ell}$ de $\widetilde
P_w$ son  polinomios de grado $w$. Más aún,
\begin{align*}
  \gr \left(\widetilde P_w^k\right)_k = w\quad \text{ y } \quad
  \gr \left(\widetilde P_w^k\right)_j< w, \text{ para }  j\neq k.
\end{align*}
\end{prop}

\begin{proof}[\it Demostración]
La $k$-ésima columna  de la matriz $\widetilde P_w=\Psi^{-1}P_w$ es el vector $\widetilde P_w^k=\Psi^{-1}P_w^k$, donde $P_w^k$ es la $k$-ésima columna de $P_w$.
Por el  Corolario \ref{tildePweigenfunction} tenemos que $\widetilde P_w^k$  es una función polinomial que satisface
$$\widetilde D \widetilde P_w^k =\lambda_w(k) \widetilde P_w^k, \qquad \widetilde E \widetilde P_w^k =\mu_w(k) \widetilde P_w^k,$$ para
$\lambda_w(k)=-(w+k)(w+k+2)$ y $\mu_w(k)=w(\tfrac\ell2-k)-k(\tfrac\ell2+1).$

Si $w'$ denota le grado de $\widetilde P_w^k $,  entonces
tenemos que $w+k-\ell\leq w'\leq w+k$ (ver  Proposición \ref{grado}).
Luego escribimos $$\widetilde P_{w}^{k}=\sum_{j=0}^{w'}  A_j u^j, \qquad \text{ con } A_j\in\CC^{\ell+1}.$$

Más aún, por la  Proposición \ref{corgrado} tenemos que el correspondiente autovalor de $D$ debe ser igual a $\lambda_{w'}(k') = -(w'+k')(w'+k'+2) $, con $0\leq k' \leq \ell$, y el coeficiente director $A_{w'}$ tiene todas sus entradas nulas, excepto para las $k'$-ésimas.
Por la Observación \ref{igualdadautoval} obtenemos que
$$w-w'= k'-k.$$

Por otra parte, $\widetilde P_w^k$ satisface $\widetilde E \widetilde P_w^k =\mu_w(k) \widetilde P_w^k$, donde
  $$\widetilde E F= (uR_2 + R_1)  F'+ M_0 F $$
  es el operador diferencial dado en Teorema \ref{hyp}.
Entonces, los coeficientes de los polinomios $\widetilde P_w^k$ satisfacen
$$\big (j R_2 + M_0-\mu_w(k) \big) A_j  +(j+1) R_1 A_{j+1} =0,\qquad
\text{ para } 0\leq j \leq w',
$$
denotando   $A_{w'+1}=0$. En particular, para $j=w'$ tenemos
\begin{equation}\label{aux2}
\big({w'} R_2  +M_0- \mu_w (k)\big ) A_{w'}=0.
\end{equation}
Por el Teorema \ref{hyp} tenemos
\begin{align*}
{w'} R_2  +M_0- \mu_w (k)I &=\sum_{j=0}^\ell \big( w'(\tfrac \ell 2 -j)-j(\tfrac \ell 2 +1)-\mu_w(k) \big) E_{jj}\\
& = \sum_{j=0}^\ell \big( w'(\tfrac \ell 2 -j)- w(\tfrac \ell 2 -k)+(k-j)(\tfrac \ell 2 +1) \big) E_{jj}.
\end{align*}
 Por \eqref{aux2} tenemos que la $k'$-ésima entrada de la matriz ${w'} R_2  +M_0- \mu_w (k) I$ debe ser cero, entonces
$$0= w'(\tfrac \ell 2 -k')- w(\tfrac \ell 2 -k)+(k-k')(\tfrac \ell 2 +1).$$
Como $w-w'= k'-k$, tenemos
$ 0=(w-w')(1+k+w)$,
lo cual implica que $w'=w$ y $k'=k$.

Por lo tanto, $\widetilde P_w^k$ es un polinomio de grado $w$ y la única entrada no nula del coeficiente director de $\widetilde P_w^k$ es la $k$-ésima.
\end{proof}

\subsection{El Producto Interno} \label{the inner product}
\

Dada una representación irreducible de dimensión finita
$\pi=\pi_{\ell}$ de $K$ en el espacio vectorial $V_\pi$, sea
$(C(G)\otimes \End (V_\pi))^{K\times K}$ el espacio de todas las funciones continuas  $\Phi:G\longrightarrow \End(V_\pi)$  tales que
$\Phi(k_1gk_2)=\pi(k_1)\Phi(g)\pi(k_2)$ para todo $g\in G$,
$k_1,k_2\in K$. Equipamos $V_\pi$ con un producto interno  tal que
$\pi(k)$ sea  unitario para todo $k\in K$. Entonces, introducimos un
producto interno en el espacio vectorial $(C(G)\otimes \End
(V_\pi))^{K\times K}$ definiendo
\begin{equation}\label{pi}
\langle \Phi_1,\Phi_2 \rangle =\int_G \tr ( \Phi_1(g)\Phi_2(g)^*)\, dg\, ,
\end{equation}
donde $dg$ denota la medida de Haar de $G$ normalizada por $\int_G
dg=1$, y $\Phi_2(g)^*$ denota la adjunta de $\Phi_2(g)$ con
respecto al producto interno en $V_\pi$.

Usando las relaciones de ortogonalidad de Schur para las representaciones irreducibles  unitarias de $G$, tenemos que  si $\Phi_1$ y $\Phi_2$ son funciones esféricas irreducibles no equivalentes, entonces
son ortogonales con respecto al producto interno $\langle\cdot ,\cdot\rangle$, i.e. $$\langle \Phi_1,\Phi_2 \rangle =0.$$
En particular, si $\Phi_1$ y $\Phi_2$ son dos funciones esféricas irreducibles de tipo $\pi=\pi_\ell$, escribimos como antes (ver \eqref{defHg})  $\Phi_1=H_1\Phi_\pi$ y
$\Phi_2=H_2\Phi_\pi$
 y denotamos $$H_1(u)=(h_0(u),\cdots, h_\ell(u))^t, \qquad H_2(u)=(f_0(u),\cdots, f_\ell(u))^t,$$
 como hicimos en la Subsección \ref{reconstruccion}.

\begin{prop}\label{prodint}  Si $\Phi_1, \Phi_2\in \left(C(G)\otimes \End (V_\pi)\right)^{K\times K}$ entonces
\begin{equation*}
\langle \Phi_1,\Phi_2 \rangle = \frac{2}{\pi}\int_{-1}^1
\sqrt{1-u^2}\sum_{j=0}^\ell {h_j(u)}\overline{f_j(u)}\, du=  \frac{2}{\pi}\int_{-1}^1
\sqrt{1-u^2} H_2^*(u) H_1(u) \, du.
\end{equation*}
\end{prop}
\begin{proof}[\it Demostración]
Consideremos el elemento $E_1=E_{14}-E_{41}\in \lieg$. Entonces, como
$\so(4)_\CC\simeq \mathfrak{sl}(2,\CC)\oplus\mathfrak{sl}(2,\CC)$,
$\text{ad }E_1$ tiene $0$ y $\pm i$ como  autovalores con multiplicidad $2$.

Sea $A=\exp \RR E_1$ el subgrupo de Lie de $G$ de todos los elementos de
la forma
$$a(t)= \exp tE_1=\left(\begin{matrix} \cos t&0& 0&
\sin t\\ 0&1&0&0	\\ 0&0&1&0\\ -\sin t&0& 0&\cos
t\end{matrix}\right)\, ,	\qquad t\in \RR.$$

Ahora el Teorema 5.10, página 190 en \cite{He00}, establece que para cada
$f\in C(G/K)$ y un apropiado $c_*$
$$\int_{G/K} f(gK)\,dg_K=c_*\int_{K/M}\Big(\int_{-\pi}^{\pi}
\delta_*(a(t))f(ka(t)K)\,dt\Big)\,dk_M\,,$$ donde la función $\delta_*:A\longrightarrow \RR $ está definida por
$$\delta_*(a(t))=\prod_{\nu\in\Sigma^+} |\sin it \nu(E_1)|,$$ y  $dg_K$ y $dk_M$
son respectivamente las medidas invariantes a izquierda en $G/K$ y
 $K/M$ nor\-ma\-li\-za\-das por $\int_{G/K} dg_K=\int_{K/M} dk_M=1$.
 Recordemos que  $M$ fue introducido en \eqref{Msubgrupo} y coincide con  el centralizador de  $A$ en $K$.
  En nuestro caso tenemos $\delta_*(a(t))=\sin^2t $.

Como la función $g\mapsto \tr(\Phi_1(g)\Phi_2(g)^*)$ es invariante
por multiplicación  a izquierda y derecha por elementos en $K$, tenemos
\begin{align}\label{pi2}
\langle \Phi_1,\Phi_2\rangle = c_* \int_{-\pi}^{\pi} \sin ^2t
\,\tr\left( \Phi_1(a(t))\Phi_2(a(t))^*\right)\,dt.
\end{align}
Además, para cada
$t\in [-\pi,0]$, tenemos que $(I-2(E_{11}+E_{22}))a(t)(I-2(E_{11}+E_{22}))=a(-t)$, con
$I-2(E_{11}+E_{22})$ en $K$. Entonces tenemos
\begin{align*}
\langle \Phi_1,\Phi_2\rangle = 2c_* \int_{0}^{\pi} \sin ^2t
\,\tr\left( \Phi_1(a(t))\Phi_2(a(t))^*\right)\,dt.
\end{align*}

Por la definición de la función auxiliar $\Phi_\pi(g)$ (ver Subsección \ref{auxiliar}), tenemos que
$$\Phi_1(a(t))\Phi_2(a(t))^*=\ H_1(a(t))H_2(a(t))^*.$$ Por lo tanto, haciendo el cambio de variables $\cos(t)=u$, obtenemos
$$\langle \Phi_1,\Phi_2\rangle = 2c_* \int_{-1}^{1} \sqrt{1-u²}\sum_{j=0}^\ell
 {h_j}(u)\overline{{f_j(}u)}du.$$
Para encontrar el valor de $c_*$  consideramos el caso trivial $\Phi_1=\Phi_2=I$ en (\ref{pi}) y (\ref{pi2}). Por lo tanto, obtenemos
$$\ell+1= c_* \int_{-\pi}^{\pi} \sin ^2t
\,(\ell+1)\,dt.$$
Entonces, tenemos que $c_*=\pi^{-1}$ y la proposición sigue.
\end{proof}

En los Teoremas \ref{puntos} y \ref{hyp} conjugamos los operadores diferenciales $D$ y $E$ a los operadores hipergeómetricos $\widetilde D$ y $\widetilde E$ dados por
$$ \widetilde D=(UT(u)\Psi(u))^{-1} D  (UT(u)\Psi(u))\quad \text{y} \quad \widetilde E=(UT(u)\Psi(u))^{-1} E  (UT(u)\Psi(u)).$$
Por lo tanto, en términos de las funciones
$$\widetilde P_1=(U T(u) \Psi(u))^{-1}H_1 \text{ y } \widetilde P_2=(U T(u) \Psi(u))^{-1}H_2,$$
tenemos
\begin{equation*}
\langle \widetilde P_1,\widetilde P_2\rangle_W=\int_{-1}^1 {\widetilde P}_2(u)^*\,W(u)\widetilde P_1(u)\,du,
\end{equation*}
donde el peso matricial $W(u)$ está dado por
\begin{equation}\label{peso}
       W(u)=\frac2\pi\sqrt{1-u^2}\Psi^*(u)T^*(u)U^*U T(u) \Psi(u).
      \end{equation}

Si consideramos que  $\Psi(u)$ es un polinomio en $u$, que $U$ es una constante matricial y que $T(u)=\sum_{j=0}^\ell (1-u^2)^{j/2}$, tenemos que $W(u)$ es una función continua en el intervalo cerrado $[-1,1]$. Entonces, $W$ es un peso matricial en $[-1,1]$ con momentos finitos de todos los órdenes.

\smallskip
Consideremos ahora la sucesión de polinomios matriciales $\{\widetilde P_w\}_{w\geq0}$ introducida en \eqref{Pwtilde}.
La $k$-ésima columna de  $\widetilde P_w(u)$  está dada por un vector
$\widetilde P_w^k(u)$ asociado a la función esférica irreducible de tipo $\pi_\ell$
$$\Phi^{({w}+\ell/2,-{k}+\ell/2)}_{\ell}.$$

\noindent Por lo tanto cuando   $(w,k)\neq(w',k')$ tenemos que
$\widetilde P_w^k$ y $\widetilde P_{w'}^{k'}$ son ortogonales con respecto a $W$, i.e.
\begin{equation}\label{columnortog}
  \langle \widetilde P_w^k, \widetilde P_{w'}^{k'} \rangle_W=0 \qquad  \text{si  $(w,k)\neq(w',k')$}.
\end{equation}
En otras palabras, esta sucesión de polinomios matriciales cuadra en la teoría de Krein, y tenemos el siguiente teorema.

\begin{thm}\label{sucesion}
Los polinomios matriciales $\widetilde P_w$, $ w\geq0$, forman una sucesión de polinomios or\-to\-go\-na\-les con respecto a $ W$, las cuales son autofunciones de los operadores diferenciales simétricos $\widetilde D$ y $\widetilde E$ que aparecen en el  Teorema \ref{hyp}.
Más aún,
 $$ \widetilde D \widetilde P_w =\widetilde P_w \Lambda_w \qquad \text{y }\qquad
\widetilde E \widetilde P_w =\widetilde P_w M_w,$$
donde $\Lambda_w= \sum_{k=0}^\ell\lambda_w(k)E_{kk}$, y  $M_w= \sum_{k=0}^\ell \mu_w(k)E_{kk}$, con
\begin{align*}
\lambda_w(k)=-(w+k)(w+k+2) \qquad \text{y}\qquad \mu_w(k)=w(\tfrac\ell2-k)-k(\tfrac\ell2+1).
\end{align*}
\end{thm}

\begin{proof}[\it Demostración]
Por la Proposición \ref{columns} obtenemos que cada  columna de $\widetilde P_w$ es un  polinomio de grado $w$.
Más aún, $\widetilde P_w$ es polinomio cuyo coeficiente director es
una matriz no singular.

Dados $w$ y $w'$, enteros no negativos, usando \eqref{columnortog} tenemos
\begin{align*}
\langle \widetilde P_{w'},\widetilde P_{w} \rangle _W&=
\int_{-1}^1 \widetilde P_w(u)^*W(u) \widetilde P_{w'}(u) \, du =\sum_{k,k'=0}^\ell \int_{-1}^1 \Big( \widetilde P_w^k(u)^*W(u)
\widetilde P_{w'}^{k'}(u) \, du \Big)\, E_{k,k'}\\
& =\sum_{k,k'=0}^\ell   \delta_{w,w'}\delta_{k,k'}   \Big(\int_{-1}^1 \widetilde P_w^k(u)^*W(u) \widetilde P_{w'}^{k'}(u)
\, du \Big)\, E_{k,k'}\\& =\delta_{w,w'}\sum_{k=0}^\ell   \int_{-1}^1 \Big(\widetilde P_w^k(u)^*W(u) \widetilde P_{w'}^{k}(u) \, du,\Big) \, E_{k,k},
\end{align*}
lo que prueba la ortogonalidad. Más aún, también muestra que $\langle \widetilde P_w,\widetilde P_{w} \rangle _W$ es una matriz diagonal. Ahora, gracias al Corolario \ref{tildePweigenfunction} solo resta demostrar que los operadores $\widetilde D$ y $\widetilde E$ son simétricos con respecto a $W$.

Haciendo algunos simples cálculos tenemos que
$$\langle\widetilde D \widetilde P_w,\widetilde P_{w'}\rangle=\delta_{w,w'}\langle \widetilde P_w,\widetilde P_{w'}\rangle \Lambda_w
  =\delta_{w,w'} \Lambda_w^* \langle \widetilde P_w,\widetilde P_{w'}\rangle =\langle \widetilde P_w,\widetilde D\widetilde P_{w'}\rangle,$$
para cada $w,w'\in\NN_0$, pues $\Lambda_w$ es real y diagonal.
Esto concluye la demostración del teorema.
\end{proof}

  \chapter{Las Esferas y los Espacios Proyectivos Reales}\label{pn}

\begin{flushright}{\it
``El amor: un apóstol ciego que rueda con la memoria esquilada a través del más hermoso de los silencios tristes y embiste desafortunadamente azul como si fuera el mar."}\\Daniel Veranda.
\end{flushright}
\

En este capítulo establecemos una directa relación entre las funciones esféricas de la esfera  $n$-dimensional $S^n\simeq\SO(n+1)/\SO(n)$ y las  funciones esféricas del  espacio proyectivo real $n$-dimensional $P^n(\mathbb{R})\simeq\SO(n+1)/\mathrm{O}(n)$. Precisamente, para $n$ impar una función en $\SO(n+1)$
es una función esférica irreducible de algún tipo $\pi\in\hat\SO(n)$ si y solo si  es una función esférica irreducible de algún tipo $\gamma\in\hat {\mathrm{O}}(n)$. Cuando $n$ es par esto también es cierto para ciertos tipos, y en los otros casos exhibimos una clara correspondencia entre las  funciones esféricas irreducibles de ambos pares  $(\SO(n+1),\SO(n))$ y $(\SO(n+1),\mathrm{O}(n))$. Finalmente demostramos que encontrar todas las funciones esféricas de un par es equivalente a hacer los mismo con el otro.

\section{Funciones Esféricas Zonales}\label{zonales}\

Es sabido que los espacios simétricos conexos compactos de rango uno son de la forma $X\simeq G/K$, donde $G$ y $K$ son:
\begin{enumerate}
\item[i)] {\ \hbox to 4cm{$G=\mathrm{SO}(n+1),$\hfill}\ \hbox to 6cm{ $K=\mathrm{SO}(n)$,\hfill}\ $X=S^n.$}
\item[ii)] {\ \hbox to 4cm{$G=\mathrm{SO}(n+1),$\hfill}\ \hbox to 6cm{ $K=\mathrm{O}(n)$,\hfill}\  $X=P^n(\RR)$.}
\item[iii)] {\ \hbox to 4cm{$G=\mathrm{SU}(n+1),$\hfill}\ \hbox to 6cm{ $K=\mathrm{S}(\mathrm{U}(n)\times\mathrm{U}(1))$,\hfill}\  $X=P^n(\CC)$.}
 \item[iv)] {\ \hbox to 4cm{$G=\mathrm{Sp}(n+1),$\hfill}\ \hbox to 6cm{ $K=\mathrm{Sp}(n)\times\mathrm{Sp}(1)$,\hfill}\  $X=P^n(\HH)$.}
\item[v)]   {\ \hbox to 4cm{$G=F_{4(-52)},$\hfill}\ \hbox to 6cm{ $K=\mathrm{Spin}(9)$,\hfill}\  $X=P^{2}(Cay)$.}
\end{enumerate}
Las funciones esféricas zonales (i.e. de  $K$-tipo trivial)  en $X\simeq G/K$ son las autofunciones del operador Laplace-Beltrami que solo depende de la distancia $d(x,o)$, $x \in X$, donde $o$ es el origen en $X$. En cada caso las llamamos $\varphi_0, \varphi_1, \varphi_2, \ldots, $ con $\varphi_0=1$, y sea $\varphi_j^*(\theta)$ la correspondiente función inducida en $[0, L]$  por $\varphi_j$, donde $L$ es el diámetro de $X$; notemos que por razones de compacidad se tiene una cantidad numerable de autofunciones.  Estas  funciones son finalmente polinomios de Jacobi
\begin{equation}\label{zonal}
\varphi_j^*(\theta) =c_j\, P_j^{(\alpha,\beta)}(\cos \lambda \theta),
\end{equation}
con $c_j$ definida por la condición $\varphi_j(0)=1$ y $\lambda$, $\alpha$ y $ \beta $  dependen del par $(G,K)$. Renormalizando la distancia podemos asumir $\lambda=1$ y $L=\pi$, (confrontar \cite[página 171]{H65}). Ahora citamos la siguiente lista de \cite[página 239]{K73}:

\begin{enumerate}
\item[i)]  {\ \hbox to 6cm{$ G/K\simeq  S^n:$\hfill}                 \ \hbox to 4.5cm{$\alpha=(n-2)/2$,\hfill}$\beta=(n-2)/2$.}
\item[ii)] {\ \hbox to 6cm{$ G/K\simeq  P^n(\RR):$\hfill}            \ \hbox to 4.5cm{$\alpha=(n-2)/2$,\hfill}$\beta=-1/2$.}
\item[iii)]{\ \hbox to 6cm{$ G/K\simeq  P^n(\CC):$\hfill}            \ \hbox to 4.5cm{$\alpha=n-1$,\hfill}$\beta=0$.}
\item[iv)] {\ \hbox to 6cm{$ G/K\simeq  P^n(\mathbb H):$\hfill}      \ \hbox to 4.5cm{$\alpha=2n-1$,\hfill}$\beta=1$.}
\item[v)]  {\ \hbox to 6cm{$ G/K\simeq  P^2(Cay):$\hfill}            \ \hbox to 4.5cm{$\alpha=7$,\hfill}$\beta=3$.}
\end{enumerate}


Por lo tanto, a primera vista, las funciones esféricas zonales en la esfera y en el espacio proyectivo real aparentemente son dos familias completamente diferentes, para aclarar este punto podemos remitirnos a la Sección \ref{appendix}. En este capítulo se prueba que conocer todas las funciones esféricas asociadas a la esfera $n$-dimensional es equivalente a conocer las funciones esféricas en el  espacio proyectivo real $n$-dimensional. Precisamente, establecemos una directa relación entre las funciones esféricas matriciales del par $(\mathrm{SO}(n+1),\mathrm{SO}(n))$ y las del par $(\mathrm{SO}(n+1),\mathrm{O}(n))$.
En primera instancia probamos que, para $n$ par las funciones esféricas de la esfera y las funciones esféricas del espacio proyectivo real son las mismas, i.e., una función $\Phi$ en $\mathrm{SO}(n+1)$ es una función esférica irreducible de tipo $\pi\in \hat{\mathrm{SO}}(n)$ si y solo si existe $\gamma\in\hat{\mathrm{O}}(n)$ tal que la función $\Phi$ es un función esférica de tipo $\gamma$. Cuando $n$ es impar hay algunos casos particulares en los cuales se tiene la misma situación que en el caso $n$ par, y mostramos que estos casos son fácilmente distinguibles con sólo mirar el peso máximo de los correspondiente $\SO(n)$-tipos. Para el caso general mostramos como cada función esférica irreducible del espacio proyectivo está explícitamente relacionada con dos funciones esféricas de la esfera, ver el Teorema \ref{Matrix}.

\section{Preliminares}
\

Como mencionamos en el Capítulo \ref{funciones_esfericas}, las funciones esféricas de tipo $\delta$ aparecen de forma natural
considerando representaciones de $G$. Si $g\mapsto U(g)$ es una
representación continua de $G$, digamos en un espacio vectorial
topológico $E$ completo, localmente convexo y Hausdorff, entonces
$$P(\delta)=\int_K \chi_\delta(k^{-1})U(k)\, dk$$ es una
proyección continua de $E$ en $P(\delta)E=E(\delta)$. La función
$\Ph:G\longrightarrow \End(E(\delta))$ definida por
$$\Phi(g)a=P_\delta \tau(g)a,\quad g\in G,\; a\in E(\delta),$$ es una función
esférica de tipo $\delta$.

 Si la representación $g\mapsto U(g)$ es irreducible  entonces la función
esférica asociada $\Ph$ es también irreducible. Recíprocamente, cualquier
función esférica en un grupo compacto $G$ se consigue de esta forma desde una representación irreducible de dimensión finita de $G$.

Ahora recodamos algunos hechos  (\cite[\S 5.5.5]{GW09}) acerca de cómo  uno obtiene las representaciones irreducibles de dimensión finita de $\OO(n)$ a partir de las representaciones irreducibles de dimensión finita de $\SO(n)$, para entender en profundidad el resultado principal del capítulo: Teoremas \ref{par}, \ref{impar} y \ref{Matrix}.

Tomemos $a\in\mathrm{O}(n)$ dependiendo de $n$:
\begin{align*}
 a&=\diag(1,\dots,1,-1), &\text{si $n$ es par,}\\
a&=\diag(-1,\dots,-1), &\text{si $n$ es impar.}
\end{align*}
Y sea $\phi$ el automorfismo de $\SO(n)$ definido por
 $$\phi(k)=aka,$$
para todo $k\in \SO(n)$. Notemos que  cuando $n$ es impar $\phi$ es trivial y $\OO(n)=\SO(n)\times F$, donde $F=\{1,a\}$. Por lo tanto en este caso las representaciones irreducibles de dimensión finita de $\OO(n)$ son de la forma
$\gamma=\pi\otimes 1$ o $\gamma=\pi\otimes \epsilon$ donde $\pi$ es una representación irreducible de dimensión finita de
$\SO(n)$ y $\epsilon$ es el carácter no trivial de $F$. Entonces tenemos el siguiente teorema.

\begin{thm} Si $n$ es impar $\OO(n)=\SO(n)\times F$, además $\hat\SO(n)\times \hat F$ puede ser identificado con el dual unitario de $\OO(n)$ bajo la biyección  $([\pi],1)\mapsto [\pi\otimes1]$ y $([\pi],\epsilon)\mapsto [\pi\otimes\epsilon]$.
\end{thm}

Si $n$ es par tenemos $\OO(n)=\SO(n)\rtimes F$. Denotemos por $V_\pi$ el espacio vectorial asociado a $\pi\in \hat\SO(n)$, entonces sea $V_{\pi_\phi}=V_\pi$ y definimos   $\pi_\phi:\SO(n)\rightarrow\End(V_{\pi_\phi})$ como la representación irreducible de $\mathrm{SO}(n)$ dada por
 $$\pi_\phi=\pi\circ\phi.$$
En esta situación consideraremos dos casos: $\pi_\phi\sim\pi$ en la Subsección \ref{equiv}  y $\pi_\phi\nsim\pi$ en la Subsección \ref{nequiv}.

\subsection{Cuando $\pi_\phi$ es Equivalente a $\pi$}\label{equiv}
\

Tomemos $A\in\GL(V)$ tal que $\pi_\phi(k)=A\pi(k)A^{-1}$ para todo $k\in \SO(n)$. Entonces
$$\pi(k)=\pi(a(aka)a)=\pi_\phi(aka)=A\pi(aka)A^{-1}=A\pi_\phi(k)A^{-1}=A^2\pi(k)A^{-2}.$$
Por lo tanto, por el Lema de Schur, tenemos $A^2=\lambda I$. Cambiando $A$ por $\sqrt{\lambda^{-1}} A$ podemos asumir que $A^2=I$. Sea $\epsilon_A$ la representación de $F$ definida por \begin{equation*}
\epsilon_A(a)=A.
\end{equation*}
 Ahora definimos $\gamma=\pi\cdot\epsilon_A:\OO(n)\rightarrow\GL(V)$ como
\begin{equation}\label{gammaA}
 \gamma(kx)=\pi(k)\epsilon_A(x), \qquad\text{              para } (k,x)\in\SO(n)\times F,
\end{equation}
 y es fácil verificar que $\gamma$ es una representación irreducible de $\OO(n)$. Más aún, si $B$ es otra solución de $\pi_\phi(k)=B\pi(k)B^{-1}$ para todo $k\in \SO(n)$, y $B^2=I$, entonces $B=\pm A$. De hecho, por el Lema de  Schur, $B=\mu A$ y $\mu^2=1$. En el conjunto de todos los pares $(\pi,A)$ introducimos la relación de equivalencia $(\pi,A)\sim(T\pi T^{-1},TAT^{-1})$, donde $T$ es una transformación lineal biyectiva de $V$ en otro espacio vectorial y denotamos $[\pi,A]$ para la clase de equivalencia de $(\pi,A)$.

\begin{prop} \label{equivalente} Cuando $n$ es par $\OO(n)=\SO(n)\rtimes F$. Asumamos que $\pi_\phi\sim\pi$.  Si $\gamma=\pi\cdot\epsilon_A$, entonces $\gamma$ es una representación irreducible de $\OO(n)$. Más aún, $\gamma'=\pi\cdot\epsilon_{-A}$, es otra representación irreducible de $\OO(n)$ no equivalente a $\gamma$. Además, el conjunto $\{[\pi,A]: \pi_\phi(k)=A\pi(k)A^{-1}, A^2=I\}$ puede ser incluido en $\hat\OO(n)$ vía la función
$[\pi,A]\mapsto [\pi\cdot\epsilon_A]$.
\end{prop}

\subsection{Cuando $\pi_\phi$ no es Equivalente a $\pi$}\label{nequiv}
\

Asumamos que $n$ es par y que $\pi_\phi\nsim\pi$. Consideramos el $\SO(n)$-módulo $V_\pi\times V_\pi$ y definimos
\begin{align}\label{gamma}
 \gamma(k)(v,w)&=(\pi(k)v,\pi_\phi(k)w),&
\gamma(ka)(v,w)=(\pi(k)w,\pi_\phi(k)v),
\end{align}
para todo $k\in \SO(n)$ y $v,w\in V_\pi$. Entonces es fácil verificar que  $\gamma:\OO(n)\rightarrow\GL(V_\pi\times V_\pi)$ es un representación de $\OO(n)$.

\begin{prop}\label{equivalente2} Asumamos que $n$ es par y $\pi\in\hat\SO(n)$, entonces $\OO(n)=\SO(n)\rtimes F$. Más aún, si $\pi_\phi\nsim\pi$ definimos
$\gamma:\SO(n)\times F\rightarrow\GL(V_\pi\times V_\pi)$ como en \eqref{gamma}, entonces $\gamma$ es una representación irreducible de $\OO(n)$. Además $\gamma':\SO(n)\times F\rightarrow\GL(V_\pi\times V_\pi)$ definida por
\begin{align*}
\gamma'(k)(v,w)&=(\phi(k)v,kw), & \gamma'(ka)(v,w)=(\phi(k)w,kv),
 \end{align*}
para todo $k\in \SO(n)$ y $v,w\in V_\pi$, es una representación irreducible de $\OO(n)$ que es equivalente a $\gamma$. Por otra parte, el conjunto $\{\{[\pi],[\pi_\phi]\}: [\pi]\in \hat \SO(n)\}$ puede ser incluido en $\hat\OO(n)$ vía la función
$\{[\pi],[\pi_\phi]\}\mapsto [\gamma]$.
\end{prop}

\begin{thm} Asumamos  que $n$ es par. Dividimos $\hat\OO(n)$ en dos conjuntos disjuntos: (a) $\{[\gamma]:\gamma_{\vert \SO(n)} \; \text{ irreducible}\}$ y (b) $\{[\gamma]:\gamma_{\vert \SO(n)} \;\text{ reducible}\}$.
\noindent\begin{enumerate} \item[(a)]
Si $[\gamma]$ está en el primer conjunto y $\pi=\gamma_{\vert \SO(n)}$, entonces
$$\pi_\phi(k)=\pi(aka)=\gamma(aka)=\gamma(a)\pi(k)\gamma(a),$$
para todo $k\in \SO(n)$. Por lo tanto $\pi_\phi\sim\pi$ y $\gamma$ es equivalente a la representación $\pi\cdot\epsilon_A$ construida por $\pi$ y $A=\gamma(a)$ en \eqref{gammaA}.
\item[(b)]
Si $[\gamma]$ está en el segundo conjunto, sea $W$ el espacio  de representación de $\gamma$. Sea $V_\pi<W$ un $\SO(n)$-módulo irreducible. Entonces $W=V_\pi\oplus V_{\pi_\phi}$ como $\SO(n)$-módulos, y $\gamma$ es equivalente a la representación $\gamma'$ definido en  $V_\pi\times V_\pi$ por \eqref{gamma}.
\end{enumerate}
\end{thm}

\subsection{Los Pesos Máximos de $\pi$ y $\pi_\phi$}	
\

Cuando $n$ es par sería muy útil saber cuándo $\pi\in\hat \SO(n)$ es equivalente a $\pi_\phi$. En esa dirección probamos un criterio muy sencillo en términos del peso máximo de $\pi$.

Dada $\ell\in\mathbb{N}$, sabemos por \cite{V92} que el peso máximo de una representación irreducible $\pi$ de $\mathrm{SO}(2\ell)$ es de la forma ${\bf m}_\pi=(m_1, m_2, m_3, \dots , m_\ell)$ $\in\mathbb{Z}^\ell$, con $$m_1\ge m_2\ge m_3\ge\dots\ge m_{\ell-1} \ge |m_\ell|.$$

 A continuación enunciamos un  simple resultado que relaciona los pesos máximos de $\pi$ y $\pi_\phi$.
\begin{thm}\label{weights}
 Si ${\bf m}_\pi=(m_1, m_2, m_3, \dots , m_\ell)$ es el peso máximo de $\pi\in\hat{\mathrm{SO}}(2\ell)$ entonces ${\bf m}_{\pi_\phi}=(m_1, m_2, m_3, \dots , -m_\ell)$ es  el peso máximo de $\pi_\phi$.
\end{thm}

Las matrices $I_{ki},\, 1\le i<k\le 2\ell$, con $-1$ en el lugar $(k,i)$, $1$ en el lugar $(i,k)$ y ceros en el resto, forma una base del álgebra de Lie $\mathfrak{so}(2\ell)$.
El subespacio generado
$$\lieh={\langle I_{21},I_{43},\dots,I_{2\ell,2\ell-1}\rangle}_\CC$$
es una subálgebra de Cartan de $\mathfrak{so}(2\ell,\CC)$.

Ahora consideramos
 $$H=i(x_1I_{21}+\dots+x_\ell I_{2\ell,2\ell-1})\in\lieh,$$
y sea $\epsilon_j\in\lieh^*$ definida por $\epsilon_j(H)=x_j$, $1\le j\le\ell$. Entonces para $1\le j<k\le\ell$, las siguientes matrices son vectores raíces de $\lieso(2\ell,\CC)$:
\begin{equation}\label{rootvectors}
\begin{split}
X_{\epsilon_j+\epsilon_k}&=I_{2k-1,2j-1}-I_{2k,2j}-i(I_{2k-1,2j}+I_{2k,2j-1}),\\
X_{-\epsilon_j-\epsilon_k}&=I_{2k-1,2j-1}-I_{2k,2j}+i(I_{2k-1,2j}+I_{2k,2j-1}),\\
X_{\epsilon_j-\epsilon_k}&=I_{2k-1,2j-1}+I_{2k,2j}-i(I_{2k-1,2j}-I_{2k,2j-1}),\\
X_{-\epsilon_j+\epsilon_k}&=I_{2k-1,2j-1}+I_{2k,2j}+i(I_{2k-1,2j}-I_{2k,2j-1}).
\end{split}
\end{equation}
Elegimos el siguiente conjunto de raíces positivas
$$\Delta^+=\{\epsilon_j+\epsilon_k, \epsilon_j-\epsilon_k: 1\le j<k\le\ell\}$$
teniendo entonces $${\bf m}_\pi=m_1\epsilon_1+m_2\epsilon_2+\ldots+m_\ell\epsilon_\ell.$$
Entonces, podemos probar el Teorema \ref{weights}.

\begin{proof}[\it Demostración] Primero demostraremos que el vector peso máximo $v_\pi$ de la representación $\pi$ es también vector peso máximo de $\pi_\phi$: Para cada vector
vector raíz $X_{\epsilon_j\pm \epsilon_k}$ con $1\le j < k <\ell$ tenemos que $\operatorname{Ad}(a) X_{\epsilon_j\pm \epsilon_k}$ $=X_{\epsilon_j\pm \epsilon_k}$. Y, cuando $k=\ell$ tenemos que $\operatorname{Ad}(a) X_{\epsilon_j\pm \epsilon_\ell}=X_{\epsilon_j\mp \epsilon_k}.$ Luego, si denotamos por $\dot\pi$ y $\dot\pi_\phi$ las representaciones de la complexificación de $\mathfrak{so}(2\ell)$ correspondiente a $\pi$ y $\pi_\phi$, respectivamente, tenemos $\dot\pi\circ \operatorname{Ad}(a)=\dot\pi_\phi$ y entonces
$$\dot\pi_\phi(X_{\epsilon_j\pm \epsilon_k}) v_\pi =  \dot\pi(\operatorname{Ad}(a) X_{\epsilon_j\pm \epsilon_k} ) v_\pi =  \dot\pi(X_{\epsilon_j\pm \epsilon_k}) v_\pi = 0,$$
para $1\le j <k<\ell$. Cuando $k=\ell$ tenemos
$$\dot\pi_\phi(X_{\epsilon_j\pm \epsilon_\ell}) v_\pi =  \dot\pi(\operatorname{Ad}(a) X_{\epsilon_j\pm \epsilon_\ell} ) v_\pi =  \dot\pi(X_{\epsilon_j\mp \epsilon_\ell}) v_\pi = 0.$$
Por lo tanto $v_\pi$ es un  vector peso máximo de $\pi_\phi$.

Notemos que $\operatorname{Ad}(a)I_{2j,2j-1}=I_{2j,2j-1}$ para $1\le j<\ell$ y que $\operatorname{Ad}(a)I_{2\ell,2\ell-1}$ $=-I_{2\ell,2\ell-1}$, entonces
$$\dot\pi_\phi(iI_{2j,2j-1}) v_\pi =  \dot\pi(\operatorname{Ad}(a) iI_{2j,2j-1}) v_\pi =  \dot\pi(iI_{2j,2j-1}) v_\pi = m_j v_\pi,$$
para $1\le j <\ell$. Cuando $k=\ell$ tenemos
$$\dot\pi_\phi(iI_{2\ell,2\ell-1}) v_\pi =  \dot\pi(\operatorname{Ad}(a) iI_{2\ell,2\ell-1} ) v_\pi =  -\dot\pi(iI_{2\ell,2\ell-1}) v_\pi =-m_\ell v_\pi.$$
Luego el peso máximo de $\pi_\phi$ es
$${\bf m}=(m_1,m_2,\ldots,m_{\ell-1},-m_\ell),$$
que es lo que queríamos demostrar.\end{proof}

\begin{cor}
 Una representación irreducible $\pi$ de $\mathrm{SO}(2\ell)$, $\ell\in\NN$, de peso máximo ${\bf m}=(m_1,m_2,\ldots,m_\ell)$ es equivalente a $\pi_\phi$ si y solo si $m_\ell=0$.
\end{cor}

\section{Funciones Esféricas en $S^{\lowercase{n}}$ y en $P^{\lowercase{n}} (\RR)$}
\

Sea $(V_\tau,\tau)$ una representación irreducible unitaria de $G=\SO(n+1)$ y $(V_\pi,\pi)$ una representación irreducible unitaria de $\SO(n)$.

Supongamos  que $n$ es impar. Entonces $\OO(n)=\SO(n)\times F$ y las representaciones unitarias irreducibles de $\OO(n)$ son de la forma $\gamma=\pi\otimes 1$ o $\gamma=\pi\otimes\epsilon$. Asumamos que $\pi$ es una subrepresentación de $\tau_{\vert_{\SO(n)}}$. Observemos que $a\in\OO(n)$ como elemento de $G$ es  $-I\in G$. Claramente $\tau(-I)=\pm I$. Tomemos $\gamma=\pi\otimes 1$ si  $\tau(-I)=I$ y $\gamma=\pi\otimes\epsilon$ si  $\tau(-I)=-I$. Entonces $\gamma$ es una subrepresentación de $\tau_{\vert_{\OO(n)}}$. Sean $\Phi^{\tau,\pi}$ y $\Phi^{\tau,\gamma}$, respectivamente, las correspondientes funciones esféricas de $(G,\SO(n))$ y $(G,\OO(n))$.
\begin{thm}\label{par}
Asumamos que $n$ es impar. Si $\Phi^{\tau,\pi}(-I)=I$ tomamos $\gamma=\pi\otimes 1$, y si $\Phi^{\tau,\pi}(-I)=-I$ tomamos $\gamma=\pi\otimes\epsilon$. Entonces $\Phi^{\tau,\pi}(g)=\Phi^{\tau,\gamma}(g)$ para todo $g\in G$.
\end{thm}
\begin{proof}[\it Demostración]
Como  $\SO(n)$-módulos $V_\tau=V_\pi\oplus V_\pi^\perp$. Pero como $\tau(a)=\tau(-I)=\pm I$ la descomposición $V_\tau=V_\pi\oplus V_\pi^\perp$ es también una $\OO(n)$-descomposición. Luego, la $\SO(n)$-proyección $P_\pi$ en $V_\pi$ es igual a la $\OO(n)$-proyección $P_\gamma$ en $V_\pi$. Por lo tanto $\Phi^{\tau,\pi}(g)=P_\pi\tau(g)P_\pi=P_\gamma\tau(g)P_\gamma=\Phi^{\tau,\gamma}(g)$, completando la prueba.
\end{proof}

\

Asumamos ahora  que $n$ es par, entonces $\OO(n)=\SO(n)\rtimes F$.
Supongamos que $\pi\in\hat \SO(n)$  y que $\pi\sim\pi_\phi$. Entonces $\gamma=\pi\cdot\epsilon_A$, donde $A\in\GL(V_\pi)$ es tal que $\pi_\phi=A\pi A^{-1}, A^2=I$, es una representación irreducible de $\OO(n)$ en $V_\pi$ como hemos visto  en Proposición \ref{equiv}. Ahora usamos este resultado para obtener el siguiente teorema.

\begin{thm}\label{impar} Asumamos que $n$ es par. Identifiquemos $a=\diag(1,\dots,1,-1)\in\OO(n)$ con
$a=\diag(1,\dots,1,-1,-1)\in\SO(n+1)$.  Supongamos que $\pi$ es una subrepresentación de $\tau_{\vert_{\SO(n)}}$ y que $\pi\sim\pi_\phi$. Sea  $A=\Phi^{\tau,\pi}(a)$ y tomemos
$\gamma=\pi\cdot\epsilon_A$. Entonces $\Phi^{\tau,\pi}(g)=\Phi^{\tau,\gamma}(g)$ para todo $g\in G$.
\end{thm}
\begin{proof}[\it Demostración]
Lo primero en demostrarse será que  $A\in\GL(V_\pi)$, $\pi_\phi=A\pi A^{-1}$ y  $A^2=I$. Esto último sigue directamente de que  $a^2=e$.

Para todo $k\in \SO(n)$ tenemos
\begin{equation}\label{equivalencia}
\pi_\phi(k)=\pi(aka)=\tau(aka)_{\vert V_\pi}=\tau(a)\tau(k)\tau(a)_{\vert V_\pi}.
\end{equation}
Por lo tanto $\tau(a)V_\pi$ es un $\SO(n)$-módulo equivalente a $\pi_\phi$. Como $\pi_\phi\sim\pi$, y por la multiplicidad uno del par $(\SO(n+1),\SO(n))$, obtenemos  que $V_\pi=\tau(a)V_\pi$. Por lo tanto $A=\tau(a)_{\vert V_\pi}\in\GL(V_\pi)$ y $\pi_\phi=A\pi A^{-1}$.
Luego $V_\pi$ es un $\OO(n)$ submódulo de $V_\tau$  y la correspondiente representación es $\gamma=\pi\cdot\epsilon_A$. Esto implica que la $\SO(n)$-proyección $P_\pi$ en $V_\pi$ es igual a la $\OO(n)$-proyección $P_\gamma$ en $V_\pi$. Por lo tanto $\Phi^{\tau,\pi}(g)=P_\pi\tau(g)P_\pi=P_\gamma\tau(g)P_\gamma=\Phi^{\tau,\gamma}(g)$. Finalmente observamos que $\Phi^{\tau,\pi}(a)=P_\pi\tau(a)P_\pi=\tau(a)_{\vert V_\pi}$, completando la demostración.
\end{proof}

\

Asumamos que $n$ es par, y tomemos $\pi\in\hat \SO(n)$ tal que $\pi\nsim\pi_\phi$. Consideramos el $\SO(n)$-módulo $V_\pi\times V_{\pi_{\phi}}$ y definamos
$\gamma(k)(v,w)=(\pi(k)v,\pi_{\phi}(k)w)$, $\gamma(ka)(v,w)=(\pi(k)w,\pi_{\phi}(k)v)$ para todo $k\in \SO(n)$, $v\in V_\pi$ y $w\in V_{\pi_{\phi}}$. Entonces $\gamma$ es una representación irreducible de $\OO(n)$ en  $V_\pi\times V_{\pi_{\phi}}$.

\begin{thm}\label{Matrix} Asumamos que $n$ es par. Identifiquemos $a=\diag(1,\dots,1,-1)\in\OO(n)$ con
$a=\diag(1,\dots,1,-1,-1)\in\SO(n+1)$. Supongamos que $\pi$ es un subrepresentación de $\tau_{\vert_{\SO(n)}}$ y que $\pi\nsim\pi_\phi$. Entonces $\tau(a)V_\pi\sim V_{\pi_{\phi}}$ como $\SO(n)$-módulos y $V_\pi\oplus \tau(a)V_\pi$ es un $\OO(n)$-submódulo irreducible  de $V_\tau$ equivalente a la representación irreducible $\gamma$ en
$V_\pi\times V_{\pi_{\phi}}$ construida arriba. Más aún,
$$\Phi^{\tau,\gamma}(g)=\left(\begin{matrix} \Phi^{\tau,\pi}(g) & \Phi^{\tau,\pi}(ga)\\ \Phi^{\tau,\pi_\phi}(ga) & \Phi^{\tau,\pi_\phi}(g)\end{matrix}\right)$$
para todo $g\in G$.
\end{thm}
\begin{proof}[\it Demostración]
Que $\tau(a)V_\pi\sim V_{\pi_\phi}$ como $\SO(n)$-módulos  sigue por \eqref{equivalencia}. Además, si hacemos la identificación $V_\pi\times V_{\pi_\phi}\sim V_\pi\oplus \tau(a)V_\pi$ vía el $\SO(n)$-isomorfismo
$(v,w)\mapsto v+\tau(a)w$, y usando de nuevo que $\pi_\phi(k)w=\tau(a)\tau(k)\tau(a)w$ (ver \eqref{equivalencia}), tenemos
\begin{equation*}
\begin{split}
\gamma(k)(v,w)&= (\pi(k)v,\pi_\phi(k)w)=\left(\pi(k)v,\tau(a)\tau(k)\tau(a)w\right)\\
&\sim\pi(k)v+\tau(k)\tau(a)w=\tau(k)(v+\tau(a)w)
\end{split}
\end{equation*}
 para todo $k\in \SO(n)$, y
$$\gamma(a)(v,w)=(w,v)\sim(w+\tau(a)v)=\tau(a)(v+\tau(a)w).$$
Esto prueba que $V_\pi\oplus \tau(a)V_\pi$ como $\OO(n)$-submódulo de $V_\tau$ es equivalente a la representación irreducible $\gamma$ en $V_\pi\times V_{\pi_{\phi}}$. Por lo tanto $P_\gamma=P_\pi\oplus P_{\pi_\phi}$.

Luego, para todo $g\in G$ tenemos,

\begin{equation*}
\begin{split}
\Phi^{\tau,\gamma}(g)&=(P_\pi\oplus P_{\pi_\phi})\tau(g)(P_\pi\oplus P_{\pi_\phi})=P_\pi\tau(g)P_\pi\oplus P_\pi\tau(g) P_{\pi_\phi}\oplus  P_{\pi_\phi}\tau(g)P_\pi\oplus  P_{\pi_\phi}\tau(g) P_{\pi_\phi}.
\end{split}
\end{equation*}
Entonces en forma matricial tenemos
$$\Phi^{\tau,\gamma}(g)=\left(\begin{matrix} \Phi^{\tau,\pi}(g) & \Phi_{12}(g)\\ \Phi_{21}(g) & \Phi^{\tau,\pi_\phi}(g)\end{matrix}\right),$$
donde $\Phi_{21}(g)=P_{\pi_\phi}\tau(g)_{|_{V_\pi}}$ y $\Phi_{12}(g)=P_{\pi}\tau(g)_{|_{\tau(a)V_{\pi}}}$.

Por la identidad $\Phi^{\tau,\gamma}(ga)=\Phi^{\tau,\gamma}(g)\tau(a)_{|_{V_\pi\oplus \tau(a)V_\pi}}$ obtenemos
$$\left(\begin{matrix} \Phi^{\tau,\pi}(ga) & \Phi_{12}(ga)\\ \Phi_{21}(ga) & \Phi^{\tau,\pi_\phi}(ga)\end{matrix}\right)=\left(\begin{matrix} \Phi^{\tau,\pi}(g) & \Phi_{12}(g)\\ \Phi_{21}(g) & \Phi^{\tau,\pi_\phi}(g)\end{matrix}\right)\left(\begin{matrix} 0 & I\\ I & 0\end{matrix}\right),$$
lo cual es equivalente a $\Phi_{12}(g)=\Phi^{\tau,\pi}(ga)$ y  $\Phi_{21}(g)=\Phi^{\tau,\pi_\phi}(ga)$.
El teorema se ha demostrado.
\end{proof}

\section{$K$-Tipos  Triviales}\label{appendix}
\

Las funciones esféricas irreducibles de $K$-tipo  trivial de  $(\SO(n+1),\SO(n))$ y $(\SO(n+1),\OO(n))$ son, respectivamente, las funciones esféricas zonales de $S^n$ y $P^n(\RR)$. De acuerdo a nuestros Teoremas  \ref{par} y \ref{impar} las  funciones esféricas zonales $\phi$ de  $P^n(\RR)$, como funciones en $\SO(n+1)$,  coinciden con las funciones esféricas zonales $\varphi$ de $S^n$ tales que $\varphi(-I)=1$.

Como dijimos en la Subsección \ref{zonales} las funciones esféricas zonales en la esfera $n$-dimensional y en el correspondiente espacio proyectivo real están, respectivamente,  dadas por
\begin{align*}
\varphi_j^*(\theta) =c_j\, P_j^{\left(\tfrac{n-2}{2},\tfrac{n-2}{2}\right)}(\cos \theta), \qquad
\phi_j^*(\theta) =c'_j\, P_j^{\left(\tfrac{n-2}{2},-\tfrac{1}{2}\right)}(\cos \theta),
\end{align*}
con $c_j$, $c'_j$ escalares tales que $\varphi_j(0)=1=\phi_j(0)$, y $0\le\theta\le\pi$.

En este apéndice explicamos esta aparente inconsistencia. Para empezar, notemos que en ambos espacios la métrica elegida está normalizada por el diámetro  $L=\pi$. Para un dado $g\in \SO(n+1)$ denotamos $\theta(g)$ la distancia la esfera entre $g\cdot o$ y el origen $o$, y análogamente denotamos $\theta'(g)$ la distancia en el espacio proyectivo entre $g\cdot o$ y el origen $o$. Entonces, no es difícil ver que
\begin{align*}
2\theta(g)&=\theta'(g),& &\text{para } 0\le\theta(g)\le\pi/2,\\
2\pi-2\theta(g)&=\theta'(g),&   &\text{para } \pi/2\le\theta(g)\le\pi.
\end{align*}
Por lo tanto tenemos que
\begin{equation}\label{cos}
 \cos(2\theta(g))=\cos(\theta'(g)),
\end{equation}
para cualquier $g\in \SO(n+1)$.
Por otra parte, por \cite[(3.1.1)]{AAR} sabemos que los polinomios de Jacobi tienen la siguiente propiedad
\begin{equation}\label{Jac}
P_{2k}^{(\alpha,\alpha)}(x)=\frac{k!(\alpha+1)_{2k}}{(2k)!(\alpha+1)_{k}}  P_{k}^{(\alpha,-1/2)}(2x^2-1).
\end{equation}
Entonces, si ponemos $x=\cos(\theta(g))$ en \eqref{Jac} tenemos que
$$
P_{2k}^{(\alpha,\alpha)}(\cos(\theta(g)))=\frac{k!(\alpha+1)_{2k}}{(2k)!(\alpha+1)_{k}}  P_{k}^{(\alpha,-1/2)}(\cos(2\theta(g))),
$$
de aquí, usando \eqref{cos} tenemos que $\varphi_{2j}^*(\theta(g))=\phi_j^*(\theta'(g))$ para todo $g\in\SO(n+1)$. En otras palabras, la siguiente identidad entre funciones esféricas zonales se mantiene: $$\varphi_{2j}=\phi_j,\qquad\text{para }j\ge0,$$ como funciones en $\SO(n+1)$.

  \chapter{Funciones Esféricas en las Esferas {$S^{\lowercase{n}}$}}\label{sn}

\begin{flushright}{\it
``El día termina cuando escribo un poema,\\ tengo miedo, Madre,\\ de morir terriblemente joven."}\\Lautaro Flores.
\end{flushright}

\

En este capítulo estudiamos las funciones esféricas de ciertos $K$-tipos en la esfera $n$-dimensional $S^{n}\simeq \SO(n+1)/\SO(n)$, para cualquier $n$.
Más precisamente, explicitamos todas las funciones esféricas escalares, incluyendo a las de tipo no trivial, y luego estudiamos todas las de tipo fundamental, describiéndolas  en  términos de funciones hipergeométricas matriciales $_2\!F_1$. Para esto trabajamos con las realizaciones explícitas de las representaciones fundamentales del grupo especial ortogonal real.
 Posteriormente construimos para cada $\SO(n)$-tipo sucesiones de polinomios ortogonales con respecto a un peso $W$, las cuales están asociadas a las funciones esféricas. Y probamos que, para cualquier $n$, $W$ admite un operador diferencial simétrico de segundo orden.

\section{Representaciones de $\SO(\lowercase{n})$ y el Álgebra $\lowercase{\lieso}(\lowercase{n},\CC)$}
\subsection{Estructura de Raíces en $\lieso(n,\CC)$}
\

  Llamemos $E_{ik}$ a la  matriz cuadrada con un $1$ en la entrada $ik$ y ceros en el resto; y   consideremos las matrices
$$I_{ki}=E_{ik}-E_{ki},\qquad 1\le i,k\le n.$$
Entonces, el conjunto $\{I_{ki}\}_{1\le i<k\le n}$ forma una base del álgebra de Lie $\lieso(n)$. Estas matrices satisfacen las siguientes relaciones
$$[I_{ki},I_{rs}]=\delta_{ks}I_{ri}+\delta_{ri}I_{sk}+\delta_{es}I_{kr}+\delta_{rk}I_{es}.$$
Si asumimos que $k>i, r>s$ entonces tenemos
$$[I_{ki},I_{es}]=I_{sk},\; [I_{ki},I_{rk}]=I_{ri},\; [I_{ki},I_{ri}]=I_{kr},\; [I_{ki},I_{ks}]=I_{es},$$
y todos los otros corchetes son cero. De aquí fácilmente sigue que el conjunto
$$\{I_{p,p-1}:2\le p\le n\}$$
genera el álgebra de Lie $\lieso(n)$.

\begin{prop}\label{rightinv}
Dado $n\in\NN$, tenemos que el operador
$$Q_n=\sum_{1\le i,k\le n}I_{ki}^2\in D(\SO(n))$$
es invariante a derecha  por $\SO(n)$, i.e. $$Q_n\in D(\SO(n))^{\SO(n)}, \quad \forall\, n\in\NN.$$
\end{prop}
\begin{proof}[\it Demostración]
Para probar que $Q_n$ es invariante a derecha por $G$ es suficiente probar que $\dot I_{p,p-1}(Q_n)=0$ para todo $2\le p\le n$. Entonces
$$\dot I_{p,p-1}(Q_n)=\sum_{1\le i,k\le n}\big([I_{p,p-1},I_{ki}]I_{ki}+I_{ki}[I_{p,p-1},I_{ki}]\big).$$
Tenemos
\begin{equation}\label{multiplicationtable}
[I_{p,p-1},I_{ki}]=\begin{cases} I_{ip}&\qquad {\text si}\quad k=p-1,\\I_{k,p-1}&\qquad {\text si}\quad i=p,\\I_{pk}&\qquad {\text si}\quad i=p-1,
\\I_{p-1,i}&\qquad{\text si}\quad k=p.
\end{cases}
\end{equation}
Entonces
\begin{align*}
\dot I_{p,p-1}(Q_n)=&\sum_{1\le i\le n}(I_{ip}I_{p-1,i}+I_{p-1,i}I_{ip})+\sum_{1\le k\le n}(I_{k,p-1}I_{kp}+I_{kp}I_{k,p-1})\\
&+\sum_{1\le k\le n}(I_{pk}I_{k,p-1}+I_{k,p-1}I_{pk})+\sum_{1\le i\le n}(I_{p-1,i}I_{p,i}+I_{p,i}I_{p-1,i})=0.
\end{align*}
Esto prueba la proposición.\end{proof}

\subsection{El Operador Diferencial $Q_{2\ell}$}
\

  Asumamos que $n=2\ell$.
Miremos una descomposición en espacios raíces de $\lieso(n)$ relacionada a los elementos $I_{ki}$, $1\le i<k\le n$.

El subespacio generado
$$\lieh={\langle I_{21},I_{43},\dots,I_{2\ell,2\ell-1}\rangle}_\CC$$
es una subálgebra de Cartan de $\lieso(n,\CC)$.
Para encontrar los vectores raíz  es conveniente  visualizar los elementos de $\lieso(n,\CC)$ como matrices $\ell\times\ell$ de bloques $2\times2$. Entonces $\lieh$ es el subespacio de todas las matrices diagonales  de bloques antisimétricos $2\times2$. El subespacio de todas las matrices $A$ con un bloque $A_{jk}$ de tamaño dos, $1\le j<k\le\ell$, en el lugar $(j,k)$ y $-A_{jk}^t$ en el lugar $(k,j)$ con ceros en todos los otros lugares, es $\ad(\lieh)$-estable. Sea
$$H=i(x_1I_{21}+\dots+x_\ell I_{2\ell,2\ell-1})\in\lieh.$$
Entonces $[H,A]=\lambda A$ si y solo si
$$x_jiI_{2j,2j-1}A_{jk}-x_kiA_{jk}I_{2k,2k-1}=\lambda A_{jk}.$$
Salvo múltiplos escalares esta ecuación lineal tiene cuatro soluciones linealmente independientes:
$$A_{jk}=\begin{pmatrix} 1&\pm i\\ \pm i&-1\end{pmatrix}\quad\text {con correspondiente}\quad \lambda=\mp(x_j+x_k),$$
$$A_{jk}=\begin{pmatrix} 1&\mp i\\ \pm i&1\end{pmatrix}\quad\text {con correspondiente}\quad \lambda=\mp(x_j-x_k).$$

Sea $\epsilon_j\in\lieh*$ definida por $\epsilon_j(H)=x_j$ para $1\le j\le\ell$. Entonces para $1\le j<k\le\ell$, las siguientes matrices son vectores raíz de $\lieso(2\ell,\CC)$:
\begin{equation}\label{rootvectorsSn}
\begin{split}
X_{\epsilon_j+\epsilon_k}&=I_{2k-1,2j-1}-I_{2k,2j}-i(I_{2k-1,2j}+I_{2k,2j-1}),\\
X_{-\epsilon_j-\epsilon_k}&=I_{2k-1,2j-1}-I_{2k,2j}+i(I_{2k-1,2j}+I_{2k,2j-1}),\\
X_{\epsilon_j-\epsilon_k}&=I_{2k-1,2j-1}+I_{2k,2j}-i(I_{2k-1,2j}-I_{2k,2j-1}),\\
X_{-\epsilon_j+\epsilon_k}&=I_{2k-1,2j-1}+I_{2k,2j}+i(I_{2k-1,2j}-I_{2k,2j-1}).
\end{split}
\end{equation}
Elegimos el siguiente conjunto de raíces positivas
$$\Delta^+=\{\epsilon_j+\epsilon_k, \epsilon_j-\epsilon_k: 1\le j<k\le\ell\},$$
tenemos que el diagrama de Dynkin de $\lieso(2\ell,\CC)$ es $D_\ell$:

\noindent

\

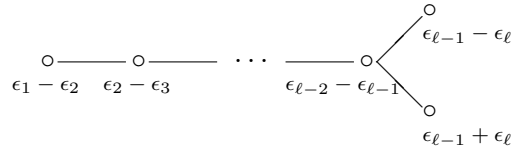
\setlength{\unitlength}{0.75mm}
\hspace{0cm}
\begin{picture}(0,15)
\put(59,0){$\circ$}
\put(62,1){\line(1,0){12}} \put(54,-4){\scriptsize$\epsilon_1-\epsilon_2$}
\put(75,0){$\circ$}
\put(78,1){\line(1,0){12}} \put(70,-4){\scriptsize$\epsilon_2-\epsilon_3$}
\put(93,1){$\dots$}
\put(102,1){\line(1,0){12}}\put(102,-4){\scriptsize$\epsilon_{\ell-2}-\epsilon_{\ell-1}$}
\put(115,0){$\circ$}
\put(118,1){\line(1,1){8}}
\put(126,9){$\circ$} \put(126,5){\scriptsize$\epsilon_{\ell-1}-\epsilon_{\ell}$}
\put(118,1){\line(1,-1){8}}
\put(126,-9){$\circ$} \put(126,-13){\scriptsize$\epsilon_{\ell-1}+\epsilon_{\ell}$}
\end{picture}

\vspace{2cm}

Mirando los bloques $2\times2$, $A_{jk}$ de las diferentes raíces raíces, que son
\begin{equation*}
\begin{split}
X_{\epsilon_j+\epsilon_k}&=\begin{pmatrix} 1&-i\\-i&-1\end{pmatrix},\qquad X_{-\epsilon_j-\epsilon_k}=\begin{pmatrix} 1&i\\i&-1\end{pmatrix},\\
X_{\epsilon_j-\epsilon_k}&=\begin{pmatrix} 1&i\\-i&1\end{pmatrix},\hspace{1.1cm}X_{-\epsilon_j+\epsilon_k}=\begin{pmatrix} 1&-i\\i&1\end{pmatrix},
\end{split}
\end{equation*}
es fácil obtener las siguientes relaciones inversas
\begin{equation*}
\begin{split}
I_{2k-1,2j-1}&=\tfrac14\big(X_{\epsilon_j+\epsilon_k}+X_{-\epsilon_j-\epsilon_k}+X_{\epsilon_j-\epsilon_k}+X_{-\epsilon_j+\epsilon_k}\big),\\
I_{2k,2j}&=\tfrac14\big(-X_{\epsilon_j+\epsilon_k}-X_{-\epsilon_j-\epsilon_k}+X_{\epsilon_j-\epsilon_k}+X_{-\epsilon_j+\epsilon_k}\big),\\
I_{2k,2j-1}&=\tfrac{i}4\big(X_{\epsilon_j+\epsilon_k}-X_{-\epsilon_j-\epsilon_k}-X_{\epsilon_j-\epsilon_k}+X_{-\epsilon_j+\epsilon_k}\big),\\
I_{2k-1,2j}&=\tfrac{i}4\big(X_{\epsilon_j+\epsilon_k}-X_{-\epsilon_j-\epsilon_k}+X_{\epsilon_j-\epsilon_k}-X_{-\epsilon_j+\epsilon_k}\big).
\end{split}
\end{equation*}
De aquí sigue que
\begin{multline*}
I_{2k-1,2j-1}^2+I_{2k,2j}^2+I_{2k,2j-1}^2+I_{2k-1,2j}^2=\\
\tfrac14\big(X_{\epsilon_j+\epsilon_k}X_{-\epsilon_j-\epsilon_k}+X_{-\epsilon_j-\epsilon_k}X_{\epsilon_j+\epsilon_k}+X_{\epsilon_j-\epsilon_k}X_{-\epsilon_j+\epsilon_k}+X_{-\epsilon_j+\epsilon_k}X_{\epsilon_j-\epsilon_k}\big).
\end{multline*}
Por lo tanto
\begin{multline*}
Q_{2\ell}=\;\sum_{1\le j\le\ell}I_{2j,2j-1}^2+\tfrac14\sum_{1\le j<k\le\ell}\big(X_{\epsilon_j+\epsilon_k}X_{-\epsilon_j-\epsilon_k}+X_{-\epsilon_j-\epsilon_k}X_{\epsilon_j+\epsilon_k}\\
\hspace{4.5cm}+X_{\epsilon_j-\epsilon_k}X_{-\epsilon_j+\epsilon_k}+X_{-\epsilon_j+\epsilon_k}X_{\epsilon_j-\epsilon_k}\big).
\end{multline*}
 Ahora usando las expresiones \eqref{rootvectorsSn} y la tabla de multiplicar \eqref{multiplicationtable} obtenemos
\begin{equation*}
\begin{split}
[X_{\epsilon_j+\epsilon_k},X_{-\epsilon_j-\epsilon_k}]&=-4i(I_{2j,2j-1}+I_{2k,2k-1}),\\
[X_{\epsilon_j-\epsilon_k},X_{-\epsilon_j+\epsilon_k}]&=-4i(I_{2j,2j-1}-I_{2k,2k-1}).
\end{split}
\end{equation*}
Entonces $Q_{2\ell}$ resulta
\begin{equation}\label{qpar}
\begin{split}
Q_{2\ell}=&\;\sum_{1\le j\le\ell}I_{2j,2j-1}^2-2\sum_{1\le j\le\ell}(\ell-j)iI_{2j,2j-1}+\sum_{1\le j<k\le\ell}\tfrac12\big(X_{-\epsilon_j-\epsilon_k}X_{\epsilon_j+\epsilon_k}+X_{-\epsilon_j+\epsilon_k}X_{\epsilon_j-\epsilon_k}\big).
\end{split}
\end{equation}

\subsection{El Operador Diferencial $Q_{2\ell+1}$}
\

Ahora miramos la descomposición en espacios raíz de $\lieso(n)$ relacionada a los elementos $I_{ki}$, $1\le i<k\le n$ cuando  $n=2\ell+1$.

El subespacio generado
$$\lieh={\langle I_{21},I_{43},\dots,I_{2\ell,2\ell-1}\rangle}_\CC$$
es una subálgebra de  Cartan de $\lieso(n,\CC)$.
Para encontrar los vectores raíz  es conveniente vi\-sua\-li\-zar los elementos de $\lieso(n,\CC)$ como matrices $(\ell+1) \times(\ell+1)$, con $\ell \times \ell$ bloques de $2\times2$ ocupando la esquina superior izquierda de las matrices cuadradas de tamaño $2\ell+1$ y  con la última columna (respectivamente fila) hecha de $\ell$ columnas (respectivamente filas) de tamaño dos y  un cero en el lugar $(2\ell+1,2\ell+1)$.  El subespacio de todas las matrices $A$ con un  bloque $A_{jk}$, $1\le j<k\le \ell$, en el lugar $(j,k)$, con el bloque $-A_{jk}^t$ en el lugar $(k,j)$ y con ceros en todos los demás lugares, es $\ad(\lieh)$-estable. Además el subespacio de todas las matrices $B$ con una columna $B_j$ de tamaño dos, $1\le j\le \ell$, en el lugar $(j,\ell+1)$, con la fila $-B_{j}^t$ en el lugar $(\ell+1,j)$ y con ceros en todo el resto, es $\ad(\lieh)$-estable.

Por otra parte $[H,B]=\lambda B$ si y solo si
$$x_jiI_{2j,2j-1}B_{j}=\lambda B_{j}.$$
Salvo múltiplos escalares esta ecuación lineal tiene dos soluciones linealmente independientes:
$$B_{j}=\begin{pmatrix} 1\\\pm i\end{pmatrix}\quad\text {con correspondiente}\quad \lambda=\mp x_j,$$

Sea $\epsilon\in\lieh^*$ definida por $\epsilon(H)=x_j$ para $1\le j\le\ell$. Entonces para $1\le j<k\le\ell$ y $1\le r\le\ell$, las siguientes matrices son vectores raíz de $\lieso(2\ell+1,\CC)$:
\begin{equation}\label{rootvectors1}
\begin{split}
X_{\epsilon_j+\epsilon_k}&=I_{2k-1,2j-1}-I_{2k,2j}-i(I_{2k-1,2j}+I_{2k,2j-1}),\\
X_{-\epsilon_j-\epsilon_k}&=I_{2k-1,2j-1}-I_{2k,2j}+i(I_{2k-1,2j}+I_{2k,2j-1}),\\
X_{\epsilon_j-\epsilon_k}&=I_{2k-1,2j-1}+I_{2k,2j}-i(I_{2k-1,2j}-I_{2k,2j-1}),\\
X_{-\epsilon_j+\epsilon_k}&=I_{2k-1,2j-1}+I_{2k,2j}+i(I_{2k-1,2j}-I_{2k,2j-1}),\\
X_{\epsilon_r}&=I_{n,2r-1}-iI_{n,2r},\\
X_{-\epsilon_r}&=I_{n,2r-1}+iI_{n,2r}.
\end{split}
\end{equation}
Elegimos el siguiente conjunto de raíces positivas
$$\Delta^+=\{\epsilon_r, \epsilon_j+\epsilon_k, \epsilon_j-\epsilon_k: 1\le r\le\ell, 1\le j<k\le\ell\},$$
el diagrama de Dynkin de $\lieso(2\ell+1,\CC)$ es $B_\ell$:

\noindent

\setlength{\unitlength}{0.75mm}
\hspace{0cm}
\begin{picture}(0,15)
\put(59,0){$\circ$}
\put(62,1){\line(1,0){12}} \put(54,-4){\scriptsize$\epsilon_1-\epsilon_2$}
\put(75,0){$\circ$}
\put(78,1){\line(1,0){12}} \put(70,-4){\scriptsize$\epsilon_2-\epsilon_3$}
\put(93,1){$\dots$}
\put(102,1){\line(1,0){12}}\put(108,-4){\scriptsize$\epsilon_{\ell-1}-\epsilon_{\ell}$}
\put(115,0){$\circ$}
\put(118,1){\line(1,0){12}}
\put(118,2){\line(1,0){12}}
\put(128,0){$>$}
\put(132,0){$\circ$} \put(131,-4){\scriptsize$\epsilon_{\ell}$}
\end{picture}

\vspace{1cm}

Mirando las columnas  $2\times 1$ de las distintas raíces, que son
\begin{equation*}
\begin{split}
X_{\epsilon_j}= \begin{pmatrix} 1\\-i\end{pmatrix},\qquad \quad X_{-\epsilon_j}= \begin{pmatrix} 1\\i\end{pmatrix},
\end{split}
\end{equation*}
es fácil obtener las relaciones inversas siguientes
$$I_{n,2r-1}=\tfrac12(X_{\epsilon_r}+X_{-\epsilon_r}),\qquad I_{n,2r}=\tfrac{i}2(X_{\epsilon_r}-X_{-\epsilon_r}).$$

De aquí sigue que
\begin{equation*}
\begin{split}
I_{n,2r-1}^2+I_{n,2r}^2&=\tfrac12(X_{\epsilon_r}X_{-\epsilon_r}+X_{-\epsilon_r}X_{\epsilon_r})=-iI_{2r,2r-1}+X_{-\epsilon_r}X_{\epsilon_r},
\end{split}
\end{equation*}
pues $[X_{\epsilon_r},X_{-\epsilon_r}]=-2iI_{2r,2r-1}$. Por lo tanto tenemos que
$$Q_{2\ell+1}=\sum_{1\le j\le2\ell}I_{n,j}^2+Q_{2\ell}=\sum_{1\le r\le2\ell}(-iI_{2r,2r-1}+X_{-\epsilon_r}X_{\epsilon_r})+Q_{2\ell}.
$$
Y entonces
\begin{equation}\label{qimpar}
\begin{split}
Q_{2\ell+1}=&\sum_{1\le j\le\ell}I_{2j,2j-1}^2-\sum_{1\le j\le\ell}(2(\ell-j)+1)iI_{2j,2j-1}\\
&+\sum_{1\le j<k\le\ell}\tfrac12\big(X_{-\epsilon_j-\epsilon_k}X_{\epsilon_j+\epsilon_k}+X_{-\epsilon_j+\epsilon_k}X_{\epsilon_j-\epsilon_k}\big)+\sum_{1\le r\le2\ell}X_{-\epsilon_r}X_{\epsilon_r}.
\end{split}
\end{equation}

\subsection{La Base de Gelfand-Tsetlin}
\

Sea $T_\mm$ una representación unitaria irreducible de $\SO(n)$ con peso máximo $\mm$ y sea $V_\mm$ el espacio de esta representación. Los pesos máximos $\mm$ de estas representaciones son dados por los números $\mm=\mm_n=(m_{1n},\dots,m_{\ell\, n})$ tales que
\begin{align*}
&m_{1n}\ge m_{2n}\ge \dots\ge m_{\ell-1,n}\ge|m_{\ell\, n}|&  &\text{si } n=2\ell,\\
&m_{1n}\ge m_{2n}\ge \dots\ge m_{\ell\, n}\ge0&               &\text{si } n=2\ell+1,
\end{align*}
con $m_{jn}$ todos enteros.

La restricción de la representación $T_\mm$ del subgrupo $\SO(2\ell+1)$ al subgrupo $\SO(2\ell)$ se des\-com\-po\-ne en la suma directa todos las representaciones $T_{\mm'}$, $\mm'=\mm_{n-1}=$ $(m_{1,n-1},\dots,m_{\ell,n-1})$, para las cuales las condiciones de entrelazamiento
$$m_{1,2\ell+1}\ge m_{1,2\ell}\ge m_{2,2\ell+1}\ge m_{2,2\ell}\ge\cdots\ge m_{\ell,2\ell+1}\ge m_{\ell,2\ell}\ge-m_{\ell,2\ell+1}$$
son satisfechas. Para las restricciones de las representaciones $T_\mm$ de $\SO(2\ell)$ al subgrupo $\SO(2\ell-1)$ las correspondientes condiciones de entrelazamiento son
$$m_{1,2\ell}\ge m_{1,2\ell-1}\ge m_{2,2\ell}\ge m_{2,2\ell-1}\ge\cdots\ge m_{\ell-1,2\ell}\ge m_{\ell-1,2\ell-1}\ge|m_{\ell,2\ell}|.$$
Todas las multiplicidades en las descomposiciones son igual a uno (ver \cite{V92}, página 362).

Si continuamos este procedimiento de restringir las representaciones irreducibles sucesivamente a los subgrupos
$$\SO(n-2)>\SO(n-3)>\cdots>\SO(2)$$
 finalmente obtenemos representaciones unidimensionales del grupo $\SO(2)$. Si tomamos un vector unitario en cada una de estas representaciones unidimensionales obtenemos una base ortonormal de la representación espacio $V_\mm$. Tal base es llamada a base de Gelfand-Tsetlin.
Los elementos de una base de Gelfand-Tsetlin $\{v(\mu)\}$ de la representación $T_\mm$ de $\SO(n)$ son rotulados por los patrones de Gelfand-Tsetlin $\mu=(m_{n},m_{n-1},\dots,m_3,m_2)$, donde las condiciones de entrelazamiento son detalladas en el siguiente diagrama.\\
 Si $n=2\ell+1$

$$\mu=\begin{matrix}
m_{1n}&{}&m_{2n}&&\cdots&&m_{\ell\, n}&{}&-m_{\ell\, n}\\
{}&m_{1,n-1}&{}&\cdots&&\cdots&&m_{\ell,n-1}&{}\\
{}&{}&\cdot&&\cdot&&\cdot&&\cdot\\
{}&{}&&\cdot&&\cdot&&\cdot&\\
{}&{}&{}&{}&m_{15}&{}&m_{25}&{}&-m_{25}\\
{}&{}&{}&{}&{}&m_{14}&{}&m_{24}&{}\\
{}&{}&{}&{}&{}&{}&m_{13}&{}&-m_{13}\\
{}&{}&{}&{}&{}&{}&{}&m_{12}&{}
\end{matrix}
$$
Si $n=2\ell$

$$\mu=\begin{matrix}
m_{1n}&{}&m_{2n}&&\cdots&&m_{\ell-1,n}&{}&m_{\ell\, n}&{}\\
{}&m_{1,n-1}&{}&\cdots&&\cdots&{}&m_{\ell-1,n-1}&{}&-m_{\ell-1,n-1}\\
{}&{}&\cdot&&\cdot&&\cdot&&\cdot\\
{}&{}&&\cdot&&\cdot&&\cdot&\\
{}&&&{}&\cdot&&\cdot&&\cdot\\
&{}&{}&{}&{}&m_{15}&{}&m_{25}&{}&-m_{25}\\
&{}&{}&{}&{}&{}&m_{14}&{}&m_{24}&{}\\
&{}&{}&{}&{}&{}&{}&m_{13}&{}&-m_{13}\\
&{}&{}&{}&{}&{}&{}&{}&m_{12}&{}
\end{matrix}
$$

\

La cadena de subgrupos $\SO(n-1)>\SO(n-2)>\cdots>\SO(2)$ define unívocamente la base  ortonormal  $\{v(\mu)\}$ salvo múltiplos escalares de valor absoluto uno en los elementos de la base.

  Primero tomemos  $n=2\ell$. Para maniobrar con todos representaciones unitarias  irreducibles de $\SO(n)$, que contienen a un $\SO(n-1)$-tipo fijo $\mm_{n-1}$, introducimos los parámetros enteros $(w,r_1,\dots,r_{\ell-1})$ tomando:
\begin{align*}
m_{1n}=&w+m_{1,n-1},& &\text{donde}\quad w\ge0,\\
m_{2n}=&r_1+m_{2,n-1},& &\text{donde}\quad 0\le r_1\le m_{1,n-1}-m_{2,n-1},\\
      &\vdots & & \quad\vdots \\
 m_{\ell-1,n}=&r_{\ell-2}+m_{\ell-1,n-1},& &\text{donde}\quad 0\le r_{\ell-2}\le m_{\ell-2,n-1}-m_{\ell-1,n-1},\\
 m_{\ell,n}=&r_{\ell-1}-m_{\ell-1,n-1},& &\text{donde}\quad 0\le r_{\ell-1}\le 2m_{\ell-1,n-1}.
\end{align*}
Además ponemos $\mm=\mm_n=\mm(w,\rr)$, $w\ge0$ y $\rr\in\Omega$ donde
\begin{equation*}
\begin{split}
\Omega=\{\rr=(r_1,\dots,r_{\ell-1}):& \;0\le r_j\le m_{j,n-1}-m_{j+1,n-1} \text{\;si\;} 1\le j\le \ell-2, \\
&\;0\le r_{\ell-1}\le 2m_{\ell-1,n-1}\}.
\end{split}
\end{equation*}

Ahora hacemos una observación importante: las representaciones irreducibles de $\SO(n-2)$ que aparecen como subrepresentaciones de la representación irreducible de $\SO(n-1)$ de peso máximo $\mm_{n-1}$ son además parametrizadas  por el conjunto $\Omega$ en la siguiente forma:
\begin{align*}
 m_{1,n-2}=&s_1+m_{2,n-1},& &\text{donde}\quad 0\le s_1\le m_{1,n-1}-m_{2,n-1},\\
m_{2,n-2}=&s_2+m_{3,n-1},& &\text{donde}\quad 0\le s_2\le m_{2,n-1}-m_{3,n-1},\\
             \vdots&& &\quad\vdots\\
m_{\ell-2,n-2}=&s_{\ell-2}+m_{\ell-1,n-1},& &\text{donde}\quad 0\le s_{\ell-2}\le m_{\ell-2,n-1}-m_{\ell-1,n-1},\\
m_{\ell-1,n-2}=&s_{\ell-1}-m_{\ell-1,n-1},& &\text{donde}\quad 0\le s_{\ell-1}\le 2m_{\ell-1,n-1}.
\end{align*}

Además ponemos $\ttt(\sss)=\mm_{n-2}(\sss)$, $\sss\in\Omega$. Si $\mm_{n-1}$ es el peso máximo de $\pi\in\hat K$, entonces la descomposición de $V_\pi$ en $M$-submódulos irreducibles, $M=\SO(n-2)$, es la siguiente
$$V_\pi=\bigoplus_{\sss\in\Omega}V_{\ttt(\sss)}.$$

\subsection{Expresiones Explícitas para $\dot\pi(Q_n)$}
\

Ahora, dado $\pi\in\hat\SO(n)$ y $v\in V_\pi$ un  vector peso máximo, por $\eqref{qpar}$ y $\eqref{qimpar}$ podemos dar el valor explícito de $\dot\pi(Q_n)v$ en términos del peso máximo de $\pi$, $\mm=(m_1,m_2,\dots,m_\ell)$.
\begin{cor}
Sea $(\pi,V_\pi)$ una representación irreducible  de $\SO(2\ell)$ de peso máximo $\mm=(m_1,m_2,\dots,m_\ell)$. Entonces, si $v$ es un vector peso máximo de $\pi$ tenemos que
\begin{align}\label{qvpar}
\dot\pi(Q_{2\ell})v=&\sum_{1\le j\le\ell}\left(-m_j^2-2(\ell-j)m_j\right)v.
\end{align}
\end{cor}

\begin{cor}
Sea $(\pi,V_\pi)$ una representación irreducible  de $\SO(2\ell+1)$ de peso máximo $\mm=(m_1,m_2,\dots,m_\ell)$. Entonces, si $v$ es a peso máximo de $\pi$ tenemos que
\begin{align}\label{qvimpar}
\dot\pi(Q_{2\ell+1})v=&\sum_{1\le j\le\ell}\left(-m_j^2-(2(\ell-j)+1)m_j\right)v.
\end{align}
\end{cor}

\section{El Operador Diferencial $\Delta$}\label{eloperadordelta}
\

Miraremos con detenimiento el operador $\Delta$ definido por
$$\Delta=\sum _{j=1}^{n} I_{n+1,j}^2,$$
para así estudiar sus autofunciones, autovalores y luego usarlo para poder entender las funciones esféricas irreducibles de fundamental tipo asociadas al par $(G,K)=(\SO(n+1),\SO(n))$.

\begin{prop} Sea $G=\SO(n+1)$ y $K=\SO(n)$.   Consideremos el siguiente operador diferencial en $G$ invariante a izquierda
$$\Delta=\sum _{j=1}^{n} I_{n+1,j}^2.$$
Entonces $\Delta$ es además invariante a derecha por $K$.
\end{prop}
\begin{proof}[\it Demostración]

Es suficiente probar que $\dot I_{p,p-1}(\Delta)=0$ para todo $2\le p\le n$. Calculamos
\begin{align*}
\dot I_{p,p-1}(\Delta)&=\,\sum _{j=1}^{n}\Big( [I_{p,p-1}, I_{n+1,j}] I_{n+1,j}+ I_{n+1,j} [I_{p,p-1}, I_{n+1,j}]\Big)\\&=
I_{n+1,p-1}I_{n+1,p}+I_{n+1,p}I_{n+1,p-1}+I_{p,n+1}I_{n+1,p-1}+I_{n+1,p-1}I_{p,n+1}=0.
\end{align*}
Esto prueba la proposición.
\end{proof}
Además observemos  que
\begin{equation}\label{delta}
Q_{n+1}=Q_{n}+\Delta.
\end{equation}

Sea $a(s)=\exp sI_{n+1,n}\in A$. Entonces
\begin{equation}\label{a(s)}
a(s)=\begin{pmatrix}I_{n-1}&0&0\\0&\cos s&\sin s\\0&-\sin s&\cos s\end{pmatrix},
\end{equation}
donde $I_{n-1}$ denota la matriz identidad de tamaño $n-1$.
Ahora queremos  obtener la expresión de $[\Delta\Phi](a(s))$ para cualquier función suave $\Phi$ en $G$ con valores en $\End(V_\pi)$ tal que $\Phi(kgk')=\pi(k)\Phi(g)\pi(k')$ para todo $g\in G$ y todos $k,k'\in K$.

Tenemos que
$$[I_{n+1,j}^2\Phi](a(s))=\frac{\partial^2}{\partial t^2}\left.\Phi(a(s)\exp tI_{n+1,j})\right|_{t=0},$$
 y  usaremos la descomposición $G=KAK$ para escribir $$a(s)\exp tI_{n+1,j}=k(s,t)a(s,t)h(s,t),$$ con $k(s,t),h(s,t)\in K$ y $a(s,t)\in A$. Si $j=n$ tenemos
$a(s)\exp tI_{n+1,n}=a(s+t)$. Entonces podemos tomar
$$a(s,t)=a(s+t),\qquad k(s,t)=h(s,t)=e.$$
Para $1\le j\le n-1$, cuando $s\notin\ZZ\pi$, podemos tomar
$$k(s,t)=\begin{pmatrix}I_{j-1}&{\bf 0}&{\bf 0}&{\bf 0}&{\bf 0}\\{\bf 0}&\frac{\sin s\cos t}{\sqrt{1-\cos^2s\cos^2t}}&{\bf 0}&\frac{\sin t}{\sqrt{1-\cos^2s\cos^2t}}&0\\{\bf 0}&{\bf 0}&I_{n-j-1}&{\bf 0}&{\bf 0}\\{\bf 0}&\frac{-\sin t}{\sqrt{1-\cos^2s\cos^2t}}&{\bf 0}&\frac{\sin s\cos t}{\sqrt{1-\cos^2s\cos^2t}}&0\\{\bf 0}&0&{\bf 0}&0&1\end{pmatrix},$$
$$h(s,t)=\begin{pmatrix}I_{j-1}&{\bf 0}&{\bf 0}&{\bf 0}&{\bf 0}\\{\bf 0}&\frac{\sin s}{\sqrt{1-\cos^2s\cos^2t}}&{\bf 0}&\frac{-\cos s\sin t}{\sqrt{1-\cos^2s\cos^2t}}&0\\{\bf 0}&{\bf 0}&I_{n-j-1}&{\bf 0}&{\bf 0}\\{\bf 0}&\frac{\cos s\sin t}{\sqrt{1-\cos^2s\cos^2t}}&{\bf 0}&\frac{\sin s}{\sqrt{1-\cos^2s\cos^2t}}&0\\{\bf 0}&0&{\bf 0}&0&1\end{pmatrix},$$
$$a(s,t)=\begin{pmatrix}I_{n-1}&{\bf 0}&{\bf 0}\\{\bf 0}&\cos s\cos t&\sqrt{1-\cos^2s\cos^2t}\\{\bf 0}&-\sqrt{1-\cos^2s\cos^2t}&\cos s\cos t\end{pmatrix}.$$
Observemos que $k(s,0)=h(s,0)=e$ y que $a(s,0)=a(s)$. Entonces
\begin{equation*}
\begin{split}
[I_{n+1,j}^2\Phi](a(s))=&\frac{\partial^2}{\partial t^2}\pi(k(s,t))\Big|_{t=0}\Phi(a(s))+2\frac{\partial}{\partial t}\pi(k(s,t))\Big|_{t=0}\frac{\partial}{\partial t}\Phi(a(s,t))\Big|_{t=0}\\
&+2\frac{\partial}{\partial t}\pi(k(s,t))\Big|_{t=o}\Phi(a(s))\frac{\partial}{\partial t}\pi(h(s,t))\Big|_{t=0}+\frac{\partial^2}{\partial t^2}\Phi(a(s,t))\Big|_{t=0}\\
&+2\frac{\partial}{\partial t}\Phi(a(s,t))\Big|_{t=0}\frac{\partial}{\partial t}\pi(h(s,t))\Big|_{t=0}+\Phi(a(s))\frac{\partial^2}{\partial t^2}\pi(h(s,t))\Big|_{t=0}.
\end{split}
\end{equation*}

  Introducimos en el grupo abeliano $A$ la coordenada $x(a(s))=\tan(s)$ para $-\pi/2<s<\pi/2$. Dado
$-\infty<x<\infty$ el elemento
$$a(x)=\begin{pmatrix} I_{n-1}&{\bf 0}&{\bf 0}\\{\bf 0}&\frac1{\sqrt{1+x^2}}&\frac{x}{\sqrt{1+x^2}}\\{\bf 0}&\frac{-x}{\sqrt{1+x^2}}&\frac1{\sqrt{1+x^2}}\end{pmatrix},$$
es el punto en $A$ de coordenada $x$. Entonces sea $$F(s)=\Phi(a(s)),\qquad \text{para } s\in\RR.$$

De la igualdad $\cos s\cos t=1/\sqrt{1+x^2}$ obtenemos
$$x(a(s,t))=\frac{\pm\sqrt{1-\cos^2 s\cos^2t}}{\cos s\cos t}$$
y
$$\frac{\partial}{\partial t}x(a(s,t))=\frac{\sin t}{\pm\sqrt{1-\cos^2s\cos^2t}\cos s\cos^2t}=\frac{\tan t}{x\cos^2s\cos^2t},$$
de aquí
$$x(a(s,0))=\tan s\qquad\text{y}\qquad \frac{\partial}{\partial t}x(a(s,t))\Big|_{t=0}=0.$$
Por lo tanto
$$\frac{\partial}{\partial t}\Phi(a(s,t))\Big|_{t=0}=F'(\tan s)\frac{\partial}{\partial t}x(a(s,t))\Big|_{t=0}=0,$$
y
\begin{align*}
[I_{n+1,j}^2\Phi](a(s))=&\frac{\partial^2}{\partial t^2}\pi(k(s,t))\Big|_{t=0}\Phi(a(s))+2\frac{\partial}{\partial t}\pi(k(s,t))\Big|_{t=o}\Phi(a(s))\frac{\partial}{\partial t}\pi(h(s,t))\Big|_{t=0}\hfill\\
&+\frac{\partial^2}{\partial t^2}\Phi(a(s,t))\Big|_{t=0}+\Phi(a(s))\frac{\partial^2}{\partial t^2}\pi(h(s,t))\Big|_{t=0}.
\end{align*}

Ahora obtenemos que
\begin{equation*}
\begin{split}
\frac{\partial^2}{\partial t^2}\Phi(a(s,t))\Big|_{t=0}&=\frac{\partial}{\partial t}\left(F'(x)\frac{\tan t}{x\cos^2 s\cos^2t}\right)\Big|_{t=0}=\frac1{\cos^2 s}\frac{\partial}{\partial t}\left(\frac{F'(x)}{x}\frac{\tan t}{\cos^2t}\right)\Big|_{t=0}\\
&=\frac1{\cos^2 s}\left(\frac{F'(x)}{x}\right)\Big|_{x=\tan s}=\frac1{\sin^2 s}F'(\tan s),
\end{split}
\end{equation*}
entonces
\begin{align*}
[I_{n+1,j}^2\Phi](a(s))=&\frac{\partial^2}{\partial t^2}\pi(k(s,t))\Big|_{t=0}F(\tan s)+F(\tan s)\frac{\partial^2}{\partial t^2}\pi(h(s,t))\Big|_{t=0}\\
&+2\frac{\partial}{\partial t}\pi(k(s,t))\Big|_{t=0}F(\tan s)\frac{\partial}{\partial t}\pi(h(s,t))\Big|_{t=0}+\frac1{\sin^2s}F'(\tan s).
\end{align*}

Además tenemos
$$\frac{\partial}{\partial t}\pi(k(s,t))\Big|_{t=0}=\dot\pi\Big(\frac{\partial}{\partial t}k(s,t)\Big|_{t=0}\Big)=\frac1{\sin s}\dot\pi(I_{n,j}),$$
y
$$\frac{\partial}{\partial t}\pi(h(s,t))\Big|_{t=0}=\dot\pi\Big(\frac{\partial}{\partial t}h(s,t)\Big|_{t=0}\Big)=-\frac{\cos s}{\sin s}\dot\pi(I_{n,j}).$$

\begin{prop}\label{derprincipal}
Si $A(s,t)=k(s,t)$ o $A(s,t)=h(s,t)$, entonces tenemos
$$\frac{\partial^2(\pi\circ A)}{\partial t^2}\Big|_{t=0}=\dot\pi\Big(\frac{\partial A}{\partial t}\Big|_{t=0}\Big)^2.$$
Más aún en cada caso, para $1\le j\le n-1$, tenemos
$$\frac{\partial^2}{\partial t^2}\pi(k(s,t))\Big|_{t=0}=\frac1{\sin^2s}\dot\pi(I_{n,j})^2,\qquad
\frac{\partial^2}{\partial t^2}\pi(h(s,t))\Big|_{t=0}=\frac{\cos^2 s}{\sin^2 s}\dot\pi(I_{n,j})^2.$$
\end{prop}

\begin{proof}[\it Demostración]
Para $|t|$ suficientemente pequeño $A(s,t)$ está cerca de la identidad de $K$, i.e. de la matriz $I_{n}$. Entonces podemos considerar la función
\begin{equation}\label{logA}
X(s,t)=\log(A(s,t))=B(s,t)-\frac{B(s,t)^2}2+\frac{B(s,t)^3}3-\cdots,
\end{equation}
donde $B(s,t)=A(s,t)-I_{n}$. Entonces
$$\pi(A(s,t))=\pi(\exp X(s,t))=\exp\dot\pi
(X(s,t))=\displaystyle\sum_{j\geq 0}\frac{\dot\pi(X(s,t))^j}{j!}.$$
Ahora derivamos con respecto a $t$ para obtener
\begin{equation}\label{derpi}
\begin{split}
\frac{\partial \, (\pi\circ A)}{\partial t}= &\, \dot\pi\Bigl(\frac{\partial X}{\partial t}\Bigr)+\tfrac 1{2!}\dot\pi\Bigl( \frac{\partial X}{\partial t}\Bigl)
\dot\pi(X)+\tfrac 1{2!}\dot\pi(X) \dot\pi\Bigl( \frac{\partial X}{\partial t}\Bigr)  + \tfrac 1{3!}\dot\pi\Bigl( \frac{\partial X}{\partial t}\Bigr)\dot\pi(X)^2\\
&+ \tfrac 1{3!}\dot\pi (X) \dot\pi\Bigl( \frac{\partial X}{\partial t}\Bigr) \dot\pi\Bigl(X\Bigr)  + \tfrac 1{3!}\dot\pi(X)^2
\dot\pi\Bigl( \frac{\partial X}{\partial t}\Bigr) +\cdots.
\end{split}
\end{equation}
Como $X(s,0)=0$, si derivamos \eqref{derpi} con respecto a $t$ y evaluamos en $(s,0)$ obtenemos
\begin{equation*}
\frac{\partial^2 \,(\pi\circ A)}{\partial t^2}\Big|_{t=0}  = \dot \pi\Bigl(\frac{\partial^2 X}{\partial t^2}\Big|_{t=0}\Bigr) +\dot\pi\Bigl(\frac{\partial X}{\partial t}\Big|_{t=0}\Bigr)^2.
\end{equation*}
Para calcular $\dfrac{\partial X}{\partial t}\Big|_{t=0}$ y
$\dfrac{\partial^2 X}{\partial t^2}\Big|_{t=0}$ derivamos
\eqref{logA} y obtenemos
\begin{align*}
\frac{\partial X}{\partial t}= \frac{\partial B}{\partial t} -\tfrac 1{2} \Bigl( \frac{\partial B}{\partial t}\Bigr)B- \tfrac 12 B \Bigl(\frac{\partial B}{\partial t}\Bigr)+ \tfrac 13 \Bigl(\frac{\partial B}{\partial t}\Bigr)B^2
 +\tfrac 13  B\Bigl(\frac{\partial B}{\partial t}\Bigr) B +\tfrac 13 B^2 \Bigl(\frac{\partial B}{\partial
t}\Bigr)+\cdots .
\end{align*}

\noindent Como $B(s,0)=0$ tenemos
$$\frac{\partial X}{\partial t}\Big|_{t=0}= \frac{\partial B}{\partial
t}\Big|_{t=0}=\frac{\partial A}{\partial t}\Big|_{t=0}.$$
Y además
$$\frac{\partial^2 X}{\partial t^2}\Big|_{t=0} =
\frac{\partial^2 A}{\partial t^2}\Big|_{t=0}
-\Big( \frac{\partial A}{\partial t}\Big|_{t=0}\Big)^2.$$

Ahora primero consideramos el caso $A(s,t)=k(s,t)$. Un cálculo directo nos lleva  a
$$\frac{\partial k}{\partial t}=\begin{pmatrix}{\bf 0}&{\bf 0}&{\bf 0}&{\bf 0}&{\bf 0}\\{\bf 0}&\frac{-\sin s\sin t}{(1-\cos^2s\cos^2t)^{3/2}}&{\bf 0}&\frac{\sin^2s\cos t}{(1-\cos^2s\cos^2t)^{3/2}}&0\\{\bf 0}&{\bf 0}&{\bf 0}&{\bf 0}&{\bf 0}\\{\bf 0}&\frac{-\sin^2 s\cos t}{(1-\cos^2s\cos^2t)^{3/2}}&{\bf 0}&\frac{-\sin s\sin t}{(1-\cos^2s\cos^2t)^{3/2}}&0\\{\bf 0}&0&{\bf 0}&0&0\end{pmatrix},$$
en particular $\dfrac{\partial k}{\partial t}\Big|_{t=0}=\dfrac1{\sin s}I_{n,j}$. Derivando una vez más con respecto a $t$ y evaluando en $t=0$ obtenemos
$\dfrac{\partial^2 k}{\partial t^2}\Big|_{t=0}=-\dfrac{1}{\sin^2 s}(E_{jj}+E_{n,n})$.
Entonces
$$\frac{\partial^2 A}{\partial t^2}\Big|_{t=0}
-\Big( \frac{\partial A}{\partial t}\Big|_{t=0}\Big)^2=-\dfrac{1}{\sin^2 s}(E_{jj}+E_{n,n})-\dfrac{1}{\sin^2 s}I_{n,j}^2=0.$$

Similarmente cuando $A(s,t)=h(s,t)$ obtenemos
$$\frac{\partial h}{\partial t}=\begin{pmatrix}{\bf 0}&{\bf 0}&{\bf 0}&{\bf 0}&{\bf 0}\\{\bf 0}&\frac{-\sin s\cos^2 s\cos t\sin t}{(1-\cos^2s\cos^2t)^{3/2}}&{\bf 0}&\frac{-\cos s\cos t\sin^2s}{(1-\cos^2s\cos^2t)^{3/2}}&0\\{\bf 0}&{\bf 0}&{\bf 0}&{\bf 0}&{\bf 0}\\{\bf 0}&\frac{\cos s\cos t\sin^2s}{(1-\cos^2s\cos^2t)^{3/2}}&{\bf 0}&\frac{-\sin s\cos^2s\cos t\sin t}{(1-\cos^2s\cos^2t)^{3/2}}&0\\{\bf 0}&0&{\bf 0}&0&0\end{pmatrix},$$
en particular $\dfrac{\partial h}{\partial t}\Big|_{t=0}=-\dfrac{\cos s}{\sin s}I_{n,j}$. Derivando una vez más con respecto a $t$ y evaluando en $t=0$ obtenemos
$\dfrac{\partial^2 h}{\partial t^2}\Big|_{t=0}=-\dfrac{\cos^2s}{\sin^2 s}(E_{jj}+E_{n,n})$.
Entonces
$$\frac{\partial^2 A}{\partial t^2}\Big|_{t=0}
-\Big( \frac{\partial A}{\partial t}\Big|_{t=0}\Big)^2=-\dfrac{\cos^2s}{\sin^2 s}(E_{jj}+E_{n,n})-\dfrac{\cos^2 s}{\sin^2 s}I_{n,j}^2=0.$$
La proposición ha quedado demostrada.\end{proof}

\

Para $j=n$ tenemos
$$[I_{n+1,n}^2\Phi](a(s))=\frac{\partial^2}{\partial t^2}\Phi(a(s)\exp tI_{n+1,n})\Big|_{t=0}=\frac{\partial^2}{\partial t^2}\Phi(a(s+t))\Big|_{t=0}.$$
Luego $x(a(s+t))=\tan(s+t)$, de aquí
\begin{flalign*}
&\frac{\partial}{\partial t}\Phi(a(s+t))=\frac1{\cos^2(s+t)}F'(\tan(s+t)),\\
&\frac{\partial^2}{\partial t^2}\Phi(a(s+t))\Big|_{t=0}=\frac1{\cos^4s}F''(\tan(s))+2\frac{\sin s}{\cos^3s}F'(\tan s).
\end{flalign*}
\begin{cor} Sea $\Phi$ una función suave de $G$ que toma valores en $\End(V_\pi)$ tal que $\Phi(kgk')=\pi(k)\Phi(g)\pi(k')$ para todo $g\in G$ y todo $k,k'\in K$.  Entonces, si $F(s)=\Phi(a(s))$ tenemos
\begin{equation*}
\begin{split}
[\Delta \Phi](a(s))=&\frac1{\cos^4s}F''(\tan s)+\Big(\frac{n-1}{\sin^2s}+\frac{2\sin s}{\cos^3s}\Big)F'(\tan s)+\frac{\cos^2s}{\sin^2 s}F(\tan s)\sum_{j=1}^{n-1}\dot\pi(I_{n,j})^2\\
&+\frac{1}{\sin^2 s}\sum_{j=1}^{n-1}\dot\pi(I_{n,j})^2F(\tan s)
-2\frac{\cos s}{\sin^2 s}\sum_{j=1}^{n-1}\dot\pi(I_{n,j})F(\tan s)\dot\pi(I_{n,j}).
\end{split}
\end{equation*}
\end{cor}

\begin{cor} Sea $F(s)=\Phi(a(s))$, donde $\Phi$ es una función esférica irreducible en $G$ de tipo $\pi\in \hat K$. Entonces
\begin{equation*}
\begin{split}
\frac1{(1+x^2)^2}F''(x)+\frac{(n-1+2x^3)(1+x^2)}{x^2}F'(x)+\frac{(1+x^2)}{x^2}\sum_{j=1}^{n-1}\dot\pi(I_{n,j})^2F(x)&\\
-2\frac{\sqrt{1+x^2}}{x^2}\sum_{j=1}^{n-1}\dot\pi(I_{n,j})F(x)\dot\pi(I_{n,j})+\frac1{x^2}F(x)\sum_{j=1}^{n-1}\dot\pi(I_{n,j})^2&=\lambda F(x),
\end{split}
\end{equation*}
para algún $\lambda\in\CC$.
\end{cor}

\begin{cor} Sea $F(s)=\Phi(a(s))$, donde $\Phi$ es una función esférica irreducible en $G$ de $K$-tipo trivial. Entonces
\begin{equation*}
\begin{split}
&\frac1{(1+x^2)^2}F''(x)+\frac{(n-1+2x^3)(1+x^2)}{x^2}F'(x)=\lambda F(x),
\end{split}
\end{equation*}
para algún $\lambda\in\CC$.
\end{cor}

\subsection{Nueva Coordenada en $A$}
\

  Tomemos en $A$ la coordenada $x(a(s))=s$. Entonces tenemos $x(a(s,t))=\arccos(\cos s\cos t)$ y
$$\frac{\partial}{\partial t}x(a(s,t))=\frac{\cos s\sin t}{\sqrt{1-\cos^2 s\cos^2t}}.$$
De esto obtenemos
$$\frac{\partial}{\partial t}x(a(s,t))\Big|_{t=0}=0,\quad{\text y}\quad \frac{\partial^2}{\partial t^2}x(a(s,t))\Big|_{t=0}=\frac{\cos s}{\sin s}.$$
Entonces
\begin{align*}
&\frac{\partial}{\partial t}\Phi(a(s,t))\Big|_{t=0}=F'(s)\frac{\partial}{\partial t}x(a(s,t))\Big|_{t=0}=0,\text{ y}\\
&\frac{\partial^2}{\partial t^2}\Phi(a(s,t))\Big|_{t=0}=\frac{\cos s}{\sin s}F'(s).
\end{align*}

Como $x(a(s+t))=s+t$ obtenemos
\begin{equation*}
\begin{split}
[I_{n+1,n}^2\Phi](a(s))&=\frac{\partial^2}{\partial t^2}\Phi(a(s)\exp tI_{n+1,n})\Big|_{t=0}=\frac{\partial^2}{\partial t^2}\Phi(a(s+t))\Big|_{t=0}\\
&=\frac{\partial^2}{\partial t^2}F(s+t)\Big|_{t=0}=F''(s).
\end{split}
\end{equation*}

Ahora obtenemos

\begin{cor} Sea $\Phi$ una función suave en $G$ con valores en $\End(V_\pi)$ tal que $\Phi(kgk')=\pi(k)\Phi(g)\pi(k')$ para todo $g\in G$ y todos $k,k'\in K$.  Entonces, si $F(s)=\Phi(a(s))$, tenemos
\begin{equation*}
\begin{split}
[\Delta\Phi](a(s))=&F''(s)+(n-1)\frac{\cos s}{\sin s}F'(s)+\frac{1}{\sin^2 s}\sum_{j=1}^{n-1}\dot\pi(I_{n,j})^2F(s)
\\&-2\frac{\cos s}{\sin^2 s}\sum_{j=1}^{n-1}\dot\pi(I_{n,j})F(s)\dot\pi(I_{n,j})+\frac{\cos^2s}{\sin^2 s}F(s)\sum_{j=1}^{n-1}\dot\pi(I_{n,j})^2.
\end{split}
\end{equation*}
\end{cor}

\begin{cor}\label{eigenvalue} Sea   $\Phi$ una función esférica irreducible en $G$ de tipo $\pi\in \hat K$. Entonces, si $F(s)=\Phi(a(s))$, tenemos
\begin{equation*}
\begin{split}
F''(s)+(n-1)\frac{\cos s}{\sin s}F'(s)+\frac1{\sin^2s}\sum_{j=1}^{n-1}\dot\pi(I_{n,j})^2F(s)
-2\frac{\cos s}{\sin^2 s}\sum_{j=1}^{n-1}\dot\pi(I_{n,j})F(s)\dot\pi(I_{n,j})&\\+\frac{\cos^2s}{\sin^2s}F(s)\sum_{j=1}^{n-1}\dot\pi(I_{n,j})^2&=\lambda F(s),
\end{split}
\end{equation*}
para algún $\lambda\in\CC$.
\end{cor}

\begin{cor}\label{eigenfunction1} Sea   $\Phi$ una función esférica irreducible en $G$ del $K$-tipo trivial. Entonces, para $F(s)=\Phi(a(s))$ tenemos
\begin{equation*}
\begin{split}
&F''(s)+(n-1)\frac{\cos s}{\sin s}F'(s)=\lambda F(s),
\end{split}
\end{equation*}
para algún $\lambda\in\CC$.
\end{cor}

  Hacemos el cambio de variables $y=(1+\cos s)/2$. Entonces $\cos s=2y-1$, $\sin^2s=4y(1-y)$ y $\tfrac{d}{dy}=-\tfrac{\sin s}{2}$. Si ponemos $H(y)=F(s)$ obtenemos
$$F'(s)=-\frac{\sin s}{2}H'(s),\qquad F''(s)=\frac{\sin^2s}{4}H''(y)-\frac{\cos s}{2}H'(y).$$

En términos de esta nueva variable, el Corolario \ref{eigenvalue}  resulta
\begin{cor}\label{eigenfunction2}
Sea   $\Phi$ una función esférica irreducible en $G$ de tipo $\pi\in \hat K$. Entonces, si $H(y)=\Phi(a(s))$ con $y=(1+\cos s)/2$, tenemos
\begin{equation*}
\begin{split}
y(1-y)H''(y)+\frac{1}{2}n(1-2y)H'(y)+\frac1{4y(1-y)}\sum_{j=1}^{n-1}\dot\pi(I_{n,j})^2H(y)&\\
+\frac{(1-2y)}{2y(1-y)}\sum_{j=1}^{n-1}\dot\pi(I_{n,j})H(y)\dot\pi(I_{n,j})+\frac{(1-2y)^2}{4y(1-y)}H(y)\sum_{j=1}^{n-1}\dot\pi(I_{n,j})^2&=\lambda H(y),
\end{split}
\end{equation*}
para algún $\lambda\in\CC$.
\end{cor}

Notemos que para cualquier $y\in[0,1]$ la función $H(y)$ es escalar cuando se restringe a un $M$-módulo, ver Proposición \ref{propesf}. Por lo tanto, si $m$ es el número de $M$-módulos contenidos en $(V,\pi)$, podemos interpretar la ecuación diferencial del Corolario \ref{eigenfunction2} como sigue
\begin{align*}
y(1-y)H''(y)+\frac{1}{2}n(1-2y)H'(y)+\frac1{4y(1-y)}N_1H(y)&\\
+\frac{(1-2y)}{2y(1-y)}M H(y)+\frac{(1-2y)^2}{4y(1-y)}N_2 H(y)&=\lambda H(y),
\end{align*}
donde $M$, $N_1$ y $N_2$ son matrices de tamaño $m\times m$ y $H:[0,1]\to \CC^m$ es la función vectorial dada por la  función matricial (que toma valores de matrices diagonales)  $H$.

Más aún, como $\sum_{j=1}^{n-1}I_{n,j}^2=Q_n-Q_{n-1}$ por la Proposición \ref{rightinv} tenemos que $\sum_{j=1}^{n-1}I_{n,j}^2\in D(\SO(n-1))^{\SO(n-1)}$, por lo tanto  $\sum_{j=1}^{n-1}\dot\pi(I_{n,j})^2$ es escalar cuando se  restringe a un $M$-módulo. De allí, $N_1=N_2$ y la ecuación de arriba es equivalente a

\begin{equation}\label{vectoreq}
\begin{split}
y(1-y)H''(y)+\frac{n}{2}(1-2y)H'(y) +\frac{(1-2y)}{2y(1-y)}M H(y)+\frac{(1-2y)^2+1}{4y(1-y)}N H(y)&=\lambda H(y)
\end{split}
\end{equation}
donde $N$ es una matriz diagonal de tamaño $m\times m$. En las siguientes secciones encontraremos las expresiones explícitas  para $M$ y $N$.
\begin{remark}\label{ene}
Es importante notar que por \eqref{qvimpar} y \eqref{qvpar} podemos inmediatamente \-ob\-te\-ner cada entrada de la matriz diagonal $N$, ya que $\sum_{j=1}^{n}I_{n+1,j}^2=Q_{n+1}-Q_{n}$.
\end{remark}

\subsection{$K$-Tipo Trivial}
\

 Para el caso particular del $K$-tipo trivial tenemos
\begin{cor}\label{eigenfunction3} Sea   $\Phi$ una función esférica irreducible en $G$ de $K$-tipo trivial. Entonces, si $H(y)=\Phi(a(s))$ con $y=(1+\cos s)/2$, tenemos
\begin{equation*}
y(1-y)H''(y)+\frac{1}{2}n(1-2y)H'(y)=\lambda H(y),
\end{equation*}
para algún $\lambda\in\CC$.
\end{cor}

Observemos que cuando $n=1$ el Corolario \ref{eigenfunction1} nos dice
$$F''(s)=\lambda F(s),$$
lo cual es satisfecho por las funciones esféricas $F(s)=e^{iks}$ de $S^1$, con $\lambda=-k^2$.

Además cuando $n=2$ el Corolario \ref{eigenfunction3} nos dice
$$y(1-y)H''(y)+(1-2y)H'(y)=\lambda H(y).$$
Las soluciones acotadas en $y=0$, salvo múltiplos escalares,  son ${}_2\!F_1(-k,k+1,1;y)$, para $k\in \NN_0$. Como el polinomio de Legendre de grado $k$ es dado por
$$P_k(x)= {}_2\!F_1 \!\!\left( \begin{matrix} -k\,,\,k+1\\
1\end{matrix} ; (1+x)/2\right),$$
obtenemos que $F(s)=P_k(\cos s)$.

\section{Los $K$-Tipos $M$-Irreducibles}
\

Sea $K=\SO(n)$ con $n=2\ell+1$ y sea $\mm_n=(m_{1n},\dots,m_{\ell n})$ un $K$-tipo tal que $V_\mm$ es irreducible como $M$-módulo, con $M=\SO(n-1)$.
Los pesos máximos $\mm_{n-1}$ de los $M$-submódulos de $V_\mm$ son aquellos que satisfacen las siguientes relaciones de entrelazamiento

$$
 \begin{array}{ccccccccccc}
m_{1n}&{}       &m_{2n}     &{}   &\cdot&{}     &m_{\ell,n}&{}            &-m_{\ell\, n}\\
    {}&m_{1,n-1}&{}         &\cdot&{}   &\cdot  &{}          &m_{\ell,n-1}&{}
          \end{array}$$

Como $V_\mm$ es irreducible como $M$-módulo sigue que $m_{1n}=\cdots=m_{\ell,n}=0$.
La recíproca es cierta también, por lo tanto $V_\mm$ es $M$-irreducible si y solo si es la representación trivial.

\
Sea $K=\SO(n)$ con $n=2\ell$ y sea $\mm_n=(m_{1n},\dots,m_{\ell n})$ un $K$-tipo tal que $V_\mm$ es irreducible como $M$-módulo.
Los pesos máximos $\mm_{n-1}$ del $M$-submódulos de $V_\mm$ son aquellos que satisfacen las siguientes relaciones de entrelazamiento
$$\begin{matrix}
m_{1n}&{}&m_{2n}&{}&\cdots&{}&m_{\ell-1,n}&{}&m_{\ell\, n}&{}\\
{}&m_{1,n-1}&{}&\cdots&{}&\cdots&{}&m_{\ell-1,n-1}&{}&-m_{\ell-1,n-1}
\end{matrix}$$
Como $V_\mm$ es irreducible como $M$-módulo sigue que $m_{1n}=\cdots=m_{\ell-1,n}=p$ y $m_{\ell n}=p-j$ con $0\le j\le 2p$, pues
$m_{\ell-1,n}\ge\vert m_{\ell n}\vert$. Esto implica que $m_{1,n-1}=\cdots =m_{\ell-2,n-1}=p$ y $m_{\ell-1,n-1}=q$ con $p\ge q\ge\max\{p-j,j-p\}$. Entonces si $0\le j\le p$ tenemos $p\ge q\ge p-j$ y por irreducibilidad tenemos  $j=0$. Similarmente si $p\le j\le 2p$ tenemos $p\ge q\ge j-p$ y por irreducibilidad tenemos $j=2p$. Por lo tanto $\mm_n=(p,\dots,p)=2p\alpha$ o $\mm_n=(p,\dots,p,-p)=2p\beta$.

La recíproca es cierta también, por lo tanto $V_\mm$ es $M$-irreducible si y solo si $\mm_n=2p\alpha$ o $\mm_n=2p\beta$ para cualquier $p\in\NN_0$.

\

Si $\Phi$ es una función esférica irreducible en $\SO(n+1)$ de tipo $\pi$, con peso máximo $\mm_n=2p\alpha$ o $\mm_n=2p\beta$, entonces por el Corolario \ref{eigenfunction2} obtenemos que la función asociada $H$ satisface
\begin{equation*}
y(1-y)H''(y)+\ell(1-2y)H'(y)+\frac{1-y}{y}\sum_{j=1}^{n-1}\dot\pi(I_{nj})^2H(y)=\lambda H(y).
\end{equation*}
Para calcular $\sum_{j=1}^{n-1}\dot\pi(I_{nj})^2$ escribimos $\sum_{j=1}^{n-1}\dot\pi(I_{nj})^2=\dot\pi(Q_{n}-Q_{n-1})$.

  Primero consideramos $\mm_n=2p\alpha$. Si $v\in V_{\mm_n}$ es un vector peso máximo, entonces
\begin{equation*}
\begin{split}
\dot\pi(Q_n)v&=-p\ell(p+\ell-1)v,\\
\dot\pi(Q_{n-1})v&=-p(\ell-1)(p+\ell-1)v,
\end{split}
\end{equation*}
ver \eqref{qvpar} y \eqref{qvimpar}. Por lo tanto
$$\sum_{j=1}^{n-1}\dot\pi(I_{nj})^2v=-p(p+\ell-1)v.$$

  Ahora consideramos $\mm_n=2p\beta$. Si $v\in V_{\mm_n}$ es un vector peso máximo, entonces
$\dot\pi(Q_n)v=-2p\ell(p+\ell-1)v$ y $\dot\pi(Q_{n-1})v=-2p(\ell-1)(p+\ell-1)v$ como antes pues en ambos casos
$\mm_{n-1}$ es el mismo.

Por lo tanto si $\mm_n=(p,\dots,p,\pm p)$ tenemos
$$\sum_{j=1}^{n-1}\dot\pi(I_{nj})^2v=-p(p+\ell-1)v.$$

De allí, si $\Phi$ es una función esférica irreducible en $\SO(n+1)$, $n=2\ell$, de tipo $\mm_n=(p,\dots,p,\pm p)\in \CC^\ell$, entonces la función escalar asociada  $H=h$ satisface
\begin{equation}\label{ecuacion}
y(1-y)h''(y)+\ell(1-2y)h'(y)-\frac{p(p+\ell-1)(1-y)}{y}h(y)=\lambda h(y).
\end{equation}

\

  Ahora computamos el autovalor $\lambda$ correspondiente a la función esférica de tipo $\pi\in\hat\SO(2\ell)$, de peso máximo $\mm_n=2p\alpha$, asociada a  la representación irreducible $\tau\in\SO(2\ell+1)$, de peso máximo $\mm_{n+1}=(w,p,\dots,p)\in\CC^\ell$.
Si $v\in V_{\mm_{n+1}}$ es un vector peso máximo, entonces por \eqref{qvimpar} tenemos
\begin{equation*}
\begin{split}
\dot\tau(Q_{n+1})v=-\left(w(w+2\ell-1)+p(\ell-1)(p+\ell-1)\right)v,
\end{split}
\end{equation*}

Si $v\in V_{\mm_{n}}$ es un vector peso máximo, entonces por \eqref{qvpar} tenemos
\begin{equation*}
\begin{split}
\dot\tau(Q_n)v=\dot\pi(Q_n)v=-p\ell(p+\ell-1)v.
\end{split}
\end{equation*}
Como $\Delta=Q_{n+1}-Q_n   $ sigue que
$$\lambda=-w(w+2\ell-1)+p(p+\ell-1).$$

Para resolver \eqref{ecuacion} escribimos $h=y^\alpha f$. Entonces obtenemos
\begin{equation*}
\begin{split}
y(1-y)y^\alpha f''+(2\alpha(1-y)+\ell(1-2y))y^\alpha f'+\left(\alpha(\alpha-1)(1-y)\right.&\\
\left.+\ell\alpha(1-2y)-p(p+\ell-1)(1-y)\right)y^{\alpha-1}f&=\lambda y^\alpha f.
\end{split}
\end{equation*}
Entonces la ecuación indicial es $\alpha(\alpha-1)+\ell\alpha-p(p+\ell-1)=0$ y $\alpha=p$ es una solución.
Si $h$ es solución de \eqref{ecuacion} tomamos $h=y^pf$, entonces obtenemos que $f$ es solución de
$$y(1-y)f''+(2p+\ell-2(p+\ell)y)f'-p\ell f=\lambda f.$$
Si reemplazamos $\lambda=-w(w+2\ell-1)+p(p+\ell-1)$ obtenemos
$$y(1-y)f''+(2p+\ell-2(p+\ell)y)f'-(p-w)(2\ell+p+w-1)f=0.$$

Sea $a=p-w, b=2\ell+p+w-1, c=2p+\ell$ entonces la ecuación de arriba resulta
$$y(1-y)f''+(c-(1+a+b)y)f'-abf=0.$$

Un sistema fundamental de soluciones  de esta ecuación alrededor de $y=0$ está dado por las siguientes funciones
$${}_2\!F_1 \!\!\left( \begin{matrix} a\,,\,b\\
c\end{matrix} ; y\right),\qquad y^{1-c}{}_2\!F_1 \!\!\left( \begin{matrix} a-c+1\,,\,b-c+1\\
2-c\end{matrix} ; y\right).$$
Como $h=y^pf$ es acotada en $y=0$ sigue que
$$h(y)=uy^p{}_2\!F_1 \!\!\left( \begin{matrix} p-w\,,\,2\ell+p+w-1\\2p+\ell\end{matrix} ; y\right)$$
donde la constante $u$ es determinada por la condición $h(1)=1$.
\begin{remark} Sea $h_w=h_w(y)$, $w\ge p$,  la función $h$ de arriba. Entonces $h_w$ es polinomio de grado $w$. Más aún, observemos que
la función $y^p$ usada para hipergeometrizar \eqref{ecuacion} es precisamente $h_p$.
\end{remark}

\

  Ahora computamos el autovalor $\lambda$ correspondiente a la  función esférica de tipo $\mm_n=2p\beta$ asociada a  una representación irreducible $\tau$ de $\SO(n+1)$ de peso máximo $\mm_{n+1}=(w,p,\dots,p)\in\CC^\ell$.
Si $v\in V_{\mm_{n+1}}$ es un vector peso máximo, obtenemos $\dot\tau (Q_{n+1})v=-(w(w+2\ell-1)+p(\ell-1)(p+\ell-1))v$.

Si $v\in V_{\mm_{n}}$ es un vector peso máximo, entonces
$\dot\pi(Q_n)v=-p\ell(p+\ell-1)v$ como arriba, pues $Q_nv$ no depende del signo de $\mm_n$.
Como $\Delta=Q_{n+1}-Q_n$  además tenemos
$$\lambda=-w(w+2\ell-1)+p(p+\ell-1).$$

Por lo tanto tenemos probado el siguiente resultado.
\begin{thm}
Las funciones escalares $H=h$ asociadas a  las funciones esféricas irreducibles de $\SO(n+1)$, $n=2\ell$, de $\SO(n)$-tipo $\mm_n=(p,\dots,p,\pm p)\in \CC^\ell$, son parametrizadas por los enteros $w\ge p$ y son dadas por
$$h_w(y)=uy^p{}_2\!F_1 \!\!\left( \begin{matrix} p-w\,,\,2\ell+p+w-1\\2p+\ell\end{matrix} ; y\right),$$
donde la constante $u$ es determinada por la condición $h_w(1)=1$.
\end{thm}

\section{El Operador Diferencial $\Delta$ para $K=\SO(2\ell)$}
\

Queremos encontrar una expresión más explícita de la ecuación diferencial
\begin{equation*}
\begin{split}
y(1-y)H''(y)+\frac{1}{2}n(1-2y)H'(y)+\frac1{4y(1-y)}\sum_{j=1}^{n-1}\dot\pi(I_{n,j})^2H(y)&\\
+\frac{(1-2y)}{2y(1-y)}\sum_{j=1}^{n-1}\dot\pi(I_{n,j})H(y)\dot\pi(I_{n,j})+\frac{(1-2y)^2}{4y(1-y)}H(y)\sum_{j=1}^{n-1}\dot\pi(I_{n,j})^2&=\lambda H(y),
\end{split}
\end{equation*}
 dada en el Corolario \ref{eigenfunction2} para el caso en que la representación $\pi\in\SO(n)$ es fundamental y $n$ par; ver además \eqref{vectoreq} y Observación \ref{ene}.

%
%
%
%

Los pesos fundamentales de $\lieso(2\ell,\CC)$ son
\begin{align*}
\lambda_p&=\lambda_{\epsilon_p-\epsilon_{p+1}}=\epsilon_1+\cdots+\epsilon_p, \quad 1\le p\le\ell-2,\\
\lambda_{\ell-1}&=\tfrac12(\epsilon_1+\cdots+\epsilon_{\ell-1}-\epsilon_\ell),\\
\lambda_{\ell}&=\tfrac12(\epsilon_1+\cdots+\epsilon_{\ell-1}+\epsilon_\ell).
\end{align*}
Por lo tanto los pesos máximos de $\lieso(2\ell,\CC)$ son combinaciones lineales enteras  no negativas
\begin{align*}
\lambda=&q_1\lambda_1+\cdots+q_\ell\lambda_\ell\\
=&\big(q_1+\cdots+q_{\ell-2}+\tfrac12(q_{\ell-1}+q_\ell)\big)\epsilon_1 + \big(q_2+\cdots+q_{\ell-2}+\tfrac12(q_{\ell-1}+q_\ell)\big)\epsilon_2\\
&+\cdots+\big(q_{\ell-2}+\tfrac12(q_{\ell-1}+q_\ell)\big)\epsilon_{\ell-2}+\tfrac12(q_{\ell-1}+q_\ell)\epsilon_{\ell-1}+\tfrac12(q_\ell-q_{\ell-1})\epsilon_{\ell}.
\end{align*}

Sea $m_j=q_j+\cdots q_{\ell-2}+\tfrac12(q_{\ell-1}+q_\ell)$ para $1\le j\le\ell-2$ y $m_{\ell-1}=\tfrac12(q_{\ell-1}+q_\ell)$, $m_{\ell}=\tfrac12(q_\ell-q_{\ell-1})$. Entonces $\lambda=m_1\epsilon_1+\cdots+m_\ell\epsilon_\ell$, donde $m_1,\dots,m_\ell$ son todos enteros o  todos medio enteros, y $m_1\ge m_2\ge \cdots\ge m_{\ell-1}\ge |m_\ell|$. La recíproca también es cierta.

Ahora observemos  que los $m_j$'s son todos enteros o todos medio enteros si y solo si $q_{\ell-1}-q_{\ell}$ es, respectivamente, par o impar. Las representaciones irreducibles de dimensión finita de $\SO(2\ell)$ son aquellas cuyos pesos máximos $\lambda$ tienen todos los $m_j$'s enteros. Entonces asumiremos  que $q_{\ell-1}-q_{\ell}$ es par. Sea
$\omega_j=\epsilon_1+\cdots+\epsilon_j$ para $1\le j\le\ell-1$, $2\alpha=\epsilon_1+\cdots+\epsilon_\ell$ y $2\beta=\epsilon_1+\cdots+\epsilon_{\ell-1}-\epsilon_\ell$. Si $q_{\ell-1}\ge q_{\ell}$ tenemos
$$\lambda=q_1\omega_1+\cdots+q_{\ell-2}\omega_{\ell-2}+\tfrac{q_{\ell-1}-q_\ell}2(2\beta)+q_{\ell}\omega_{\ell-1}.$$
Si $q_{\ell}\ge q_{\ell-1}$ tenemos
$$\lambda=q_1\omega_1+\cdots+q_{\ell-2}\omega_{\ell-2}+q_{\ell-1}\omega_{\ell-1}+\tfrac{q_{\ell}-q_{\ell-1}}2(2\alpha).$$
Por lo tanto los pesos máximos de las representaciones irreducibles de dimensión finita de $\SO(2\ell)$ pueden ser escritas de una única forma como una combinación lineal entera no negativa:
\begin{align*}\lambda=a_1\omega_1+\cdots+a_{\ell-1}\omega_{\ell-1}+a_\ell(2\alpha)&&\text{o}&& \lambda=a_1\omega_1+\cdots+a_{\ell-1}\omega_{\ell-1}+a_\ell(2\beta).\end{align*}

\
Estamos interesados en calcular la expresión explícita de $M$ en \eqref{vectoreq}, i.e. queremos calcular
\begin{equation*}
\sum_{j=1}^{n-1}\dot\pi(I_{nj})P_{\sss}\dot\pi(I_{nj})|_{V_{\ttt(\rr)}}=\lambda(\rr,\sss)I_{V_{\ttt(\rr)}},
\end{equation*}
para cualquier representación irreducible de dimensión finita $\pi$ de $K=\SO(n)$ y $\rr,\sss\in\Omega$. En esta sección tomamos $n=2\ell$ y trabajaremos con las representaciones fundamentales de $K$: aquellas de pesos máximos $\omega_1,\dots,\omega_{\ell-1}$.

Comenzamos con el $K$-módulo estándar $V$. El vector  $\ee_1-i\ee_2$ es el único vector dominante salvo multiplicidad, y su peso es $\epsilon_1$. Entonces $V$ es un $K$-módulo irreducible de peso máximo $\omega_1=(1,0,\dots,0)\in\CC^\ell$. Entonces $V$ es la suma de dos $M$-submódulos, $M=\SO(n-1)$, específicamente $V=V_0\oplus V_1$ de pesos máximos $\ttt(0)=(0,0,\dots,0)\in\CC^{\ell-1}$ y $\ttt(1)=(1,0,\dots,0)\in\CC^{\ell-1}$, respectivamente. El subespacio generado por $\ee_n$ coincide con $V_0$ y $V_1$ es el subespacio generado por $\{\ee_1,\dots,\ee_{n-1}\}$. Claramente $\ee_n$ es un $M$ vector peso máximo en $V_0$ y $\ee_1-i\ee_2$ es un $M$ vector peso máximo en $V_1$.

Sea $1\le j\le n-1$. Tenemos
\begin{equation}\label{standard}
\dot\pi(I_{nj})\ee_i=\begin{cases} -\ee_n&\qquad {\text si}\quad i=j,\\0&\qquad {\text si}\quad i\ne j,\;1\le i\le n-1,\\\ee_j&\qquad {\text si}\quad i=n.
\end{cases}
\end{equation}

Para obtener $\lambda(0,0)$ es suficiente calcular $\sum_{j=1}^{n-1}\dot\pi(I_{nj})P_{0}\dot\pi(I_{nj})\ee_n$. Tenemos $\dot\pi(I_{nj})\ee_n=\ee_j$,  por lo tanto $P_{0}\dot\pi(I_{nj})\ee_n=0$ y $\lambda(0,0)=0$.

 Para obtener $\lambda(0,1)$ es suficiente calcular $\sum_{j=1}^{n-1}\dot\pi(I_{nj})P_{1}\dot\pi(I_{nj})\ee_n$. Tenemos $\dot\pi(I_{nj})\ee_n=\ee_j$,  por lo tanto $P_{1}\dot\pi(I_{nj})\ee_n=\ee_j$ y $\dot\pi(I_{nj})P_{1}\dot\pi(I_{nj})\ee_n=-\ee_n$. De allí $\lambda(0,1)=-(n-1)$.

Similarmente, para obtener $\lambda(1,0)$ es suficiente calcular $\sum_{j=1}^{n-1}\dot\pi(I_{nj})P_{0}\dot\pi(I_{nj})\ee_1$. Tenemos $\dot\pi(I_{nj})\ee_1=-\delta_{1j}\ee_n$, por lo tanto $P_{0}\dot\pi(I_{nj})\ee_1=-\delta_{1j}\ee_n$ y  $\dot\pi(I_{nj})P_{0}\dot\pi(I_{nj})\ee_1=-\delta_{1j}\ee_j$. De allí $\lambda(1,0)=-1$.

 Además es claro ahora que $\sum_{j=1}^{n-1}\dot\pi(I_{nj})P_{1}\dot\pi(I_{nj})\ee_1=0$, de aquí $\lambda(1,1)=0$.\\
Por lo tanto cuando $\pi$ es la representación estándar de $K$ tenemos
$$(\lambda(r,s)) _{0\le r,s\le1}=\begin{pmatrix} 0&1-n\\-1&0 \end{pmatrix}.$$

\

Ahora consideraremos el  $K$-módulo $\Lambda^p(V)$ para $1\le p\le \ell-1$. El vector
$(\ee_1-i\ee_2)\wedge(\ee_3-i\ee_4)\wedge\cdots\wedge(\ee_{2p-1}-i\ee_{2p})$ es el único  vector dominante salvo multiplicidad, y su peso es $\omega_p=(1,\dots,1,0,\dots,0)\in\CC^\ell$ con $p$ unos. Entonces $\Lambda^p(V)$ es la suma de dos $M$-submódulos,
$$\Lambda^p(V)=\Lambda^{p-1}(V_1)\wedge\ee_n\oplus \Lambda^p(V_1),$$
 de pesos máximos $\ttt(0)=(1,\dots,1,0,\dots,0)\in\CC^{\ell-1}$ con $p-1$ unos, y $\ttt(1)=(1,\dots,1,0,\dots,0)\in\CC^{\ell-1}$ con $p$ unos, respectivamente. Es fácil  ver que  $(\ee_1-i\ee_2)\wedge(\ee_3-i\ee_4)\wedge\cdots\wedge(\ee_{2p-3}-i\ee_{2p-2})\wedge\ee_n$ es un $M$ vector peso máximo en $\Lambda^{p-1}(V_1)\wedge\ee_n$ y que $(\ee_1-i\ee_2)\wedge(\ee_3-i\ee_4)\wedge\cdots\wedge(\ee_{2p-1}-i\ee_{2p})$ es un $M$ vector peso máximo en $\Lambda^p(V_1)$.

Para obtener $\lambda(0,0)$ es suficiente calcular
$$\sum_{j=1}^{n-1}\dot\pi(I_{nj})P_{0}\dot\pi(I_{nj})(\ee_1\wedge\cdots\wedge\ee_{p-1}\wedge\ee_n).$$
Tenemos $\dot\pi(I_{nj})(\ee_1\wedge\cdots\wedge\ee_{p-1}\wedge\ee_n)=\ee_1\wedge\cdots\wedge\ee_{p-1}\wedge\ee_j$,  por lo tanto $P_{0}\dot\pi(I_{nj})(\ee_1\wedge\cdots\wedge\ee_{p-1}\wedge\ee_n)=0$ y $\lambda(0,0)=0$.

Para obtener $\lambda(0,1)$ es suficiente calcular
$$\sum_{j=1}^{n-1}\dot\pi(I_{nj})P_{1}\dot\pi(I_{nj})(\ee_1\wedge\cdots\wedge\ee_{p-1}\wedge\ee_n).$$ Tenemos
\begin{equation*}
P_1\dot\pi(I_{nj})(\ee_1\wedge\cdots\wedge\ee_{p-1}\wedge\ee_n)=\begin{cases} 0&\; {\text si}\; 1\le j\le p-1,\\ \ee_1\wedge\cdots\wedge\ee_{p-1}\wedge\ee_j &\; {\text si}\; p\le j\le n-1.
\end{cases}
\end{equation*}
Por lo tanto tenemos $\dot\pi(I_{nj})P_1\dot\pi(I_{nj})(\ee_1\wedge\cdots\wedge\ee_{p-1}\wedge\ee_n)=$
\begin{equation*}
\begin{cases} 0&1\le j\le p-1,\\ -\ee_1\wedge\cdots\wedge\ee_{p-1}\wedge\ee_n&p\le j\le n-1.
\end{cases}
\end{equation*}
De allí $\lambda(0,1)=-(n-p)$.

Similarmente, para obtener $\lambda(1,0)$ es suficiente calcular
$$\sum_{j=1}^{n-1}\dot\pi(I_{nj})P_{0}\dot\pi(I_{nj})(\ee_1\wedge\cdots\wedge\ee_{p}).$$
Tenemos
\begin{equation*}
\dot\pi(I_{nj})(\ee_1\wedge\cdots\wedge\ee_{p})=\begin{cases} -\ee_1\wedge\cdots\wedge\ee_n\wedge\cdots\wedge\ee_{p}&\; \text{ si}\; 1\le j\le p,\\ 0 &\; \text{ si}\; p+1\le j\le n-1,
\end{cases}
\end{equation*}
donde $\ee_n$ aparece en el lugar $j$. Por lo tanto
\begin{equation*}
\dot\pi(I_{nj})P_0\dot\pi(I_{nj})(\ee_1\wedge\cdots\wedge\ee_{p})=\begin{cases} -\ee_1\wedge\cdots\wedge\ee_{p}&1\le j\le p,\\ 0 &p+1\le j\le n-1.
\end{cases}
\end{equation*}
De allí $\lambda(1,0)=-p$.

Además es claro ahora que $\sum_{j=1}^{n-1}\dot\pi(I_{nj})P_{1}\dot\pi(I_{nj})(\ee_1\wedge\cdots\wedge\ee_{p})=0$, de aquí $\lambda(1,1)=0$.

Por lo tanto, cuando $\pi$ es la representación estándar de $K$ en $\Lambda^p(V)$, $1\le p\le n-1$, tenemos
$$(\lambda(r,s)) _{0\le r,s\le1}=\begin{pmatrix} 0&p-n\\-p&0 \end{pmatrix}.$$

Por lo tanto, en esta situación tenemos una versión más explícita del Corolario \ref{eigenfunction2}; ver \eqref{vectoreq} y Observación \ref{ene}.
\begin{cor}\label{operator2l}
Sea   $\Phi$ una función esférica irreducible en $G$ de tipo $\pi\in \hat \SO(n)$, $n=2\ell$. Si el peso máximo de $\pi$ es de la forma $(1,\dots,1,0,\dots,0)\in\CC^\ell$, con $p$ unos, $1\le p\le \ell-1$, entonces la función $H:[0,1]\to\End(\CC^2)$ asociada a  $\Phi$ satisface
\begin{equation*}
\begin{split}
y(1-y)H''(y)+\frac{1}{2}n(1-2y)H'(y)+\frac{(1-2y)^2+1}{4y(1-y)}\left(\begin{smallmatrix} p-n &0\\0&-p \end{smallmatrix}\right)H(y)&\\
\hskip.3cm+\frac{(1-2y)}{2y(1-y)}\left(\begin{smallmatrix} 0&p-n\\-p&0 \end{smallmatrix}\right)H(y)&=\lambda H(y),
\end{split}
\end{equation*}
para algún $\lambda\in\CC$.
\end{cor}

\section{El Operador Diferencial $\Delta$ para $K=\SO(2\ell+1)$}
\

Queremos encontrar una expresión más explícita de la ecuación diferencial
\begin{equation*}
\begin{split}
y(1-y)H''(y)+\frac{1}{2}n(1-2y)H'(y)+\frac1{4y(1-y)}\sum_{j=1}^{n-1}\dot\pi(I_{n,j})^2H(y)&\\
\frac{(1-2y)}{2y(1-y)}\sum_{j=1}^{n-1}\dot\pi(I_{n,j})H(y)\dot\pi(I_{n,j})+\frac{(1-2y)^2}{4y(1-y)}H(y)\sum_{j=1}^{n-1}\dot\pi(I_{n,j})^2&=\lambda H(y),
\end{split}
\end{equation*}
 dada en el Corolario \ref{eigenfunction2} para el caso en que la representación $\pi\in\SO(n)$ es fundamental y $n$ impar; ver además \eqref{vectoreq} y Observación \ref{ene}.

%
%
%
%
%

Estamos interesados en  calcular
\begin{equation*}
\sum_{j=1}^{n-1}\dot\pi(I_{nj})P_{\sss}\dot\pi(I_{nj})|_{V_{\ttt(\rr)}}=\lambda(\rr,\sss)I_{V_{\ttt(\rr)}},
\end{equation*}
para cualquier representación irreducible de dimensión finita $\pi$ de $K=\SO(n)$ y $\rr,\sss\in\Omega$. En esta sección tomamos $n=2\ell+1$ y trabajaremos  con las  representaciones fundamentales de $K$, aquellas de pesos máximos
\begin{align*}
\lambda_p&=\epsilon_1+\cdots+\epsilon_p, \quad 1\le p\le\ell.
\end{align*}
El caso $\lambda_\ell$ es particularmente diferente a los otros, y lo trataremos posteriormente por separado.

Consideremos el  $K$-módulo $\Lambda^p(V)$ para $1\le p\le \ell-1$. El vector
$(\ee_1-i\ee_2)\wedge(\ee_3-i\ee_4)\wedge\cdots\wedge(\ee_{2p-1}-i\ee_{2p})$ es el único  vector dominante salvo multiplicidad, y su peso es $\omega_p=(1,\dots,1,0,\dots,0)\in\CC^\ell$ con $p$ unos. Entonces $\Lambda^p(V)$ es la suma de dos $M$-submódulos,
$$\Lambda^p(V)=\Lambda^{p-1}(V_1)\wedge\ee_n\oplus \Lambda^p(V_1)$$
 de pesos máximos $\ttt(0)=(1,\dots,1,0,\dots,0)\in\CC^{\ell-1}$ con $p-1$ unos, y $\ttt(1)=(1,\dots,1,0,\dots,0)\in\CC^{\ell-1}$ con $p$ unos, respectivamente. Es fácil  ver que  $(\ee_1-i\ee_2)\wedge(\ee_3-i\ee_4)\wedge\cdots\wedge(\ee_{2p-3}-i\ee_{2p-2})\wedge\ee_n$ es un $M$ vector peso máximo en $\Lambda^{p-1}(V_1)\wedge\ee_n$ y que $(\ee_1-i\ee_2)\wedge(\ee_3-i\ee_4)\wedge\cdots\wedge(\ee_{2p-1}-i\ee_{2p})$ es un $M$ vector peso máximo en $\Lambda^p(V_1)$.

Para obtener $\lambda(0,0)$ es suficiente calcular
$$\sum_{j=1}^{n-1}\dot\pi(I_{nj})P_{0}\dot\pi(I_{nj})(\ee_1\wedge\cdots\wedge\ee_{p-1}\wedge\ee_n).$$
Tenemos $\dot\pi(I_{nj})(\ee_1\wedge\cdots\wedge\ee_{p-1}\wedge\ee_n)=\ee_1\wedge\cdots\wedge\ee_{p-1}\wedge\ee_j$,  por lo tanto $P_{0}\dot\pi(I_{nj})(\ee_1\wedge\cdots\wedge\ee_{p-1}\wedge\ee_n)=0$ y $\lambda(0,0)=0$.

Para obtener $\lambda(0,1)$ es suficiente calcular
$$\sum_{j=1}^{n-1}\dot\pi(I_{nj})P_{1}\dot\pi(I_{nj})(\ee_1\wedge\cdots\wedge\ee_{p-1}\wedge\ee_n).$$ Tenemos
\begin{equation*}
P_1\dot\pi(I_{nj})(\ee_1\wedge\cdots\wedge\ee_{p-1}\wedge\ee_n)=\begin{cases} 0&\; \text{si}\; 1\le j\le p-1,\\ \ee_1\wedge\cdots\wedge\ee_{p-1}\wedge\ee_j &\; \text{si}\; p\le j\le n-1.
\end{cases}
\end{equation*}
Por lo tanto $\dot\pi(I_{nj})P_1\dot\pi(I_{nj})(\ee_1\wedge\cdots\wedge\ee_{p-1}\wedge\ee_n)=$
\begin{equation*}
\begin{cases} 0&\text{si } 1\le j\le p-1,\\
 -\ee_1\wedge\cdots\wedge\ee_{p-1}\wedge\ee_n&\text{si } p\le j\le n-1.
\end{cases}
\end{equation*}
De allí $\lambda(0,1)=-(n-p)$.

Similarmente, para obtener $\lambda(1,0)$ es suficiente calcular
$$\sum_{j=1}^{n-1}\dot\pi(I_{nj})P_{0}\dot\pi(I_{nj})(\ee_1\wedge\cdots\wedge\ee_{p}).$$
Tenemos que  $\dot\pi(I_{nj})(\ee_1\wedge\cdots\wedge\ee_{p})=$
\begin{equation*}
\begin{cases} -\ee_1\wedge\cdots\wedge\ee_n\wedge\cdots\wedge\ee_{p}&\; \text{ si}\; 1\le j\le p,\\ 0 &\; \text{ si}\; p+1\le j\le n-1,
\end{cases}
\end{equation*}
donde $\ee_n$ aparece en el lugar $j$. Por lo tanto
\begin{equation*}
\dot\pi(I_{nj})P_0\dot\pi(I_{nj})(\ee_1\wedge\cdots\wedge\ee_{p})=\begin{cases} -\ee_1\wedge\cdots\wedge\ee_{p}   &\text{si } 1\le j\le p,\\
 0 &\text{si } p+1\le j\le n-1.
\end{cases}
\end{equation*}
De allí $\lambda(1,0)=-p$.

Además es claro ahora que $\sum_{j=1}^{n-1}\dot\pi(I_{nj})P_{1}\dot\pi(I_{nj})(\ee_1\wedge\cdots\wedge\ee_{p})=0$, de aquí $\lambda(1,1)=0$.\\
Por lo tanto, cuando $\pi$ es la representación estándar de $K$ en $\Lambda^p(V)$, $1\le p\le n-1$, tenemos
$$(\lambda(r,s)) _{0\le r,s\le1}=\begin{pmatrix} 0&p-n\\-p&0 \end{pmatrix}.$$

Entonces, en esta situación tenemos una versión más explícita de Corolario \ref{eigenfunction2}; ver \eqref{vectoreq} y Observación \ref{ene}.
\begin{cor}\label{operator2l+1}
Sea   $\Phi$ una función esférica irreducible en $G$ de tipo $\pi\in \hat \SO(n)$, $n=2\ell+1$. Si el peso máximo de $\pi$ es de la forma $(1,\dots,1,0,\dots,0)$ $\in\CC^\ell$, con $p$ unos, $1\le p\le \ell-1$, entonces la función $H:[0,1]\to\End(\CC^2)$ asociada a  $\Phi$ satisface
\begin{equation*}
\begin{split}
y(1-y)H''(y)+\frac{1}{2}n(1-2y)H'(y)+\frac{(1-2y)^2+1}{4y(1-y)}\left(\begin{smallmatrix} p-n &0\\0&-p \end{smallmatrix}\right)H(y)&\\
\hskip.3cm+\frac{(1-2y)}{2y(1-y)}\left(\begin{smallmatrix} 0&p-n\\-p&0 \end{smallmatrix}\right)H(y)&=\lambda H(y),
\end{split}
\end{equation*}
para algún $\lambda\in\CC$.
\end{cor}

\subsection{Cuando el $K$-Tipo Fundamental tiene Peso Máximo $\lambda_\ell=(1,\dots,1)\in\CC^\ell$}\label{111}
\

Ahora consideramos el  $K$-módulo irreducible $\Lambda^\ell(V)$. El vector
$v=(\ee_1-i\ee_2)\wedge(\ee_3-i\ee_4)\wedge\cdots\wedge(\ee_{2\ell-1}-i\ee_{2\ell})$ es el único vector dominante salvo múltiplos escalares, y su peso es $\lambda_\ell=(1,1,\dots,1)$. Entonces $\Lambda^\ell(V)$ es la suma de tres $M$-módulos,
\begin{equation}\label{3M-mod}
\Lambda^\ell(V)=V_{2\alpha}\oplus \Lambda^{\ell-1}(V) \wedge\ee_n \oplus V_{2\beta}
\end{equation}
de pesos máximos  $\ttt(1)=2\alpha=(1,\dots,1)$, $\ttt(0)=(1,\dots,1,0)$ y $\ttt(-1)=2\beta=(1,\dots,1,-1)$, respectivamente.

Los vectores
\begin{align*}
v_1=&(\ee_1-i\ee_2)\wedge(\ee_3-i\ee_4)\wedge\cdots\wedge(\ee_{2\ell-1}-i\ee_{2\ell}),\\ v_0=&-(\ee_1-i\ee_2)\wedge(\ee_3-i\ee_4)\wedge\cdots\wedge(\ee_{2\ell-3}-i\ee_{2\ell-2})\wedge \ee_n,\\
v_{-1}=&(\ee_1-i\ee_2)\wedge(\ee_3-i\ee_4)\wedge\cdots\wedge(\ee_{2\ell-1}+i\ee_{2\ell}),
\end{align*}
son $M$ vector peso máximo en $ V_{2\alpha}$, en $\Lambda^{\ell-1}(V) \wedge\ee_n$ y en $ V_{2\beta}$ respectivamente. Además, llamemos $P_1$, $P_0$ y $P_{-1}$ a las  respectivas proyecciones en $ V_{2\alpha}$, $\Lambda^{\ell-1}(V) \wedge\ee_n$ y $ V_{2\beta}$, de acuerdo a la descomposición \eqref{3M-mod}.

Si $1\le j\le\ell$, entonces
\begin{equation*}
\dot\pi(I_{n,2j-1})(\ee_{2k-1}-i\ee_{2k})=\begin{cases} 0&\; {\text si}\; k\ne j,\\
-\ee_n &\; {\text si}\; k=j,
\end{cases}
\end{equation*}
\begin{equation*}
\dot\pi(I_{n,2j})(\ee_{2k-1}-i\ee_{2k})=\begin{cases} 0&\; {\text si}\; k\ne j,\\
i\ee_n &\; {\text si}\; k=j,
\end{cases}
\end{equation*}
por lo tanto, es fácil ver que $P_0 \dot\pi(I_{n,2j-1}) v_0=P_0 \dot\pi(I_{n,2j}) v_0=0$ y que
$P_\mathbf r \dot\pi(I_{n,2j-1}) v_\mathbf s=P_\mathbf r \dot\pi(I_{n,2j}) v_\mathbf s=0$
 cuando
 $\mathbf s\ne 0$ y $\mathbf r\ne 0$; i.e. $$\lambda(0,0 )=\lambda(-1,-1 )=\lambda(1,-1 )=\lambda(1,-1 )=\lambda(1,1 )=0.$$
Más aún, es fácil  ver que, para $1\le j\le \ell$ y $\mathbf r$ igual a $1$ o $-1$, tenemos $$\dot\pi(I_{n,2j-1})P_0\dot\pi(I_{n,2j-1})v_\mathbf r+\dot\pi(I_{n,2j})P_0\dot\pi(I_{n,2j})v_\mathbf r=-v_\mathbf r,$$
entonces $\lambda(-1,0 )=\lambda(1,0 )=-\ell$. Por lo tanto, solo resta estudiar
\begin{align*}
\dot\pi(I_{n,2j-1})P_\mathbf s\dot\pi(I_{n,2j-1})v_0&&\text{y}&&\dot\pi(I_{n,2j})P_\mathbf s\dot\pi(I_{n,2j})v_0,\end{align*}
para $1\le j\le \ell$ y $\mathbf s=1,-1$.

 Primero miramos las expresiones de $X_{-\varepsilon_j-\varepsilon_\ell}v_{1}\in V_{2\alpha}$ y $X_{-\varepsilon_j+\varepsilon_\ell}v_{-1}\in V_{2\beta}$ (ver \eqref{rootvectorsSn}). Tenemos
\begin{align*}
\dot\pi\left(X_{-\varepsilon_j-\varepsilon_\ell}\right)\left((\ee_{2j-1}-i\ee_{2j})\wedge(\ee_{2\ell-1}-i\ee_{2\ell})\right)
&=4i(\ee_{2j-1})\wedge(\ee_{2j})+4i(\ee_{2\ell-1})\wedge(\ee_{2\ell})\\
=4&i\left((\ee_{2j-1}-i\ee_{2j})\wedge(\ee_{2j})+(\ee_{2\ell-1}-i\ee_{2\ell})\wedge(\ee_{2\ell})\right),\\
\dot\pi\left(X_{-\varepsilon_j+\varepsilon_\ell}\right)\left((\ee_{2j-1}-i\ee_{2j})\wedge(\ee_{2\ell-1}+i\ee_{2\ell})\right)
&=4i(\ee_{2j-1})\wedge(\ee_{2j})-4i(\ee_{2\ell-1})\wedge(\ee_{2\ell})\\
=4&i\left((\ee_{2j-1}-i\ee_{2j})\wedge(\ee_{2j})-(\ee_{2\ell-1}+i\ee_{2\ell})\wedge(\ee_{2\ell})\right).\\
\end{align*}
Por lo tanto, para $1\le j< \ell$,
\begin{align*}
\dot\pi\left(X_{-\varepsilon_j-\varepsilon_\ell}\right)v_1=&4i\big((\ee_1-i\ee_2)\wedge\cdots\wedge(\ee_{2(\ell-1)-1}-i\ee_{2(\ell-1)})\wedge\ee_{2j}\\
&+(\ee_1-i\ee_2)\wedge\cdots\wedge{(\ee_{2j-3}-i\ee_{2j-2})}\wedge(\ee_{2\ell-1}-i\ee_{2\ell})\wedge(\ee_{2j+1}-i\ee_{2j+2})\wedge\\
&\cdots\wedge(\ee_{2(\ell-1)-1}-i\ee_{2(\ell-1)})\wedge\ee_{2\ell}\big),\\
\dot\pi\left(X_{-\varepsilon_j+\varepsilon_\ell}\right)v_{-1}=&4i\big((\ee_1-i\ee_2)\wedge\cdots\wedge(\ee_{2(\ell-1)-1}-i\ee_{2(\ell-1)})\wedge\ee_{2j}\\
&-(\ee_1-i\ee_2)\wedge\cdots\wedge{(\ee_{2j-3}-i\ee_{2j-2})}\wedge(\ee_{2\ell-1}-i\ee_{2\ell})\wedge(\ee_{2j+1}-i\ee_{2j+2})\wedge\\
&\cdots\wedge(\ee_{2(\ell-1)-1}-i\ee_{2(\ell-1)})\wedge\ee_{2\ell}\big).
\end{align*}
De allí, para $1\le j< \ell$,
\begin{align*}
\dot\pi (I_{n,2j})v_0&=(\ee_1-i\ee_2)\wedge\cdots\wedge(\ee_{2(\ell-1)-1}-i\ee_{2(\ell-1)})\wedge\ee_{2j}\\
&=\tfrac{-i}{8}\left(\dot\pi\left(X_{-\varepsilon_j-\varepsilon_\ell}\right)v_1+
\dot\pi\left(X_{-\varepsilon_j+\varepsilon_\ell}\right)v_{-1}\right),\\
\dot\pi (I_{n,2j-1})v_0&=(\ee_1-i\ee_2)\wedge\cdots\wedge(\ee_{2(\ell-1)-1}-i\ee_{2(\ell-1)})\wedge\ee_{2j-1}\\
&=(\ee_1-i\ee_2)\wedge\cdots\wedge(\ee_{2(\ell-1)-1}-i\ee_{2(\ell-1)})\wedge i\ee_{2j}\\
&=\tfrac{1}{8}\left(\dot\pi\left(X_{-\varepsilon_j-\varepsilon_\ell}\right)v_1+
\dot\pi\left(X_{-\varepsilon_j+\varepsilon_\ell}\right)v_{-1}\right),
\end{align*}
y para $j=\ell$ tenemos
\begin{align*}
\dot\pi (I_{n,2\ell})v_0=\tfrac{1}{2i}(-v_1+v_{-1}),&&
\dot\pi (I_{n,2\ell-1})v_0=\tfrac{1}{2}(v_1+v_{-1}).
\end{align*}
De allí, para $1\le j< \ell$,
\begin{align*}\dot\pi \left(I_{n,2j-1}\right)P_1\dot\pi \left(I_{n,2j-1}\right)v_0&=\tfrac{1}{8}\dot\pi \left(I_{n,2j-1}\right)\dot\pi\left(X_{-\varepsilon_j-\varepsilon_\ell}\right)v_1\\
&=\tfrac i2\big((\ee_1-i\ee_2)\wedge\cdots\wedge(-\ee_{n})\wedge\cdots\wedge(\ee_{2(\ell-1)-1}-i\ee_{2(\ell-1)})\wedge\ee_{2j}\\
&=\tfrac i2\big((\ee_1-i\ee_2)\wedge\cdots\wedge\ee_{2j}\wedge\cdots\wedge(\ee_{2(\ell-1)-1}-i\ee_{2(\ell-1)})\wedge\ee_{n},\\
\dot\pi \left(I_{n,2j}\right)P_1\dot\pi \left(I_{n,2j}\right)v_0&=\tfrac{-i}{8}\dot\pi \left(I_{n,2j}\right)\dot\pi\left(X_{-\varepsilon_j-\varepsilon_\ell}\right)v_1\\
&=\tfrac {1}{2}\big((\ee_1-i\ee_2)\wedge\cdots\wedge\ee_{2j-1}\wedge\cdots\wedge(\ee_{2(\ell-1)-1}-i\ee_{2(\ell-1)})\wedge(-\ee_{n})\\
&=-\tfrac 12\big((\ee_1-i\ee_2)\wedge\cdots\wedge\ee_{2j-1}\wedge\cdots\wedge(\ee_{2(\ell-1)-1}-i\ee_{2(\ell-1)})\wedge\ee_{n},
\end{align*}
y entonces, para $1\le j< \ell$,
\begin{align*}\dot\pi \left(I_{n,2j-1}\right)P_1\dot\pi \left(I_{n,2j-1}\right)v_0+\dot\pi \left(I_{n,2j}\right)v_0P_1\dot\pi \left(I_{n,2j}\right)v_0=-\tfrac12v_0.\\
\end{align*}
Por lo tanto, como
\begin{align*}
\dot\pi \left(I_{n,2\ell}\right)P_1\dot\pi (I_{n,2\ell})v_0&=-\tfrac{1}{2i}\dot\pi \left(I_{n,2\ell}\right)v_1=-\tfrac12v_0,\\
\dot\pi \left(I_{n,2\ell-1}\right)P_1\dot\pi (I_{n,2\ell-1})v_0&=\tfrac{1}{2}\dot\pi \left(I_{n,2\ell-1}\right)v_1=-\tfrac12v_0,\\
\end{align*}
tenemos que
$$\lambda(0,1)=-\frac{\ell+1}{2}.$$
De manera análoga también se obtiene que
$$\lambda(0,-1)=-\frac{\ell+1}{2}.$$
De allí
$$(\lambda(r,s)) _{-1\le r,s\le1}=\begin{pmatrix} 0&-\ell&0\\-\tfrac{\ell+1}{2}&0&-\tfrac{\ell+1}{2}\\0&-\ell&0 \end{pmatrix}.$$

Entonces, en esta situación tenemos una versión más explícita de Corolario \ref{eigenfunction2}; ver \eqref{vectoreq} y Observación \ref{ene}.
\begin{cor}\label{operator2l+12}
Sea  $\Phi$ una función esférica irreducible en $G$ de tipo $\pi\in \hat \SO(n)$, $n=2\ell+1$. Si el peso máximo de $\pi$ es de la forma $(1,\dots,1)\in\CC^\ell$, entonces la función $H:[0,1]\to\CC^3$ asociada a  $\Phi$ satisface
\begin{equation*}
\begin{split}
y(1-y)H''(y)+\frac{1}{2}n(1-2y)H'(y)+\frac{(1-2y)^2+1}{4y(1-y)}\left(\begin{smallmatrix} -\ell&0&0\\0&-\ell-1&0\\0&0&-\ell \end{smallmatrix}\right)H(y)&\\
\hskip.3cm+\frac{(1-2y)}{2y(1-y)}\left(\begin{smallmatrix} 0&-\ell&0\\-\tfrac{\ell+1}{2}&0&-\tfrac{\ell+1}{2}\\0&-\ell&0 \end{smallmatrix}\right)H(y)&=\lambda H(y),
\end{split}
\end{equation*}
para algún $\lambda\in\CC$.
\end{cor}

\section{Las Funciones Esféricas de  $K$-Tipos Fundamentales}\label{hiper}
\

Como ya mencionamos, tanto para $n$ par como para $n$ impar, las representaciones i\-rre\-du\-ci\-bles de peso máximo $(1,\dots,1,0,\dots,0)$ son representaciones fundamentales. La única representación fundamental no comprendida entre éstas es la de peso máximo $(1,\dots,1)$, la cual tiene lugar solo cuando $n$ es impar.

Sea $n=2\ell$, las funciones esféricas irreducibles de $K$-tipo $\mm_{n}=(1,\dots,$ $1,0,\dots,0)\in\CC^\ell$, con $p$ unos $1\le p\le\ell-1$, son aquellas asociadas a las representaciones irreducibles de $G$ de pesos máximos de la forma
$\mm_{n+1}=(w+1,1,\dots,1,\delta,0,\dots,0)$ $\in\CC^\ell$ tales que $\delta=0,1$.

Ahora consideremos el $K$-módulo $\Lambda^p(\CC^{n})$ el cual tiene peso máximo $\mm_{n}$.

 Para $w=0$ y $\delta=0$ consideramos el $G$-módulo $\Lambda^p(\CC^{n+1})$ cuyo peso máximo es $\mm_{n+1}$, y tenemos la siguiente descomposición en $K$-módulos
    \begin{equation*}
\begin{split}
\Lambda^p(\CC^{n+1})&=\Lambda^p(\CC^n)\oplus \Lambda^{p-1}(\CC^n)\wedge\ee_{n+1},
\end{split}
\end{equation*}
donde $\Lambda^p(\CC^{n})$ es la suma de dos $\SO(n-1)$-módulos:
$$\Lambda^p(\CC^{n})=\Lambda^p(\CC^{n-1})\oplus\Lambda^{p-1}(\CC^n)\wedge\ee_n.$$
Observemos que
\begin{align*}
a(s)&(\ee_1\wedge\cdots\wedge\ee_{p-1}\wedge\ee_n)=\ee_1\wedge\cdots\wedge\ee_{p-1}\wedge(\cos s\, \ee_n-\sin s\,\ee_{n+1})\\
&=\cos s(\ee_1\wedge\cdots\wedge\ee_{p-1}\wedge\ee_n)-\sin s(\ee_1\wedge\cdots\wedge\ee_{p-1}\wedge\ee_{n+1}).
\end{align*}
De allí, si $\Phi_{0}$ es la función esférica asociada a  las representaciones irreducibles de $G$ de peso máximo
$\mm_{n+1}=(1,1,\dots,1,\delta,0,\dots,0)\in\CC^\ell$ con $\delta=0$, tenemos que $$\Phi_0(s)(\ee_1\wedge\cdots\wedge\ee_{p-1}\wedge\ee_n)=\cos s\,(\ee_1\wedge\cdots\wedge\ee_{p-1}\wedge\ee_n).$$
Además tenemos que $a(s)(\ee_1\wedge\cdots\wedge\ee_p)=\ee_1\wedge\cdots\wedge\ee_p$. Entonces la función vectorial $F_0(s)$ dada por la función esférica irreducible $=\Phi_0(a(s))$ es
$$F_0(s)=\left(\begin{matrix} \cos s\\1\end{matrix}\right).$$

Para $w=0$ y $\delta=1$ consideramos el $G$-módulo $\Lambda^{p+1}(\CC^{n+1})$ cuyo peso máximo es $\mm_{n+1}$, y para $1\le p\le\ell-1$ tenemos la siguiente descomposición en  $K$-módulos
\begin{equation*}
\begin{split}
\Lambda^{p+1}(\CC^{n+1})&=\Lambda^{p+1}(\CC^n)\oplus \Lambda^{p}(\CC^n)\wedge\ee_{n+1},
\end{split}
\end{equation*}
donde $\Lambda^{p}(\CC^n)\wedge\ee_{n+1}$ es la suma de dos $\SO(n-1)$-módulos:
$$\Lambda^{p}(\CC^n)\wedge\ee_{n+1}=\Lambda^{p}(\CC^{n-1})\wedge\ee_{n+1} \oplus\Lambda^{p-1}(\CC^{n-1})\wedge\ee_n\wedge\ee_{n+1}.$$
Observemos que
$a(s)(\ee_1\wedge\cdots\wedge\ee_{p-1}\wedge\ee_n\wedge\ee_{n+1})=\ee_1\wedge\cdots\wedge\ee_{p-1}\wedge\ee_n\wedge\ee_{n+1}$.
De allí, si $\Phi_{1}$ es la función esférica asociada a  las representaciones irreducibles de $G$ de peso máximo
$\mm_{n+1}=(1,1,\dots,1,\delta,0,\dots,0)\in\CC^\ell$ con $\delta=1$, tenemos que $ \Phi_1(s)(\ee_1\wedge\cdots\wedge\ee_{p-1}\wedge\ee_n\wedge\ee_{n+1})=\cos s\,(\ee_1\wedge\cdots\wedge\ee_{p-1}\wedge\ee_n\wedge\ee_{n+1})$. Además tenemos que
\begin{equation*}
\begin{split}
a(s)(\ee_1\wedge\cdots\wedge\ee_p\wedge\ee_{n+1})&=
\ee_1\wedge\cdots\wedge\ee_{p-1}\wedge(\sin s \,\ee_n+\cos s\,\ee_{n+1})\\
&=\sin s(\ee_1\wedge\cdots\wedge\ee_{p-1}\wedge\ee_n)+\cos s (\ee_1\wedge\cdots\wedge\ee_{p-1}\wedge\ee_{n+1}).
\end{split}
\end{equation*}

Entonces la función vectorial $F_1(s)$ dado por la función esférica irreducible $=\Phi_1(a(s))$ es
$$F_1(s)=\left(\begin{matrix} 1\\\cos s\end{matrix}\right).$$
Por lo tanto, después del cambio de variables $\cos s=2y-1$, la función matricial hecha con las columnas $F_0$ y $F_1$ es
\begin{equation}\label{Psi}
\Psi(y)=\left(\begin{matrix} 2y-1&1\\1&2y-1\end{matrix}\right).
\end{equation}

Cada columna de $\Psi$ satisface la ecuación diferencial dada en la Proposición \ref{operator2l}. Y es fácil comprobar que, aún cuando $n$ es impar, tenemos
\begin{equation*}
\begin{split}
y(1-y)\Psi''+\frac{n}{2}(1-2y)\Psi'+\frac{1+(1-2y)^2}{4y(1-y)}\left(\begin{matrix}p-n&0\\0&-p\end{matrix}\right)\Psi&\\
+\frac{(1-2y)}{2y(1-y)}\left(\begin{matrix} 0&p-n\\-p&0\end{matrix}\right)\Psi
&=\Psi\left(\begin{matrix} -p&0\\0&p-n\end{matrix}\right).
\end{split}
\end{equation*}

\begin{thm}\label{ecpar} La función $\Psi$ puede ser usada para hipergeometrizar las ecuaciones diferenciales dadas en Proposiciones \ref{operator2l} y \ref{operator2l+1}. Precisamente, si $n$ es de la forma $2\ell$ o $2\ell+1$ y $H=\Psi P$ es la función vectorial asociada a  una función esférica irreducible en $G=\SO(n+1)$ de $K$-tipo $\mm_n=(1,\dots,1,0,\dots,0)\in\CC^\ell$, con $p$ unos y $1\le p\le\ell-1$, tenemos que $DP=\lambda P$ para algún $\lambda\in\CC$, donde $D$ es el operador diferencial dado por
\begin{equation*}
\begin{split}
y(1-y)P''-\left(\begin{matrix} (\frac n2+1)(2y-1)&-1\\-1&(\frac n2+1)(2y-1)\end{matrix}\right)P'
-\left(\begin{matrix} p&0\\0&n-p\end{matrix}\right)P.
\end{split}
\end{equation*}
\end{thm}

\begin{proof}[\it Demostración]   Escribamos $H=\Psi P$. Entonces
\begin{equation*}
\begin{split}
y(1-y)P''+\big(2y(1-y)\Psi^{-1}\Psi'+\frac n2(1-2y)I\big)P'+\Psi^{-1}\Biggr(y(1-y)\Psi''
+\frac n2(1-2y)\Psi'&\\+\frac{1+(1-2y)^2}{4y(1-y)}\left(\begin{matrix}p-n&0\\0&-p\end{matrix}\right)\Psi+\frac{(1-2y)}{2y(1-y)}\left(\begin{matrix} 0&p-n\\-p&0\end{matrix}\right)\Psi\Biggr)P&=\lambda P.
\end{split}
\end{equation*}
Ahora calculamos
$$2y(1-y)\Psi^{-1}\Psi'=\frac{4y(1-y)}{4y(y-1)}\left(\begin{matrix} 2y-1&-1\\-1&2y-1\end{matrix}\right)=-\left(\begin{matrix} 2y-1&-1\\-1&2y-1\end{matrix}\right).$$
Por lo tanto
\begin{equation*}
\begin{split}
y(1-y)P''&-\left(\begin{matrix} (\frac n2+1)(2y-1)&-1\\-1&(\frac n2+1)(2y-1)\end{matrix}\right)P'-\left(\begin{matrix} \lambda+p&0\\0&\lambda+n-p\end{matrix}\right)P=0.
\end{split}
\end{equation*}
Esto completa la prueba del teorema.
\end{proof}

\
 Ahora consideramos el otro caso. Sea $n=2\ell+1$, y nos concentremos en las funciones esféricas irreducibles $\Phi_{w,\delta}$ de tipo $\mm_{n}=(1,\dots,1)\in\CC^\ell$, las cuales están asociadas a  las representaciones irreducibles de $\SO(n+1)$ de pesos máximos de la forma
$\mm_{n+1}=(w+1,1,\dots,1,\delta)$ $\in\CC^{\ell+1}$ tales que el siguiente patrón se mantiene
$$
 \begin{array}{cccccccccccccccccc}
 w+1 &{} &1    &\cdots   &1  &{}     &\delta &{}  \\
  {} &1   &\cdots  &{}   &\cdots    &1   &{} &-1
          \end{array}.$$
Como antes, armamos  la función $\Psi$ cuyas columnas son dadas por las funciones esféricas $\Phi_{0,\delta}$, $\delta=-1,0,1$. Cuando $w=0$, se puede calcular con poca dificultad usando \cite[página 364, ecuación (8)]{V92}. Por lo tanto tenemos
$$\Psi(y)=\left(\begin{matrix} 2y-1+i\sqrt{1-(1-2y)^2}&1&2y-1-i\sqrt{1-(1-2y)^2}\\1&2y-1&1\\2y-1-i\sqrt{1-(1-2y)^2}&1&2y-1+i\sqrt{1-(1-2y)^2}\end{matrix}\right).$$

Cada columna de $\Psi$ satisface la ecuación diferencial dado en Proposición \ref{operator2l+12}. Y es fácil comprobar que tenemos
\begin{equation*}
\begin{split}
y(1-y)H''(y)+\frac{1}{2}n(1-2y)H'(y)+\frac{(1-2y)^2+1}{4y(1-y)}\left(\begin{smallmatrix} -\ell&0&0\\0&-\ell-1&0\\0&0&-\ell \end{smallmatrix}\right)\Psi(y)&\\
+\frac{(1-2y)}{2y(1-y)}\left(\begin{smallmatrix} 0&-\ell&0\\-\tfrac{\ell+1}{2}&0&-\tfrac{\ell+1}{2}\\0&-\ell&0 \end{smallmatrix}\right)\Psi(y)&=\Psi(y)\left(\begin{smallmatrix} -\ell&0&0\\0&-\ell-1&0\\0&0&-\ell \end{smallmatrix}\right),
\end{split}
\end{equation*}

\begin{thm}\label{ecimpar} La función $\Psi$ puede ser usada para hipergeometrizar la ecuación diferencial dada en Proposición \ref{operator2l+12}. Precisamente, si $n=2\ell+1$ y $H=\Psi P$ es la función vectorial asociada a  una función esférica irreducible en $G=\SO(n+1)$ de $K$-tipo $\mm_n=(1,\dots,1)\in\CC^\ell$, tenemos que $DP=\lambda P$ para algún $\lambda\in\CC$, donde $D$ es el operador diferencial dado por
\begin{equation*}
\begin{split}
DP=y(1-y)P''+\left(-(n+2)yI+\left(\begin{smallmatrix} (n+2)/2&1/2&0\\1&(n+2)/2&1\\0&1/2&(n+2)/2\end{smallmatrix}\right)\right)P'
+\left(\begin{smallmatrix} -\ell&0&0\\0&-\ell-1&0\\0&0&-\ell \end{smallmatrix}\right)P.
\end{split}
\end{equation*}
\end{thm}

\begin{proof}[\it Demostración]   Sea $H=\Psi P$. Entonces
\begin{equation*}
\begin{split}
y(1-y)P''+\big(2y(1-y)\Psi^{-1}\Psi'+\frac n2(1-2y)I\big)P'+\Psi^{-1}\Biggr(y(1-y)\Psi''+\frac n2(1-2y)\Psi'&\\
+\frac{1+(1-2y)^2}{4y(1-y)}\left(\begin{smallmatrix} -\ell&0&0\\0&-\ell-1&0\\0&0&-\ell \end{smallmatrix}\right)\Psi
+\frac{(1-2y)}{2y(1-y)}\left(\begin{smallmatrix} 0&-\ell&0\\-\tfrac{\ell+1}{2}&0&-\tfrac{\ell+1}{2}\\0&-\ell&0 \end{smallmatrix}\right)\Psi\Biggr)P&=\lambda P.
\end{split}
\end{equation*}
Ahora calculamos
$$2y(1-y)\Psi^{-1}\Psi'+\frac n2(1-2y)I=-(n+2)yI+\left(\begin{smallmatrix} (n+2)/2&1/2&0\\1&(n+2)/2&1\\0&1/2&(n+2)/2\end{smallmatrix}\right).$$
Por lo tanto
\begin{equation*}
\begin{split}
y(1-y)P''+\left(-(n+2)yI+\left(\begin{smallmatrix} (n+2)/2&1/2&0\\1&(n+2)/2&1\\0&1/2&(n+2)/2\end{smallmatrix}\right)\right)P'+\left(\left(\begin{smallmatrix} -\ell&0&0\\0&-\ell-1&0\\0&0&-\ell \end{smallmatrix}\right)-\lambda I\right)P=0.
\end{split}
\end{equation*}
Esto completa la demostración del teorema.
\end{proof}

\subsection{Autovalores Posibles}
\

Como dijimos, cuando $n=2\ell$ las funciones esféricas irreducibles del par $(\SO(n+1),\SO(n))$, de tipo $\mm_n=(1,\dots,1,0\dots,0)\in\CC^\ell$ con $p$ unos, $1\leq p\leq\ell$, son aquellas asociadas a  las
representaciones irreducibles $\tau$ de $G$ de pesos máximos de la forma \(\mm_{n+1}=(w+1,1,\dots,1,\delta,0,\dots,0)\) \(\in\CC^\ell\) con $p-1$ unos, tales que $\delta=0,1$.

Sea $\Phi_{w,\delta}$ la correspondiente función esférica. Entonces $\Delta \Phi_{w,\delta}=\lambda\Phi_{w,\delta}$, donde el autovalor
$\lambda=\lambda_n(w,\delta)$ puede ser calculado usando que $\Delta=Q_{n+1}-Q_n$. Si $v\in V_{\mm_{n+1}}$ es un vector peso máximo por \eqref{qvimpar} tenemos
\begin{equation*}
\begin{split}
\dot\tau(Q_{2\ell+1})v&=-\big((w+1)^2+(2\ell-1)(w+1)+(2\ell-p)(p-1)+2\delta(\ell-p)\big)v.
\end{split}
\end{equation*}
Si $v\in V_{\mm_{2\ell}}$ es un vector peso máximo, entonces por \eqref{qvpar} tenemos
\begin{equation*}
\begin{split}
\dot\pi(Q_n)v=-p(2\ell-p)v.
\end{split}
\end{equation*}
Como $\Delta=Q_{n+1}-Q_n$ sigue que
\begin{equation*}
\begin{split}
\lambda_{2\ell}(w,\delta)&=-(w+1)^2-(2\ell-1)(w+1)+(2\ell-p)-2\delta(\ell-p).
\end{split}
\end{equation*}
Análogamente, obtenemos que los autovalores de las funciones esféricas $\Phi_{w,\delta}$ del par $(\SO(2\ell+2),\SO(2\ell+1))$ son de la forma
\begin{equation*}
\begin{split}
\lambda_{2\ell+1}(w,\delta)&=-(w+1)(w+2\ell+1)+2\ell-p+1-\delta2(\ell-p)-\delta^2.
\end{split}
\end{equation*}
Por lo tanto, tenemos que los autovalores de las funciones esféricas $\Phi_{w,\delta}$ del par $(\SO(n+1),\SO(n))$ son de la forma
\begin{equation}\label{lambda}
\begin{split}
\lambda_{n}(w,\delta)=\begin{cases} -(w+1)(w+n)+n-p \quad&\text{si}\quad \delta=0,\\
              -(w+1)(w+n)+p \quad&\text{si}\quad \delta=\pm1.
              \end{cases}
\end{split}
\end{equation}
Observemos que $\delta$ puede ser $-1$ si y solo si $n=2\ell+1$ y $p=\ell$.

Comenzamos por concentrarnos en la situación $1\le p<\ell$ y mirando el caso $\delta=0$. La función vectorial $P_{w,0}$ asociada a  $\Phi_{w,0}$ satisface la siguiente ecuación diferencial, ver Teorema \ref{ecpar},
\begin{equation}\label{ecpar0}
\begin{split}
y(1-y)P''-\left(\begin{matrix} (n/2+1)(2y-1)&-1\\-1&(n/2+1)(2y-1)\end{matrix}\right)P'&\\
-\left(\begin{matrix} -(w+1)(w+n)+n&0\\0&-(w+1)(w+n)+2(n-p)\end{matrix}\right)P&=0.
\end{split}
\end{equation}
Ahora, sea $$C=\left(\begin{matrix} n/2+1&1\\1&n/2+1\end{matrix}\right),$$
y busquemos matrices $A,B$ tales que $ A+B=(n+1)I$ y
$$AB=\left(\begin{matrix}-(w+1)(w+n)+n&0\\0&-(w+1)(w+n)+2(n-p) \end{matrix}\right).$$
Una solución es
$$A=\left(\begin{matrix}-w&0\\0&a\end{matrix}\right),\quad B=\left(\begin{matrix}w+1+n&0\\0&n+1-a\end{matrix}\right),$$
donde $a$ es una solución de $x^2-(n+1)x-(w+1)(w+n)+2(n-p)=0$.

Similarmente para $\delta=1$ la función vectorial $P_{w,1}$ asociada a  $\Phi_{w,1}$ satisface la siguiente ecuación diferencial,
\begin{equation}\label{ecpar1}
\begin{split}
y(1-y)P''-\left(\begin{matrix} (n/2+1)(2y-1)&-1\\-1&(n/2+1)(2y-1)\end{matrix}\right)P'&\\
-\left(\begin{matrix} -(w+1)(w+n)+2p&0\\0&-(w+1)(w+n)+n\end{matrix}\right)P&=0.
\end{split}
\end{equation}

Ahora sea $$C=\left(\begin{matrix} n/2+1&1\\1&n/2+1\end{matrix}\right),$$
y busquemos matrices $A,B$ tales que
\begin{align*} A+B=(n+1)I, && AB=\left(\begin{matrix}-(w+1)(w+n)+2p&0\\0&-(w+1)(w+n)+n \end{matrix}\right).
\end{align*}
Una solución es
\begin{align*}
A=\left(\begin{matrix}b&0\\0&-w\end{matrix}\right),&& B=\left(\begin{matrix}n+1-b&0\\0&w+1+n\end{matrix}\right),
\end{align*}
donde $b$ es una solución de $x^2-(n+1)x-(w+1)(w+n)+2p=0$.

Ahora probaremos que ni $a$ ni $b$ están en $\ZZ$.
\begin{lem}\label{roots}
Un polinomio de la forma $x^2-( n+1)x-(w+1)(w+ n)+2j$ no tiene raíces en $\mathbb Z$, para $1\leq j\le n-1$. Y cada una de sus raíces está en $\RR.$
\end{lem}
\begin{proof}[\it Demostración]
Para la última afirmación solo observemos que $( n+1)^2+4((w+1)(w+ n)-2j)$ es un número positivo.

Para probar el resto del lema es suficiente comprobar que $(n+1)^2+4((w+1)(w+ n)-2j)$ no es el cuadrado de un número entero.
Observemos que $( n+1)^2+4((w+1)(w+ n)-2j)=(2w+ n+1)^2+8( {n/2}-j)$.

Entonces, cuando $ {n/2}\le j\le  n-1$, usando que $k^2=\sum_{h=1}^{k}(2h-1)$ para cada $k\in\NN$, solo necesitamos probar que $-8( {n/2}-j)$ no es $\sum_{h=m}^{2w+ n+1}(2h-1)$ para cualquier $m$.
Como $j- {n/2}\le {n/2}-1$ y $w\ge0$ tenemos que
$$8(j- {n/2})\le8( {n/2}-1)<2(2w+ n)-1 + 2(2w+ n+1)-1,$$
de aquí $-8( {n/2}-j)<\sum_{h=m}^{2w+ n+1}(2h-1)$ para cada $m$ menor que $2w+ n+1$. Por otra parte, para $m=2w+ n+1$ tenemos que $2(2w+ n+1)-1$ es un número impar, a diferencia de $-8( {n/2}-j)$.

Si $1\leq j\leq {n/2}$, solo necesitamos probar que $8( {n/2}-j)$ no es de la forma $\sum_{h={2w+ n+2}}^m(2h - 1)$ para cualquier $m$.
En efecto, como $j\ge1$ y $w+1>0$ tenemos que $8( {n/2}-j)<8w+16 {n/2}+4$ y por lo tanto $$8( {n/2}-j)<2(2w+ n+1)-1 + 2(2w+ n+2)-1,$$
de aquí $8( {n/2}-j)<\sum_{h=2w+ n+2}^{m}(2h-1)$ para cada $m$ mayor que $2w+ n+2$.
 Por otra parte, para $m=2w+ n+2$ tenemos que $2(2w+ n+2)-1$ es un número impar, a diferencia de $8( {n/2}-j)$.
\end{proof}

\begin{cor}\label{corpar0}
Dado $n=2\ell$ o $2\ell+1$, la función vectorial $P_{w,0}$, asociada a  la función esférica irreducible $\Phi_{w,0}$ del par $(\SO(n+1),\SO(n))$ de tipo $\mm_n=(1,\dots,1,0,\dots,0)$, con $p$ unos ($1\le p<\ell$), satisface la siguiente ecuación diferencial
\begin{equation*}
\begin{split}
&y(1-y)P''+\left(C_0-y\left(A_0+B_0+1\right)\right)P'-(A_0B_0)P=0,
\end{split}
\end{equation*}
donde
\begin{align*}
C_0=\left(\begin{matrix}  {\frac{n}{2}}+1&1\\1& {\frac{n}{2}}+1\end{matrix}\right),&&
A_0=\left(\begin{matrix}-w&0\\0&a\end{matrix}\right),&&B_0=\left(\begin{matrix}w+ n+1&0\\0& n+1-a\end{matrix}\right),
\end{align*}
con $a\in\RR-\ZZ$.
\end{cor}

\begin{cor}\label{corpar1}
Dado $n=2\ell$ o $2\ell+1$, la función vectorial $P_{w,1}$, asociada a  la función esférica irreducible $\Phi_{w,1}$ del par $(\SO(n+1),\SO(n))$ de tipo $\mm_n=(1,\dots,1,0,\dots,0)$, con $p$ unos ($1\le p<\ell$), satisface la siguiente ecuación diferencial
\begin{equation*}
\begin{split}
&y(1-y)P''+\left(C_1-y\left(A_1+B_1+1\right)\right)P'-(A_1B_1)P=0,
\end{split}
\end{equation*}
donde
\begin{align*}
C_1=\left(\begin{matrix}  {\frac{n}{2}}+1&1\\1& {\frac{n}{2}}+1\end{matrix}\right),&&
A_1=\left(\begin{matrix}b&0\\0&-w\end{matrix}\right),&&B_1=\left(\begin{matrix} n+1-b&0\\0&w+ n+1\end{matrix}\right),
\end{align*}
con $b\in\RR-\ZZ$.
\end{cor}

En una forma muy similar obtendremos resultados similares en el caso de funciones esféricas del par $(\SO(2\ell+2),\SO(2\ell+1))$ de tipo $\mm_{2\ell+1}=(1,\dots,1)$:

\begin{cor}\label{corimpar}
Dado $n=2\ell+1$, la función vectorial $P_{w,\delta}$, asociada a  la función esférica irreducible $\Phi_{w,\delta}$ del par $(\SO(n+1),\SO(n))$ de tipo $\mm_n=(1,\dots,1)$ (entonces $\delta=-1,0,1$), satisface el siguiente ecuación diferencial
\begin{equation*}
\begin{split}
&y(1-y)P''+\left(C_\delta -y\left(A_\delta +B_\delta +1\right)\right)P'-(A_\delta B_\delta )P=0,
\end{split}
\end{equation*}
donde
\begin{align*}
A_0 &=\left(\begin{matrix}-w&0&0\\0&a&0\\0&0&-w\end{matrix}\right),&B_0 &=\left(\begin{matrix} w+n+1&0&0\\0& n+1-a\\0&0&w+n+1\end{matrix}\right),\\
A_{\pm1} &=\left(\begin{matrix}b&0&0\\0&-w&0\\0&0&b\end{matrix}\right),&B_{\pm1} &=\left(\begin{matrix} n+1-b&0&0\\0&w+ n+1\\0&0&n+1-b\end{matrix}\right),\\
&&C_\delta=&\left(\begin{smallmatrix} (n+2)/2&1/2&0\\1&(n+2)/2&1\\0&1/2&(n+2)/2\end{smallmatrix}\right),
\end{align*}
con $a,b\in\RR-\ZZ$.
\end{cor}
\begin{proof}[\it Demostración]
 Por el Teorema \ref{ecimpar} y \eqref{lambda} sabemos que
\begin{multline*}
y(1-y)P''+\left(-(n+2)yI+\left(\begin{smallmatrix} (n+2)/2&1/2&0\\1&(n+2)/2&1\\0&1/2&(n+2)/2\end{smallmatrix}\right)\right)P'+\left(\begin{smallmatrix} -\ell&0&0\\0&-\ell-1&0\\0&0&-\ell \end{smallmatrix}\right)P\\=(-(w+1)(w+n)+\ell+1-\delta^2)P.
\end{multline*}
Ahora tomamos $C_\delta=\left(\begin{smallmatrix} (n+2)/2&1/2&0\\1&(n+2)/2&1\\0&1/2&(n+2)/2\end{smallmatrix}\right)$, para $\delta=-1,0,1$, y
\begin{align*}
A_0 =\left(\begin{matrix}-w&0&0\\0&a&0\\0&0&-w\end{matrix}\right),B_0 =\left(\begin{matrix} w+n+1&0&0\\0& n+1-a\\0&0&w+n+1\end{matrix}\right),\\
A_{\pm1} =\left(\begin{matrix}b&0&0\\0&-w&0\\0&0&b\end{matrix}\right),B_{\pm1} =\left(\begin{matrix} n+1-b&0&0\\0&w+ n+1\\0&0&n+1-b\end{matrix}\right),
\end{align*}
donde $a$ es una raíz de $x(n+1-x)-(w+1)(w+n)-2(\ell-1)$ y $b$ es una raíz de $x(n+1-x)-(w+1)(w+n)-2(\ell)$; por el lema \ref{roots} sabemos que ambos $a$ y $b$ están en $\RR-\ZZ$.
Y ahora es fácil ver que $A_\delta+B_\delta+I=(n+2)I$ y que
$$-A_\delta B_\delta=\left(\begin{smallmatrix} -\ell&0&0\\0&-\ell-1&0\\0&0&-\ell \end{smallmatrix}\right)-(-(w+1)(w+n)+\ell+1-\delta^2)I.$$
Esto concluye la prueba del corolario.
\end{proof}

\subsection{Autofunciones Polinomiales del Operador Hipergeométrico $D$}
\

Sea $D $ el operador diferencial en $(0,1)$ introducidos en el Teorema \ref{ecpar}:

\begin{equation}\label{DSn}
\begin{split}
DP=y(1-y)P''-\left(\begin{matrix} (n/2+1)(2y-1)&-1\\-1&(n/2+1)(2y-1)\end{matrix}\right)P'-\left(\begin{matrix} p&0\\0&n-p\end{matrix}\right)P.
\end{split}
\end{equation}

Comenzamos por estudiar las soluciones polinomiales con valores en $\CC^{2}$ de (\ref{ecpar0}) y (\ref{ecpar1}), que son casos particulares de la ecuación diferencial matricial $DP=\lambda P$:
\begin{align}\label{ochoSn}
&y(1-y) P'' + \left(C-y (A+B+1)\right)  P'+  \left(AB\right)P=0,
\end{align}
con
\begin{align*}
C=\left(\begin{matrix} \frac{n}2+1&1\\1&\frac{n}2+1\end{matrix}\right),
\quad A+B=\left(\begin{matrix} n+1&0\\0&n+1\end{matrix}\right),
\quad AB=\left(\begin{matrix} -\lambda-p&0\\0&-\lambda-n+p\end{matrix}\right),
\end{align*}
 con $n=2\ell$ o $ 2\ell+1$, $\ell\in\NN$, y $p$ un entero tal que $1\leq p\leq\ell-1$. Y si tomamos $\lambda= -(w+2)(w+n)+n-p$ o $\lambda= -(w+2)(w+n)+p$
obtenemos exactamente las ecuaciones \eqref{ecpar0} o \eqref{ecpar1} respectivamente.

La ecuación (\ref{ochoSn}) es un caso particular de la  ecuación diferencial hipergeométrica estudiada en \cite{T03}. Como los autovalores de $C$, que son de la forma $n/2$ y $n/2+2$, no están en $-\NN_0$ la función $P$ es determinada por $P_0=P(0)$. Para $|y|<1$ está dada  por
 \begin{align}\label{sum}
P(y)={}_2\!F_1\left(\begin{smallmatrix}A,B\\  C \end{smallmatrix}; y\right)P_0=\sum_{j=0}^{\infty}\frac{y^j}{j!} (C;A;B)_j P_0, \qquad P_0\in \CC^2,
\end{align}
donde el símbolo $(C;A;B)_j$ se define inductivamente por
\begin{align*}
(C;A;B)_0   =1,&&(C;A;B)_{j+1} =
\left(C+j\right)^{-1}\left(A+j\right)\left(B+j\right) (C;A;B)_j,
\end{align*}
para todo $j\geq 0$.

Por lo tanto existe una  solución polinomial de \eqref{ochoSn} si y solo si el coeficiente $(C;A;B)_j$ es una matriz singular
para algún $j\in \ZZ$. Como la matriz $C+j$ es invertible para todo $j\in\NN_0$, tenemos que existe una solución polinomial de grado
$\kappa$ para $DP=\lambda P$ si y solo si  existe
$P_0\in\CC^{2}$ tal que
$(C;A;B)_{\kappa}P_0\neq 0$ y
$(A+\kappa)(B+\kappa)(C;A;B)_{\kappa}P_0=0$.

Ahora fácilmente observamos que en el Corolario \ref{corpar0} y en el Corolario \ref{corpar1} el primer y único $j$ para el cual $(A+j)(B+j)$ es no singular es $j=w+1$, y su núcleo (de dimensión $1$) es el subespacio generado por $\left(\begin{smallmatrix} 1\\ 0 \end{smallmatrix}\right)$ y $\left(\begin{smallmatrix} 0\\1  \end{smallmatrix}\right)$ respectivamente. Por lo tanto tenemos el siguiente resultado,
\begin{thm}\label{polsol2l}
Dado $n=2\ell$, $p$ y $w$, las ecuaciones diferenciales en los  Corolarios \ref{corpar0} y \ref{corpar1},
\begin{align*}
y(1-y)P''-\left(\begin{matrix} (\ell+1)(2y-1)&-1\\-1&(\ell+1)(2y-1)\end{matrix}\right)P'&\\
-\left(\begin{matrix} -(w+1)(w+2\ell)+2\ell&0\\0&-(w+1)(w+2\ell)+2(2\ell-p)\end{matrix}\right)P&=0\\
\end{align*}
y
\begin{align*}
y(1-y)P''-\left(\begin{matrix} (\ell+1)(2y-1)&-1\\-1&(\ell+1)(2y-1)\end{matrix}\right)P'&\\
-\left(\begin{matrix} -(w+1)(w+2\ell)+2p&0\\0&-(w+1)(w+2\ell)+2\ell\end{matrix}\right)P&=0,
\end{align*}
respectivamente, tienen solo una solución polinomial salvo múltiplos escalares, cada una. Más aún, en ambos casos el grado del polinomio es $w$ y los coeficientes directores son múltiplos de $\left(\begin{smallmatrix} 1\\ 0 \end{smallmatrix}\right)$ o $\left(\begin{smallmatrix} 0\\ 1 \end{smallmatrix}\right)$ respectivamente.

\end{thm}
\subsection{Funciones Esféricas y Soluciones Polinomiales de $DF=\lambda F   $}
\

Consideremos $\widetilde D $, el operador diferencial en $(0,1)$ introducido en el Corolario \ref{eigenfunction2}:
\begin{equation}\label{Dtilde}
\begin{split}
\widetilde D H  =&y(1-y)H''(y)+\frac{1}{2}n(1-2y)H'(y)\\
&+\frac{1+(1-2y)^2}{4y(1-y)}\sum_{j=1}^{n-2}\dot\pi(I_{n-1,j})^2H(y)+\frac{(1-2y)}{2y(1-y)}\sum_{j=1}^{n-2}\dot\pi(I_{n-1,j})H(y)\dot\pi(I_{n-1,j}).
\end{split}
\end{equation}
Recordemos que el operador $D$ por \eqref{DSn} satisface $$D=\Psi \widetilde D \Psi^{-1},$$ donde
$$\Psi(y)=\left(\begin{matrix} 2y-1&1\\1&2y-1\end{matrix}\right)$$
es la función matricial dada en \eqref{Psi} y luego empleada para hipergeometrizar en el Teorema \ref{ecpar}.

Nos concentramos en los siguientes espacios vectoriales de  funciones con valores en $\CC^{2}$:
\begin{align*}
S_\lambda=&\{H=H(y):\widetilde D H=\lambda H,\text{ analítica en }y=1\},\\
W_\lambda=&\{P=P(y):DP=\lambda P,\text{ analítica en }y=1\}.
\end{align*}

Por el Teorema \ref{ecpar} sabemos que la correspondencia $F\mapsto\Psi F$ es una transformación lineal inyectiva de $W_\lambda$ en $S_\lambda$. Ahora probaremos que sesta transformación es biyectiva.

\begin{thm}\label{isomo}
La transformación lineal $P\mapsto\Psi P$ es un isomorfismo de $W_\lambda$ en $S_\lambda$.
\end{thm}
\begin{proof}[\it Demostración]
La función vectorial $P(y)\in W_\lambda$ es una solución de la ecuación hipergeométrica (\ref{ochoSn})
y, como hemos mencionado, tal función es caracterizada por su valor en $0$, por lo tanto $\dim(W_\lambda)=2$.

Por otra parte, si $H\in S_\lambda$ los coeficientes $H_j$ de $H(y)=\sum_{j=0}^\infty H_jy^j$ satisfacen la siguiente relación recursiva

\begin{align*} 0=&j(j-1)H_j-2(j-1)(j-2)H_{j-1}+(j-2)(j-3)H_{j-2}+n/2\, j H_j \\
               &-3n/2\, (j-1) H_{j-1}+n (j-2) H_{j-2} + \begin{pmatrix} p-n &0\\0&-p \end{pmatrix}(\tfrac12 H_j- H_{j-1}+H_{j-2})\\
               &+\begin{pmatrix} 0&p-n\\-p&0 \end{pmatrix}(\tfrac12 H_j-H_{j-1})-\lambda H_{j-1}-\lambda H_{j-2}
\end{align*}
para todo $j\geq0$, donde tomamos $H_{-1}=0=H_{-2}$.

Considerando el caso $j=0$ obtenemos
\begin{align*} 0=\begin{pmatrix} p-n &0\\0&-p \end{pmatrix}\tfrac12 H_0+\begin{pmatrix} 0&p-n\\-p&0 \end{pmatrix}\tfrac12 H_0
\end{align*}
entonces $H_0$ es un múltiplo de $\left(\begin{smallmatrix} 1\\-1  \end{smallmatrix}\right)$, digamos $H_0=s\left(\begin{smallmatrix} 1\\-1  \end{smallmatrix}\right)$ con $s\in\CC$.

Del caso $j=1$ obtenemos
\begin{align*} 0=&n/2\,  H_1 + \begin{pmatrix} p-n &0\\0&-p \end{pmatrix}(\tfrac12 H_1- H_{0}) +\begin{pmatrix} 0&p-n\\-p&0 \end{pmatrix}(\tfrac12 H_1-H_{0})-\lambda H_{0},
\end{align*}
entonces

\begin{align}\label{H_1} 2            H_0\lambda             =& \begin{pmatrix} p &p-n\\-p&-p+n \end{pmatrix} H_1.
\end{align}
Por lo tanto $H_1$ es de la forma $\tfrac{2\lambda s}{p}\left(\begin{smallmatrix} 1\\0  \end{smallmatrix}\right)+ r\left(\begin{smallmatrix} n-p\\p  \end{smallmatrix}\right)$ con $r\in\CC$.

Para $j>1$ tenemos que $H_j$ es determinada por $H_{j-1}$ y $H_{j-2}$. Por lo tanto $\dim S_\lambda $ es $2$. El teorema sigue.
\end{proof}

\section{El Producto Interno}\label{innerprod}
\

Dado una representación irreducible de dimensión finita $\pi$ de $\SO(n)$
en el espacio vectorial $V_\pi$ sea $(C(\SO(n+1))\otimes \End (V_\pi))^{\SO(n)\times
\SO(n)}$ el espacio de todas las funciones continuas $\Phi:\SO(n+1)\longrightarrow
\End(V_\pi)$  tales que $\Phi(k_1gk_2)=\pi(k_1)\Phi(g)\pi(k_2)$ para todo $g\in \SO(n+1)$, $k_1,k_2\in \SO(n)$.   Equipemos $V_\pi$ con un producto interno tal que $\pi(k)$ resulte unitaria para todo $k\in \SO(n)$. Entonces introducimos un producto interno en el espacio vectorial $(C(\SO(n+1))\otimes \End
(V_\pi))^{\SO(n)\times \SO(n)}$ definiendo
\begin{equation*}
\langle \Phi_1,\Phi_2 \rangle =\int_{\SO(n+1)} \tr ( \Phi_1(g)\Phi_2(g)^*)\, dg\, ,
\end{equation*}
donde $dg$ denota la medida de Haar en $\SO(n+1)$ normalizada por $\int_G dg=1$, y
donde $\Phi_2(g)^*$ denota la adjunta de $\Phi_2(g)$
con respecto al producto interno en $V_\pi$.

Usando las relaciones de ortogonalidad de Schur para las representaciones irreducibles unitarias
de $\SO(n+1)$, tenemos que si $\Phi_1$ y $\Phi_2$ son funciones esféricas irreducibles no equivalentes entonces
son ortogonales con respecto al producto interno $\langle\cdot ,\cdot\rangle$, i.e.
\begin{equation}\label{ort}
\langle \Phi_1,\Phi_2 \rangle =0.
\end{equation}

Recordemos que, dada una función esférica irreducible $\Phi$ de tipo $\pi$ del par $(\SO(n+1),\SO(n))$, la función $\Phi(a(s))$ es a valores escalares cuando se restringe a cualquier $\SO(n-1)$-módulo (ver \eqref{a(s)} para $a(s)$). Denotaremos por $m$ la cantidad de de $\SO(n-1)$-submódulos de $\pi$, y por $d_1,d_2,\dots,d_m$ las respectivas dimensiones de cada uno de aquellos submódulos.

En particular, si $\Phi_1$ y $\Phi_2$ son dos funciones esféricas irreducibles de tipo $(\pi,V_\pi)\in\hat \SO(n)$, consideramos las funciones vectoriales $H_1(y)$ y $H_2(y)$ dadas por las funciones matriciales diagonales $\Phi_1 ( a (s))$ y $\Phi_2 ( a (s))$, con $y=(\cos(s)+1)/2$,  respectivamente; denotando
$$H_1(u)=(h_0(u),\cdots, h_m(u))^t, \qquad H_2(u)=(f_0(u),\cdots, f_m(u))^t.$$

\begin{prop}\label{productoint}
Si $\Phi_1, \Phi_2\in \left(C(\SO(n+1))\otimes \End (V_\pi)\right)^{\SO(n)\times \SO(2n)}$ entonces
$$\langle \Phi_1,\Phi_2\rangle =\frac{(n-1)!}{\Gamma(n/2)^2} \int_{0}^{1}(y(1- y) )^{n/2-1}  \, \sum_{i=0}^m d_i \,h_i(y)\overline{f_i(y)}\,dy.$$
\end{prop}

\begin{proof}[\it Demostración]
Sea $A=\exp \RR I_{n,n-1}$ el subgrupo de Lie de $G$ de todos los elementos de
la forma
$$a(s)=\exp sI_{n,n-1}=\begin{pmatrix}I_{n-2}&0&0\\0&\cos s&\sin s\\0&-\sin s&\cos s\end{pmatrix}\, ,\qquad s\in \RR,$$
donde $I_{n-2}$ denota  la matriz identidad de tamaño $n-2$.

Ahora el Teorema 5.10, página 190 en \cite{H62}, establece que para cada
$f\in C(G/K)$ y un apropiado $c_*$
$$\int_{G/K} f(gK)\,dg_K=c_*\int_{K/M}\Big(\int_{-\pi}^{\pi}
\delta_*(a(s))f(ka(s)K)\,ds\Big)\,dk_M\,,$$
 donde la función $\delta_*:A\longrightarrow \RR $ es definida por
$$\delta_*(a(s))=\prod_{\nu\in\Sigma^+} |\sin i s \nu(I_{n,n-1})|,$$
    y  $dg_K$ y $dk_M$
son respectivamente las medidas invariantes en $G/K$ y
 $K/M$, normalizadas por $\int_{G/K} dg_K=\int_{K/M} dk_M=1$. En nuestro caso tenemos $\delta_*(a(s))=\sin^{n-1}s $.

Como la función $g\mapsto \tr(\Phi_1(g)\Phi_2(g)^*)$ es invariante a izquierda y derecha por multiplicación por elementos en $K$, tenemos
\begin{align*}
\langle \Phi_1,\Phi_2\rangle =\int_G\tr ( \Phi_1(g)\Phi_2(g)^*)dg=2 c_* \int_{0}^{\pi} \sin ^{n-1} s\,\tr\left( \Phi_1(a(s)\Phi_2(a(s))^*\right)\,ds.
\end{align*}
Si ponemos $y=\tfrac12(\cos s +1)$ para $0<s<\pi$ nos queda
$$\tr\left( \Phi_1(a(s)\Phi_2(a(s))^*\right)=\sum_{i=0}^m d_i\,h_i(y)\overline{f_i(y)}.$$
 Entonces
$$\langle \Phi_1,\Phi_2\rangle =4c_*  \int_{0}^{1}(4y(1- y ))^{(n-2)/2}  \,\sum_{i=0}^m d_i\,h_i(y)\overline{f_i(y)}\,dy.$$

Para encontrar el valor de  $c_*$ consideramos el caso trivial
$\Phi_1=\Phi_2= I$, teniendo entonces
$$1=4c_*  \int_{0}^{1}(4y(1- y ))^{(n-2)/2}  \,dy,$$
pues $\dim(V_\pi)=d_1+d_2+\dots+d_m$. Por lo tanto
 $$c_* =\frac{1}{4^{n/2}}\frac{\Gamma(n)}{\Gamma(n/2)\Gamma(n/2)},$$
  y la proposición sigue.
\end{proof}

\begin{prop}\label{Deltasim}
 Si $\Phi_1,\Phi_2\in\left(C^\infty(G)\otimes \End
(V_\pi)\right)^{K\times K}$ entonces
$$\langle \Delta\Phi_1, \Phi_2\rangle=\langle \Phi_1,\Delta \Phi_2\rangle.$$
\end{prop}
\begin{proof}[\it Demostración]
Si aplicamos un campo vectorial invariante a izquierda  $X\in \lieg$ a la función
$g\mapsto\tr(\Phi_1(g)\Phi_2(g)^*)$ en $G$ y luego integramos sobre $G$ obtenemos
$$0= \int_G \tr\left( (X\Phi_1)(g)\Phi_2(g)^*\right)\, dg+\int_G \tr\left(
\Phi_1(g)(X\Phi_2)(g)^*\right)\, dg.$$ Por lo tanto $\langle
X\Phi_1,\Phi_2\rangle=-\langle \Phi_1,X\Phi_2\rangle$. Ahora sea
$\tau:\lieg_\CC\longrightarrow \lieg_\CC$ la conjugación de
$\lieg_\CC$ con respecto a la forma real $\lieg$. Entonces $-\tau$
se extiende a un único operador ${}^*$ involutivo antilineal en $D(G)$ tal que $\left(D_1D_2\right)^*=D_2^*D_1^*$ para todo
$D_1,D_2\in D(G)$. Esto sigue fácilmente del hecho de que el álgebra universal envolvente sobre $\CC$ de $\lieg$ es canónicamente isomorfa a $D(G)$. Entonces se tiene que $\langle
D\Phi_1, \Phi_2\rangle=\langle \Phi_1,D^* \Phi_2\rangle$.

Es fácil verificar que  $\Delta^*=\Delta$.
\end{proof}

Ahora enunciaremos y demostraremos el siguiente resultado, el cual es análogo al Teorema 3.9 de \cite{RT06}.

\begin{thm}\label{pol}
Sea $H$ la función sobre  $\CC^m$ asociada a  una función esférica irreducible $\Phi$ del par $(\mathrm{SO}(n+1),\mathrm{SO}(n))$ de tipo fundamental. Si $P=\Psi^{-1}H$, entonces $P$ es polinomial.
\end{thm}
\begin{proof}[\it Demostración]
Sabemos que la función $H$ es analítica en $[0,1]$, y por el Corolario \ref{operator2l} sabemos que es una autofunción del operador $\widetilde D$ (ver \eqref{Dtilde}), por lo tanto por el  Teorema \ref{isomo} la función $P=\Psi^{-1}H$ es una autofunción analítica de $D$ en el intervalo cerrado  $[0,1]$.

Sea $V=V(y)$ el peso matricial con soporte en el intervalo $[0,1]$ definido por
$$V(y)=\frac{(n-1)!}{\Gamma(n/2)^2} (y(1- y) )^{n/2-1} diag(d_1,d_2,\dots,d_m).
$$
Sigue de la Proposición \ref{productoint} y de la Proposición \ref{Deltasim} que $\widetilde D$ es un operador simétrico con res\-pec\-to a la forma bilineal hermítica sesquilineal   a valores matriciales definida para funciones matriciales $C^\infty$ en $[0,1]$ por
\begin{equation*}
\langle H_1,H_2 \rangle_V =\int_0^1 H_2^*(y)V(y)H_1(y)dy.
\end{equation*}
Entonces, como $D=\Psi^{-1}\widetilde D \Psi$, tenemos que $D$ es un operador simétrico con respecto a
\begin{equation*}
\langle P,Q \rangle_W =\int_0^1 Q^*(y)W(y)P(y)dy,
\end{equation*}
donde \begin{equation}\label{W}
                            W=\Psi^{*}  V \Psi.
                           \end{equation}
Entonces, tenemos que $(W,D)$ es un par clásico en el sentido de  \cite{GPT03}, ver además \cite{D97}. Como el peso $W$ tiene momentos finitos existe una sucesión $\{Q_r\}_{r\geq0}$ de $m\times m$ polinomios ortonormales matriciales, tales que para $r>0$ se tiene $DQ_r=Q_r\Lambda_r$ donde $\Lambda_r$ es una  matriz diagonal real.

Ahora consideramos el espacio $L^2$ de todas las funciones $P$  en $[0,1]$ a valores en $\CC^{m}$ con el producto interno
$$(P_1,P_2)_W=\int_0^1P_2^*(y)W(y)P_1(y)dy.$$
Sea $\{e_1,e_2,\dots,e_m\}$ la base canónica de $\CC^m$. Entonces
$$(Q_re_j,Q_se_i)=e_i^*\langle Q_r,Q_s\rangle_W e_j=\delta_{r,s}\delta_{i,j}.$$
Por lo tanto, para $r\geq0$, $j=1, 2,\dots,m$, $\{Q_re_j\}$ es una familia de polinomios ortonormales con valores en  $\CC^m$ tal que
$$D(Q_re_j)=(DQ_r)e_j=Q_r\Lambda_r e_j=\lambda_r^j(Q_r e_j),$$
donde $\Lambda_r=diag(\lambda_r^1,\lambda_r^2,\dots,\lambda_r^m)$.

Ahora escribimos nuestra función $P=\sum_{r,j}a_{r,j}Q_re_j$, donde $a_{r,j}=(P,Q_re_j)_W$. Entonces la serie converge no solo en la norma $L^2$ sino además en la topología de basada en la convergencia uniforme de funciones y sus derivadas sucesivas.

Por lo tanto, $$\lambda P=DP=\sum_{r,j}a_{r,j}\lambda_r^jQ_re_j.$$
Entonces $a_{r,j}=0$ si $\lambda_{r,j}\neq\lambda$. Como $\dim W_\lambda <\infty$ sigue que $P$ es un polinomio.
\end{proof}
Como antes, para una representación fundamental  $\pi\in\hat\SO(n)$, sea $\Phi_{w,\delta}$  la función esférica irreducible del par $(\SO(n+1),\SO(n))$ dada por la representación  $\tau\in\hat\SO(n+1)$ con peso máximo de la forma \(\mm_\tau=(w+1,1 ,\dots,1,\) \(\delta,0\dots,0)\).

Por lo tanto, combinando los Teoremas \ref{pol} y  \ref{polsol2l} tenemos el siguiente resultado.

 \begin{thm}\label{columnsSn}
Cada función esférica irreducible $\Phi_{w,\delta}$ del par $(\SO(n+1),\SO(n))$ con $n=2\ell$ o $2\ell+1$, de tipo $\mm_n=(1,\dots,1,0\dots,0)\in\CC^\ell$ con $p$ unos ($p<\ell$), corresponde a una función vectorial $P_{w,\delta}$ ($w\ge0$, $\delta=0,1$), la cual es un polinomio de grado $w$; el coeficiente director de $P_{w,0}$ y $P_{w,1}$ es un múltiplo de $\left(\begin{smallmatrix} 1\\ 0 \end{smallmatrix}\right)$ y $\left(\begin{smallmatrix} 0\\ 1 \end{smallmatrix}\right)$ respectivamente.
 Precisamente
\begin{equation*}
P_{w,\delta}(y)={}_2\!F_1\left(\begin{smallmatrix}A_\delta,B_\delta\\  C_\delta \end{smallmatrix}; y\right)P_0=\sum_{j=0}^{w}\frac{y^j}{j!} (C_\delta;A_\delta;B_\delta)_j  P_{w,\delta}(0),
\end{equation*}
donde
\begin{align*}
C_0=\left(\begin{matrix}  {\frac{n}{2}}+1&1\\1& {\frac{n}{2}}+1\end{matrix}\right),
A_0=\left(\begin{matrix}-w&0\\0&a\end{matrix}\right),B_0=\left(\begin{matrix}w+ n+1&0\\0& n+1-a\end{matrix}\right),\\
C_1=\left(\begin{matrix}  {\frac{n}{2}}+1&1\\1& {\frac{n}{2}}+1\end{matrix}\right),
A_1=\left(\begin{matrix}b&0\\0&-w\end{matrix}\right),B_1=\left(\begin{matrix} n+1-b&0\\0&w+ n+1\end{matrix}\right),
\end{align*}
con $a$ y $b$ soluciones (en $\RR -\ZZ$) de $x^2-(2\ell+1)x-(w+1)(w+2\ell)+2(2\ell-p)=0$ y $x^2-(2\ell+1)x-(w+1)(w+2\ell)+2p=0$, respectivamente.
Más aún, el valor de $P_{w,\delta}(0)$ puede ser calculado.
 \end{thm}
\begin{proof}[\it Demostración]
Es consecuencia de los Teoremas \ref{polsol2l} y  \ref{pol} y de los Corolarios \ref{corpar0} y \ref{corpar1}.

Además el vector $P_{w,\delta}(0)$ es determinado por dos condiciones. Consideremos el caso $\delta=0$; la primera condición   es que $\Phi_{w,\delta}(e)=1$, y por lo tanto
$$\Psi(1)\sum_{j=0}^{w}\frac{1}{j!} (C_\delta;A_\delta;B_\delta)_j  P_{w,0}(0)=\left(\begin{smallmatrix} 1\\ 1 \end{smallmatrix}\right)
.$$
Como $$\Psi(1)=\left(\begin{matrix}  1&1\\1& 1\end{matrix}\right),$$
tenemos que
$$\sum_{j=0}^{w}\frac{1}{j!} (C_\delta;A_\delta;B_\delta)_j  P_{w,0}(0)=\left(\begin{smallmatrix} 1\\ 0 \end{smallmatrix}\right)+c\left(\begin{smallmatrix} 1\\ -1 \end{smallmatrix}\right),\qquad \text{para algún } c\in\CC.$$

La segunda condición es que el coeficiente director de $P_{w,0}(y)$ es un múltiplo escalar  de $\left(\begin{smallmatrix} 1\\ 0 \end{smallmatrix}\right)$. De allí
$$\frac{1}{w!} (C_\delta;A_\delta;B_\delta)_w  P_{w,0}(0)=c' \left(\begin{smallmatrix} 1\\ 0 \end{smallmatrix}\right),\qquad \text{para algún }c'\in\CC.$$
Recordemos que $(C_\delta;A_\delta;B_\delta)_w$ es invertible, y entonces es fácil comprobar que no hay más que un solo posible valor para $P_{w,0}(0)$. Similarmente podemos probar lo mismo para el caso $\delta=1$.
\end{proof}

\begin{remark}\label{Ps}{Es importante observar que si  $(w,\delta)\neq(w',\delta')$, por \eqref{ort} tenemos $$\langle P_{w,\delta},P_{w',\delta'} \rangle_W=\langle \Phi_{w,\delta},\Phi_{w',\delta'}\rangle =0.$$}
\end{remark}

\section{Polinomios Matriciales Ortogonales}\label{POM}
\

En esta subsección, dado $n$ de la forma $2\ell$ o $2\ell+1$, para  $1\le p\le\ell-1$ construiremos una sucesión de polinomios matriciales ortogonales
$\{P_w\}_{w\ge0}$ directamente relacionados a funciones esféricas de tipo $\pi\in\hat \SO(n)$ de peso máximo $\mm_\pi=(1,\dots,1,0\dots,0)\in\CC^\ell$, con $p$ unos.

 Dado un entero no negativo $w$ y $\delta =0,1,$ podemos consideramos $\Phi_{w,\delta},$
  la función esférica irreducible de tipo $\pi$ asociada a  la
representación irreducible $\tau\in\hat \SO(n)$ de peso máximo de la forma $\mm_{ \tau}=(w+1,1,\dots,1,\delta,0,\dots,0)$ con $p-1$ unos.

Recordemos que, como $\pi$ tiene solo dos $\SO(2\ell-1)$-submódulos, podemos interpretar la  función matricial diagonal $\Phi_{w,\delta}(a(s))$, $s\in[0,1]$, como una matriz diagonal $2\times 2$.

Ahora consideramos la función vectorial
 $$P_{w,\delta}:[0,1]\to \CC^2$$
dada por la función diagonal  $P(y)=\Psi^{-1}\Phi_{w,\delta}(a(s)),$ con $\cos(s)=2y-1$.
Entonces,  definimos  la función matricial $$P_w= P_w(y),$$ cuya $\delta$-ésima columna  ($\delta =0,1$) está dada por el polinomio $P_{w,\delta}(y)$ con valores en $\CC^{2}$.

Consideramos la forma bilineal antisimétrica a valores matriciales definida para funciones matriciales $C^\infty$ $2\times 2$ en $[0,1]$ por
 \begin{equation*}
\langle P,Q \rangle_W =\int_0^1 Q^*(y)W(y)P(y)dy,
\end{equation*}
donde $$W(y)=\Psi^*(y)\frac{(n-1)!}{\Gamma(n/2)^2}  (y(1- y) )^{n/2-1}\left(\begin{matrix}
d_1&0\\
0&d_2
\end{matrix}\right) \Psi(y),$$ con $d_1$ y $d_2$ las dimensiones de los respectivos $M$-módulos de $\pi$, ver \eqref{W}.

Entonces enunciamos el siguiente teorema.
\begin{thm}
 Los polinomios matriciales $P_w$, $ w\geq0$, forman una sucesión de polinomios ortogonales con respecto a $ W$, además son autofunciones del operador simétrico diferencial $D$ en \eqref{DSn}.
Más aún,
 $$ D P_w = P_w \left(\begin{smallmatrix}\lambda(w,0) & 0\\0& \lambda(w,1)\end{smallmatrix}\right),$$
donde
 \begin{equation*}
\begin{split}
\lambda(w,\delta)=\begin{cases} -(w+1)(w+2\ell)+2\ell-p \quad&{\text si}\quad \delta=0,\\
              -(w+1)(w+2\ell)+p \quad&{\text si}\quad \delta=1.
              \end{cases}
\end{split}
\end{equation*}\end{thm}

\begin{proof}[\it Demostración]
 Por el Teorema \ref{ecpar} tenemos
que la $\delta$-ésima columna de $P_w$ es una autofunción del operador
$D$ con autovalor $\lambda(w,\delta)$, ver \eqref{lambda} y \eqref{DSn}. Por lo tanto tenemos $$DP_w=P_w \Lambda_w,$$ con $$\Lambda_w=\left(\begin{matrix}\lambda(w,0) & 0\\0& \lambda(w,1)\end{matrix}\right).$$

Por el Teorema \ref{columnsSn} sabemos que cada  columna de $ P_w$ es un polinomio de grado $w$ y, más aún, que $ P_w$ es un polinomio cuyo coeficiente director es una matriz diagonal no singular.

Dados $w$ y $w'$, enteros no negativos, usando Observación \ref{Ps} tenemos
\begin{align*}
\langle P_{w'},P_w \rangle _W&=
\int_{0}^1 P_w(y)^*W(y) P_{w'}(y) \, du =\sum_{\delta,\delta'=0}^1 \int_{0}^1 \Big( P_{w,\delta}(y)^*W(y)
P_{w',\delta'}(y) \, du \Big)\, E_{\delta,\delta'}\\
& =\sum_{\delta,\delta'=0}^1   \delta_{w,w'}\delta_{\delta,\delta'}   \Big(\int_{0}^1 P_{w,\delta}(y)^*W(y) P_{w',\delta'}(y)
\, du \Big)\, E_{\delta,\delta'}\\
& =\delta_{w,w'}\sum_{\delta=0}^1   \int_{0}^1 \Big(P_{w,\delta}(y)^*W(y) P_{w'\delta}(y) \, du,\Big) \, E_{\delta,\delta},
\end{align*}
lo cual prueba la ortogonalidad. Más aún, además nos muestra que $\langle P_w, P_{w} \rangle _W$ es una matriz diagonal. Además, haciendo unas pocas cuentas simples tenemos que
$$\langle D P_w,P_{w'}\rangle
=\delta_{w,w'}\langle P_w,P_{w'}\rangle \Lambda_w
=\delta_{w,w'} \Lambda_w^* \langle P_w,P_{w'}\rangle
=\langle P_w, DP_{w'}\rangle,$$
para cada $w,w'\in\NN_0$, ya que $\Lambda_w$ es real y diagonal.
Esto concluye la prueba del teorema.
\end{proof}



\cleardoublepage
\addcontentsline{toc}{chapter}{Bibliografía}
 \bibliographystyle{alpha} 
\bibliography{esfericas} 


\vfill \hfill \it Consérvate sano.
\end{document}